\documentclass{article}
\usepackage{latexsym,amsfonts,amsthm,amsmath,amscd,amssymb}
\usepackage[dvips]{graphicx}

\newtheorem{lemma}{Lemma}[section]
\newtheorem{thm}[lemma]{Theorem}
\newtheorem{rem}[lemma]{Remark}
\newtheorem{prop}[lemma]{Proposition}
\newtheorem{cor}[lemma]{Corollary}

\newtheorem{oss}[lemma]{Observation}
\newtheorem{example}[lemma]{Example}
\newtheorem{defn}[lemma]{Definition}

\newcommand{\cl}{C \kern -0.1em \ell}

\newcommand\matC{{\mathbb{C}}}
\newcommand\matZ{{\mathbb{Z}}}
\newcommand\matR{{\mathbb{R}}}
\newcommand\matN{{\mathbb{N}}}

\def\sistema#1{\left\{\begin{array}{l}#1\end{array}\right.}

\newcommand\calD{{\mathcal D}}

\begin{document}

\title{Diffeological Dirac operators and diffeological gluing}

\author{Ekaterina~{\textsc Pervova}}

\maketitle

\begin{abstract}
\noindent This manuscript attempts to present a way in which the classical construction of the Dirac operator can be carried over to the setting of diffeology. A more specific aim is to describe a procedure 
for gluing together two usual Dirac operators and to explain in what sense the result is again a Dirac operator. Since versions of cut-and-paste (surgery) operations have already appeared in the context of 
Atiyah-Singer theory, we specify that our gluing procedure is designed to lead to spaces that are not smooth manifolds in any ordinary sense, and since much attention has been paid in recent years to Dirac 
operators on spaces with singularities, we also specify that our approach is more of a piecewise-linear nature (although, hopefully, singular spaces in a more analytic sense will enter the picture sooner or 
later; but this work is not yet about them).

To define a diffeological Dirac operator, we describe the diffeological counterparts of all the main components (with the exception of the tangent bundle, of which we use a simplistic version; on the other hand, 
in the standard case it does become the usual tangent bundle): the diffeological analogue of a vector bundle, called a \emph{pseudo-bundle} here, endowed with \emph{pseudo-metric} playing a role of a 
Riemannian metric, the spaces of sections of such pseudo-bundles, the pseudo-bundles of Clifford algebras and those of exterior algebras (as specific instances of Clifford modules), and then the diffeological 
analogue of differential forms, which in particular provides a standard counterpart of the cotangent bundle, the pseudo-bundle $\Lambda^1(X)$. The dual pseudo-bundle $(\Lambda^1(X))^*$ is what is used 
in place of the tangent bundle. We then consider a notion of diffeological connection, a straightforward generalization of the standard notion, leading also to the notions of Levi-Civita connections and Clifford 
connections. A diffeological Dirac operator on a diffeological pseudo-bundle $E$ of Clifford modules is then the composition of a given Clifford action of $\cl(\Lambda^1(X),g^{\Lambda})$, where $g^{\Lambda}$ 
is a fixed pseudo-metric on $\Lambda^1(X)$, with a (Clifford) connection on $E$.

To give a more concrete angle to our treatment, we concentrate a lot on the interactions of these constructions with the operation of the so-called \emph{diffeological gluing}. On the level of underlying sets it is 
the same as what is usually called gluing, and the diffeology assigned to the resulting space is a standard quotient diffeology (although it is a rather weak one and so does probably risk being too weak for any 
potential applications). For each of the above-listed notion we outline how it behaves under the gluing procedure, perhaps under additional assumptions on the gluing map (or maps, as the case might be). 
These assumptions, although they progressively get more restrictive, do allow to treat spaces more general than smooth manifolds. Finally, we do attempt to give example whenever possible and for illustrative
purposes; according to one's personal taste, these may or may not appear artificial. The majority of the statements are cited without proofs, with references to other works (more restricted in scope) where such 
can be found.

\end{abstract}

\section*{Introduction}

Diffeology as a subject, introduced by Souriau in the 80's \cite{So1,So2}, belongs among various attempts made over the years to extend the usual setting of Differential Calculus and/or Differential
Geometry.\footnote{This depends on who you ask, and when.} Many of these attempts were particularly aimed to address the needs of mathematical physics, such as smooth structures \`a la
Sikorski or \`a la Fr\"olicher, the issue being that many objects naturally appearing in, say, noncommutative geometry, such as irrational tori, orbifolds, spaces of connections on principal
bundles in Yang-Mills theory... are not smooth manifolds and do not easily lend themselves to more standard ways of treatment. A rather comprehensive summary of other attempts made to develop
a common setting for such objects can be found in \cite{St}.

The few words just said about the general location of diffeology in the mathematical landscape do justify the attempt to look at the eventual extension of the Atiyah-Singer theory to its setting, and certain 
attempts in that direction have already been made (see \cite{magnot1}); this theory does find itself at the crossroads of all the same subjects. How successful, or useful, further attempts in this sense might be, 
is a different matter, but, as it is commonly said, you never know until you try.

\paragraph{A diffeological space and its diffeology} The notion of a \emph{diffeological space} is a simple and elegant extension of the notion of a smooth manifold. Such a spaceis a set $X$ endowed with 
a diffeological structure analogous to a smooth atlas of a smooth manifold. The charts of such an atlas are maps from domains of Euclidean spaces into $X$, but the difference is that these domains have 
varying dimensions (all possible finite ones, for the definition rigorously stated). On the other hand, two charts with intersecting ranges are related by a smoothsubstitution wherever appropriate (in analogy 
with smooth manifolds), and the ranges of all the charts cover $X$. Constant maps are formally included in the atlas, whose proper name is the \emph{diffeology} of $X$ (or its \emph{diffeological structure}). 
Later on we give precise definitions; for the moment it suffices to think of a diffeological space as an analogue of a smooth manifold, more general in that the charts do not have to have the same dimension.

\paragraph{The basic notion of the Dirac operator} The most basic definition of a Dirac operator is as follows.

\begin{defn}
A \textbf{Dirac operator} $D$ is the following composition of maps:
$$D=c\circ\nabla^E,\mbox{ where }\nabla^E:C^{\infty}(M,E)\to C^{\infty}(M,T^*M\otimes E)\mbox{ and }c:C^{\infty}(M,T^*M\otimes E)\to C^{\infty}(M,E).$$ The data that appear in this definition have
the following meaning:
\begin{itemize}
\item $M$ is a Riemannian manifold, whose metric is denoted by $g$;
\item the metric that $g$ induces on the cotangent bundle $T^*M$ is also denoted by $g$;
\item $E$ is a \emph{bundle of Clifford modules} over $M$. This means that each fibre $E_x$, with $x\in M$, of $E$ is a Clifford module over the corresponding Clifford algebra $\cl(T_x^*M,g_x)$
and the action of the latter depends smoothly on $x$. Finally, $E$ is assumed to be endowed with a Hermitian metric;
\item $\nabla^E$ is a connection on $E$, compatible with the above hermitian metric;
\item $c:C^{\infty}(M,T^*M\otimes E)\to C^{\infty}(M,E)$ is the map that is pointwise given by the Clifford action on $E$:
$$c(\psi\otimes s)(x)=c_x(\psi_x)(s_x)\in E_x.$$
\end{itemize}
\end{defn}

There are two other conditions that are usually imposed on $c$ and on $\nabla^E$, respectively. The condition on $c$ asks that the action of $\cl(T_x^*M,g_x)$ be unitary; whereas the connection 
$\nabla^E$ must be a \textbf{Clifford connection}, that is,
$$\nabla_X^E(c(\psi)s)=c(\nabla_X^{LC}\psi)(s)+c(\psi)\nabla_X^Es,$$ where $\nabla^{LC}$ is the Levi-Civita connection on the cotangent bundle.

\paragraph{The diffeological version} In principle, obtaining a diffeological version of the Dirac operator is an obvious matter: just replace each item appearing in the above list by its diffeological
counterpart, and consider the same composition $c\circ\nabla^E$. The trouble is that for some of these items a diffeological counterpart has not yet been defined, or it has been little studied. This regards 
even the most basic items, such as the notion of a Riemannian manifold; the main problem is that there is no standard construction of the tangent space for diffeological spaces (although there are several 
proposed versions, see the discussion in \cite{iglesiasBook}, the constructions in \cite{HeTangent}, \cite{HMtangent}, \cite{CWtangent}, and references therein).  

On the other hand, the center of the standard construction of the Dirac operator (as it is described above) is the cotangent bundle, and this does have a rather well-developed diffeological counterpart. 
Coupled with the construction of the so-called \emph{diffeological vector pseudo-bundle} (see\cite{iglFibre}, \cite{iglesiasBook}, \cite{vincent}, \cite{CWtangent}, \cite{pseudobundle}), that in diffeology takes 
place of a smooth vector bundle, it provides a reasonable (or at least one that is not unreasonable) starting point for the construction of a diffeological Dirac operator. For that, both are endowed with a 
diffeological counterpart of a Riemannian metric, called \emph{pseudo-metric} $g$ (see \cite{pseudometric-pseudobundle}), and we note right away that fibrewise this is not a scalar product (which in most
cases does not exist on a finite-dimensional diffeological vector space, see \cite{iglesiasBook}), but rather a smooth symmetric semi-definite positive bilinear form with the minimal possible degree of 
degeneracy.

The pseudo-bundle that takes place of the cotangent bundle (see \cite{iglesiasBook} for a recent and comprehensive exposition) is called the \emph{bundle of values of differential forms} $\Lambda^1(X)$ (we
 will always call it a pseudo-bundle, since more often than not it is not really a bundle, not being locally trivial). An element of the pseudo-bundle $\Lambda^1(X)$ is a collection of usual differential $1$-forms 
associated one to each diffeological chart. Such collection is required to be invariant under the usual smooth substitutions on the domains of charts. Whereas, whenever some kind of tangent vectors is needed, 
we use the dual pseudo-bundle $(\Lambda^1(X))^*$ of $\Lambda^1(X)$ (the duality is meant in the diffeological sense, which extends the standard one; it was introduced in \cite{vincent}). This choice is formal 
and is based on the existence of the obvious natural pairing between the elements of $(\Lambda^1(X))^*$ and $\Lambda^1(X)$, but there is no clear geometrical interpretation attached to it; however, if $X$ is 
a usual smooth manifold then $\Lambda^1(X)$ is the cotangent bundle $T^*X$ and $(\Lambda^1(X))^*$ is indeed $TX$.

Assuming, as we will do throughout, that $\Lambda^1(X)$ has finite-dimensional fibres, it can be endowed with a \emph{diffeological pseudo-metric} (or simply a \emph{pseudo-metric}). On each fibre, a 
pseudo-metric is a semi-definite positive symmetric bilinear form that satisfies the usual requirement of smoothness, both within a fibre and across the fibres. It is defined, more generally, on any 
finite-dimensional pseudo-bundle $V$, although sometimes it may not exist. Assuming that it does, we obtain the corresponding pseudo-bundle of Clifford algebras $\cl(V,g)$, where $V$ is the pseudo-bundle 
and $g$ is the chosen pseudo-metric on it, in a more or less straightforward manner; in particular, we obtain a Clifford algebra $\cl(\Lambda^1(X),g^{\Lambda})$, where $g^{\Lambda}$ is a pseudo-metric 
on $\Lambda^1(X)$ (once again, assuming it exists). Even more straightforward is the construction of the pseudo-bundle of exterior algebras $\bigwedge V$, which, as in the standard case, turns out to be a 
pseudo-bundle of Clifford modules over $\cl(V,g)$. There is also the more abstract definition of a pseudo-bundle of Clifford modules, that is analogous to the standard one and which we consider at some length. 

Subsequently, we consider the pseudo-bundle of differential $1$-forms $\Lambda^1(X)$ and its dual $(\Lambda^1(X))^*$, particularly to their interactions with the \emph{diffeological gluing} procedure. Our 
treatment of them is not comprehensive and is for the most part subject to significant restrictions. Still, we do what at the moment is doable (and seems reasonable to do). We then consider \emph{diffeological 
connections}, with the three varieties of them: a diffeological connection on an abstract pseudo-bundle $V$, a \emph{diffeological Levi-Civita connection} on $\Lambda^1(X)$, and a \emph{Clifford connection} 
on a pseudo-bundle of Clifford modules over $\cl(\Lambda^1(X),g^{\Lambda})$, where $g^{\Lambda}$ is some pseudo-metric on $\Lambda^1(X)$ (again, assuming it exists). The end result of the entire 
discussion, the notion of a Dirac operator, is then indeed fully analogous to the standard one, as outlined in the beginning of this introduction.

\paragraph{Definition of a diffeological Dirac operator} The initial data for our version of a diffeological Dirac operator thus consist of the following:
\begin{itemize}
\item a diffeological space $X$;
\item a pseudo-metric $g^{\Lambda}$ on $\Lambda^1(X)$;
\item a pseudo-bundle $E$ of Clifford modules over $\cl(\Lambda^1(X),g^{\Lambda})$, with Clifford action $c$ that determines the usual map $C^{\infty}(X,\Lambda^1(X)\otimes E)\to C^{\infty}(X,E)$;
\item a diffeological connection $\nabla^E:C^{\infty}(X,E)\to C^{\infty}(X,\Lambda^1(X)\otimes E)$.
\end{itemize}
The corresponding operator is then, as usual, $D=c\circ\nabla^E$. A lot of what we do consists in considering how all these notions interact with the \emph{diffeological gluing} construction. 

\paragraph{The diffeological gluing} The term \emph{diffeological gluing} (see \cite{pseudobundle}) stands a simple procedure that allows to obtain out of two diffeological spaces a third one; it then gets 
extended to the case of pseudo-bundles, spaces of smooth maps, etc. This notion mimics the usual topological gluing (a classic instance is a wedge of two smooth manifolds). The diffeology that the result 
is endowed with is usually weaker than other natural diffeologies on the same space; an interesting example due to Watts \cite{watts} shows, for instance, that the gluing diffeology on the union of the two 
coordinate axes in $\matR^2$ is strictly weaker that its diffeology as a subset of $\matR^2$. 

The spaces and pseudo-bundles obtained by gluing can also be a useful testing ground for diffeological constructions: they can be quite simple while being different from any space carrying a smooth 
structure in the usual sense, and the weakness of gluing diffeologies increases the likelihood of quickly revealing the impossibility of a such-and-such construction in any one of all potential diffeologies
(of course, on the other hand, it may give rise to false hopes of something being true more often when it actually is). It is also curious to observe that gluing \emph{to} a one-point space provides a natural 
setting for considering the usual $\delta$-functions as plots (so in particular as smooth maps).

\paragraph{The structure} The text naturally splits into three parts. In the first of them (Sections 1-3) we collect the introductory material, such as some basic facts regarding Dirac operators (Section 1), 
diffeological spaces and particularly vector spaces (Section 2), and diffeological differential 1-forms as they have been treated elsewhere (Section 3). The second part (Sections 4-7) deals with pseudo-bundles 
and related notions. The pseudo-bundles themselves, and the diffeological gluing procedure, are discussed in Section 4. We then consider pseudo-metrics on them (Section 5), spaces of smooth sections 
(Section 6), and finally the pseudo-bundles of Clifford algebras and those of the exterior algebras associated to a given pseudo-bundle carrying a pseudo-metric (Section 7; not all of this material is necessary). 
The third part (Sections 8-13) treats the rest: differential 1-forms and particularly their behavior under gluing, with a complete answer being reached only under the assumption of the gluing map being a 
diffeological diffeomorphism, the assumption carried from that point onwards (Section 8), the dual pseudo-bundle $(\Lambda^1(X))^*$ (Section 9), diffeological connections (Section 10), the analogue of 
Levi-Civita connections as connections on $\Lambda^1(X)$ endowed with a pseudo-metric (Section 11). In Section 12 we say what we can about Clifford connections, and in the concluding Section 13 we 
wrap everything together, stating and then illustrating via examples the resulting notion of a diffeological Dirac operator.

\paragraph{What is \emph{not} in here} Here is a very brief and incomplete list of things that we do not even attempt to treat in the present work.
\begin{itemize}
\item We have already mentioned that the gluing diffeology, a rather weak one on its own, is a precursor to stronger and therefore more useful diffeologies. We do not discuss any such extension;
\item we give no applications. In particular, all our examples are for illustrative purposes only and might appear artificial to some;
\item we say almost nothing about the index. This is left for future work;
\item there exists an established notion of the diffeological de Rham cohomology, but we do not really discuss it. Neither do we consider the potential de Rham operator;
\item a great number of other things.
\end{itemize}

\paragraph{Acknowledgments}\footnote{\emph{``Cercare e saper riconoscere chi e cosa, in mezzo all'inferno, non \`e inferno, e farlo durare, e dargli spazio''} (I. Calvino)} This paper is meant to be a collection,
in a single place, and a summary, of other projects carried out separately (all united by the same theme, however). As such, it came out too lengthy, and so its destiny is uncertain.\footnote{``[...] but the delight 
and pride of Aul\"e is in the deed of making, and in the thing made, and neither in possession, nor in his own mastery; wherefore he gives and hoards not, and is free from care, passing ever on to some new 
work.'' J. Tolkien, in ``Silmarillion''.} It also took forever to complete; yet, whatever becomes of it, it has been, and still is, a satisfying process in a way that goes much beyond the satisfaction that one might draw
from having just one more item to add to one's publication list. And, if nothing more, it led to various other papers being written along the way; they would not have come into being otherwise. These are 
among the reasons why completing this paper is of particular significance to me; and its existence is in large part due to contribution from many other people, first of all, Prof. Riccardo Zucchi (who, without 
knowing it, gave me its idea) and Prof. Paolo Piazza (I first learnt the Atiyah-Singer theory from his notes on the subject). Also, quite a few anonymous referees made very useful comments on the papers 
originating from this project, for which I am grateful to all of them.

\section{The Dirac operator}\label{dirac:operator:defn:sect}

In this section (which is rigorously for a non-specialist) we recall some of the main notions regarding the Dirac operator, mostly following the exposition in \cite{heat-kernel}.

\subsection{Clifford algebras and Clifford modules}

These are the most basic constructions that come into play when defining the Dirac operator.

\paragraph{Clifford algebras} There is more than one way to define a Clifford algebra; a more constructive one is as follows.

\begin{defn}
Let $V$ be a vector space equipped with a symmetric bilinear form $q(\, ,\,)$. The \textbf{Clifford algebra} $\cl(V,q)$ associated to $V$ and $q$ is the quotient of $T(V)/I(V)$ of the tensor algebra
$T(V)=\sum_rV^{\otimes r}$ by the ideal $I(V)\subset T(V)$ generated by all the elements of the form $v\otimes w+w\otimes v+2q(v,w)$, where $v,w\in V$.
\end{defn}

The natural projection $\pi_{\cl}:T(V)\to\cl(V,q)$ is a universal map in the following sense: if $\varphi:\cl(V,q)\to A$ is a map from $V$ to an algebra $A$ that satisfies $\varphi(v)^2=-4(v,v)1$ then there
is a unique algebra homomorphism $t:\cl(V,q)\to A$ such that $\varphi=t\circ\pi_{\cl}$. An easy example of a Clifford algebra is the exterior algebra of a given vector space, which corresponds to
the bilinear form being identically zero.

\paragraph{Clifford modules} Let $V$ be a vector space endowed with a symmetric bilinear form $q$.

\begin{defn}
A \textbf{Clifford module} is a vector space $E$ endowed with an action of the algebra $\cl(V,q)$, that is, a unital algebra homomorphism $c:\cl(V,q)\to\mbox{End}(E)$.
\end{defn}

If the space $E$ is Euclidean, \emph{i.e.}, if it is endowed with a scalar product $g$, then there is the notion of a \textbf{unitary action}, as a homomorphism $c:\cl(V,q)\to\mbox{End}(E)$ such that
$c(v)$ is an orthogonal transformation for each $v\in V$, \emph{i.e.},
$$g(c(v_1),c(v_2))=g(v_1,v_2).$$ More precisely, let $a\mapsto a^*$ be the anti-automorphism of $T(V)$ such that $v$ is sent to $-v$; this obviously induces an
automorphism $v\mapsto v^*$ of $\cl(V,q)$.

\begin{defn}
A Clifford module $E$ over $\cl(V,q)$ endowed with a scalar product is \textbf{self-adjoint} if $c(v^*)=c(v)^*$. This is equivalent to the operators $c(v)$ with $v\in V$ being skew-adjoint.
\end{defn}

\paragraph{The exterior algebra as a Clifford module} The exterior algebra $\bigwedge V$ of $V$ is a standard example of a Clifford module over $\cl(V,q)$; let us describe the action of the latter
on the former. Let $\varepsilon(v)\alpha$ denote the exterior product of $v$ with $\alpha$, and let $i(v)\alpha$ stand for the contraction with the covector $q(v,\cdot)\in V^*$:
$$i(v)(w_1\wedge\ldots\wedge w_l)=\sum_{j=1}^l(-1)^{j+1}w_1\wedge\ldots\wedge q(v,w_j)\wedge\ldots\wedge w_l.$$ The Clifford action on $\bigwedge V$ is then defined by the formula:
$$c(v)\alpha=\varepsilon(v)\alpha-i(v)\alpha,$$ which defines a homomorphism $V\to\mbox{End}(\bigwedge V)$; it is extended to a homomorphism defined on $\cl(V,q)$ by linearity
(obviously) and with the tensor product being substituted by the composition.

To see that it is indeed a Clifford module action, it is sufficient to consider the identity $\varepsilon(v)i(w)+i(w)\varepsilon(v)=q(v,w)$ (see \cite{heat-kernel}, p. 101). If $q$ is positive definite (that is,
if it is a scalar product), the operator $i(v)$ is the adjoint of $\varepsilon(v)$, so the Clifford module $\bigwedge V$ is also self-adjoint.

\paragraph{Isomorphism of the graded algebras $\cl(V,q)$ and $\bigwedge V$} Both of these algebras have a natural grading (see below), and there is a standard isomorphism between
them that respects the grading.

\begin{defn} \emph{(\cite{heat-kernel}, Definition 3.4)}
The \textbf{symbol map} $\sigma:\cl(V,q)\to\bigwedge V$ is defined in terms of the Clifford module structure on $\bigwedge V$ by
$$\sigma(v)=c(v)1\in\bigwedge V,$$ where $1\in\bigwedge^0V$ is the unit of the exterior algebra $\bigwedge V$.
\end{defn}

Suppose that $q$ is a scalar product; then the symbol map has an inverse, called the \textbf{quantization map}, which is described as follows. Let $\{e_i\}$ be an orthonormal basis of $V$, and
let $c_i$ be the element of $\cl(V,q)$ corresponding to $e_i$. The quantization map $\mbox{\textbf{c}}:\bigwedge V\to\cl(V,q)$ is given by the formula
$$\mbox{\textbf{c}}(e_{i_1}\wedge\ldots\wedge e_{i_j})=c_{i_1}\ldots c_{i_j}.$$ This preserves the natural $\matZ_2$-grading of the two modules (also see below).

\paragraph{Grading and filtration on $\cl(V,q)$} As has just been alluded to, every Clifford algebra $\cl(V,q)$ carries the following $\matZ_2$-grading:
$$\cl(V,q)=\cl(V,q)^0\oplus\cl(V,q)^1,$$ where $\cl(V,q)^0$ is the subspace generated by the products of an even number of elements of $V$, while $\cl(V,q)^1$ is the subspace
generated by the products of an odd number of elements of $V$; this is well-defined because $I(V)$ is generated by elements of even degree in $T(V)$.

Besides, $\cl(V,q)$ inherits from $T(V)$ its filtration $T(V)=\bigoplus_k(\oplus_{r=0}^kV^{\otimes r})$, via the natural projection. Therefore
$$\cl(V,q)=\oplus_k \cl^k(V,q),\mbox{ where }\cl^k(V,q)=\{v\in\cl(V,q)\,|\,\exists u\in\oplus_{r=0}^kV^{\otimes r}\mbox{ such that }[u]=v\}.$$ The use of natural projections allows
also to define a surjective algebra homomorphism $$V^{\otimes k}\to\cl^k(V,q)/\cl^{k-1}(V,q).$$ Considering the kernel of the latter, one sees that
$$\cl^k(V,q)/\cl^{k-1}(V,q)\cong\bigwedge^kV.$$ This implies that the graded algebra associated to the just-described filtration on $\cl(V,q)$, namely, the algebra
$\oplus_k\cl^k(V,q)/\cl^{k-1}(V,q)$, is isomorphic to the exterior algebra $\bigwedge V$ (in particular, $\dim(\cl(V,q))=2^{\dim V}$). The following statement provides a summary of what has
just been said.

\begin{prop} \emph{(\cite{heat-kernel}, Proposition 3.6)}
The graded algebra $\cl(V,q)$ is naturally isomorphic to the exterior algebra $\bigwedge V$, the isomorphism being given by sending $v_1\wedge\ldots\wedge v_i\in\bigwedge^i V$
to $\sigma_i(v_1\ldots v_i)\in\cl^i(V,q)$.

The symbol map $\sigma$ extends the symbol map $\sigma_i:\cl^i(V,q)\to\cl(V,q)^i\cong\bigwedge^iV$, in the sense that if $a\in\cl^i(V,q)$ then $\sigma(a)_{[i]}=\sigma_i(a)$.
The filtration $\cl^i(V,q)$ may be written
$$\cl^i(V,q)=\oplus_{k=0}^i\cl_k(V),\mbox{ where }\cl_k(V)=\mbox{\textbf{c}}(\bigwedge^kV).$$
\end{prop}

Using the symbol map $\sigma$, the Clifford algebra $\cl(V,q)$ may be identified with the exterior algebra $\bigwedge^* V$ with a twisted, or quantized, multiplication $\alpha\cdot_q\beta$.

\subsection{Clifford connections}

As has already been said, the type of the connection that is typically used in constructing a Dirac operator is a Clifford connection. It is defined as follows.

\begin{defn}
Let $E$ be a Clifford module, and let $\nabla^E$ be a connection on it. We say that $\nabla^E$ is a \textbf{Clifford connection} if for any $a\in C^{\infty}(M,Cl(M))$ and $X\in C^{\infty}(M,TM)$ we have
$$[\nabla_X^E,c(a)]=c(\nabla_X(a)),$$ where $\nabla_X$ is the Levi-Civita covariant derivative extended to the bundle $\cl(M)$.
\end{defn}

When associated to a Clifford connection, the Dirac operator described in the Introduction may be written, in local coordinates, as:
$$D=\sum_ic(dx^i)\nabla_{\partial_i}^E.$$ There always exists a Clifford connection on any bundle of unitary Clifford modules. In particular, recall from the previous section
that if $(M,g)$ is a Riemannian manifold then $\bigwedge^*M$ is naturally a bundle of Clifford modules (via the action $c$ of thebundle of Clifford algebras $\cl(T^*M,g^*)$); the connection
$\nabla^{\bigwedge^*M}$ induced on $\bigwedge^*M$ by the Levi-Civita connection on $M$ is a Clifford connection.

\begin{example}
The operator
$$c\circ\nabla^{\bigwedge^*M}:C^{\infty}(M,\bigwedge^*M)\to C^{\infty}(M,\bigwedge^*M)$$ is a Dirac operator. It is called the Gauss-Bonnet operator, or the Euler operator.
\end{example}

\subsection{What does it mean for a given operator to be a Dirac operator?}

Recall that so far we have only given a constructive definition of a Dirac operator. There is however an abstract, and more general definition of this notion, which is as follows:

\begin{defn}\label{dirac:defn} \emph{(\cite{heat-kernel}, Definition 3.36)}
A \textbf{Dirac operator} $D$ on a $\matZ_2$-graded vector bundle $E$ is a first-order differential operator of odd parity on $E$,
$$D:C^{\infty}(M,E^{\pm})\to C^{\infty}(M,E^{\mp}),$$ such that $D^2$ is a generalized Laplacian.
\end{defn}

This way of defining the Dirac operator shows that the relation of such with (bundles of) Clifford modules goes in two directions. To be specific, in the first, constructive, definition, a Dirac
operator is associated to a bundle of Clifford modules. On the other hand, if we are given a Dirac operator, in the sense of the definition just cited, on a $\matZ_2$-graded vector bundle $E$,
then$E$ inherits a natural Clifford module structure, over the bundle of Clifford algebras $\cl(T^*M,g^*)$, whose Clifford action (which, again, it suffices to define on $T^*M$) is described by the
following statement.

\begin{prop}\label{dirac:clifford:prop} \emph{(\cite{heat-kernel}, Proposition 3.38)}
The action of $T^*M$ on $E$ defined by
$$[D,f]=c(df),\mbox{ where }f\mbox{ is a smooth function on }M,$$ is a Clifford action, which is self-adjoint with respect to a metric on $E$ if the operator $D$ is symmetric. Conversely, any
differential operator $D$ such that $[D,f]=c(df)$ for all $f\in C^{\infty}(M)$ is a Dirac operator.
\end{prop}

It can be easily observed that a Dirac operator $D=\sum_ic(dx^i)\nabla_{\partial_i}^E$ (meaning the one obtained by the explicit construction outlined in the Introduction) defined with
respect to a Clifford connection is indeed a Dirac operator in the sense of Definition \ref{dirac:defn}, since it is easy to calculate that $[D,f]=c(df)$.

\subsection{The index space of Dirac operators}

Let $E$ be a $\matZ_2$-graded vector bundle on a compact Riemannian manifold $M$, and let $D:C^{\infty}(M,E)\to C^{\infty}(M,E)$ be a self-adjoint Dirac operator. We denote by $D^{\pm}$
the restrictionsof $D$ to $C^{\infty}(M,E^{\pm})$, that is,
$$D=\left(
\begin{array}{cc}
0 & D^- \\
D^+ & 0
\end{array}
\right),$$ where $D^-=(D^+)^*$. The \emph{index} of $D$ is defined as follows. First, if $E=E^+\oplus E^-$ is a finite-dimensional superspace, define its dimension to be
$$\dim(E)=\dim(E^+)-\dim(E^-);$$ note that the superspace $\mbox{ker}(D)$ is finite-dimensional.

\begin{defn}
The \textbf{index space} of the self-adjoint Dirac operator $D$ is its kernel
$$\mbox{ker}(D)=\mbox{ker}(D^+)\oplus\mbox{ker}(D^-).$$ The \textbf{index} of $D$ is the dimension of the superspace $\mbox{ker}(D)$:
$$\mbox{ind}(D)=\dim(\mbox{ker}(D^+))-\dim(\mbox{ker}(D^-)).$$
\end{defn}

One remarkable property of the index is that it does not really depend on the whole of the operator, but rather on its domain of definition, that is:

\begin{thm} \emph{(\cite{heat-kernel}, Theorem 3.51)}
The index is an invariant of the manifold $M$ and of the Clifford module $E$.
\end{thm}

In other words, if $D^z$ is a one-parameter family of operators on $C^{\infty}(M,E)$ which are Dirac operators with respect to a family of metrics $g^z$ on $M$ and Clifford actions $c^z$ of
$\cl(M,g^z)$ on $E$, then the index of $D^z$ is independent of $z$. This follows from the famous McKean-Singer formula, that expresses the index of $D$ as the supertrace of $e^{-tD^2}$,
equal to the integral over $M$ of the supertrace of the heat kernel of the Laplacian $D^2$:

\begin{thm} \emph{(McKean-Singer)}
Let $\langle x|e^{-tD^2}|y\rangle$ be the heat kernel of the operator $D^2$. Then for any $t>0$
$$\mbox{ind}(D)=\mbox{Str}(e^{-tD^2})=\int_M\mbox{Str}(\langle x|e^{-tD^2}|x\rangle)dx.$$
\end{thm}

The application of this to a family $D^z$ is then straightforward (see \cite{heat-kernel}, the proof of Theorem 3.50): $\frac{d}{dz}\mbox{ind}(D^z)=\frac{d}{dz}\mbox{Str}(e^{-t(D^z)^2})=
-t\mbox{Str}([\frac{dD^z}{dz},D^ze^{-t(D^z)^2}])$. Then that the latter supertrace is equal to $0$, is a general fact that holds for any two differential operators on $E$ such that the second of them
has a smooth kernel (see \cite{heat-kernel}, Lemma 3.49).

\subsection{Classical examples of the Dirac operator}

For completeness, we now list some classical linear first-order differential operators of differential geometry, which turn out to be Dirac operators.

\paragraph{The De Rham operator} Let $M$ be a Riemannian manifold, let $\mathcal{A}^i(M)$ be its $i$th De Rham cohomology group, and let $d_i:\mathcal{A}^i(M)\to\mathcal{A}^{i+1}(M)$ be
the usual exterior derivative operator. The bundle $\bigwedge(T^*M)$ is a Clifford module via a certain standard action, that we recall in the next Section; the Levi-Civita connection is compatible
with this action. Then the following is well-known.

\begin{prop} (\cite{heat-kernel}, Proposition 3.53)
The Dirac operator associated to the Clifford module $\bigwedge(T^*M)$ and its Levi-Civita connection is the operator $d+d^*$, where
$$d^*:\mathcal{A}^{\bullet}(M)\to\mathcal{A}^{\bullet-1}(M)$$ is the adjoint of the exterior differential $d$.
\end{prop}

The square $dd^*+d^*d$ of the operator $d+d^*$ is the so-called Laplace-Beltrami operator and is a generalized Laplacian; for reasons of general interest, we mention that the following is true.

\begin{prop}
The kernel of the Laplace-Beltrami operator $dd^*+d^*d$ on $\mathcal{A}^i(M)$ is naturally isomorphic to the De Rham cohomology space $H^i(M,\matR)$. The index of the Dirac operator $d+d^*$
on $\mathcal{A}(M)$ is equal to the Euler number of the manifold $M$.
\end{prop}

\paragraph{The signature operator} This operator is constructed as in the previous example (thus the Dirac operator is the same), but the definition of the $\matZ_2$-grading on the Clifford module
$\bigwedge(T^*M)$ is changed; indeed, this $\matZ_2$-grading comes from the Hodge star operator.

\begin{defn}
If $V$ is an oriented Euclidean vector space with complexification $V_{\matC}$, the \textbf{Hodge star} operator $\star$ on $\bigwedge V_{\matC}$ equals to the action of the chirality element
$\Gamma\in\cl(V)\otimes\matC$.\footnote{This definition, that comes from \cite{heat-kernel}, Definition 3.57, differs from the usual one by a power of $i$ so that $\star^2=1$.}
\end{defn}

Applying this operator to each fibre of the complexified exterior bundle $\bigwedge_{\matC}(T^*M)=\bigwedge(T^*M)\otimes_{\matR}\matC$ of an oriented $n$-dimensional Riemannian manifold $M$,
we obtain the Hodge star operator:
$$\star:\Lambda_{\matC}^kT^*M\to\Lambda_{\matC}^{n-k}T^*M.$$ Since $\star^2=1$ and $\star$ anticommutes with the Clifford action $c$, it can be used to define another $\matZ_2$-grading
on the exterior bundle $\bigwedge(T^*M)\otimes\matC$: the differential forms satisfying $\star\alpha=+\alpha$ are called \textbf{self-dual}, and those satisfying $\star\alpha=-\alpha$ are called
\textbf{anti-self-dual}. Note that, in particular, the following is then true:

\begin{prop}
The index of the signature operator $d+d^*$ is equal to the signature $\sigma(M)$ of the manifold $M$.
\end{prop}

\paragraph{The Dirac operator on a spin manifold} Let $\mathcal{L}$ be the spinor bundle over an even-dimensional manifold $M$. The most basic example of the Dirac operator is the Dirac operator
associated to the Levi-Civita connection $\nabla^{\mathcal{L}}$ on $\mathcal{L}$; this is usually called \textbf{the} Dirac operator. More generally, one can consider the Dirac operator
$D_{W\otimes\mathcal{L}}$ on a twisted spinor bundle $W\otimes\mathcal{L}$ with respect to a Clifford connection of the form
$\nabla^{W\otimes\mathcal{L}}=\nabla^W\otimes 1+1\otimes\nabla^{\mathcal{L}}$. This more general operator, together with the so-called Lichnerowicz formula (see, for instance, \cite{heat-kernel},
Theorem 3.52), allows to obtain the following result, that applies to the above (non-twisted) Dirac operator:

\begin{prop} \emph{(Lichnerowicz)}
If $M$ is a compact spin manifold whose scalar curvature is non-negative, and strictly positive at at least one point, then the kernel of the Dirac operator on the spinor bundle $\mathcal{L}$
vanishes; in particular, its index is zero.
\end{prop}

\paragraph{The $\bar{\partial}$-operator on a K\"{a}hler manifold} We limit ourselves to just a few brief remarks about this operator; further details can be found in \cite{heat-kernel}, Section 3.6.
Let $M$ be a K\"{a}hler manifold, and let $W$ be a holomorphic vector bundle with a Hermitian metric over $M$. Then:

\begin{prop} \emph{(\cite{heat-kernel}, Proposition 3.67)}
The tensor product of the Levi-Civita connection with the canonical connection of $W$ is a Clifford connection on the Clifford module
$\bigwedge(T^{0,1}M)^*\otimes W$, with associated Dirac operator $\sqrt2(\bar{\partial}+\bar{\partial}^*)$.
\end{prop}

There is a relation between this operator and the \emph{Dolbeaut cohomology} of the holomorphic bundle $W$, which is somewhat similar to that between the De Rham operator and the
De Rham cohomology. More precisely, the following is true.

\begin{thm} \emph{(Hodge)}
The kernel of the Dirac operator $\sqrt2(\bar{\partial}+\bar{\partial}^*)$ on the Clifford module $\bigwedge(T^{0,1}M)^*\otimes W$ is naturally isomorphic to the sheaf cohomology space
$H^{\bullet}(M,\mathcal{O}(W))$.
\end{thm}

\begin{cor}
The index of $\bar{\partial}+\bar{\partial}^*$ on the Clifford module $\bigwedge(T^{0,1}M)^*\otimes W$ is equal to the Euler number of the holomorphic vector bundle $W$:
$$\mbox{ind}(\bar{\partial}+\bar{\partial}^*)=\mbox{Eul}(W)=\sum(-1)^i\dim(H^i(M,\mathcal{O}(W))).$$
\end{cor}

\section{Diffeology: the main notions}\label{diffeology:defn:sect}

In this section we review (some of) the main notions of diffeology, starting from what a diffeological space is, and ending with the concept of a diffeological bundle (not necessarily a vector bundle)
and that of a connection on it, as this notion appears in \cite{iglesiasBook}. (The other parts of the section are also based on the same source).

\subsection{Diffeological spaces and smooth maps}

We start by giving the precise definitions of these basic objects.

\paragraph{The concept} We first recalling the notion of a diffeological space and that of a smooth map between such spaces.

\begin{defn} \emph{(\cite{So2})}
A \textbf{diffeological space} is a pair $(X,\calD_X)$ where $X$ is a set and $\calD_X$ is a specified collection of maps $U\to X$ (called \textbf{plots}) for each open set $U$ in $\matR^n$ and for
each $n\in\matN$, such that for all open subsets $U\subseteq\matR^n$ and $V\subseteq\matR^m$ the following three conditions are satisfied:
\begin{enumerate}
\item (The covering condition) Every constant map $U\to X$ is a plot;
\item (The smooth compatibility condition) If $U\to X$ is a plot and $V\to U$ is a smooth map (in the usual sense) then the composition $V\to U\to X$ is also a plot;
\item (The sheaf condition) If $U=\cup_iU_i$ is an open cover and $U\to X$ is a set map such that each restriction $U_i\to X$ is a plot then the entire map $U\to X$ is a plot as well.
\end{enumerate}
\end{defn}
Usually, instead of $(X,\calD_X)$ one writes simply $X$ to denote a diffeological space. A standard example of a diffeological space is a smooth manifold $M$, endowed with the diffeology consisting
of all smooth maps into $M$; this diffeology is called the \textbf{standard diffeology} of $M$.

Let now $X$ and $Y$ be two diffeological spaces, and let $f:X\to Y$ be a set map. We say that $f$ is \textbf{smooth} if for every plot $p:U\to X$ of $X$ the composition $f\circ p$ is a plot  of $Y$.
The typical notation $C^{\infty}(X,Y)$ is used to denote the set of all smooth maps from $X$ to $Y$.

\paragraph{The D-topology} There is a canonical topology underlying every diffeological structure on a given set, the so-called D-topology; this notion appeared in \cite{iglFibre}.\footnote{A frequent
restriction on the choice of a diffeology on a given topological space is that the corresponding D-topology coincide with the given one.} It is defined by imposing that a subset $A\subset X$ of a
diffeological space $X$ is open for the D-topology (and is said to be \textbf{D-open}) if and only if $p^{-1}(A)$ is open for every plot $p$ of $X$. In case of a smooth manifold with the standard
diffeology, the D-topology is the same as the usual topology on the manifold; this is frequently the case also for non-standard diffeologies. This is due to the fact that, as established in
\cite{CSW_Dtopology} (Theorem 3.7), the D-topology is completely determined smooth curves, in the sense that a subset $A$ of $X$ is D-open if and only if $p^{-1}(A)$ is open
for every $p\in C^{\infty}(\matR,X)$.

\paragraph{Comparing diffeologies} Given a set $X$, the set of all possible diffeologies on $X$ is partially ordered by inclusion (with respect to which it forms a complete lattice). More precisely, a
diffeology $\calD$ on $X$ is said to be \textbf{finer} than another diffeology $\calD'$ if $\calD\subset\calD'$ (whereas $\calD'$ is said to be \textbf{coarser} than $\calD$). Among all diffeologies,
there is the finest one (the natural \textbf{discrete diffeology}, which consists of all locally constant maps $U\to X$) and the coarsest one (which consists of \emph{all} possible maps $U\to X$,
for all $U\subseteq\matR^n$ and for all $n\in\matN$ and is called the \textbf{coarse diffeology}). Furthermore, due to the above-mentioned structure of a lattice on the set of all diffeologies on a
given $X$, it is frequently possible to claim the existence of the finest, or the coarsest, diffeology possessing a certain desirable property. A number of definitions are of this type.

\paragraph{The generated diffeology} A lot of specific examples are constructed via this simple notion. Given a set $X$ and a set of maps $\mathcal{A}=\{U_i\to X\}$ into $X$, where each $U_i$ is
a domain in some $\matR^{m_i}$, there exists the finest diffeology on $X$ that contains $\mathcal{A}$. This diffeology is called the \textbf{diffeology generated by $\mathcal{A}$}; its plots are
precisely the maps that locally are either constant or filter through a map in $\mathcal{A}$.

\paragraph{Pushforwards and pullbacks of diffeologies} Let $X$ be a diffeological space and let $X'$ be any set. Given an arbitrary map $f:X\to X'$, there exists a finest diffeology on $X'$ such that
$f$ is smooth; this diffeology is called the \textbf{pushforward of the diffeology of $X$ by the map $f$}. Its plots are precisely the compositions of plots of $X$ with the map $f$. If, \emph{vice versa},
we have a map $f:X'\to X$, there is the coarsest diffeology on $X'$ such that $f$ is smooth; it is called the \textbf{pullback of the diffeology of $X$ by the map $f$}. A map $p:U\to X'$ is a plot for
this pullback diffeology if and only if $f\circ p$ is a plot of $X$.

\paragraph{The quotient diffeology} Any quotient of a diffeological space is itself a diffeological space for a canonical choice of the diffeology. Namely, if $X$ is a diffeological space and $\sim$ is
an equivalence relation on $X$, the \textbf{quotient diffeology on $X/\sim$} is the pushforward of the diffeology of $X$ by the natural projection $X\to X/\sim$.

\paragraph{The subset diffeology} Let $X$ be a diffeological space, and let $Y\subseteq X$ be its subset. The \textbf{subset diffeology} on $Y$ is the coarsest diffeology on $Y$ making the inclusion
map $Y\hookrightarrow X$ smooth; it consists of all maps $U\to Y$ such that $U\to Y\hookrightarrow X$ is a plot of $X$ (less formally, we can say that the subset diffeology consists of all plots of $X$
whose image is contained in $Y\subset X$).

\paragraph{Disjoint unions and products of diffeological spaces} Let $\{X_i\}_{i\in I}$ be a collection of diffeological spaces. The \textbf{disjoint union} of $\{X_i\}_{i\in I}$ is the usual disjoint
union $\coprod_{i\in I}X_i=\{(i,x)\,|\,i\in I\mbox{ and }x\in X_i\}$, endowed the so-called \textbf{disjoint union}, or \textbf{sum diffeology} that is the \emph{finest} diffeology such that each natural
injection $X_i\to\coprod_{i\in I}X_i$ is smooth. Locally, every plot of this diffeology is a plot of one of the components of the disjoint union. The \textbf{product diffeology} $\calD$ on the product
$\prod_{i\in I}X_i$ is the \emph{coarsest} diffeology such that for each index $i\in I$ the natural projection $\pi_i:\prod_{i\in I}X_i\to X_i$ is smooth. If the collection of the spaces $X_i$ is finite, then
any plot of the product diffeology on $X_1\times\ldots\times X_n$ is an $n$-tuple of form $(p_1,\ldots,p_n)$, where each $p_i$ is a plot of $X_i$.

\paragraph{The functional diffeology} Let $X$ and $Y$ be two diffeological spaces. The \textbf{functional diffeology} on the set $C^{\infty}(X,Y)$ of all smooth maps from $X$ to $Y$ is the coarsest
diffeology for which the following map, called the \textbf{evaluation map} is smooth:
$$\mbox{\textsc{ev}}:C^{\infty}(X,Y)\times X\to Y\mbox{ and }\mbox{\textsc{ev}}(f,x)=f(x).$$ Occasionally, one speaks of a functional diffeology, which is any diffeology such that $\mbox{\textsc{ev}}$
is smooth.

These are the main notions of diffeology that we will use; occasionally some other term will be needed, at which point we will recall it as we go along.

\subsection{Diffeological vector spaces}

The concept of a diffeological vector space is the obvious one: it is a set $X$ that is both a diffeological space and a vector space such that the operations of addition and scalar multiplication are
smooth (with respect to the diffeology).

\paragraph{The definition} Let $V$ be a vector space (over real numbers and in most cases finite-dimensional, although the definition that follows is more general). A \textbf{vector space
diffeology} on $V$ is any diffeology such that the addition and the scalar multiplication are smooth, that is,
$$[(u,v)\mapsto u+v]\in C^{\infty}(V\times V,V)\mbox{ and }[(\lambda,v)\mapsto\lambda v]\in C^{\infty}(\matR\times V,V),$$ where $V\times V$ and $\matR\times V$ are equipped with the product
diffeology. Equipped with such a diffeology, $V$ is called a \textbf{diffeological vector space}.

The following observation could be useful to make the distinction between $V$ diffeological vector space, and $V$ diffeological space \emph{proper}. Since the constant maps are plots for any
diffeology and the scalar multiplication is smooth with respect to the standard diffeology of $\matR$, any vector space diffeology on a given $V$ includes maps of form $f(x)v$ for any fixed
$v\in V$ and for any smooth map $f:\matR\to\matR$; furthermore, since the addition is smooth, any vector space diffeology includes all finite sums of such maps. This immediately implies, for
instance, that any vector space diffeology on $\matR^n$ includes all usual smooth maps (since they write as $\sum_{i=1}^nf_i(x)e_i$).\footnote{This is not the case for a non-vector space diffeology
of $\matR^n$; the simplest example is the discrete diffeology, which consists of constant maps only. This is not a vector space diffeology, since the scalar multiplication is not smooth. A more intricate
example is that of the so-called \emph{wire diffeology}, one generated by the set $C^{\infty}(\matR,\matR^n)$. For this diffeology, the scalar multiplication is smooth, but the addition is not.}

Given two diffeological vector spaces $V$ and $W$, the space of all smooth linear maps between them is denoted by $L^{\infty}(V,W)=L(V,W)\cap C^{\infty}(V,W)$; it is endowed with the
functional diffeology, with respect to which it becomes a diffeological vector space. A \textbf{subspace} of a diffeological vector space $V$ is a vector subspace of $V$ endowed with the subset
diffeology; it is, again, a diffeological vector space on its own. Finally, if $V$ is a diffeological vector space and $W\leqslant V$ is a subspace of it then the quotient $V/W$ is a diffeological
vector space with respect to the quotient diffeology.

\paragraph{The direct sum of diffeological vector spaces} Let $\{V_i\}_{i\in I}$ be a family of diffeological vector spaces. Consider the usual direct sum $V=\oplus_{i\in I}V_i$ of this family;
then $V$, equipped with the product diffeology, is a diffeological vector space.

\paragraph{Euclidean structure on diffeological vector spaces} A diffeological vector space $V$ is \textbf{Euclidean} if it is endowed with a scalar product that is smooth with respect to the
diffeology of $V$ and the standard diffeology of $\matR$; that is, if there is a fixed map $\langle\cdot,\cdot\rangle:V\times V\to\matR$ that has the usual properties of bilinearity,
symmetricity, and definite-positiveness and that is smooth with respect to the diffeological product structure on $V\times V$ and the standard diffeology on $\matR$. Note that a finite-dimensional
diffeological vector space admits a smooth scalar product if and only if it is diffeomorphic to some $\matR^n$ with the standard diffeology (see \cite{iglesiasBook}, Ex. 70 on p. 74 and its
solution). Thus, a finite-dimensional diffeological vector space (or a bundle of such) is endowed, not with a scalar product, but with a ``minimally degenerate'' smooth symmetric bilinear form,
which we call a \textbf{pseudo-metric} (see Section \ref{pseudometric:sect}).

\paragraph{Fine diffeology on vector spaces} The \textbf{fine diffeology} on a vector space $\matR$ is the \emph{finest} vector space diffeology on it; endowed with such, $V$ is called a
\emph{fine vector space}. Note that \emph{any} linear map between two fine vector spaces is smooth (\cite{iglesiasBook}, 3.9). An example of a fine vector space is $\matR^n$ with its
\textbf{standard diffeology}, \emph{i.e.}, the diffeology that consists of all the usual smooth maps with values in $\matR^n$.\footnote{It is easy to see that this set of maps is indeed a
(vector space) diffeology. Furthermore, it is the finest one, since, as we have already observed above, it is contained in any other vector space diffeology.}

\paragraph{Smooth linear and bilinear maps} In the case of diffeological vector spaces it frequently happens that the space $L^{\infty}(V,\matR)$ of all smooth linear maps $V\to\matR$ (where
$V$ is a diffeological vector space and $\matR$ is considered with its standard diffeology), and more generally, the space $L^{\infty}(V,W)$ of all smooth linear maps $V\to W$, is \emph{a
priori} smaller than the space of all linear maps between the respective spaces; see Example 3.11 of \cite{wu} (and also \cite{multilinear}, Example 3.1). In fact, such examples can easily
be found for all finite-dimensional vector spaces. Accordingly, the same issue presents itself for bilinear maps; given $V$, $W$ two diffeological vector spaces, let $B(V,W)$ be the set of bilinear
maps on $V$ with values in $W$, and let $B^{\infty}(V,W)$ be the set of those bilinear maps that are smooth with respect to the product diffeology on $V\times V$ and the given diffeology on $W$.
Just as for linear maps, the space $B^{\infty}(V,W)$ is frequently a proper subspace of $B(V,W)$, although some of the usual isomorphisms continue to exist.

\paragraph{The dual of a diffeological vector space} The \textbf{diffeological dual} $V^*$ (see \cite{vincent}, \cite{wu}) of a diffeological vector space $V$ is the set $L^{\infty}(V,\matR)$ of
all smooth linear maps $V\to\matR$, endowed with the functional diffeology. As all spaces of smooth linear maps (see above), it is a diffeological vector space, which in general is not isomorphic
to $V$. For one thing, in the finite-dimensional case it almost always has a smaller dimension: as shown in \cite{pseudometric}, the diffeological dual of a finite-dimensional diffeological vector
space is always a standard space (in particular, it is a fine space), so the equality $L^{\infty}(V,\matR)=L(V,\matR)$ holds if and only if $V$ is a standard space.\footnote{Note also that, as
shown in \cite{multilinear}, Proposition 4.4, if $V^*$ and $V$ are isomorphic as vector spaces then they are also diffeomorphic.} The matters become less straightforward in the infinite-dimensional
case, which in this work we do not consider.

\paragraph{The tensor product} The definition of the diffeological tensor product was given first in \cite{vincent} and then in \cite{wu} (see Section 3). Let $V_1$, ..., $V_n$ be diffeological vector
spaces, and let $V_1\times\ldots\times V_n$ be their free product, endowed with the finest vector space diffeology that contains the product diffeology on the Cartesian product of $V_1,\ldots,V_n$.
Let $T:V_1\times\ldots\times V_n\to V_1\otimes\ldots\otimes V_n$ be the universal map onto their tensor product as vector spaces, and let $Z\leqslant V_1\times\ldots\times V_n$ be the kernel of
$T$. The \textbf{tensor product diffeology} on $V_1\otimes\ldots\otimes V_n$ is the quotient diffeology on $V_1\otimes\ldots\otimes V_n=(V_1\times\ldots\times V_n)/Z$ coming from the free product
diffeology on the free product of the spaces $V_1,\ldots,V_n$; the free product diffeology is in turn defined as the finest vector space diffeology on the free product $V_1*\ldots*V_n$ that contains the 
product diffeology. The diffeological tensor product thus defined possesses the usual universal property established in \cite{vincent}, Theorem 2.3.5: for any diffeological vector spaces $V_1,\ldots,V_n,W$ 
the space $L^{\infty}(V_1\otimes\ldots\otimes V_n,W)$ of all smooth linear maps $V_1\otimes\ldots\otimes V_n\to W$ (considered, as usual, with the functional diffeology) is diffeomorphic to the space 
$\mbox{Mult}^{\infty}(V_1\times\ldots\times V_n,W)$ of all smooth (for the product diffeology) multilinear maps $V_1\times\ldots\times V_n\to W$ (also endowed with the functional diffeology).

\paragraph{The spaces $V^*\otimes W^*$ and $(V\otimes W)^*$} The standard diffeomorphism between these two spaces continues to hold, in the sense that that the usual isomorphism
(when it exists) is smooth.

\subsection{Diffeological bundles}

The notion of a \textbf{diffeological fibre bundle} was first studied in \cite{iglFibre}; a more recent exposition appears in \cite{iglesiasBook}, Chapter 8. A smooth surjective map $\pi:T\to B$
is a \textbf{fibration} if there exists a diffeological space $F$ such that the pullback of $\pi$ by any plot $p$ of $B$ is locally trivial, with fibre $F$. The latter condition means that there
exists a cover of $B$ by a family of D-open sets $\{U_i\}_{i\in I}$ such that the restriction of $\pi$ over each $U_i$ is trivial with fibre $F$. There is also another definition of a diffeological fibre
bundle available in \cite{iglesiasBook}, 8.8, involving the notion of a \emph{diffeological groupoid} (we do not recall it since we will not use it).

\paragraph{Principal diffeological fibre bundles} Let $X$ be a diffeological space, and let $G$ be a diffeological group.\footnote{A diffeological group is a group $G$ endowed with a
diffeology such that the group product map $G\times G\to G$ and the inverse element map $G\to G$ are smooth.} Denote by $g\mapsto g_X$ a smooth action of $G$ on $X$, that is, a smooth
homomorphism from $G$ to $\mbox{Diff}(X)$. Let $F$ be the \emph{action map}:
$$F:X\times G\to X\times X\mbox{ with }F(x,g)=(x,g_X(x)).$$ Then the following is true (see the Proposition in Section 8.11 of \cite{iglesiasBook}): if $F$ is an induction\footnote{A map $f:X\to
Y$ between two diffeological spaces is called an induction if $f$ is a diffeomorphism of $X$ with the image $\mbox{Im}(f)$, the latter endowed with the subset diffeology.} then the projection $\pi$
from $X$ to its quotient $X/G$ is a diffeological fibration, with the group $G$ as fibre. In this case we say that the action of $G$ on $X$ is \textbf{principal}. Now, if a surjection $\pi:X\to Q$ is
equivalent to $\mbox{class}:X\to G/H$, that is, if there exists a diffeomorphism $\varphi:G/H\to Q$ such that $\pi=\varphi\circ\mbox{class}$, we shall say that $\pi$ is a \textbf{principal fibration},
or a \textbf{principal fibre bundle}, with structure group $G$.

\paragraph{Associated fibre bundles} Let $\pi:T\to B$ be a principal fibre bundle with structure group $G$, and let $E$ be a diffeological space together with a smooth action of $G$, that is,
a smooth homeomorphism $g\mapsto g_E$ from $G$ to $\mbox{Diff}(E)$. Let $X=T\times_G E$ be the quotient of $T\times E$ by the diagonal action of $G$:
$$g_{T\times E}:(t,e)\mapsto (g_T(t),g_E(e)).$$ Let $p:X\to B$ be the projection $\mbox{class}(t,e)\mapsto\pi(t)$. Then the projection $p$ is a diffeological fibre bundle, with fibre $E$; it is called
the fibre bundle \textbf{associated with $\pi$ by the action of $G$ on $E$}.

\subsection{Connections on diffeological bundles}

As discussed in \cite{iglesiasBook} (see Foreword to 8.32), there is not yet an immutable notion of connection in diffeology. We briefly recall the definition given in the just-cited source, although in
Section 10 we will attempt to develop a notion of diffeological connection, on diffeological vector bundles, or extensions of such, following a different approach (which mimics the standard one).

Let $Y$ be a diffeological space, and let $\mbox{Paths}_{loc}(Y)$ be the space of local paths in $Y$, \emph{i.e.}, the set of $1$-plots of $Y$ defined on open intervals,
$$\mbox{Paths}_{loc}(Y)=\{\tilde{c}\in C^{\infty}((a,b),Y)|a,b\in\matR\},$$ equipped with the functional diffeology induced by the functional diffeology of the $1$-plots of $Y$. Let us denote by
$\mbox{tbPaths}_{loc}(Y)$ equipped with the sub-diffeology of the product diffeology,
$$\mbox{tbPaths}_{loc}(Y)=\{(\tilde{c},t)\in\mbox{Paths}_{loc}(Y)\times\matR|t\in\mbox{def}(\tilde{c})\}.$$ Finally, let $\pi:Y\to X$ be a principal diffeological fibration with the structure group $G$, and
let $(g,y)\mapsto g_Y(y)$ denote the action of $G$ on $Y$.

\begin{defn}
A \textbf{connection} on the $G$-principal fibre bundle $\pi:Y\to X$ is any smooth map
$$\Omega:\mbox{tbPaths}_{loc}(Y)\to\mbox{Paths}_{loc}(Y)$$ satisfying the following series of conditions:
\begin{enumerate}
\item \emph{Domain}. $\mbox{def}(\Omega(\tilde{c},t))=\mbox{def}(\tilde{c})$.
\item \emph{Lifting}. $\pi\circ\Omega(\tilde{c},t)=\pi\circ\tilde{c}$.
\item \emph{Basepoint}. $\Omega(\tilde{c},t)(t)=\tilde{c}(t)$.
\item \emph{Reduction}. $\Omega(\gamma\cdot\tilde{c},t)=\gamma(t)_Y\circ\Omega(\tilde{c},t)$, where $\gamma:\mbox{def}(\tilde{c})\to G$ is any smooth map and
$\gamma\cdot\tilde{c}=[s\mapsto\gamma(s)\gamma(\tilde{c}(s))]$.
\item \emph{Locality}. $\Omega(\tilde{c}\circ f,s)=\Omega(\tilde{c},f(s))\circ f$, where $f$ is any smooth local path defined on an open domain with values in $\mbox{def}(\tilde{c})$.
\item \emph{Projector}. $\Omega(\Omega(\tilde{c},t),t)=\Omega(\tilde{c},t)$.
\end{enumerate}
\end{defn}

The local path $\Omega(\tilde{c},t)$ is the \textbf{horizontal projection} of $\tilde{c}$ pointed at $t$; it is a \textbf{horizontal path} for the connection $\Omega$. It is also denoted by $\Omega_t(\tilde{c})$
or $\Omega(\tilde{c})(t)$.

\section{Differential $k$-forms on a diffeological space}\label{diffeological:forms:sect}

As is observed in the Foreword to Chapter 6 of \cite{iglesiasBook}, the definition of a diffeology on a set $X$ by means of maps to $X$ allows for a relatively simple extension of those standard
constructions that are based on geometric covariant objects. In particular, there exists a well-established notion of diffeological differential forms, even if an agreed-upon concept of tangent
vectors is not there yet.\footnote{The extension just mentioned goes up to by-now standard version of the diffeological De Rham calculus; in this section we give its brief description.}

\subsection{Bundles of differential $k$-forms}

Differential forms on diffeological spaces are defined by their evaluations on the plots, which are regarded as pull-backs of the forms by the plots; these pullbacks are ordinary smooth forms.
The condition on the pullbacks, to represent a differential form of a diffeological space, expresses just the condition of compatibility under composition.

\subsubsection{What is a differential $k$-form}

We now state, and illustrate, the definition of a differential $k$-form on a diffeological space, showing how it admits a sort of finite description via the concept of the generated
diffeology.\footnote{For this description to be finite, that is, to be given by a finite list of usual $k$-forms, the diffeology must be generated by a finite set of plots.}

\paragraph{Differential forms on diffeological spaces} Let $X$ be a diffeological space. A \textbf{differential $k$-form on $X$} is a map $\alpha$ that associates with every plot $p:U\to X$ a smooth
$k$-form, denoted by $\alpha(p)$, defined on $U$ and satisfying the following condition: for every smooth map $g:V\to U$, with $V$ a domain, it must hold that
$$\alpha(p\circ g)=g^*(\alpha(p)).$$ This condition is called the \textbf{smooth compatibility condition}, and we say that $\alpha(p)$ \textbf{represents} the differential form $\alpha$ in the plot $p$.
The set of differential $k$-forms on $X$ is denoted by $\Omega^k(X)$.

\paragraph{The differential of a smooth function} The simplest example of a differential form on a diffeological space is, of course, the differential of a smooth function; it is defined as follows.
Let $X$ be a diffeological space, and let $f:X\to\matR$ be a smooth (for the standard diffeology on $\matR$) function. Then for any plot $p:U\to X$ the composition $f\circ p$ is a smooth function
$U\to\matR$ in the usual sense. Associating to each plot $p$ the $1$-form $d(f\circ p)$, that is
$$\calD_X\ni p\mapsto d(f\circ p)\in\Lambda^1(U),$$ satisfies the smooth compatibility condition and therefore defines a differential $1$-form on the diffeological space $X$;
it is called the \textbf{differential of $f$}.

As an easy example, consider $X=\matR$ with the diffeology generated by the plot $p:\matR\ni x\mapsto |x|$; by the definition of the generated diffeology, this means that locally every non-constant
plot has form $q:U\ni u\mapsto|h(u)|$ for some domain $U\subset\matR^m$ and an ordinary smooth function $h:U\to\matR$. Thus, $f:X\to\matR$ given by $f(x)=x^2$ is a smooth function;
we have $(f\circ q)(u)=(h(u))^2$. The differential $df$ is then the $1$-form given by $df(q)=2h(u)h'(u)du$.

\paragraph{Differential forms through generating families} Let $X$ be a diffeological space, let $\calD$ be its diffeology, and let $\mathcal{A}=\{p_i:U_i\to X\,|\,i\in I\}$ be a set of plots that
generates $\calD$. A collection $\{\alpha(p_i)\,|\,p_i\in\mathcal{A},\,i\in I\}$ of smooth $k$-forms yields the values of a differential form $\alpha\in\Omega^k(X)$ if and only if the following two
conditions are satisfied:
\begin{enumerate}
\item For all $p_i\in\mathcal{A}$, $\alpha(p_i)$ is a smooth $k$-form defined on the domain of definition of $p_i$ (that is, on $U_i$);
\item For all $p_i,p_j\in\mathcal{A}$, for every smooth $g:V\to g(V)=U_i$, and for every smooth $g':V\to g'(V)=U_j$, the following holds:
$$p_i\circ g=p_j\circ g'\Rightarrow g^*(\alpha(p_i))=(g')^*(\alpha(p_j)).$$
\end{enumerate}

\paragraph{A sample differential $1$-form} Let us now construct a differential $1$-form which is not the differential of a smooth function. Let $X$ be $\matR^2$ endowed with the vector space
diffeology generated by the plot $x\mapsto|x|$; this means that a generic (non-constant) plot of it locally has form $u\mapsto(f_1(u),f_2(u)+g_1(u)|h_1(u)|+\ldots+g_k(u)|h_k(u)|)$. It follows that
a generating (in the usual sense, not vector space sense) set for this diffeology can be given by the infinite sequence $\{q_k\}_{k\geqslant 1}$ of plots of the following form:
$$q_k:\matR^{2k+2}\ni(u_1,\ldots,u_k,v_1,\ldots,v_k,z_1,z_2)\mapsto (z_1,z_2+v_1|u_1|+\ldots+v_k|u_k|).$$ Note that any two consecutive plots in this sequence are related by a smooth
substitution $q_k=q_{k+1}\circ g_k$, where $g_k:\matR^{2k+2}\to\matR^{2k+4}$ acts by
$$g_k(u_1,\ldots,u_k,v_1,\ldots,v_k,z_1,z_2)=(u_1,\ldots,u_k,v_1,\ldots,v_k,0,0,z_1,z_2).$$ The smooth compatibility condition means therefore that, for a prospective form $\omega$ on $X$,
we should have
$$\omega(q_k)=g_k^*(\omega(q_{k+1})).$$

This condition suggests the following description of a given differential $1$-form $\omega$ on $X$. Choose an index $k$ and let $\omega_0:=\omega(q_k)$. By the smooth compatibility condition,
the forms $\omega(q_1)$, $\ldots$, $\omega(q_{k-1})$ are then uniquely defined by the form $\omega_0$. On the other hand, we can see that, for instance, $\omega(q_{k+1})=\omega_0+\omega'$,
where $\omega'$ is any $1$-form on $\matR^{2k+4}$ vanishing on the hyperplane of equation $u_{k+1}=v_{k+1}=0$ (this could be the zero form, or any $1$-form in the
$C^{\infty}(\matR^{2k+4},\matR)$-span of $u_{k+1}dv_{k+1},v_{k+1}du_{k+1}$, \emph{et cetera}). Here we consider $\omega_0$ (which is a $1$-form on $\matR^{2k+2}$) as a form on $\matR^{2k+4}$
with respect to the inclusion $\matR^{2k+2}\subset\matR^{2k+4}$ that identifies the former with the hyperplane of equation $u_{k+1}=v_{k+1}=0$; notice that this description suggests that $\omega_0$
can be any $1$-form on $\matR^{2k+2}$; furthermore, it does define the differential $1$-form $\omega$ on $X$ up to quotienting over (some set of) vanishing $1$-forms.

Finally, a specific example of a $1$-form can be (this is probably one of the simplest examples) given by:
$$\omega(q_k)=\sum_{i=1}^k u_idv_i.$$

\paragraph{Closed and exact forms} The meaning of these terms for differential forms on diffeological spaces is the same as in the standard case; we wish to illustrate them. The \emph{exact}
$k$-forms are differentials of $(k-1)$-forms, such as differentials of smooth functions in the case of $k=1$. It is easy to see that the notion of the differential easily extends to the case of arbitrary
$k$; given a differential $k$-form $\alpha$, the $(k+1)$-form $d\alpha$ is given by associating to each plot $p:U\to X$ of $X$ the form $d(\alpha(p))\in\Lambda^{k+1}(U)$, the usual differential of
the form $\alpha(p)$. The fact that this assignment does define a differential form on the diffeological space $X$, \emph{i.e.}, that the smooth compatibility condition still holds, follows from the
standard properties of differential forms (those on domains of Euclidean spaces). The \emph{closed} forms, as usual, are those whose differential is the zero form; recall that for a diffeological
space $X$ this means that $\omega(p)\equiv 0$ for any plot $p$ of $X$ (this trivially means that every form $\omega(p)$ is closed).

One immediate question at this point is the following one. Let $\omega$ be a differential $1$-form on a diffeological space $X$ such that for every plot $p:U\to X$ of $X$ the usual $1$-form
$\omega(p)\in\Lambda^1(U)$ is exact; does this mean that $\omega$ is the differential of some smooth function on $X$? The answer is positive, as follows from Sect. 6.31 and 6.34 of
\cite{iglesiasBook}.

\subsubsection{The vector space $\Omega^k(X)$ and the pseudo-bundle $\Lambda^k(X)$}

As has been said already, $\Omega^k(X)$ denotes the set of all differential $k$-forms on $X$. By analogy with the usual differential forms (of which their diffeological counterparts are generalizations),
one can expect that it is not just a set. Indeed, it has a natural structure of a diffeological vector space.

\paragraph{The functional diffeology on $\Omega^k(X)$} The vector space structure on $\Omega^k(X)$ is given by the following operations: for all $\alpha,\alpha'\in\Omega^k(X)$, for all
$s\in\matR$, and for all plots $p$ of $X$ we have
$$\sistema{(\alpha+\alpha')(p)=\alpha(p)+\alpha'(p)\\ (s\cdot\alpha)(p)=s\cdot\alpha(p)};$$ the sum $\alpha(f)+\alpha'(p)$ and the product by scalar $s\cdot\alpha(p)$ of smooth differential forms are
pointwise. It is also a diffeological vector space, for the functional diffeology, whose characterization is as follows.

Consider the set of all maps $\phi:U'\to\Omega^k(X)$, for all domains $U'\subset\matR^m$ and for all $m\in\matN$, that satisfy the following condition: for every plot $p:U\to X$ the map
$U'\times U\to\Lambda^k(\matR^n)$ given by $(s,r)\mapsto(\phi(s)(p))(r)$ is smooth. This collection of plots forms a vector space diffeology on $\Omega^k(X)$, called the \textbf{standard functional
diffeology}, or simply the \textbf{functional diffeology}, on $\Omega^k(X)$.

\paragraph{The fibre $\Lambda_x^k(X)$} There is a natural quotienting of $\Omega^k(X)$, which gives, at every point $x\in X$, the set of all distinct values, at $x$, of the differential $k$-forms on $X$.
The resulting set is denoted by $\Lambda_x^k(X)$ and is defined as follows.

Let $X$ be a diffeological space, and let $x$ be a point of it. A plot $p:U\to X$ is said to be \textbf{centered at $x$} if $U\ni 0$ and $p(0)=x$. Let $k$ be an integer, and let us consider the following
equivalence relation $\sim_x$: two $k$-forms $\alpha,\beta\in\Omega^k(X)$ are equivalent, $\alpha\sim_x\beta$, and are said to \textbf{have the same value in $x$} if for every plot $p$ centered at
$x$, we have $\alpha(p)(0)=\beta(p)(0)$ ($\Leftrightarrow(\alpha-\beta)(f)(0)=0$). The class of $\alpha$ for the equivalence relation $\sim_x$ is called \textbf{the value of $\alpha$ at the point $x$};
we occasionally denote it by $\alpha_x$. The set of all the values at the point $x$, for all $k$-forms on $X$, is the set $\Lambda_x^k(X)$,
$$\Lambda_x^k(X)=\Omega^k(X)/\sim_x=\{\alpha_x\,|\,\alpha\in\Omega^k(X)\}.$$

An element $\alpha\in\Lambda_x^k(X)$ is said to be a \textbf{$k$-form of $X$ at the point $x$} (and $x$ is said to be the \textbf{basepoint of $\alpha$}). The space $\Lambda_x^k(X)$ is then
called the \textbf{space of $k$-forms of $X$ at the point $x$}.

A form $\alpha$ \textbf{vanishes at the point $x$} if and only if, for every plot $p$ centered at $x$ we have $\alpha(p)(0)=0$ (so $\alpha\sim_x 0$, and we will write $\alpha_x=0_x$). Two $k$-forms
$\alpha$ and $\beta$ have the same value at the point $x$ if and only if their difference vanishes at this point: $(\alpha-\beta)_x=0$. The set $\{\alpha\in\Omega^k(X)\,|\,\alpha_x=0_x\}$ of the
$k$-forms of $X$ vanishing at the point $x$ is a vector subspace of $\Omega^k(X)$; furthermore,
$$\Lambda_x^k(X)=\Omega^k(X)/\{\alpha\in\Omega^k(X)\,|\,\alpha_x=0_x\}.$$ In particular, as a quotient of a diffeological vector space by a vector subspace, the space $\Lambda_x^k(X)$ is
naturally a diffeological vector space; the addition and the scalar multiplication on $\Lambda_x^k(X)$ are given respectively by $\alpha_x+\beta_x=(\alpha+\beta)_x$ and $s(\alpha_x)=(s\alpha)_x$.

\paragraph{The $k$-forms bundle $\Lambda^k(X)$} Let $X$ be a diffeological space; consider the vector space $\Lambda_x^k(X)$ of values of $k$-forms at the point $x\in X$. The \textbf{bundle of
$k$-forms over $X$}, denoted by $\Lambda^k(X)$, is the union of all spaces $\Lambda_x^k(X)$:
$$\Lambda^k(X)=\coprod_{x\in X}\Lambda_x^k(X)=\{(x,\alpha)\,|\,\alpha\in\Lambda_x^k(X)\}.$$ It has an obvious structure of a bundle over $X$. In fact, most often it is a pseudo-bundle and not
a true bundle (we will illustrate later on that in many cases it is not locally trivial); we will tend to call it pseudo-bundle, although in the original sources it is just called a bundle.

The pseudo-bundle $\Lambda^k(X)$ is endowed with the diffeology that is the pushforward of the product diffeology on $X\times\Omega^k(X)$ by the projection
$\chi^{\Lambda}:X\times\Omega^k(X)\to\Lambda^k(X)$ acting by $\chi^{\Lambda}(x,\alpha)=(x,\alpha_x)$.\footnote{We stress again that $\Lambda^k(X)$ is \emph{not} the diffeological disjoint union
of the spaces $\Lambda_x^k(X)$} Note that for this diffeology the natural projection $\pi^{\Lambda}:\Lambda^k(X)\to X$ is a local subduction;\footnote{A surjective map $f:X\to Y$ between two
diffeological spaces is called a \textbf{subduction} if the diffeology of $Y$ coincides with the pushforward of the diffeology of $X$ by $f$.} furthermore, each subspace $(\pi^{\Lambda})^{-1}(x)$ is
smoothly isomorphic to $\Lambda_x^k(X)$.

\paragraph{The plots of $\Lambda^k(X)$} A map $p:u\mapsto(p_1(u),p_2(u))\in\Lambda^k(X)$ defined on some domain $U$ is a plot of $\Lambda^k(X)$ if and only if the following two
conditions are fulfilled:
\begin{enumerate}
\item The map $p_1$ is a plot of $X$;
\item For all $u_0\in U$ there exists an open neighborhood $U'$ of $u_0$ and a plot $q:U'\to\Omega^k(X)$ (recall that $\Omega^k(X)$ is considered with its functional diffeology described above)
such that for all $u'\in U'$ we have $p_2(u')=(q(u'))(p_1)(u')$.
\end{enumerate}

\subsection{The corresponding approach to tangent and cotangent spaces}

At least from the formal point of view, the construction described in the previous two sections allows to define the corresponding concept of the cotangent space and the tangent one. These are not
standardly used definitions (which do not exist yet; see \cite{iglesiasBook}, particularly p. 167, for a very useful discussion of this state of matters), just the most straightforward consequences of
the above construction.

\paragraph{The cotangent bundle $\Lambda^1(X)$} By a formal analogy, a version of the \emph{cotangent bundle of $X$} could be defined as $\Lambda^1(X)$, although it is not determined (at least
not \emph{a priori}) by the duality to some tangent bundle (and its geometric meaning is not clear). The so-called space of tangent $1$-vectors (that we briefly describe below, although we will not make any 
further use of it) can be defined by the diffeological duality to $\Lambda^1(X)$; although in general it provides a finer construction than the entire dual pseudo-bundle $(\Lambda^1(X))^*$.

\paragraph{Tangent $k$-vectors at a point} We describe here the space of so-called \textbf{tangent $k$-vectors} to a diffeological space (see \cite{iglesiasBook}). Let $X$ be a diffeological space, let
$p:\matR^k\supset U\to X$ be a plot of $X$, let $u\in U$, and let $\alpha\in\Omega^k(X)$. Then $\alpha(p)(u)\in\Lambda^k(\matR^k)$, so as a $k$-form on $\matR^k$, it is proportional to the standard
volume form $\mbox{vol}_k$. Let $\dot{p}(u)(\alpha)$ be the coefficient of proportionality, so we have a map $\dot{p}(u):\Omega^k(X)\to\matR$ that acts by the rule:
$$\dot{p}(u):\alpha\mapsto\frac{\alpha(p)(u)}{\mbox{vol}_k}.$$ This map is smooth and so belongs to the dual vector space of $\Omega^k(X)$. Let us state again that such map is associated to
every plot $p$ defined on (a sub-domain of) $\matR^k$ and possesses the following property:
$$\forall\,\,\alpha\in\Omega^k(X)\mbox{ we have }\alpha(p)(u)=\dot{p}(u)\mbox{vol}_k.$$ In particular, for $k=1$ we would just have $\alpha(p)(u)=\dot{p}(u)du$.

\paragraph{The generating set $S_x^k(X)$} The space of tangent $k$-vectors has the following generating set. For an arbitrary $x\in X$, consider $S_x^k(X)\subset(\Omega^k(X))^*$ defined as
$$S_x^k(X)=\{\dot{p}(0)\in(\Omega^k(X))^*\,|\,p:\matR^k\to X\mbox{ is smooth, and }p(0)=x\}.$$ For all $x\in X$ we have $0\in S_x^k(X)$; moreover, for all $x\in X$, for all $\nu\in S_x^k(X)$, and for
all $s\in\matR$ we have $s\cdot\nu\in S_x^k(X)$. Therefore $S_x^k(X)$ is a star-shaped subset of $(\Omega^k(X))^*$ with origin $0$; however, \emph{a priori} it is not a vector subspace.

\paragraph{The spaces $T_x^k(X)$} The space of \textbf{tangent $k$-vectors of $X$ at the point $x$} is then defined as the subspace of $(\Omega^k(X))^*$ generated by the set $S_x^k(X)$, that is,
$$T_x^k(X)=\mbox{Span}(S_x^k(X))=\{\nu=\sum_{i=1}^n\nu_i\,|\,n\in\mathbb{N},\mbox{ and } \forall i=1,\ldots,n\,\,\,\nu_i\in S_x^p(X)\}.$$

\paragraph{The pseudo-bundles $T^k(X)$} For a fixed $k$, the pseudo-bundle composed of all $k$-vector spaces $T_x^k(X)$, when $x$ runs over the whole $X$, defines the \textbf{$k$-vector
pseudo-bundle of $X$}. It is denoted by $T^k(X)$ and is pointwise defined as
$$T^k(X)=\{(x,\nu)\,|\,x\in X\mbox{ and }\nu\in T_x^k(X)\}.$$ This space is endowed with the following diffeology. A map $p:U\to T^k(X)$ with $p(u)=(p_X(u),p_T(u))$ is a plot of $T^k(X)$ if for all
$u\in U$ there exist an open neighbourhood $U'$ of $u$ and a finite family $q_1,\ldots,q_l$ of plots of $\matR^k\supset U'\to X$ such that $q_i(0)=p_x(u)$ for all $i=1,\ldots,l$ and
$p_T(u)=\sum_{i=1}^l\dot{q}_i(0)$.

Note that the restriction of this diffeology to each $T_x^k(X)$ is a vector space diffeology; furthermore, the zero section $X\to T^k(X)$ is smooth. Thus, $T^k(X)$ is a \emph{diffeological vector space over
$X$} in the terminology of \cite{CWtangent} (which is the same object as a \emph{regular vector bundle} in \cite{vincent} and belongs to a wider class of diffeological bundles introduced originally
in \cite{iglFibre}; further on we use the term \emph{diffeological vector pseudo-bundle} for this type of object).

Finally, as suggested in \cite{iglesiasBook}, the pseudo-bundle $T^1(X)$ can be regarded as the tangent bundle of $X$, although we will not do it. More in general and as already mentioned, this is not
the standard theory of the tangent bundle (such theory does not exist yet; there is a recent summary of various attempts to develop one, see \cite{CWtangent}, Section 3.4).

\subsection{The diffeological De Rham calculus}

The content of this section is a brief summary of the exposition of the subject in \cite{iglesiasBook}, that we include for completeness only.

\paragraph{The exterior derivative of forms} We have already spoken about the differentials of $k$-forms, and this is just another name for it. Recall that if $\alpha$ is a $k$-form on a diffeological
space $X$ then the equality $(d\alpha)(p)=d(\alpha(p))$, for all plots $p$ of $X$, yields a well-defined, which is called the \textbf{exterior derivative of $\alpha$}, or the \textbf{differential of $\alpha$}.
This is a smooth linear operator
$$d:\Omega^k(X)\to\Omega^{k+1}(X),$$ where $\Omega^k(X)$ and $\Omega^{k+1}(X)$ are equipped with their usual functional diffeology.

\paragraph{The exterior product} Let $X$ be a diffeological space, and let $\alpha\in\Omega^k(X)$ and $\beta\in\Omega^l(X)$. The \textbf{exterior product} $\alpha\wedge\beta$ is the differential
$(k+l)$-form defined on $X$ by
$$(\alpha\wedge\beta)(p)=\alpha(p)\wedge\beta(p)$$ for all plots $p$ of $X$ (it is easy to see that this is well-defined). Regarded as the map
$$\wedge:\Omega^k(X)\times\Omega^l(X)\to\Omega^{k+l}(X)\mbox{ with }\wedge(\alpha,\beta)=\alpha\wedge\beta,$$ the exterior product is smooth and bilinear with respect to the above-mentioned
functional diffeology of the spaces of forms.

\paragraph{The De Rham cohomology groups} Let $X$ be a diffeological space. The exterior derivative defined above satisfies the \textbf{coboundary condition}
$$d:\Omega^k(X)\to\Omega^{k+1}(X)\mbox{ for }k\geqslant 0\mbox{ and }d\circ d=0.$$ Thus, there is a chain complex of real vector spaces $\Omega^{\star}(X)=\{\Omega^k(X)\}_{k=0}^{\infty}$ with
a coboundary operator $d$, to which the usual construction of a cohomology theory can be applied in an obvious fashion: define the space of \textbf{$k$-cocycles} as the kernel in $\Omega^k(X)$ of
the operator $d$, and the space of \textbf{$k$-boundaries} as the image, in $\Omega^k(X)$, of the operator $d$. These are denoted by
$$Z_{dR}^k(X)=\mbox{Ker}[d:\Omega^k(X)\to\Omega]\,\,\,\,\mbox{ and }\,\,\,\,B_{dR}^k(X)=d(\Omega^{k-1}(X))\subset Z_{dR}^k(X)\mbox{ with }B_{dR}^0(X)=\{0\}.$$
The \textbf{De Rham cohomology groups} of $X$ are then defined as the quotients of the spaces of cocycles by the (respective) spaces of coboundaries:
$$H_{dR}^k(X)=Z_{dR}^k(X)/B_{dR}^k(X).$$ These groups are diffeological vector spaces for the following diffeology: both $Z_{dR}^k(X)$ and $B_{dR}^k(X)$ are endowed with the subset diffeology
coming from the functional diffeology on $\Omega^k(X)$, and the quotient $H_{dR}^k(X)$ is then endowed with the quotient diffeology. If $X$ is a smooth manifold endowed with the standard diffeology,
this construction gives the usual De Rham cohomology of $X$.

\paragraph{The De Rham homomorphism} For the sake of completeness, we also mention the \textbf{De Rham homomorphism}. Let $X$ be a diffeological space, and let $\alpha\in\Omega^k(X)$ for
a positive $k$. The integration of $\alpha$ on the cubic $k$-chains (see \cite{iglesiasBook}, Section 6.65 for details) defines a cubic $k$-cochain $f_{\alpha}$ for the reduced cubic cohomology.
If the form $\alpha$ is closed, \emph{i.e.}, $d\alpha=0$, then $f_{\alpha}$ is closed as a cochain. Thus, the integration of $\alpha$ on chains defines a morphism from $Z_{dR}^k(X)$ to $Z^k(X)$
(the group of cubic cochains on $X$, see \cite{iglesiasBook}, Section 6.63). If $\alpha$ is exact then $f_{\alpha}$ is exact as a cochain. Hence, the integration on a chain defines a linear map,
denoted by $h_{dR}^k$ from $H_{dR}^k(X)$ to $H^k(X)$ (the cubic cohomology of $X$); this is the morphism that is called the De Rham homomorphism.

\subsection{Examples of differential forms on diffeological spaces}

For illustrative purposes, we provide here a few examples illustrating the basic constructions having to do with differential forms on diffeological spaces.

\subsubsection{Forms and differentials of functions}

We start with some simple, but still non-trivial examples, one of a diffeological differential $1$-form and another of the differential of a diffeologically smooth function.

\paragraph{A diffeological differential form on a space with non-standard diffeology} Let $X=\matR^2$ be endowed with the vector space diffeology generated by $p:x\mapsto|x|e_2$. Notice that
it is possible (see \cite{iglesiasBook}, Section 6.41) to describe a differential form on a given diffeological space $(X,\calD)$, with the diffeology $\calD$ generated by some set $\mathcal{A}$,
by assigning the local form $\omega_f$ to all maps $f\in\mathcal{A}$, and only to those; this assignment must of course satisfy some appropriate additional conditions (see \cite{iglesiasBook}):
namely, that if $p:U\to X$ and $q:V\to X$ are two plots of $X$, and $f:U'\to U$, $g:U'\to V$ are usual smooth functions such that $p\circ f=q\circ g$, then $f^*(\omega_p)=g^*(\omega_q)$.

However, $\mathcal{A}$ must generate $\calD$ in the sense of the usual generated diffeology. Whereas in our case $p$ generates the diffeology of (our specific) $X$ in the sense of vector space
diffeology; as a diffeology \emph{proper} it needs a larger generating set.

To find such a set, recall that a plot of $X$ is (locally) a finite linear combination, with smooth functional coefficients, of compositions of $p$ with ordinary smooth functions, namely, it has
the form $U\ni u\mapsto(h(u),f_0(u)+f_1(u)|g_1(u)|+\ldots+f_k(u)|g_k(u)|)$ for some smooth functions $h,f_0,f_1,\ldots,f_k,g_1,\ldots,g_k$. Its generating set, in the sense of just a diffeological space,
can therefore be represented by the following infinite family of plots $\{q_k\}_{k\geqslant 0}$:
$$q_k:\matR^{2k+2}\ni(u_1,\ldots,u_k,v_1,\ldots,v_k,z_1,z_2)\mapsto(z_1,z_2+v_1|u_1|+\ldots+v_k|u_k|).$$

Let us now consider the mutual relationships between these plots. It is quite easy to find that for all $k$ we have
$$q_k=q_{k+1}\circ f_k,$$ where $f_k:\matR^{2k+2}\to\matR^{2k+4}$ is defined by
$$f_k(u_1,\ldots,u_k,v_1,\ldots,v_k,z_1,z_2)=(u_1,\ldots,u_k,v_1,\ldots,v_k,0,0,z_1,z_2),$$ and therefore we must have
$$\omega(q_k)=f_k^*(\omega(q_{k+1})).$$ Thus, the following choice satisfies the smooth compatibility condition:
$$\omega(q_k)=d(z_1z_2)+\sum_{i=1}^kd(u_iv_i)=d(z_1z_2+\sum_{i=1}^ku_iv_i).$$ Indeed,
$$f_k^*(\omega(q_{k+1}))=d((z_1z_2+\sum_{i=1}^{k+1}u_iv_i)\circ g_k)=d(z_1z_2+\sum_{i=1}^ku_iv_i)=\omega(q_k).$$

Do note that in the example just made, the $1$-form constructed assigns to each plot an \emph{exact} form. As is easy to see, this does not mean that, considered as a differential $1$-form on the
diffeological space $X$, it is an exact form in the usual sense, \emph{i.e.}, that it is a differential of some smooth function on $X$ (see below).

\paragraph{Differential of a function} Let us now consider differentials of functions, using as the domain of definition the diffeological space of the previous example, $X=\matR^2$ with the vector space
diffeology generated by $x\mapsto(0,|x|)$. Recall that a generating set for this diffeology is given by by the sequence $\{q_k\}_{k=1}^{\infty}$ of plots, where $q_k:\matR^{2k+2}\to X$
acts by $q_k(u_1,\ldots,u_k,v_1,\ldots,v_k,z_1,z_2)= (z_1,z_2+v_1|u_1|+\ldots+v_k|u_k|)$. If $f:X\to\matR$ is a (diffeologically) smooth function then all $f\circ q_k$ must be smooth
in the usual sense, so $f(x,y)$ for $(x,y)\in X=\matR^2$ must essentially be a usual smooth function depending on $x$ only.\footnote{This implies, in particular, that the $1$-form
$\omega$ on $X$ is not an exact form, although its values on all plots are exact.}

On the other hand, let $X$ have the same underlying space $\matR^2$, but its diffeology be the product diffeology corresponding to $\matR^2=\matR\times\matR$, with the first factor carrying the
standard diffeology, and the second, one generated by $y\mapsto|y|$ (we mean the diffeology \emph{proper}). Thus, the generating set for the diffeology on $X$ is given by a single plot $p:\matR^2\to X$,
$p(x,y)=(x,|y|)$; this is because a generic plot of $X$ has form $(u,v)\mapsto(f(u),|g(v)|)$, for $u\in U\subseteq\matR^m$, $v\in V\subseteq\matR^l$, and $f:U\to\matR$, $g:V\to\matR$ some ordinary
smooth functions.

This description of a generating set for the diffeology on $X$ implies two things. First, any $1$-form is determined by assigning a usual $1$-form on $\matR^2$ to $p$ (and this can be any arbitrary
$1$-form), and, second, a smooth function $f:X\to\matR$ can be any usual smooth function which is even in the second variable. Thus, $f(x,y)=xy$ is not a smooth function for the chosen diffeology on
$X$, while $f(x,y)=xy^2$ is. Its differential from the diffeological point of view is obviously the same as its usual differential, $(df)(p)=y^2dx+2xydy$.

On the other hand, what is interesting to note is that $f$ does \emph{not} have to be smooth in the usual sense; for example, $f(x,y)=\mbox{sgn}(y)$ is a smooth function on $X$, if we define
$\mbox{sgn}(y)=\left\{\begin{array}{ll} 1 & \mbox{if }y\geqslant 0,\\ -1 & \mbox{if }y<0.\end{array}\right.$. Indeed, it suffices to verify that the composition $f\circ p$ is smooth in the usual sense,
and since $f(p(x,y))=f(x,|y|)=\mbox{sgn}(|y|)\equiv 1$, it is a constant function, so it is smooth (its differential does vanish everywhere). Yet, in the usual sense it is not even continuous.

\subsubsection{Spaces of forms}

Let us now consider the whole spaces $\Omega^k(X)$.

\paragraph{$\matR^2$ with the diffeology generated by $x\mapsto|x|e_2$} Recall the space $X$ mentioned in the second of the examples above; its underlying space is $\matR^2$, and its
diffeology is generated by the plot $p:\matR^2\ni(x,y)\mapsto(x,|y|)$. Thus, for $k\geqslant 1$ each $k$-form on $X$ is uniquely defined by a usual $k$-form on $\matR^2$; in fact, due to the local
property of diffeological forms, the correspondence is bijective. In particular, this means that only the spaces $\Omega^0(X)$, $\Omega^1(X)$, and $\Omega^2(X)$ are non-trivial; moreover, the latter
two are the same as the analogous spaces of the standard $\matR^2$.

On the other hand, the space $\Omega^0(X)$ is diffeomorphic to the space $C^{\infty}(X,\matR)$ of \emph{diffeologically} smooth functions (see \cite{iglesiasBook}, Section 6.31). We have already
noted that this is not the same as the space $C^{\infty}(\matR^2,\matR)$ of the usual smooth functions, nor is the former strictly contained in the latter. Let us now consider these statements
in detail.\vspace{2mm}

\noindent\emph{The space $\Omega^1(X)$} As we have just said, if $\omega\in\Omega^1(X)$, it is uniquely defined by the $1$-form $\omega(p)$ on $\matR^2$. The form $\omega(p)$ can be any
usual $1$-form on $\matR^2$, so we have a bijective correspondence $\Omega^1(X)\leftrightarrow C^{\infty}(\matR^2,\Lambda^1(\matR^2))$.

Let us now check that the assignment $\varphi_1:\Omega^1(X)\ni\omega\mapsto\omega(p)\in C^{\infty}(\matR^2,\Lambda^1(\matR^2))$ determines a smooth map for the functional diffeology on
$\Omega^1(X)$. This is a direct consequence of the diffeology's definition; indeed, let $q:U'\to\Omega^1(X)$ be a plot for this functional diffeology, and let $p$ be the above generating plot for the
diffeology on $X$. The map $q$ being a plot means that, just as for any plot of $X$, we have for $p$ that
$$U'\times\matR^2\ni(u',u)\mapsto(q(u')(p))(u)=\varphi_1(q(u'))(u)\in\Lambda^1(\matR^2)$$ is smooth, which is precisely the meaning of the map $\varphi_1$ being smooth.

Let us now consider the smoothness of its inverse; the latter is obtained by assigning to each $2$-form $\alpha$ on $\matR^2$ the form $\omega$ given by $\omega(p)=\alpha$. Let
$q':U'\to C^{\infty}(\matR^2,\Lambda^1(\matR^2))$ be a plot of $C^{\infty}(\matR^2,\Lambda^1(\matR^2))$, where $U'$ is a domain; the composition $q=\varphi^{-1}\circ q'$ yields a map
$q:U'\to\Omega^1(X)$ defined by $q(u')(p)=q'(u')$. Let us also consider a generic plot $p':U\to X$; since the diffeology of $X$ is generated by $p$, locally (thus, we assume that $U$ is small enough)
$p'=p\circ f$ for some ordinary smooth function $f:U\to\matR^2$. Therefore $q(u')(p')=f^*(q(u')(p))=f^*(q'(u'))$, which implies that the requirement for $q$ to be a plot of $\Omega^1(X)$ is indeed
satisfied.\vspace{2mm}

\noindent\emph{The space $\Omega^2(X)$} This case is completely analogous to the case of $\Omega^1(X)$; assigning to each $\omega\in\Omega^2(X)$ the form
$\omega(p)\in C^{\infty}(\matR^2,\Lambda^2(\matR^2))$ yields a diffeomorphism of these two spaces. Furthermore, since every $2$-form on $\matR^2$ is proportional to the volume form,
this implies that the $2$-forms on $X$ are essentially defined by the usual smooth functions of two variables; and this is in contrast with the case that immediately follows.\vspace{2mm}

\noindent\emph{The space $\Omega^0(X)$} As we have said already, the space $\Omega^0(X)$ coincides with the space $C^{\infty}(X,\matR)$ of all smooth functions on $X$; and we have also said,
that there is a (significant, in some sense) difference between $C^{\infty}(X,\matR)$ and $C^{\infty}(\matR^2,\matR)$. In particular, we noted that a usual smooth function on $\matR^2$ is
diffeologically smooth as a function on $X$ if and only if it is even in $y$, that is, if it filters through $(x,y)\mapsto(x,y^2)$. On the other hand, there are function on $\matR^2$ that are not even
continuous, let alone smooth, in the usual sense; yet, they belong to $C^{\infty}(X,\matR)$. We have already given an example of such function, which we called $\mbox{sgn}$ and defined as
$\mbox{sgn}(x,y)=1$ if $y\geqslant 0$ and $\mbox{sgn}(x,y)=-1$ if $y<0$. There is therefore a substantial difference between $\Omega^0(X)$ and $\Omega^2(X)$, since the latter is essentially
defined by the usual smooth functions on $\matR^2$ (see above). This suggests (at least at first glance) that the analogue of the Poincar\'e duality would be more intricate for diffeological spaces
(although our present observation is not meant to be precise).

\paragraph{$\matR^2$ with a non-standard vector space diffeology} We now turn to the space $X$ of that appear in the first example of this Appendix. Unlike the previous case, we do not aim to give
a complete description of its $\Omega^k$'s; rather, we consider a few instances of plots of $\Omega^1(X)$.

So, let $q:U'\to\Omega^1(X)$ be such a plot. It is defined by specifying for each $u'\in U'$ the differential $1$-form $q(u')$ on $X$. Recall that $q(u')$ is determined by an infinite sequence of
$1$-forms $\{q(u')(q_k)\}_{k=0}^{\infty}$, where each $q_k$ is the plot of $X$ defined by
$$q_k:\matR^{2k+2}\to X\mbox{ with }q_k(u_1,\ldots,u_k,v_1,\ldots,v_k,z_1,z_2)=(z_1,z_2+v_1|u_1|+\ldots+v_k|u_k|);$$ for the form $q(u')$ to be well-defined, we must have
$q(u')(q_k)=f_k^*(q(u')(q_{k+1}))$, where $f_k:\matR^{2k+2}\to\matR^{2k+4}$ is the map
$$f_k(u_1,\ldots,u_k,v_1,\ldots,v_k,z_1,z_2)=(u_1,\ldots,u_k,v_1,\ldots,v_k,0,0,z_1,z_2).$$

In general, to obtain a $1$-form $\omega$ on $X$ it suffices to choose, for some fixed $k$, the form $\omega(q_k)$; this uniquely defines the forms $\omega(q_i)$ for $i<k$. The form
$\omega(q_{k+1})$ is far from being unique; however, it differs from $\omega(q_k)$ by adding a term that is a differential $1$-form on the (standard) $2$-plane of coordinates $(u_{k+1},v_{k+1})$
and that vanishes at $0$.\footnote{This leads to the existence of a natural quotient space for $\Omega^1(X)$, where the sequence $\{\omega(q_k)\}$ is uniquely defined by just one term of it;
this is the space $\Lambda^1(X)$ below, whose definition we recall below.} Such a term could be, for instance, $u_{k+1}dv_{k+1}$; so, to give an example of a plot of $\Omega^1(X)$, we could take
$\omega_0\in\Omega_1(X)$ defined by
$$\omega_0(q_k)=z_1dz_2+\sum_{i=1}^ku_idv_i,$$ whose product by any ordinary smooth function $f:U'\to\matR$ would be a plot $q:U'\to\Omega^1(X)$, with
$$q(u')(q_k)=f(u')(z_1dz_2+\sum_{i=1}^ku_idv_i)$$ (sums of such expressions, for varying choice of smooth coefficients $f$'s, would yield other plots). This is because $\Omega^1(X)$ is a
diffeological vector space; this fact, moreover, has further implications in that locally every plot $q:U'\to\Omega^1(X)$ is a finite sum of constant plots (thus essentially of $1$-forms on $X$)
with usually smooth functional coefficients.

To conclude this example, we stress again that up to forms vanishing at $0$, we can represent any $\omega_0\in\Omega^1(X)$ in the following way:
$$\omega_0(q_k)=f_1(z_1,z_2)dz_1+f_2(z_1,z_2)dz_2+\sum_{i=1}^ku_idv_i$$ for some ordinary smooth $f_1,f_2$, thus $\omega_0$ is essentially determined by a usual smooth $1$-form
in variables $z_1,z_2$.

\section{What is a vector bundle in diffeology?}\label{diffeological:vector:bundle:sect}

We now turn to describing the diffeological counterparts of those components of the construction of the Dirac operator, that are perhaps most different from the usual smooth case.
The concept we illustrate first is that of a diffeological vector pseudo-bundle, an analogue of a usual smooth vector bundle obtained from the latter via a couple of rather obvious modifications;
one is that the notion of diffeological smoothness is used instead of the usual smoothness, and the other consists in dropping the requirement of local triviality. As one can expect, it is this
latter change that makes the most difference. Note also that pseudo-bundles appear under different names in the literature (see, for instance, \cite{iglFibre}, \cite{vincent}, \cite{CWtangent}).

\subsection{Why pseudo-bundles?}

As has just been said, this is an existing notion; let us briefly motivate its appearance (as well as our use of it). There exists, of course, an absolutely straightforward notion of a
diffeological vector bundle (see, for instance, \cite{iglFibre}, \cite{vincent}, or \cite{iglesiasBook} for a comprehensive exposition), as a diffeological bundle $\pi:V\to X$ (see above)
such that each fibre is a diffeological vector space for the subset diffeology, with the addition $V\times_X V\to V$ and the scalar multiplication $\matR\times V\to V$ being smooth, and admitting
an atlas of local trivializations. It is then a technicality to extend all the usual considerations regarding vector bundles to the category of diffeological spaces.

This has obviously been done (see the above sources), but for the reasons that we are about to explain, and that were already present in the above-cited works where the extended notion appears,
this is not sufficient. Such reasons (which can of course be described in many ways) can also be illustrated by the work of Christensen-Wu \cite{CWtangent} on internal tangent spaces and
bundles, specifically by Example 4.3 of \cite{CWtangent}, which we now recall.

Let $X=\{(x,y)\in\matR^2\,|\,xy=0\}$, that is, the union of the coordinate axes in $\matR^2$; endow it with the subset diffeology relative to the standard diffeology $\matR^2$. The internal tangent
space at a point $x\in X$ is denoted by $T_x^{dvs}(X)$; the corresponding internal tangent bundle is denoted by $T^{dvs}(X)$. The latter is a rigorous construction of Christensen-Wu, which,
when applied to the specific example of $X$, yields the following: the internal tangent space at the origin is $\matR^2$, while it is $\matR$ elsewhere. Thus, the internal tangent bundle of $X$ is
``almost'' a diffeological vector space bundle: each fibre is a diffeological vector space, and the addition and the scalar multiplication are smooth over $X$. But it does not satisfy the
essential condition to be a \emph{bundle}: it is not locally trivial, since its fibre is not always the same.

The difficulty cannot be dispensed with the naive manner of considering objects that are true vector bundles everywhere except at some isolated points (which would include the example just
made). The reason is that by Proposition 3.6 of \cite{CWtangent} (see also Proposition 4.13(2) of same), the internal tangent bundles respect the direct product of diffeological spaces, that is, if
$X$ and $Y$ are two diffeological spaces, and $x\in X$ and $y\in Y$ are two points, then the tangent space $T_{(x,y)}^{dvs}(X\times Y)$ at the point $(x,y)\in X\times Y$ is isomorphic as a
diffeological vector space to the direct product of the respective single tangent spaces, that is, to $T_x^{dvs}(X)\times T_y^{dvs}(Y)$. Applying this statement to the case of $X=\{(x,y)\in\matR^2\,|\,xy=0\}$
and $Y=\matR^1$, we conclude that the internal tangent space of the direct product $X\times Y$ is $\matR^3$ at any point of the line $\{(0,0)\}\times\matR$, and it is $\matR^2$ elsewhere.

Obviously, just using direct products, one can construct a multitude of similar examples. What we mean here by ``similar example'' is a pair $(X,Y)$, where $X$ is a diffeological space, $Y\subset X$ is
a subset of it, and the internal tangent space behaves as follows: all spaces $T_x^{dvs}(X)$ for $x\in X\setminus Y$ are diffeomorphic to a given diffeological vector space $V_0$, while all spaces
$T_y^{dvs}(Y)$ are distinct from $V_0$. Furthermore, the inequality $T_y^{dvs}(Y)\neq V_0$ would be true already at the level of vector spaces, and the subset $Y$ does not admit a simple
characterization.

We end this section with observing, informally, that even the description in the preceding paragraph does not give a complete picture. Indeed, consider $X\subset\matR^3$ defined as
$X=\{(x,y,z)\in\matR^3\,|\,z=0\}\cup\{(x,y,z)\in\matR^3\,|\,x=y=0\}$, \emph{i.e.}, the union of the horizontal plane $z=0$ and the vertical axis; endow $X$ with the subset diffeology coming from
$\matR^3$. Then it can be deduced from the characterization of internal tangent spaces by Christensen-Wu in \cite{CWtangent}, Proposition 3.3 and Proposition 3.4, that the internal tangent space
at a point of $X$ is $\matR^3$ at the origin, it is $\matR^2$ at any point of the plane which is not the origin, and it is $\matR$ at a point of the vertical line which, again, is not the origin. Note,
finally, that on the connected components of $X\setminus\{(0,0,0)\}$ the corresponding connected components of the internal tangent bundles are however true diffeological vector bundles,
of which, then, $T^{dvs}(X)$ is somehow ``pieced together'' (and this final remark opens the way to the discussion carried out in the next section).

\subsection{Diffeological vector pseudo-bundles}

The considerations in the previous section are probably sufficient to justify focusing our attention on this weaker notion of a vector bundle, which is obtained, in addition to intending smoothness in
the diffeological sense, by dropping the requirement of the existence of an atlas of local trivializations. We obtain what we call a \textbf{diffeological vector pseudo-bundle}; this is precisely the same
object as described by the following definition.

\subsubsection{The notion of a diffeological vector pseudo-bundle}

What we are giving below is the definition a diffeological vector space over a given diffeological space $X$, as it was given in \cite{CWtangent}; a diffeological vector pseudo-bundle is precisely
the same object (we explain the reason for the change in terminology later).

\begin{defn} \emph{(\cite{CWtangent}, Definition 4.5)}
Let $X$ be a diffeological space. A \textbf{diffeological vector space over $X$} is a map $\pi:V\to X$ where $V$ is a diffeological space and $\pi$ is a smooth map $V\to X$ such that each of the
fibres $\pi^{-1}(x)$ is endowed with a vector space structure for which the following properties hold: 1) the addition map $V\times_X V\to V$ is smooth with respect to the diffeology of $V$ and the
subset diffeology on $V\times_X V$ (relative to the product diffeology on $V\times V$); 2) the scalar multiplication map $\matR\times V\to V$ is smooth for the product diffeology on $\matR\times V$;
3) the zero section $X\to V$ is smooth.
\end{defn}

Obviously, all usual smooth vector bundles over smooth manifolds fall under this definition, if both the base space and the total space are considered with the standard diffeology. In the above
section we already described Example 4.3 of \cite{CWtangent}, which is an instance of a diffeological vector pseudo-bundle ($=$ diffeological vector space over $X$) that is not locally trivial.

\paragraph{The choice of terminology} The notion of a diffeological vector space over another diffeological space actually coincides with that of a \emph{regular vector bundle} of \cite{vincent}, and
is a particular instance of a diffeological fibre bundle that  appeared first in \cite{iglFibre}. We will use the term \textbf{diffeological vector pseudo-bundle} to refer to precisely the same object,
in order to avoid confusion between a single diffeological vector space (a $V$ with a vector space structure and a vector space diffeology) and a collection of such, \emph{i.e.}, a diffeological
counterpart of a vector bundle; and we avoid the term of Vincent, as well as that of \cite{iglFibre}, to stress the fact that in general these objects are not really bundles.

\subsubsection{Constructing diffeological vector pseudo-bundles}

As we have already indicated, sometimes diffeological vector pseudo-bundles naturally arise in other contexts, as it occurs in the Christensen-Wu example. Otherwise, there are a few systematic
ways of obtaining them, that we now describe.

\paragraph{The diffeology on fibres} Diffeological vector pseudo-bundles may arise from so-called vector space \emph{pre-bundles} (this term appears in \cite{vincent}). A pre-bundle is a smooth
surjective map $V\to X$ between two diffeological spaces such that the pre-image of each point has vector space structure but the subset diffeology on it might be finer than a vector space diffeology.
That this can actually happen is demonstrated, once again, by the Example 4.3 of \cite{CWtangent}: the internal tangent bundle of the coordinate axes in $\matR^2$ considered with the Hector's
diffeology (see Definition 4.1 of \cite{CWtangent} for details). As is shown in \cite{CWtangent}, the tangent space at the origin is \emph{not} a diffeological vector space for the subset
diffeology.\footnote{The existence of such examples motivates the introduction of the dvs diffeology on internal tangent bundles by the authors.}

In both \cite{vincent} (Theorem 5.1.6) and \cite{CWtangent} (Proposition 4.16), it is shown that the diffeology on the total space can be ``expanded'' to obtain a diffeological vector space bundle.
We now cite the latter result.

\begin{prop}\label{find:vector:bundle:diff-gy:prop} \emph{(\cite{CWtangent}, Proposition 4.6)}
Let $\pi:V\to X$ be a smooth surjective map between diffeological spaces, and suppose that each fibre of $\pi$ has a vector space structure. Then there is a smallest diffeology $\calD$ on $V$
which contains the given diffeology and which makes $V$ into a diffeological vector pseudo-bundle.
\end{prop}

The diffeology whose existence is affirmed in this proposition can be described explicitly (see \cite{CWtangent}, Remark 4.7). It is the diffeology generated by the linear combinations of plots of $V$
and the composite of the zero section with plots of $X$. More precisely, a map $\matR^m\ni U\to V$ is a plot of $\calD$ if and only if it is locally of form $u\mapsto r_1(u)q_1(u)+\ldots+r_k(u)q_k(u)$,
where $r_1,\ldots,r_k:U\to\matR$ are usual smooth functions (plots for the standard diffeology on $\matR$) and $q_1,\ldots,q_k:U\to V$ are plots for the pre-existing diffeology of $V$ such that there is
a single plot $p:U\to X$ of $X$ for which $\pi\circ q_i=p$ for all $i$.

\paragraph{Generating a vector pseudo-bundle diffeology} The above-mentioned proposition comes in handy, in particular, when constructing specific examples. An instance of this, and one which
we will encounter frequently below, is the case of the standard projection $\pi$ of $\matR^n$ onto its first $k$ coordinates (\emph{i.e.}, on $\matR^k$).

In this situation we will often have a map $p:U\to\matR^n$ such that $\pi\circ p$ is smooth for the chosen diffeology on $\matR^k$ and we wish to put a diffeology on $\matR^n$ that contains  $p$
and makes $\matR^n$ into a diffeological vector space over $\matR^k$. Then the smallest diffeology with this property is the diffeology obtained by applying Proposition \ref{find:vector:bundle:diff-gy:prop}
to the usual diffeology generated by $p$;\footnote{Possibly with whatever assumption allows us to obtain the usual underlying D-topology on $\matR^n$.} we call this diffeology a \textbf{diffeological vector
pseudo-bundle diffeology generated by $p$}.

Note that in general this diffeology is different not only from the usual diffeology generated by $p$, but also from the vector space diffeology on $\matR^n$ generated by it (the reason for the latter
is obviously that the operations on the fibres are not the usual vector space operations of $\matR^n$). Let us state the precise definition.

\begin{defn}
Let $V$ be a set, let $X$ be a diffeological space, and let $\pi:V\to X$ be a set map such that that for every $x\in X$ the pre-image $\pi^{-1}(x)$ carries a vector space structure.\footnote{We could also
say that $V$ is a disjoint union of vector spaces indexed by some $X$, and this indexing set is endowed with some diffeology; $\pi$ assigns to every $v\in V$ the index of the space to which $x$ belongs.}
Let $A$ be any collection of maps $p_i:U_i\to V$ such that each $p_i$ has the property that $\pi\circ p_i$ is a plot of $X$ (which implies, in particular, that every $U_i$ is a domain of some $\matR^{m_i}$),
and let $\calD$ be the diffeology on $V$ generated by $A$. The \textbf{pseudo-bundle diffeology on $V$ generated by $A$} is the smallest diffeology containing $\calD$ and such that the subset diffeology
on each $\pi^{-1}(x)$ is a vector space diffeology.
\end{defn}

Let us illustrate this definition with two simple examples that indicate the differences between the (diffeological space) diffeology, the vector space diffeology, and the pseudo-bundle diffeology, all three
generated by the same map on a vector space acting as the total space of a vector space fibration.

\begin{example}
Let $V=\matR^2$, and let $\pi:V\to\matR$ be the projection onto its first coordinate, so the base space $X$ is $\matR$, which we consider with its standard diffeology. Consider first the map
$p_1:\matR\to V$ defined by $p_1(u)=(0,|u|)$. The usual diffeology (\emph{i.e.}, without taking into account the vector space structure) generated on $V$ by $p_1$ is not a vector space diffeology;
indeed, the subset diffeology on the $x$-axis (which is a vector subspace of $V$) is the discrete diffeology,\footnote{Recall that this is the diffeology that includes constant maps only.} which is not
a vector space diffeology.\footnote{Whereas a vector space diffeology induces, again, a vector space diffeology on any subspace.} The vector space diffeology on $V$ generated by $p_1$ is the
product diffeology corresponding to the presentation $V=\matR\times\matR$, where the first copy of $\matR$ is endowed with the standard diffeology and the second one with the vector
space diffeology on $\matR$ generated by the plot $u\mapsto|u|$.

Let now $p_2:\matR^2\to V$ act by $p_2(u,v)=(u,u|v|)$. Let us first show that the usual diffeology on $V$ generated by $p_2$ is not a vector space diffeology. Indeed, any vector space diffeology on
$\matR^2$, the space underlying $V$, contains all plots of form $y\mapsto(0,f(y))$ for all usual smooth functions $f:\matR\to\matR$; such that a map cannot be obtained by smooth substitution in $p_2$,
because of the first coordinate being $0$. For the same reason it is not a pseudo-bundle diffeology, since the subset diffeology over $0$ would have to contain all smooth functions.

Let us now consider the vector space diffeology generated by $p_2$. A non-constant plot of it locally has form $(u,v)\mapsto(f_1(u)+\ldots+f_k(u),f_1(u)|g_1(v)|+\ldots+f_k(u)|g_k(v)|)$,
for some ordinary smooth functions $f_i,g_i$. Let us consider the subset diffeology on the fibre over some point $(x_0,0)$. A plot of it consists of all maps $(u,v)\mapsto f_1(u)|g_1(v)|+\ldots+f_k(u)|g_k(v)|$
such that $(f_1+\ldots+f_k)(u)\equiv x_0$, \emph{i.e.} $f_k=x_0-f_1-\ldots-f_{k-1}$. Thus, except in the case $x_0=0$, this property is not preserved by the summation. This means that the
diffeology in question is not a pseudo-bundle diffeology.
\end{example}

\paragraph{Constructing a pseudo-bundle from simpler ones} For many examples of diffeological vector pseudo-bundles (see \cite{vincent}, \cite{CWtangent}) such an object can be seen as the
result of a kind of a diffeological gluing of standard vector bundles along subsets composed of whole fibres. In the next section we describe a formalization of this idea, an operation of the
so-called \emph{diffeological gluing} (introduced in \cite{pseudobundle}, where more details appear).

\subsection{In place of local trivializations: diffeological gluing}\label{secn:diff:pseudo-bundles:diff:gluing:sect}

The main difference of diffeological vector pseudo-bundles from their usual counterparts (smooth vector bundles) is the absence of local trivializations; in this section we describe what can be a partial
substitute for these. The line of thinking that we follow is to consider the usual local trivializations as elementary pieces (building blocks) from which the whole bundle is reconstructed by some sort
of assembling (gluing along diffeomorphisms). We therefore adopt the same approach in the case of diffeological vector pseudo-bundles, by expanding the assortment of these building blocks,
as well as (more importantly) admitting gluings along maps that may not be diffeomorphisms and, especially, may not be defined on open\footnote{For D-topology, they are called \emph{D-open}.}
subsets of diffeological spaces being glued. What we obtain is the procedure of diffeological gluing, first between two diffeological spaces, then between two diffeological vector pseudo-bundles, along
smooth maps (in the case of pseudo-bundles, linear on fibres), which \emph{a priori} can be defined on any kind of subset. It should be noted that this procedure does extend the notion of local
trivializations, in the sense that a usual smooth vector bundle admitting a finite atlas of local trivializations, considered with its standard diffeology, is indeed the result of the diffeological gluing of the
trivial bundles that compose the atlas, along the transition maps (see \cite{pseudometric-pseudobundle} for details).

\paragraph{The idea of the construction} Suppose that we have two (locally trivial, so true vector bundles) diffeological vector pseudo-bundles $\pi_1:V_1\to X_1$ and $\pi_2:V_2\to X_2$. We wish to
describe a gluing operation on these, that would give us an object similar in nature to that appearing in the example(s) of Christensen-Wu. This obviously requires, first of all, a smooth map
$f:X_1\supset Y\to X_2$ and its lift to a smooth map $\tilde{f}:\pi_1^{-1}(Y)\to V_2$; this lift should be linear on the fibres.\footnote{It is of course essential that $\tilde{f}$ be defined on the whole fibres;
if not, the gluing can be done, of course, but some fibres of the result would not be vector spaces; thus, the result would not be a vector pseudo-bundle.} The idea is then to perform the usual
gluing (which we will talk more about below, perhaps in more in detail) simultaneously of $X_1$ to $X_2$ via $f$, and of $V_1$ to $V_2$ via $\tilde{f}$; and then to specify the diffeology obtained.

As an illustration (or clarification), we comment right away on how this construction would relate, for instance, to the example of the coordinate axes in $\matR^2$. It is not meant to produce it
immediately; rather, it describes the first step in the construction, by setting $X_1$ one of the axes with its subset diffeology and $V_1$ the corresponding internal tangent bundle (which is the usual
tangent bundle to $\matR$), the subset $Y$ is the origin, and finally $X_2$ is a single point and the corresponding bundle is the map $\pi_2$ that sends (another copy of) $\matR^2$, with the standard
diffeology, to this point. The map $f$ is obvious and sends the origin $(0,0)\in\matR^2$ to the point that composes $X_2$. The lift $\tilde{f}$ must be specified, since the options \emph{a priori} are
numerous. Indeed, $\tilde{f}$ is defined on $\matR=\pi_1^{-1}(0,0)$ and so is a linear map from $\matR$ to $\matR^2=\pi_2^{-1}(X_2)$. It thus can be described by its image, which is either the
origin or any $1$-dimensional linear subspace of $\matR^2$, with an uncountable number of possibilities in the latter case.\footnote{Since the diffeology on $\matR=\pi_1^{-1}(0,0)$ is
standard, it suffices for it to be a linear map; its smoothness then is automatic.}

\paragraph{Gluing of two diffeological spaces} Due to the concept of the quotient diffeology (see the definition in Section \ref{diffeology:defn:sect} and the original paper \cite{iglOrb}), the construction
is quite straightforward. On the level of the underlying topological spaces, it is a usual topological gluing, whose precise definition is as follows. Suppose we have two diffeological spaces $X_1$ and
$X_2$ that are glued together along some smooth map $f:X_1\supset Y\to X_2$ (the smoothness of $f$ is with respect to the subset diffeology of $Y$). We set $X_1\cup_f X_2=(X_1\sqcup X_2)/\sim$,
where $\sim$ is the following equivalence relation: $x_1\sim x_1$ if $x_1\in X_1\setminus Y$, $x_2\sim x_2$ if $x_2\in X_2\setminus f(Y)$, and $x_1\sim x_2$ if $x_1\in X_1$ and $x_2=f(x_1)$.

Now, for $X_1$ and $X_2$ diffeological spaces, there is a natural diffeology on $X_1\sqcup X_2$, namely the disjoint union diffeology (see Section \ref{diffeology:defn:sect}); and for whatever
equivalence relation exists on a diffeological space (which is $X_1\sqcup X_2$ in this case) there is the standard quotient diffeology on the quotient space. This is the \textbf{gluing diffeology}
on $X_1\cup_f X_2$.

\begin{example}\label{axes:by:gluing:wire:ex}
Take $X_1\subset\matR^2$ the $x$-coordinate axis and $X_2\subset\matR^2$ the $y$-coordinate axis; both are considered with the subset diffeology of $\matR^2$ (so it is the standard
diffeology of $\matR$). Gluing them at the origin yields, from the topological point of view, the same space that appears in the Christensen-Wu example. Note that its gluing diffeology, as has just
been described, is strictly finer than the subset diffeology of $\matR^2$, see Example 2.67 in \cite{watts}.
\end{example}

\paragraph{Gluing together diffeological vector pseudo-bundles} The operation of gluing of two diffeological spaces easily extends to the definition of gluing of two diffeological vector pseudo-bundles.
Let $X_1$ and $X_2$ be two diffeological spaces; let $\pi_1:V_1\to X_1$ and $\pi_2:V_2\to X_2$ be two diffeological vector pseudo-bundles over $X_1$ and $X_2$ respectively. Let
$f:X_1\supset Y\to X_2$ be a smooth map (we will frequently assume it to be injective, although this is not always necessary). Let also $\tilde{f}:\pi_1^{-1}(Y)\to V_2$ be a smooth map that is linear on
fibres and such that $\pi_2\circ\tilde{f}=f\circ(\pi_1)|_{\pi_1^{-1}(Y)}$ (namely, $\tilde{f}$ is a smooth fibrewise linear lift of $f$). The latter property yields a well-defined map
$$\pi:V_1\cup_{\tilde{f}}V_2\to X_1\cup_f X_2.$$ This map is indeed a diffeological vector pseudo-bundle (see \cite{pseudobundle}).

\begin{thm}
The map $\pi:V_1\cup_{\tilde{f}}V_2\to X_1\cup_f X_2$ is a diffeological vector pseudo-bundle.
\end{thm}

\paragraph{The extension from the standard case} As we stated in the beginning of the present section, the operation of diffeological gluing can be seen as an extension to diffeological vector
pseudo-bundles of the usual representation of smooth vector bundles by an atlas of local trivializations (this is also true of representing smooth manifolds by their atlas). The idea is to replace
the atlas of local trivializations by a finite collection of standard bundles of form $\matR^n\to\matR^k$ and with the difference that this collection would generally include bundles of different
dimensions (be of form $\{\matR^{n_i}\to\matR^{k_i}\}_{i=1}^m$, with $n_1,\ldots,n_m$ and $k_1,\ldots,k_m$ being in general not all equal to the same fixed $n$ and the same fixed $k$, respectively)
and the transition charts would consist of smooth injections rather than diffeomorphisms; perhaps calling such a collection a \emph{pseudo-atlas}.

Such an extension has its limits (we do not claim that every diffeological space is a result of diffeological gluings of a reasonable selection of other diffeological spaces; so much less we
make a similar claim regarding pseudo-bundles --- both claims in their full generality are in fact false). On the other hand, it is indeed an extension of the usual notion, in the sense that a smooth
vector bundle $\pi:E\to M$ which admits a finite atlas of local trivializations (for instance, if $M$ is compact) can be seen as a diffeological vector pseudo-bundle, with respect to its standard
diffeology, obtained by a finite number of diffeological gluings of several copies of the standard bundle $\matR^n\times\matR^k\to\matR^k$. The precise form of this statement (which is found in
\cite{pseudometric-pseudobundle}) is as follows.

To specify the notation, recall that a smooth vector bundle of rank $k$ is a smooth map $\pi:E\to M$ between two smooth manifolds $E$ and $M$ (we assume that $M$ admits a finite smooth atlas)
such that
\begin{enumerate}
\item[a)] for every $x\in M$ the fibre $\pi^{-1}(x)$ carries a vector space structure;
\item[b)] $M$ admits a finite atlas $\{(U_i,\varphi_i:U_i\to\matR^n)\}_{i=1}^m$ such that for every index $i=1,\ldots,m$ there is a fixed diffeomorphism $\psi_i:\pi^{-1}(U_i)\to U_i\times\matR^k$ such that
for every $x\in U_i$ the restriction of $\psi_i$ to the fibre $\pi^{-1}(x)$ is an vector space isomorphism $\pi^{-1}(x)\to\{x\}\times\matR^k$.
\end{enumerate}
The transition functions are then defined, for every pair of indices $i,j$ such that the intersection $U_i\cap U_j$ is non-empty, by setting
$g_{ij}(x)=\psi_j|_{\pi^{-1}(x)}\circ\left(\psi_i^{-1}|_{\{x\}\times\matR^k}\right)$ for all $x\in U_i\cap U_j$, which is a vector space isomorphism $\{x\}\times\matR^k\to\{x\}\times\matR^k$. Thus, each map
$g_{ij}$ is a map $U_i\cap U_j\to\mbox{GL}_k(\matR)$. We can then state the following:

\begin{thm}\label{smooth:bundle:equals:diffeol:bundle:thm}
Let $\pi:E\to M$ be a smooth vector bundle of rank $k$ over an $n$-dimensional manifold $M$ that admits a finite atlas of $m$ local trivializations. Then $\pi$ is the result of gluing of $m$
diffeological vector pseudo-bundles $\pi_i:\matR^{n+k}\to\matR^n$, where
$$\pi_i=\varphi_i\circ\pi\circ\psi_i^{-1}\circ(\varphi_i^{-1}\times\mbox{Id}_{\matR^k})\mbox{ for }i=1,\ldots,m,$$ and the gluing $(\tilde{f}_{ij},f_{ij})$ between $\pi_i$ and $\pi_j$ is given by the maps
$$f_{ij}=\varphi_j\circ(\varphi_i|_{U_i\cap U_j})^{-1},\mbox{ and}$$
$$\tilde{f}_{ij}(y,v)=(f_{ij}(y),g_{ij}(\varphi_i^{-1}(y))v),\mbox{ for }y\in U_i\cap U_j\mbox{ and }v\in\pi_i^{-1}(y).$$
\end{thm}

\paragraph{The motivation for the choice of diffeology} We conclude this section with commenting on the reasons for our specific choice of the diffeology on a diffeological space obtained by gluing of
two other spaces (in particular, when the gluing is between two total spaces of pseudo-bundles). These reasons have to do with the fact that the gluing diffeology as defined above possesses some
important properties that we state below; these properties are heavily used in the proofs of various results that are cited below, although we do not give the proofs themselves. (Another reason, of course,
is that this is a very natural definition).

\begin{lemma}
Let $X_1$ and $X_2$ be two diffeological spaces, and let $f:X_1\supset Y\to X_2$ be a smooth map. Then:\\
1) the obvious inclusions $i_1:X_1\setminus Y\hookrightarrow X_1\cup_f X_2$ and $i_2:X_2\hookrightarrow X_1\cup_f X_2$ are inductions;\\
2) every plot $p:U\to X_1\cup_f X_2$ locally has the following characterization: either there exists a plot $p_2:U\supset U'\to X_2$ of $X_2$ such that $p|_{U'}=i_2\circ p_2$, or there exists a plot
$p_1:U\supset U'\to X_1$ of $X_1$ such that
$$p(u')=\left\{\begin{array}{ll}i_1(p_1(u')) & \mbox{if }p_1(u')\in X_1\setminus Y,\\
i_2(f(p_1(u'))) & \mbox{if }p_1(u')\in Y.\end{array}\right.$$
\end{lemma}

Note in particular that $i_1(X_1\setminus Y)$ and $i_2(X_2)$ cover $X_1\cup_f X_2$.

\paragraph{On the alternative choices} The gluing diffeology is perhaps the one most closely related to the initial two diffeologies. There are of course other choices that we now briefly mention (these
do not enjoy the properties indicated in the above Lemma, in particular, they do not have the crucial property 2); crucial in that it much helps with various proofs).

\begin{example}\label{pseudobundle:coarse:ex}
Let $X$ be the diffeological space of the Example \ref{axes:by:gluing:wire:ex}, and let us describe a diffeological vector pseudo-bundle over $X$ which has non-standard fibres. Write $X=X_1\cup X_2$,
where $X_1=\{(x,0)\}$ (the $x$-axis) and $X_2=\{(0,y)\}$ (the $y$-axis). We now take three copies of $\matR^2$, which we denote by $V_1$, $V_2$, and $V_0$; we will identify, as needed, $X_1$ with
the $x$-axis of $V_1$ and $X_2$ with the $y$-axis of $V_2$. Their crossing point, the origin, will be identified with the origin of $V_0$.

Consider the projections $\pi_1:V_1\to X_1$ (the projection on the $x$-axis), $\pi_2:V_2\to X_2$ (the projection to the $y$-axis), and $\pi_0:V_0\to X_1\cap X_2\in X$ (the projection of the whole space
to the origin). Consider also maps $f_1:V_1\supset\{(0,y)\}\to V_0$, where $f_1(0,y)=(0,y)$, and $f_2:V_2\supset\{(x,0)\}\to V_0$, where $f_2(x,0)=(x,0)$; denote by $V$ the result of the usual topological
gluing of $V_1$ and $V_2$ to $V_0$ along the maps $f_1$ and $f_2$ respectively. Clearly, the maps $\pi_1$, $\pi_2$, $\pi_0$ yield a well-defined map $\pi:V\to X$ (which is continuous in the usual
sense). Furthermore, the pre-image $\pi^{-1}(x)$ of every $x\in X$ inherits a vector space structure from one of $V_1$, $V_2$, $V_0$.

The space $X$ already carrying a diffeology $\calD_X$, there is a well-defined pullback, which we denote by $\calD_V$, of it to $V$, via the map $\pi$; let us show that $\calD_V$ induces the coarse
diffeology on fibres. Let $p:U\to V$ be a plot of $\calD_V$; then $\pi\circ p$ is a plot of $\calD_X$, \emph{i.e.}, as has been observed above, it is either a map of form $u\mapsto(p_1(u),0)$ or a
map of form $u\mapsto(0,p_2(u))$, where $p_1$ and $p_2$ are ordinary smooth maps with values in $\matR$. This implies, first of all, that the image of $p$ is contained in either $V_1\cup V_0$ or
$V_2\cup V_0$. If it is entirely contained in $V_0$, it can be any map, since its composition with $\pi$ is always a constant map. This implies (recall that the pullback diffeology is the coarsest diffeology
such that the pulling-back map is smooth) that $V_0$ has coarse diffeology. Furthermore, if the image of $p$ is contained in, say, $V_1$, then writing it as $p(u)=(p_1(u),p_2(u))$, we obtain that
$p_1$ is an ordinarily smooth map, while $p_2$ is any map. Therefore the fibre over any point of $X_1$ has coarse diffeology; an analogous conclusion can be obtained for $X_2$. Thus, the subset
diffeology on any fibre of $V$ is the coarse diffeology, although the diffeology of $V$ as a whole is not the coarse one.\footnote{It is quite clear that the pullback diffeology is way too big; typically, we
would like to at least preserve the ordinary topology of $V$. For this to happen, the subset diffeology on fibres should include continuous maps only.}
\end{example}

\subsection{Constructing a desired fibre}

For diffeological vector pseudo-bundles, there continue to exist all the same operations that are performed with the usual smooth vector bundles, such as taking direct sums, tensor products, and dual
pseudo-bundles. They were probably first described in \cite{vincent} and indeed necessitate a separate description, because they cannot be defined in the classic way, using local trivializations, since
pseudo-bundles do not have them. The typical procedure for the pseudo-bundles is to describe them on individual fibres (where they are just the same operations on diffeological vector spaces ---
these are described in \cite{vincent} and \cite{wu}, and were briefly recalled above), and then explain which diffeology is assigned on the union of the new fibres thus obtained. 

\paragraph{Sub-bundles and quotient pseudo-bundles} It is useful to note that if we are given a pseudo-bundle $\pi:V\to X$, and a vector subspace $W_x\leqslant\pi^{-1}(x)$ of each fibre, then the union
$W=\cup_{x\in X}W_x\subset V$, endowed with the subset diffeology, is always a diffeological vector pseudo-bundle on its own. Likewise, the collection of fibrewise quotient spaces
$V/W=\cup_{x\in X}\pi^{-1}(x)/W_x$ is trivially a quotient space of $V$; the corresponding quotient diffeology induces the same (subset) diffeology on each fibre $\pi^{-1}(x)/W_x$ of it; the latter
diffeology coincides with the quotient diffeology relative to the projection $\pi^{-1}(x)\to \pi^{-1}(x)/W_x$. These properties come in handy when describing the further constructions with pseudo-bundles.

\paragraph{The direct product bundle} Let $\pi_1:V_1\to X$ and $\pi_2:V_2\to X$ be two diffeological vector pseudo-bundles with the same base space. The total space of the product bundle consists
of fibrewise direct products, $V_1\times_X V_2=\cup_{x\in X}\pi_1^{-1}(x)\times\pi_2^{-1}(x)$. The \textbf{product bundle diffeology}\footnote{The result is usual a pseudo-bundle rather than a standard
(locally trivial) bundle; we do call it a product bundle, just as we say \emph{sub-bundle} rather than the cumbersome \emph{sub-pseudo-bundle}, to avoid making the terminology too
complicated.} (see \cite{vincent}, Definition 4.3.1) is the coarsest diffeology such that the fibrewise defined projections are smooth; this diffeology includes, for instance, for each $x\in X$ all maps
of form $(p_1,p_2)$, where $p_i:U\to(V_i)_x$ is a plot of $(V_i)_x$ for $i=1,2$.

\paragraph{The direct sum pseudo-bundle} It suffices to add to the above direct product bundle the operations on all fibres (that are defined in the usual manner), to get a well-defined direct sum
pseudo-bundle $\pi_1\oplus\pi_2:V_1\oplus V_2\to X$. Each fibre of it is the usual direct sum of the corresponding fibres, in the sense of diffeological vector spaces,
$(\pi_1\oplus\pi_2)^{-1}(x)=\pi_1^{-1}(x)\oplus\pi_2^{-1}(x)$ for all $x\in X$.

\paragraph{The tensor product pseudo-bundle} This notion was also described in \cite{vincent} (see Definition 5.2.1); it is again defined fibrewise as the collection of the tensor products of all fibres
over the same point. Its diffeology is defined as follows. Let $\pi_1\times\pi_2:V_1\times V_2\to X$ be the direct product bundle; for each $x\in X$ let
$\phi_x:\pi_1^{-1}(x)\times\pi_2^{-1}(x)\to\pi_1^{-1}(x)\otimes\pi_2^{-1}(x)$ be the universal map onto the corresponding tensor product of diffeological vector spaces. The collection of maps $\phi_x$
defines a map $\phi:V_1\times V_2\to V_1\otimes V_2=:\cup_{x\in X}\pi_1^{-1}(x)\otimes\pi_2^{-1}(x)$. Let also $Z_x$ be the kernel of $\phi_x$, for all $x\in X$; recall that $Z:=\cup_{x\in X}Z_x$ is
a vector sub-bundle for the subset diffeology. The \textbf{tensor product pseudo-bundle diffeology} on the total space space $V_1\otimes V_2$ tensor product pseudo-bundle
$\pi_1\otimes\pi_2:V_1\otimes V_2\to X$ is the pushforward of the diffeology of $V_1\times V_2$ by the map $\phi$. Equivalently, it can be described as the quotient diffeology on the quotient
pseudo-bundle $(V_1\times V_2)/Z$ (the equivalence follows from the above-mentioned properties of quotient pseudo-bundles). Each fibre $(\pi_1\otimes\pi_2)^{-1}(x)$ of the tensor product
pseudo-bundle is diffeomorphic, as a diffeological vector space, to the tensor product $\pi_1^{-1}(x)\otimes\pi_2^{-1}(x)$ of the corresponding fibres.

\paragraph{The dual pseudo-bundle} It remains to define the dual pseudo-bundle, which is the most intricate case. This definition is also available in \cite{vincent}, Definition 5.3.1. Let $\pi:V\to X$
be a diffeological vector pseudo-bundle; the dual (pseudo-)bundle of it is obtained by taking the union $\cup_{x\in X}(\pi^{-1}(x))^*=:V^*$ (where $(\pi^{-1}(x))^*$ is the diffeological dual of the
diffeological vector space $\pi^{-1}(x)$) with the obvious projection, which we denote $\pi^*$. The \textbf{dual bundle diffeology} on $V^*$ is the finest diffeology on $V^*$ such that: 1) the composition
of any plot with $\pi^*$ is a plot of $X$; and 2) the subset diffeology on each fibre coincides with its diffeology as the diffeological dual $(\pi^{-1}(x))^*$ of fibre $\pi^{-1}(x)$.

The proof that such a diffeology exists, and more explicit characterization of its plots were given in \cite{vincent}. This explicit characterization is as follows.

\begin{lemma} \emph{(\cite{vincent}, Definition 5.3.1 and Proposition 5.3.2)}
Let $U$ be a domain of some $\matR^l$. A map $p:U\to V^*$ is a plot for the dual bundle diffeology on $V^*$ if and only if for every plot $q:U'\to V$ the map $Y'\to\matR$ acting by $(u,u')\mapsto p(u)(q(u'))$,
where $Y'=\{(u,u')|\pi^*(p(u))=\pi(q(u'))\in X\}\subset U\times U'$, is smooth for the subset diffeology of $Y'\subset\matR^{l+\dim(U')}$ and the standard diffeology of $\matR$.
\end{lemma}

Just like it happens with the diffeological duals of diffeological vector spaces, the dual pseudo-bundles can be quite different from what one obtains in the usual smooth case. The following is a more
extreme case of this.

\begin{example}
Let $\pi:V\to X$ be the diffeological vector pseudo-bundle with the base space $X$ of Example \ref{axes:by:gluing:wire:ex}, where $V$ is the space (endowed with the appropriate $\pi$) that we have
constructed in the Example \ref{pseudobundle:coarse:ex}. Since all fibres have coarse diffeology, their diffeological duals are always zero spaces, which means that the dual bundle in this case is just a
trivial covering, in the usual meaning of the term, of $X$ by itself.
\end{example}

On the other hand, if we have a usual smooth vector bundle (of finite rank and dimension) then it can be seen as a diffeological vector pseudo-bundle for its standard diffeology. Then, as long as
it admits a finite atlas of local trivializations, its diffeological dual pseudo-bundle is the same as its usual dual bundle (see \cite{pseudometric-pseudobundle}). In fact, the usual construction
of the dual bundle via local trivializations is mimicked by the concept of the diffeological gluing described in the previous section.

Finally, here is an example which is somewhat in between.

\begin{example}\label{dual:diff-gy:|xy|:ex}
Let $V$ be $\matR^2$ endowed with the vector pseudo-bundle diffeology generated by the plot $q:\matR^2\to\matR^2$, with $q(x,y)=(x,|xy|)$, let $X$ be the standard $\matR$, which we identify
with the $x$-axis of $\matR^2=V$, and let $\pi$ be the projection of the latter onto its first coordinate. Let now $U$ be a domain of some $\matR^m$, and let $p:U\to V^*$ be a plot for the dual
pseudo-bundle diffeology. Note that in some sense we can write $p(u)=(p_1(u),p_2(u))$, where $p_1(u)=\pi^*(p(u))$ determines the fibre (which is given by the first coordinate) and $p_2(u)$
determines a linear map on this fibre; this map can be identified with $p_2(u)e^2$. Now, the smoothness of the projection $\pi^*$ is equivalent to $p_1$ being an ordinary smooth map $U\to\matR$.
As for $p_2$, consider the evaluation of $p(u)$ on the plot $q$; we get that $p(u)(q(x,y))=p_2(u)|xy|$. This implies that outside of the subset $Y'\setminus p_1^{-1}(0)$ the function $p_2$ must be
identically zero (this agrees with the fact that it is $p_2$ that defines the subset diffeology of each fibre; recall that, except for the fibre $(\pi^*)^{-1}(0)$, the dual is trivial and, in particular,
is standard, while, on the other hand, the diffeology generated by the zero function, or any other constant function, is precisely the standard one). Furthermore, if $p_1^{-1}(0)$ has non-empty interior
(which it might well have, since the only restriction on $p_1$ is that it be smooth), then for any open subset $U_1\subset\mbox{Int}(p_1^{-1}(0))$ such that its closure is also contained in
$\mbox{Int}(p_1^{-1}(0))$, and for any smooth function $U_1\to\matR$, we can find $p_2$ that coincides on $U_1$ with this function and satisfies all the required conditions. This also agrees
with the fact that the fibre at $0$ is the standard $\matR$.

To summarize the above discussion, we state that $p=(p_1,p_2)$, where $p_1$ and $p_2$ are two smooth functions such that $p_2^{-1}(0)\supset(Y'\setminus p_1^{-1}(0))$.\footnote{We could
also summarize this as $Y'=p_1^{-1}(0)\cup p_2^{-1}(0)$.} Note also that this condition ensures also the smoothness of any composition of $p$ with $q\circ f$, for any smooth function $f$.
\end{example}

\subsection{Trivial bundles $\matR^n\to\matR^k$}

What at the moment we mean by a trivial bundle is one whose underlying topological map is the projection $\pi_{n,k}$ of $\matR^n$ to its first $k$ coordinates, \emph{i.e.}, to the subspace
that is naturally identified with $\matR^k$. As we illustrate below, there are choices of diffeologies on $\matR^n$ (particularly) and on $\matR^k$ (this seems less important) that make the same bundle
non-trivial from the diffeological point of view, meaning that, although all fibres are isomorphic as vector spaces, they are not diffeomorphic as diffeological (vector) spaces. In this section we
collect several examples of this, but also of other diffeologies that can be put on the two spaces, producing different instances of pseudo-bundles. Finally, we observe that such pseudo-bundles, being
among the simplest ones, can be used as building blocks for assembling, via the gluing construction, a wealth of more complicated pseudo-bundles. It is this class of pseudo-bundles to
which we intend to apply our further considerations; this does leave aside many (even simple) instances of pseudo-bundles that one might consider, but on the other hand it still produces a reasonably
wide class of them.

\paragraph{Diffeologically trivial pseudo-bundles} From the point of view of the above-mentioned building-blocks' idea, these are the natural starting point (we will explain later why we shall avoid
non-locally trivial pseudo-bundles of form $\matR^n\to\matR^k$). For these, we first consider the largest diffeology possible: the pullback diffeology.

\begin{example}\label{coarse:fibres:over:standard:ex}
Consider the projection $\pi$ of $\matR^2$ to its $x$-axis $X\cong\matR$, the latter being endowed with the standard diffeology of $\matR$; the pre-image $\pi^{-1}(x)$ of any point $x\in X$ has an
obvious vector structure.\footnote{This vector space structure is obtained by representing $\matR^2$ as the direct product $\matR\times\matR$ (with respect to the standard coordinates); each
fibre then has form $\{x\}\times\matR$, and the vector space structure is that of the second factor.} Let us endow $\matR^2$ with the pullback of this diffeology by the map $\pi$; let $p:U\to\matR^2$ be
a plot of this pullback diffeology, written as $p(u)=(p_1(u),p_2(u))$. Then $(\pi\circ p)(u)=p_1(u)$, and this has to be an ordinary smooth map. But since no condition is thus imposed on $p_2$, and the
pullback diffeology is defined as the \emph{coarsest} diffeology such that $\pi\circ p$ is a plot of $X$ (\emph{i.e.}, simply smooth in this case), $p_2$ can be \emph{any} map. In particular, every fibre of
$\pi$ has coarse diffeology.
\end{example}

\begin{example}
Consider $V=\matR^n$ (a finite-dimensional diffeological vector space whose underlying space is identified with $\matR^n$) and $X=\matR^k$, which is naturally identified with the subspace of $V$
spanned by the first $k$ coordinate axes. This defines the obvious projection $\pi_{n,k}:V\to X$.

Let $\calD_X$ be any vector space diffeology\footnote{It does not have to be vector space diffeology; our choice is somewhat arbitrary.} on $X$; let $\calD_V$ be its pullback to $V$ by the map $\pi_{n,k}$.
Let $p:U\to V=\matR^n$ be a plot for $\calD_V$ written as $p(u)=(p_1(u),\ldots,p_n(u))$. Then by the definition of pullback diffeology the map $u\mapsto(p_1(u),\ldots,p_k(u))$ is a plot of $\calD_X$, while
the map $u\mapsto(p_{k+1}(u),\ldots,p_n(u))$ is a plot for the coarse diffeology on $\matR^{n-k}$ (and the statement is also \emph{vice versa}). Thus, denoting the coarse diffeology on $\matR^{n-k}$ by
$\hat{\calD}_{n-k}$, we can write $\calD_V=\calD_X\times\hat{\calD}_{n-k}$, with the obvious meaning.
\end{example}

The pullback diffeology is the largest diffeology making the projection $\pi_{n,k}$ smooth,\footnote{The smallest of such diffeologies is obviously the fine (standard) diffeology on $\matR^n$.} but it is not
particularly interesting nor is it desirable (if nothing else, it does not induce the usual topology on $\matR^n$). Many other diffeologies can be constructed, however, by taking any vector space diffeology
$\tilde{\calD}_{n-k}$ on $\matR^{n-k}$ and setting $\calD_V$ to be the product diffeology on $\matR^n=\matR^k\times\matR^{n-k}$ coming from $\calD_X$ on the first factor and $\tilde{\calD}_{n-k}$ on
the second. The biggest ``sensible'' choice for the diffeology $\tilde{\calD}_{n-k}$ seems to be that of the diffeology consisting of all continuous, with respect to the usual topology, maps to $\matR^{n-k}$;
this is the largest diffeology whose underlying D-topology coincides with the usual topology of $\matR^n$; it also seems reasonable to ask the same of the diffeology on the base space $\matR^k$.

\paragraph{A non-locally trivial pseudo-bundle $\matR^2\to\matR$} Although we have said already that we will not include such instances in our main treatment, for reasons of completeness we
describe an example of how such a pseudo-bundle arises (and what it looks like).

\begin{example}
Let us take $V=\matR^2$ and $X=\matR$ identified with the $x$-axis of $V$. Endow $X$ with the standard diffeology and $V$ with the pseudo-bundle diffeology generated by the map $p:\matR^2\to V$
acting by $p(x,y)=(x,|xy|)$. The map $\pi=\pi_{2,1}$ is obviously smooth with respect to this diffeology; let us consider the diffeology on a given fibre $\pi^{-1}(x_0)$ for an arbitrary $x_0\in X=\matR$.
The subset diffeology on this fibre is the vector space diffeology generated by the map $y\mapsto |x_0|\cdot|y|$; this is the standard diffeology if $x_0=0$ (the generating plot is just a constant map,
so the generated diffeology is the finest vector space diffeology, namely, the standard one), and a non-standard one if $x_0\neq 0$. This implies the bundle in question is not trivial as a
diffeological vector pseudo-bundle, although it is so as a topological bundle. In fact, it is not even locally trivial as diffeological pseudo-bundle: its fibres are all isomorphic as usual vector spaces,
but they are not all diffeomorphic, since there is one fibre whose diffeology is different from that of all the others.
\end{example}

\subsection{Gluing and operations: commutativity conditions and diffeomorphisms}

We now turn to considering the behavior of the operation of gluing for pseudo-bundles with respect to the (diffeological counterparts of) usual vector bundles' operations on them (see \cite{pseudobundle}). In 
all cases but one there is a simple description of this behavior --- they commute, --- and the one (expected) exception is the operation of taking the dual pseudo-bundle; the reason stems from the fact that, on 
one hand, the operation of gluing is not symmetric and, on the other, the operation of taking duals is covariant.

Before proceeding with the details of what has just been said, we introduce some further notation notation, that we will use from this point onwards. Since in what follows we will frequently find
ourselves working with more than one pair of pseudo-bundles, each one forming a glued pseudo-bundle, we modify the notation for the corresponding standard inductions. Specifically, if we have
a pseudo-bundle $\xi_1\cup_{(\tilde{h},h)}\xi_2:W_1\cup_{\tilde{h}}W_2\to Z_1\cup_h Z_2$, where $h$ has the domain of definition $Y$, then we write
$j_1^{W_1}:W_1\setminus\tilde{h}^{-1}(Y)\to W_1\cup_{\tilde{h}}W_2$ for the corresponding standard induction (which was previously denoted just by $j_1$). Likewise, we will have
$j_2^{W_2}:W_2\to W_1\cup_{\tilde{h}}W_2$, $i_1^{Z_1}:Z_1\setminus Y\to Z_1\cup_h Z_2$, and $i_2^{Z_2}:Z_2\to Z_1\cup_h Z_2$.

\subsubsection{The switch map}

As follows from its definition, the operation of gluing for diffeological spaces is asymmetric. However, if we assume that gluing map $f$ is a diffeomorphism with its image then obviously, we
can use its inverse to perform the gluing in the reverse order, with the two results, $X_1\cup_f X_2$ and $X_2\cup_{f^{-1}}X_1$, being canonically diffeomorphic via the so-called \textbf{switch map}
$$\varphi_{X_1\leftrightarrow X_2}:X_1\cup_f X_2\to X_2\cup_{f^{-1}}X_1.$$ Using the notation just introduced, this map can be described by
$$\left\{\begin{array}{ll}
\varphi_{X_1\leftrightarrow X_2}(i_1^{X_1}(x))=i_2^{X_1}(x) & \mbox{for }x\in X_1\setminus Y,\\
\varphi_{X_1\leftrightarrow X_2}(i_2^{X_2}(f(x)))=i_2^{X_1}(x) & \mbox{for }x\in Y,\\
\varphi_{X_1\leftrightarrow X_2}(i_2^{X_2}(x))=i_1^{X_2}(x) & \mbox{for }x\in X_2\setminus Y.\end{array}\right.$$ This is well-defined, not only because the maps $i_1^{X_1}$ and $i_2^{X_2}$
are injective with disjoint ranges covering $X_1\cup_f X_2$, but also because $f$ is a diffeomorphism with its image.

\subsubsection{Gluing and operations}

Diffeological gluing of pseudo-bundles is relatively well-behaved with respect to the usual operations on vector bundles. More precisely, it commutes with the direct sum and the tensor product,
while he situation is somewhat more complicated for the dual pseudo-bundles, see \cite{pseudobundles} (the facts needed are recalled below).

\paragraph{Direct sum} Gluing of diffeological vector pseudo-bundles commutes with the direct sum in the following sense. Given a gluing along $(\tilde{f},f)$ of a pseudo-bundle $\pi_1:V_1\to X_1$
to a pseudo-bundle $\pi_2:V_2\to X_2$, as well as a gluing along $(\tilde{f}',f)$ of a pseudo-bundle $\pi_1':V_1'\to X_1$ to a pseudo-bundle $\pi_2':V_2'\to X_2$, there are two natural
pseudo-bundles that can be obtained by applying to them the operations of gluing and direct sum, namely
$$(\pi_1\cup_{(\tilde{f},f)}\pi_2)\oplus(\pi_1'\cup_{(\tilde{f}',f)}\pi_2'):(V_1\cup_{\tilde{f}}V_2)\oplus(V_1'\cup_{\tilde{f}'}V_2')\to X_1\cup_f X_2\mbox{ and }$$
$$(\pi_1\oplus\pi_1')\cup_{(\tilde{f}\oplus\tilde{f}',f)}(\pi_2\oplus\pi_2'):(V_1\oplus V_1')\cup_{\tilde{f}\oplus\tilde{f}'}(V_2\oplus V_2')\to X_1\cup_f X_2;$$ they are diffeomorphic as pseudo-bundles, that is,
there exists a fibrewise linear diffeomorphism
$$\Phi_{\cup,\oplus}:(V_1\cup_{\tilde{f}}V_2)\oplus(V_1'\cup_{\tilde{f}'}V_2')\to(V_1\oplus V_1')\cup_{\tilde{f}\oplus\tilde{f}'}(V_2\oplus V_2')$$
(see below) that covers the identity map on the base $X_1\cup_f X_2$.

\paragraph{Tensor product} What has just been said about the direct sum, applies to the tensor product as well. Specifically, we obtain two \emph{a priori} different pseudo-bundles
$$(\pi_1\cup_{(\tilde{f},f)}\pi_2)\otimes(\pi_1'\cup_{(\tilde{f}',f)}\pi_2'):(V_1\cup_{\tilde{f}}V_2)\otimes(V_1'\cup_{\tilde{f}'}V_2')\to X_1\cup_f X_2\mbox{ and }$$
$$(\pi_1\otimes\pi_1')\cup_{(\tilde{f}\otimes\tilde{f}',f)}(\pi_2\otimes\pi_2'):(V_1\otimes V_1')\cup_{\tilde{f}\otimes\tilde{f}'}(V_2\otimes V_2')\to X_1\cup_f X_2,$$ which turn out to be diffeomorphic via
$$\Phi_{\cup,\otimes}:(V_1\cup_{\tilde{f}}V_2)\otimes(V_1'\cup_{\tilde{f}'}V_2')\to(V_1\otimes V_1')\cup_{\tilde{f}\otimes\tilde{f}'}(V_2\otimes V_2')$$ covering the identity on $X_1\cup_f X_2$.

\paragraph{The dual pseudo-bundle} The case of dual pseudo-bundles is substantially different. For one thing, to even make sense of the commutativity question, we must assume that $f$ is invertible
(and so the above-mentioned switch map is defined). However, even with this assumption, the operation of gluing does \emph{not} commute with that of taking duals, for the following reason.
Let $\pi_1:V_1\to X_1$ and $\pi_2:V_2\to X_2$ be two diffeological vector pseudo-bundles, and let $(\tilde{f},f)$ be a gluing between them; consider the pseudo-bundle
$\pi_1\cup_{(\tilde{f},f)}\pi_2:V_1\cup_{\tilde{f}}V_2\to X_1\cup_f X_2$ and the corresponding dual pseudo-bundle
$$(\pi_1\cup_{(\tilde{f},f)}\pi_2)^*:(V_1\cup_{\tilde{f}}V_2)^*\to X_1\cup_f X_2;$$ compare it with the result of the induced gluing, which is along the pair $(\tilde{f}^*,f)$, of $\pi_2^*:V_2^*\to X_2$ to
$\pi_1^*:V_1^*\to X_1$, that is, the pseudo-bundle
$$\pi_2^*\cup_{(\tilde{f}^*,f)}\pi_1^*:V_2^*\cup_{\tilde{f}^*}V_1^*\to X_2\cup_{f^{-1}}X_1.$$ It then follows from the construction itself that for any $y\in Y$ (recall that $Y$ is the domain of gluing) we have
$$((\pi_1\cup_{(\tilde{f},f)}\pi_2)^*)^{-1}(i_2^{X_2}(f(y)))\cong(\pi_2^{-1}(f(y)))^*\mbox{ and }(\pi_2^*\cup_{(\tilde{f}^*,f)}\pi_1^*)^{-1}(i_2^{X_1}(y))\cong(\pi_1^{-1}(y))^*;$$
since $i_2^{X_2}(f(y))$ and $i_2^{X_1}(y)$ are related by the switch map, for the two pseudo-bundles to be diffeomorphic (in a way that we want them to be) the two vector spaces $(\pi_2^{-1}(f(y)))^*$
and $(\pi_1^{-1}(y))^*$ must be diffeomorphic, and \emph{a priori} they are not.\footnote{They may have different dimensions.}

Thus, we obtain one condition necessary for there being a diffeomorphism $(V_1\cup_{\tilde{f}}V_2)^*\cong V_2^*\cup_{\tilde{f}^*}V_1^*$, which is that $(\pi_2^{-1}(f(y)))^*\cong(\pi_1^{-1}(y))^*$
for all $y\in Y$. We do note right away that this condition may not be sufficient, in the sense that two pseudo-bundles over the same base may have all the respective fibres diffeomorphic, without
being diffeomorphic themselves (this can be illustrated by the standard example of open annulus and open M\"obius strip, both of which, equipped with the standard diffeology\footnote{The one
determined by their usual smooth structure.} can be seen as pseudo-bundles over the circle). Thus, in general we impose a certain \emph{gluing-dual commutativity condition} (see below) as
an assumption, although later on we will also discuss how it correlates with other conditions (see Section 5).

\subsubsection{The commutativity diffeomorphisms}

We now say more about the commutativity diffeomorphisms mentioned in the previous section.

\paragraph{The diffeomorphism $\Phi_{\cup,\oplus}$} As we have stated above, this diffeomorphism always exists and is defined as a map
$$\Phi_{\cup,\oplus}:(V_1\cup_{\tilde{f}}V_2)\oplus(V_1'\cup_{\tilde{f}'}V_2')\to(V_1\oplus V_1')\cup_{\tilde{f}\oplus\tilde{f}'}(V_2\oplus V_2'),$$ that covers the identity map on  $X_1\cup_f X_2$ and
is given by the following identities:
$$\Phi_{\cup,\oplus}\circ(j_1^{V_1}\oplus j_1^{V_1'})=j_1^{V_1\oplus V_1'}\,\,\mbox{ and }\,\, \Phi_{\cup,\oplus}\circ(j_2^{V_2}\oplus j_2^{V_2'})=j_2^{V_2\oplus V_2'}.$$

\paragraph{The diffeomorphism $\Phi_{\cup,\otimes}$} Also in this case, there is always a diffeomorphism
$$\Phi_{\cup,\otimes}:(V_1\cup_{\tilde{f}}V_2)\otimes(V_1'\cup_{\tilde{f}'}V_2')\to(V_1\otimes V_1')\cup_{\tilde{f}\otimes\tilde{f}'}(V_2\otimes V_2')$$ covering the identity map. It is given by
$$\Phi_{\cup,\otimes}\circ(j_1^{V_1}\otimes j_1^{V_1'})=j_1^{V_1\otimes V_1'}\,\,\mbox{ and }\,\, \Phi_{\cup,\otimes}\circ(j_2^{V_2}\otimes j_2^{V_2'})=j_2^{V_2\otimes V_2'}.$$

\paragraph{The gluing-dual commutativity conditions, and diffeomorphism $\Phi_{\cup,*}$} We have already explained that the gluing-dual commutativity is far from being always present. Here we
define what it actually means for this commutativity to occur, without discussing under which conditions it does (later on we discuss some instances, but we do not have a complete answer).
Specifically, we say that the \textbf{gluing-dual commutativity condition} holds, if there exists a diffeomorphism
$$\Phi_{\cup,*}:(V_1\cup_{\tilde{f}}V_2)^*\to V_2^*\cup_{\tilde{f}^*}V_1^*$$ that covers the switch map, that is,
$$(\pi_2^*\cup_{(\tilde{f}^*,f^{-1})}\pi_1^*)\circ\Phi_{\cup,*}=\varphi_{X_1\leftrightarrow X_2}\circ(\pi_1\cup_{(\tilde{f},f)}\pi_2)^*,$$ and such that the following are true:
$$\left\{\begin{array}{ll}
\Phi_{\cup,*}\circ((j_1^{V_1})^*)^{-1}=j_2^{V_1^*} & \mbox{on }(\pi_2^*\cup_{(\tilde{f}^*,f^{-1})}\pi_1^*)^{-1}(i_2^{X_1}(X_1\setminus Y)),\\
\Phi_{\cup,*}\circ((j_2^{V_2})^*)^{-1}=j_2^{V_1^*}\circ\tilde{f}^* & \mbox{on }(\pi_2^*\cup_{(\tilde{f}^*,f^{-1})}\pi_1^*)^{-1}(i_2^{X_1}(Y)),\\
\Phi_{\cup,*}\circ((j_2^{V_2})^*)^{-1}=j_1^{V_2^*} & \mbox{on }(\pi_2^*\cup_{(\tilde{f}^*,f^{-1})}\pi_1^*)^{-1}(i_1^{X_2}(X_2\setminus f(Y))).\end{array}\right.$$

\section{Diffeological pseudo-metrics on diffeological vector pseudo-bundles}\label{pseudo-metrics:on:pseudo-bundles:sect}

To proceed with our discussion we now need a diffeological counterpart of a Riemannian metric, and it is not immediately clear what this should be. In this section we consider a notion of a
\emph{pseudo-metric on a pseudo-bundle}, something which comes as close as possible to the standard notion, although it has its own limitations, the first of which is that it does not always exist.
Specifically, in this section we show that a pseudo-metric, which on any individual fibre is supposed to be the best possible substitute for the scalar product (such a substitute can easily be defined for
any finite-dimensional diffeological vector space, and is called just a pseudo-metric on such), does not always exist on the pseudo-bundle as a whole.

Another item that we point out is that our aim at this moment is to define a pseudo-metric, meant as a diffeological counterpart of a Riemannian metric, on a generic diffeological vector pseudo-bundle,
although the proper analogy would be to put it on a suitable model of the tangent bundle. But, as we already pointed out in the Introduction and in Section 3, there is not yet a standard theory of
tangent spaces and tangent bundles for diffeological spaces, although various attempts to develop such have been made, see \cite{HeTangent}, \cite{HMtangent}, \cite{CWtangent}, and references
therein. We therefore avoid tying ourselves down to a specific construction in favor of a more abstract treatment, applicable to any diffeological vector pseudo-bundle (with finite-dimensional fibres). 
The material in this section is based on \cite{pseudometric-pseudobundle} and in part on \cite{exterior-algebras-pseudobundles}.

\subsection{The case of a single vector space}\label{pseudometric:sect}

As recalled in the Introduction, already in the case of a finite-dimensional\footnote{It is more complicated in the infinite-dimensional case, which we do not consider.} diffeological vector space the
appropriate analogue of the scalar product is not, in fact, a scalar product. What this means that a finite-dimensional diffeological vector space $V$ admits a smooth symmetric definite-positive bilinear
form $V\times V\to\matR$ if and only if $V$ is diffeomorphic to the standard $\matR^n$ for an appropriate $n$ (see \cite{iglesiasBook}, p. 74, Ex. 70). It follows that a generic diffeological vector space
the notion of the scalar product must be replaced by something that, for the given space, comes as close as possible to the scalar product, \emph{i.e.}, a smooth form of the maximal rank possible.
We call such a form a \textbf{pseudo-metric on $V$}; it turns out (not surprisingly) that its rank is the dimension of its diffeological dual.

\paragraph{The absence of scalar products} The following easy example shows why a diffeological vector space typically does not admit a smooth scalar product. We stress how the presence of just
one non-smooth plot is sufficient to prevent the existence of such.

\begin{example}\label{Rn:vector:nontrivial:ex}
Let $V=\matR^n$, and let $v_0\in V$ be any non-zero vector. Let $p:\matR\to V$ be defined as $p(x)=|x|v_0$; let $\calD$ be any vector space diffeology on $V$ that contains $p$ as a plot.\footnote{Such
diffeology does certainly exist; for instance, the coarse diffeology would do.} Suppose that $A$ is a symmetric $n\times n$ matrix, and assume that the bilinear form $\langle v|w\rangle_A=v^tAw$
associated to $A$ is smooth with respect to $\calD$ and the standard diffeology on $\matR$. We claim that $A$ is degenerate.

Indeed, $\langle v|w\rangle_A$ being smooth implies, in particular, that for any two plots $p_1,p_2:\matR\to V$ of $V$ the composition map $\langle \cdot|\cdot\rangle_A\circ (p_1,p_2):\matR\to\matR$
is smooth in the usual sense; this map acts as $\matR\ni x\mapsto(p_1(x))^tAp_2(x)$. Let $w\in V$ be an arbitrary vector; denote by $c_w:\matR\to V$ the constant map that sends everything to $w$,
$c_w(x)=w$ for all $x\in\matR$. Such a map is a plot for any diffeology on $V$. But then $(\langle \cdot|\cdot\rangle_A\circ(p,c_w))(x)=|x|v_0^tAw$; the only way for this to be smooth is to have
$v_0^tAw=0$, and since there was no assumption on $w$, this implies that $\langle v_0|\cdot\rangle_A$ is identically zero on the whole of $V$, \emph{i.e.}, that $A$ is degenerate. In other words,
$V$ does not admit a smooth scalar product.
\end{example}

Note that the above example would work just the same if we had taken $p(x)=f(x)v_0$ with $f(x)$ \emph{any} function $\matR\to\matR$ that is not differentiable (for instance) in at least one point.
We now cite the reverse statement, giving its proof for illustrative purposes.

\begin{prop} \emph{(This is actually a solution to Exercise 70 on p. 74 of \cite{iglesiasBook}.)\footnote{It is similar to the solution of the exercise given on p. 387 of \cite{iglesiasBook}; perhaps it is a
bit more direct.}} Let $V$ be $\matR^n$ endowed with a vector space diffeology $\calD$ such that there exists a smooth scalar product. Then every plot $p$ of $\calD$ is a smooth map in the usual
sense.
\end{prop}

\begin{proof}
Let $A$ an $n\times n$ non-degenerate symmetric matrix such that the associated bilinear form on $V$ is smooth with respect to the diffeology $\calD$, and let $\{v_1,\ldots,v_n\}$ be its eigenvector
basis. Let $\lambda_i$ be the eigenvalue relative to the eigenvector $v_i$.

Let $p:U\to V$ be a plot of $\calD$; we wish to show that it is smooth as a map $U\to\matR^n$. Recall that $\langle\cdot|\cdot\rangle_A$ being smooth implies that for any two plots $p_1,p_2:U\to V$
the composition $\langle\cdot|\cdot\rangle_A\circ(p_1,p_2)$ is smooth as a map $U\to\matR$. Let $c_i:U\to V$ be the constant map $c_i(x)=v_i$; this is of course a plot of $\calD$. Set $p_1=p$ and
$p_2=c_i$; then the above composition map writes as $\lambda_i\langle p(x)|v_i\rangle$, where $\langle\cdot|\cdot\rangle$ is the canonical scalar product on $\matR^n$.

Since $A$ is non-degenerate, all $\lambda_i$ are non-zero; this implies that each function $\langle p(x)|v_i\rangle$ is a smooth map. And since $v_1,\ldots,v_n$ form a basis of $V$, this implies
that for any $v\in V$ the function $\langle p(x)|v\rangle$ is a smooth one. In particular, this is true for any $e_j$ in the canonical basis of $\matR^n$; and in the case $v=e_j$ the scalar product
$\langle p(x)|e_j\rangle$ is just the $j$-th component of $p(x)$. Thus, we obtain that all the components of $p$ are smooth functions, therefore $p$ is a smooth map.
\end{proof}

Note also that the example given prior to this proposition can easily be extended to obtain a finite-dimensional diffeological vector space, with not too large a diffeology, such that the only smooth
linear map is the zero map. Namely, it suffices to take the vector space diffeology generated generated by the $n$ maps $x\mapsto|x|e_i$ for $i=1,\ldots,n$. By the same reasoning as in the example,
applied $n$ times, one sees that a linear map, assumed to be smooth, must necessarily be the zero map, although the diffeology in question is a very specific one.

\paragraph{A pseudo-metric: the best possible substitute} This is a very natural notion. A pseudo-metric on a finite-dimensional diffeological vector space $V$ is a smooth symmetric bilinear of
form of maximal rank possible. It is easy to see (\cite{pseudometric}) that this maximal rank is the dimension of the diffeological dual of $V$. Thus, we have the following definition:

\begin{defn}
Let $V$ be a diffeological vector space of finite dimension $n$, and let $\varphi:V\times V\to\matR$ be a smooth symmetric positive semidefinite bilinear form on it. We say that $\varphi$ is a 
\textbf{pseudo-metric} if the multiplicity of its eigenvalue $0$ is equal to $n-\dim(V^*)$.
\end{defn}

Such a pseudo-metric always exists on any finite-dimensional diffeological vector space. One interesting use of it is that it naturally determines, in $V$, its (unique) subspace maximal for the
following two properties: its subset diffeology is the standard one, and it splits off as a smooth direct summand.\footnote{Meaning that the diffeology on $V$ coincides with the corresponding vector
space sum diffeology.} The restriction of the pseudo-metric on this subspace is the usual scalar product.

\paragraph{The dual of a pseudo-metric} We now consider induced pseudo-metrics on the diffeological dual spaces (and then on the dual pseudo-bundles). The situation is rather simple here: for
any finite-dimensional diffeological vector space, a pseudo-metric on it induces a true metric on the diffeological dual, which, in particular, turns out to be a standard diffeological vector space of
the appropriate dimension (see \cite{pseudometric}).

\begin{thm}\label{dual:diff-gy:standard:thm}
Let $V$ be a finite-dimensional diffeological vector space, and let $V^*$ be its diffeological dual. Then the functional diffeology on $V^*$ is standard.
\end{thm}

The proof of this statement is actually carried out using a pseudo-metric to construct a basis of $V^*$ that generates its standard diffeology; thus, it is analogous to what happens in the
usual smooth case. Also in complete analogy with the standard case, to a pseudo-metric on $V$ there corresponds a true metric on $V^*$:

\begin{cor}
Any pseudo-metric $\langle\cdot|\cdot\rangle_A$ on $V$ induces a true metric on the diffeological dual $V^*$ of $V$, via the natural pairing that assigns to each $v\in V$ the smooth linear functional
$\langle\cdot|v\rangle_A$.
\end{cor}

It is quite easy to see that this is a \emph{vice versa} statement: if $\langle\cdot|\cdot\rangle_B$ is a smooth scalar product on $V^*$ then it suffices to take an orthonormal (with respect to the
canonical scalar product associated to the standard structure on $V^*$) basis $\{f_1,\ldots,f_k\}$ of $V^*$ to get a pseudo-metric on $V$ that induces $\langle\cdot|\cdot\rangle_B$: this pseudo-metric
is given by $\sum_{i=1}^k f_i\otimes f_i$.

\paragraph{An example of a pseudo-metric on a diffeological vector space} Let us consider $V=\matR^3$ endowed with the vector space diffeology generated by the plot $p:\matR\to V$ acting by
$p(x)=(0,|x|,|x|)$. We first note that the diffeological dual of $V$ is generated by the maps $e^1$ and $e^2-e^3$ (where $\{e^1,e^2,e^3\}$ is, obviously, the canonical basis of the usual dual of
$\matR^3$). In particular, we have that $\dim(V^*)=2$. It is then easy to see that any smooth symmetric bilinear form on $V$ is given by a matrix of form
${\small{\left(\begin{array}{ccc} c & a & -a \\
                                  a & b & -b \\
                                  -a & -b & b \end{array}\right)}}$
for some $a,b,c\in\matR$. A specific example can be obtained by taking, for instance, $a=1$, $b=c=2$, which gives
$A={\small{\left(\begin{array}{ccc} 2 & 1 & -1 \\
                                  1 & 2 & -2 \\
                                  -1 & -2 & 2 \end{array}\right)}}$.
Finally, the diffeological dual of $V$ is generated by the vectors $e^1$ and $e^2-e^3$, which form its basis; with respect to this basis, the induced metric on $V^*$ has matrix
$\frac19{\small{\left(\begin{array}{cc} 6 & 5 \\
                           5 & 6 \end{array}\right)}}$.

\subsection{Pseudo-metrics on diffeological vector pseudo-bundles}

Given a diffeological vector pseudo-bundle $\pi:V\to X$ such that its fibres are finite-dimensional, there is an obvious way to define a pseudo-metric on $V$ (it extends pretty much \emph{verbatim}
from the definition of a Riemannian metric, just the notion of a scalar product gets replaced by that of a pseudo-metric).

\begin{defn}
A \textbf{pseudo-metric} on the diffeological vector pseudo-bundle $\pi:V\to X$ is any smooth section $g$ of the corresponding pseudo-bundle $\pi^*\otimes\pi^*:V^*\otimes V^*\to X$ such that for
every $x\in X$ $g(x)$ is a pseudo-metric on $\pi^{-1}(x)$.
\end{defn}

Note that by Theorem 2.3.5 of \cite{vincent} $g(x)$ is indeed a bilinear map on $\pi^{-1}(x)$, so the notion is well-defined; on the other hand, it naturally presents existence questions. Indeed, just
as it happens with the non-existence, in general, of a smooth scalar product for diffeological vector spaces, also a pseudo-metric on a diffeological vector pseudo-bundle, in the sense of the definition
just given, might easily fail to exist. We are not able to give a complete answer to this question, but we do observe that the existence of a pseudo-metric on a pseudo-bundle seems to be related
to the pseudo-bundle being, or not, locally trivial. In this section we provide some preliminary observations, as well as some technical remarks; we provide explicit examples of when a pseudo-metric
does not exist in Section \ref{pseudometric:bundles:existence:sect}, while in Section \ref{pseudometric:gluing:bundles:sect}, on the other hand, we discuss the interactions of the gluing construction
with pseudo-metrics, showing how, under certain natural conditions, gluing two pseudo-bundles endowed with a pseudo-metric each allows to obtain again a pseudo-bundle with a pseudo-metric.

\paragraph{A diffeological vector pseudo-bundle $\matR^n\to\matR^k$} Our first examples are based, as underlying topological map, on the standard projection of $\matR^n$ onto its first $k$
coordinates, therefore on $\matR^k$. While in most cases the diffeology on the total space $\matR^n$ is a product diffeology (corresponding to its presentation as $\matR^k\times\matR^{n-k}$),
it is also among such pseudo-bundles, the simplest ones from the topological point of view, that we find non-locally trivial pseudo-bundles, see Example \ref{R2:to:R:nontrivial:ex} above, and
instances of those that do not admit a pseudo-metric (as we will see, this also occurs for the Example \ref{R2:to:R:nontrivial:ex}).

\paragraph{Notation for topologically trivial pseudo-bundles} The notation that we use in the examples below, particularly for pseudo-metrics, is an \emph{ad hoc} choice designed to apply not
more than to the instances being described. Pretty much always we use pseudo-bundles that are based on the projection of some $V=\matR^n$ onto its first $k$ coordinates, and so each fibre is of
form $\{x\}\times\matR^{n-k}$ with $x\in\matR^k$; the vector space structure (but not the diffeology) is that of the factor $\matR^{n-k}$. It follows that the dual of each fibre can be viewed as an
element of $\mbox{Span}(e^{k+1},\ldots,e^n)$, so a generic element of the dual bundle $V^*$ can be written as $((x_1,\ldots,x_k),a_{k+1}e^{k+1}+\ldots+a_ne^n)$, or, more briefly, as
$$(x,a_{k+1}e^{k+1}+\ldots+a_ne^n)\in V^*,$$ although this is not \emph{vice versa}, in the sense that not all such expressions define an element of $V^*$, the diffeological dual being in general
a proper subset of the usual dual. By extension, then, a smooth bilinear form and, more specifically, a prospective pseudo-metric in particular writes as
$$(x,\sum_{i,j=k+1}^n a_{ij}e^i\otimes e^j)\in V^*\otimes V^*$$ (once again, with various restrictions on the coefficients to account for the fact that in general not all such expressions would
define a smooth form on $V$).

\paragraph{A sample pseudo-metric on a trivial non-standard pseudo-bundle} The following is an example of a pseudo-metric on a pseudo-bundle, which is both topologically and diffeologically
trivial, but has non-standard fibre.

\begin{example}\label{R3:to:R:z-nontrivial:ex}
Let $n=3$ and $k=1$, so we have $\pi:\matR^3\to\matR$ given by $\pi(x,y,z)=x$; endow $\matR$ with the standard diffeology and $\matR^3$ with the finest vector space diffeology generated by
the map $p:U=\matR^2\to\matR^3$ acting by $p(u_1,u_2)=(u_1,0,|u_2|)$. This is a diffeology seen before and is a simple example of a non-standard diffeology; recall in particular that the fibre is
diffeomorphic to $\matR^2$ with the vector space diffeology generated by the map $u\mapsto|u|e_2$. Setting $g(x)=(x,(x^2+1)e^2\otimes e^2)$ gives a pseudo-metric on this bundle (if we consider
the latter expression as a bilinear form in the coordinates $y,z$ on the fibre $\pi^{-1}(x)$). In fact, it is easy to see that any pseudo-metric on this pseudo-bundle writes, in global coordinates of
$\matR^3$, as $g(x)=(x,f(x)e^2\otimes e^2)$, where $f:\matR\to\matR$ is a smooth everywhere positive function.
\end{example}

The example is particularly simple in that it admits a constant pseudo-metric (we could take $f$ to be a positive constant). Now, starting from the standard bundles, as well as simple pseudo-bundles
such as the one above, and utilizing the gluing procedure for pseudo-bundles, it is easy to construct more complicated examples of pseudo-bundles carrying a pseudo-metric, in particular, examples
where both the total space and the base space are non-trivial topologically.\footnote{Non-trivial at this moment means first of all not being homeomorphic to $\matR^m$ for some $m$, so includes
contractible spaces; but it is also easy to obtain spaces with non-trivial homotopy.}

\paragraph{A topologically non-trivial pseudo-bundle with a pseudo-metric} We now give an example of a pseudo-bundle which is not a topologically trivial one (meaning that it is different from
all the projections $\matR^n\to\matR^k$). It is obtained by the gluing construction from two standard projections.

\begin{example}
The following pseudo-bundle is obtained by gluing together two copies of the standard (trivial) bundle $\matR^2\to\matR$, with the bundle map $(x,y)\mapsto x$. Namely, denote by $V$ the result of
gluing of $\matR^2$, written in coordinates $(x_1,y_1)$, to $\matR^2$, written in coordinates $(x_2,y_2)$, via the map $\tilde{f}(0,y_1)=(0,y_2)$. From the topological point of view, this space is
homeomorphic to the union of two coordinate planes in $\matR^3$, for instance, $V=\{(0,y,z)\}\cup\{(x,0,z)\}$. Notice that the gluing diffeology on $V$ coming from the first representation
(where the two copies of $\matR^2$ are endowed with the standard diffeology) is the same as the subset diffeology relative to the standard diffeology of $\matR^3$ coming from the second
representation.

The space $V$ is the total space of our pseudo-bundle; define $X$ to be the wedge, at their respective origins, of two copies of standard $\matR$. Once again, it can be represented as the result of gluing
and as the subset $X=\{(0,y,0)\}\cup\{(x,0,0)\}$ of $\matR^3$; just as for $V$, the resulting diffeology is the same. Relative to the first representation, the pseudo-bundle map acts by $(x_1,y_1)\mapsto x_1$
and $(x_2,y_2)\mapsto x_2$ (which is well-defined with respect to the gluing); relative to the second, it is the restriction on $V\subset\matR^3$ of the standard projection of $\matR^3$ onto its first
coordinate.

As for the choice of a pseudo-metric, it is of course natural to choose them separately on each of the two pseudo-bundles. It is also natural that on the two fibres being identified, the choices must be
compatible; we can use the fact that we have actually two identical pseudo-bundles, with gluing by identity (see the first presentation of each), and so choose the same pseudo-metric on each.
Observe finally that any pseudo-metric on the projection $\matR^2\to\matR$, seen as a pseudo-bundle in our sense, writes in the form $x\mapsto(x,f(x)e^2\otimes e^2)$ for any smooth (in the usual
sense) everywhere positive function $f:\matR\to\matR$. Other compatible choices would simply be $f_1$ and $f_2$ with the same properties as just stated and satisfying $f_1(0)=f_2(0)$ (such as
$f_1(x)=e^x$ and $f_2(x)=x^2+1$, for instance).
\end{example}

\subsection{Existence and non-existence of pseudo-metrics}\label{pseudometric:bundles:existence:sect}

It is quite easy to see that there are diffeological vector pseudo-bundles that do not admit pseudo-metrics; below we give an example of such. On the other hand, the construction of gluing
applied to compatible maps frequently allows to obtain pseudo-metrics on many pseudo-bundles that result from gluings between the domains and ranges of these pseudo-metrics. This
suggests that some highly nontrivial pseudo-bundles carrying a pseudo-metric could be obtained starting with a collection of standard bundles (ones modeled on projections $\matR^n\to\matR^k$
and carrying a product diffeology) and performing a multitude of gluings.

\paragraph{Non-existence of pseudo-metrics} Let us consider the following example of a non locally trivial pseudo-bundle (it is one of the simplest examples of such; we have already encountered it
above).

\begin{example}\label{R2:to:R:nontrivial:ex}
Consider the usual projection $\pi$ of $\matR^2$ on its $x$-axis $\matR$; endow $\matR$ with the standard diffeology, and endow $\matR^2$ with the pseudo-bundle diffeology generated by the plot
$p:\matR^2\to\matR^2$ given by $p(x,y)=(x,|xy|)$.
\end{example}

\begin{lemma}
The pseudo-bundle of Example \ref{R2:to:R:nontrivial:ex} does not admit a pseudo-metric.
\end{lemma}

\begin{proof}
We have already observed that all fibres of $\pi$, except one, have non-standard diffeology; the only standard fibre is $\pi^{-1}(0)$. Now, as vector spaces all these fibres are isomorphic to $\matR$,
so they have dimension one. This allows us to conclude that for any $x\neq 0$ the diffeological dual of $\pi^{-1}(x)$ is trivial; thus, any pseudo-metric, being an element of the tensor product of this
dual with itself, is a zero map. On the other hand, the space $\pi^{-1}(0)$ is standard and so admits a non-zero pseudo-metric (in fact, a true metric). Observe also that the dual bundle $\pi^*$ of
$\pi$ is not a true bundle even from a topological point of view. Indeed, its total space $V^*$ is the union of two copies of $\matR$ joined at the origin; one of these copies projects (trivially) to
$\matR$ which is the base space and the other gets sent to the origin of the base. It is also easy to see that $V^*$ is diffeomorphic, as a diffeological space, to the union of coordinate axes in $\matR^2$
considered with the subset diffeology.

It follows from the remarks above that the potential pseudo-metric $g$ on $V$ writes as $g(x)=(x,\delta(x)e^2\otimes e^2)$, where $\delta:\matR\to\matR$ is a version of the $\delta$-function, one
given by $\delta(x)=0$ for $x\neq 0$ and $\delta(0)=1$ (or any other positive constant).\footnote{This is the only possibility for $g$, up to choosing the specific value of $\delta(0)$.} Let us check
whether such a $g$ defines a smooth section of $V^*\otimes V^*$. By an extension (to the case of tensor products) of our characterization of plots of dual bundles we should check that the evaluation
of $g$ on $q\otimes s$, where $q,s$ with $q,s:U\to V$ are two arbitrary plots of $V$, is a smooth function (on an appropriate domain of definition). Now, if $q(u)=(q_1(u),q_2(u))$ and
$s(u)=(s_1(u),s_2(u))$, the domain of definition is the set of $u$ such that $q_1(u)=s_1(u)$ and the evaluation is the function $u\mapsto\delta(q_1(u))q_2(u)s_2(u)$; for it to be smooth, we must
have $q_1(u)=0\Rightarrow q_2(u)s_2(u)=0$ (otherwise the function would not even be continuous), which does not have to happen. Since we have already observed the $g$ proposed is essentially
the only choice for a pseudo-metric on this pseudo-bundle, we must conclude that in the sense of the definition given this pseudo-bundle does not admit any pseudo-metric.\footnote{That we obtained
zero-maps-only conclusion is an extreme which does not have to happen. Below we will see that in analogous situations with fibres of higher dimension(s), there is an almost-pseudo-metric, meaning
that does give one on most fibres, but, just as in this case, not all of them.}
\end{proof}

Examples similar to the above ones can easily be constructed for any dimension (and the absolute value function can be replaced by any function which is not smooth in at least one point of its domain).
This also allows us to observe that, given $\pi:V\to X$ a  diffeological vector pseudo-bundle whose image under the forgetful functor into the category of topological spaces\footnote{That is, if
we do not take into account the diffeologies of $V$ and $X$.} is a usual vector bundle, the corresponding (diffeological) dual $\pi^*:V^*\to X$ may not be a topological vector bundle. This also
indicates why we cannot limit the discussion to just locally trivial diffeological vector pseudo-bundles (in addition to all the reasons already listed in the previous section and coming from \cite{iglFibre},
\cite{CWtangent}, and others).

\paragraph{Existence} What has been said in the previous paragraph, shows that the existence issue for pseudo-metrics cannot be avoided via simple measures, such as, for instance, imposing some
obvious restrictions on the class of pseudo-bundles under consideration. Furthermore, there does not yet seem to be a complete answer to when a pseudo-metric does or does not exist.\footnote{I am
not aware of existence of one.} Thus, in the rest of this section we attempt to use the gluing construction as an approach to this issue, focusing on the following very natural question: given a gluing
between two pseudo-bundles carrying a pseudo-metric each, under what conditions is there an induced pseudo-metric on the resulting pseudo-bundle?

\subsection{Pseudo-metrics and gluing}\label{pseudometric:gluing:bundles:sect}

We now consider the interaction between pseudo-metrics and diffeological gluing. The starting point is an immediately obvious one: a pseudo-metric on a pseudo-bundle is a collection of
pseudo-metrics on all fibres, and when a gluing is performed, each fibre of the result corresponds to a fibre of one of the two factors of gluing. Thus, the resulting pseudo-bundle comes, it as well, with
a collection of (diffeological vector space) pseudo-metrics on each fibre. The real issue is, is this collection a pseudo-metric on the whole pseudo-bundle, \emph{i.e.}, does it depend smoothly on the
point in the base space?

\subsubsection{Preliminary considerations}

Consider two diffeological vector pseudo-bundles $\pi_1:V_1\to X_1$ and $\pi_2:V_2\to X_2$, each endowed with a pseudo-metric, denoted respectively by $g_1$ and $g_2$. This means that we are
given two smooth maps
$$g_1:X_1\to V_1^*\otimes V_1^*\mbox{ and }g_2:X_2\to V_2^*\otimes V_2^*.$$ Let us also fix a gluing of $V_1$ to $V_2$ along $(\tilde{f},f)$; recall that there is then an induced gluing of
$V_1^*\otimes V_1^*$ to $V_2^*\otimes V_2^*$. Under some conditions, $g_1$ and $g_2$ will be compatible, as smooth maps, with the latter gluing, but the result of gluing of one to the other is
not a pseudo-metric on the pseudo-bundle $\pi_1\cup_{(\tilde{f},f)}\pi_2:V_1\cup_{\tilde{f}}V_2\to X_1\cup_f X_2$ (although it might be used to define one).

Indeed, a pseudo-metric $g$ on the latter pseudo-bundle is first of all a map of form
$$X_1\cup_f X_2\to(V_1\cup_{\tilde{f}}V_2)^*\otimes(V_1\cup_{\tilde{f}}V_2)^*.$$ On the other hand, the above-mentioned induced gluing is along the map $\tilde{f}^*\otimes\tilde{f}^*$, which goes
$V_2^*\otimes V_2^*\to V_1^*\otimes V_1^*$, and, if anything, it covers the inverse of $f$, which therefore we must assume exists (and is smooth). Assuming it does, the gluing along
$(f^{-1},\tilde{f}^*\otimes\tilde{f}^*)$ yields the map
$$g_2\cup_{(f^{-1},\tilde{f}^*\otimes\tilde{f}^*)}g_1:X_2\cup_{f^{-1}}X_1\to(V_2^*\otimes V_2^*)\cup_{\tilde{f}^*\otimes\tilde{f}^*}(V_1^*\otimes V_1^*),$$ which does not have the same shape
as a pseudo-metric should have. Finally, assuming that $f$ is a diffeomorphism of its domain with its image, then the switch map (see above) yields a diffeomorphism between $X_1\cup_f X_2$ and
$X_2\cup_{f^{-1}}X_1$; the main issue then is whether there is a diffeomorphism between the spaces
$$(V_1\cup_{\tilde{f}}V_2)^*\otimes(V_1\cup_{\tilde{f}}V_2)^*\mbox{ and }(V_2^*\otimes V_2^*)\cup_{\tilde{f}^*\otimes\tilde{f}^*}(V_1^*\otimes V_1^*)$$ that covers it. The answer thus depends
on the commutativity conditions, and the corresponding commutativity diffeomorphisms, that were discussed in Section 4.

\subsubsection{Compatible pseudo-metrics}

As has already been mentioned, in order to speak of an induced pseudo-metric on a pseudo-bundle obtained by gluing, the existing pseudo-metrics on the factors of this gluing should satisfy some
natural compatibility condition. Indeed, the basic operation of gluing (between diffeological spaces) is that of identifying $\pi_1^{-1}(y)$ with (a subspace of) $\pi_2^{-1}(f(y))$. Since both are endowed
with a pseudo-metric, $g_1(y)$ the former and $g_2(f(y))$ the latter, it stands to reason that the identification map (the corresponding restriction of $\tilde{f}$) should preserve the pseudo-metrics. 

\begin{defn}\label{compatible:pseudo-metrics:on:two:pseudo-bundles:defn}
Let $g_1$ be a pseudo-metric on $V_1$, and let $g_2$ be a pseudo-metric on $V_2$. We say that $g_1$ and $g_2$ are \textbf{compatible} (with the gluing along $(\tilde{f},f)$) if for all $y\in Y$ and for all 
$v_1,v_2\in\pi_1^{-1}(y)$ we have that
$$g_1(y)(v_1,v_2)=g_2(f(y))(\tilde{f}(v_1),\tilde{f}(v_2)).$$  
\end{defn}

The following is then true.

\begin{lemma}
Suppose that $f$ is invertible. Then the pseudo-metrics $g_1$ and $g_2$ are compatible if and only if they are $(f^{-1},\tilde{f}^*\otimes\tilde{f}^*)$-compatible.
\end{lemma}

\subsubsection{Choosing the commutativity diffeomorphisms}

Assuming that $g_1$ and $g_2$ are compatible, we have a well-defined smooth map
$$g_2\cup_{(f^{-1},\tilde{f}^*\otimes\tilde{f}^*)}g_1:X_2\cup_{f^{-1}} X_1\to(V_2^*\otimes V_2^*)\cup_{\tilde{f}^*\otimes\tilde{f}^*}(V_1^*\otimes V_1^*).$$ Since the tensor product always commutes
with gluing, we have
$$\Phi_{\otimes,\cup}^{V_2^*,V_1^*}:(V_2^*\otimes V_2^*)\cup_{\tilde{f}^*\otimes\tilde{f}^*}(V_1^*\otimes V_1^*)\to(V_2^*\cup_{\tilde{f}^*}V_1^*)\otimes(V_2^*\cup_{\tilde{f}^*}V_1^*),$$ where
$\Phi_{\cup,\otimes}^{V_2^*,V_1^*}$ is the version of the commutativity diffeomorphism $\Phi_{\cup,\otimes}$ described in Section 4.7.3, obtained by taking both $V_1$ and $V_1'$ to be $V_2^*$, and
$V_2$ and $V_2'$ to be $V_1^*$ (the upper index that appears in the present notation serves to remind us of this), while switching the lower indexes indicates taking the inverse:
$\Phi_{\otimes,\cup}=(\Phi_{\cup,\otimes})^{-1}$.

The composition $\Phi_{\otimes,\cup}^{V_2^*,V_1^*}\circ(g_2\cup_{(f^{-1},\tilde{f}^*\otimes\tilde{f}^*)}g_1)$ is therefore a map
$$X_2\cup_{f^{-1}}X_1\to(V_2^*\cup_{\tilde{f}^*}V_1^*)\otimes(V_2^*\cup_{\tilde{f}^*}V_1^*).$$ To turn it into the desired form, that is, a map
$$X_1\cup_f X_2\to(V_1\cup_{\tilde{f}}V_2)^*\otimes(V_1\cup_{\tilde{f}}V_2)^*,$$ we obviously need to add the switch map and the tensor square of the appropriate gluing-dual commutativity
diffeomorphism. More precisely, we first pre-compose it with the switch map
$$\varphi_{X_1\leftrightarrow X_2}:X_1\cup_f X_2\to X_2\cup_{f^{-1}}X_1,$$ obtaining the map
$$\Phi_{\otimes,\cup}^{V_2^*,V_1^*}\circ(g_2\cup_{(f^{-1},\tilde{f}^*\otimes\tilde{f}^*)}g_1)\circ(\varphi_{X_1\leftrightarrow X_2}):X_1\cup_f X_2\to(V_2^*\cup_{\tilde{f}^*}V_1^*)\otimes
(V_2^*\cup_{\tilde{f}^*}V_1^*).$$

It is useful to observe at this point that the construction carried out so far does not require any additional assumptions except for $f$ being a diffeomorphism with its image. However, the just-obtained
composition map is not yet a pseudo-metric; for it to be one, we need it to take values in the tensor product of $(V_1\cup_{\tilde{f}^*}V_2)^*$ with itself. This is where the possibility of continuing the
construction depends on whether the gluing-dual commutativity is satisfied. Indeed, this condition is equivalent the existence of a diffeomorphism
$$(\Phi_{\cup,*})^{-1}:V_2^*\cup_{\tilde{f}^*}V_1^*\to(V_1\cup_{\tilde{f}}V_2)^*,$$ the inverse of the commutativity diffeomorphism described in Section 4.7.3. It is easy to check then, that if such exists,
then the composition
$$\left((\Phi_{\cup,*})^{-1}\otimes(\Phi_{\cup,*})^{-1}\right)\circ\Phi_{\otimes,\cup}^{V_2^*,V_1^*}\circ(g_2\cup_{(f^{-1},\tilde{f}^*\otimes\tilde{f}^*)}g_1)\circ(\varphi_{X_1\leftrightarrow X_2})$$
is a pseudo-metric on the pseudo-bundle $V_1\cup_{\tilde{f}}V_2$.

\begin{oss}
Although the above construction might appear complicated at first glance, it corresponds to a very simple idea: since each fibre of $V_1\cup_{\tilde{f}}V_2$ is canonically identified with one of either
$V_1$ or $V_2$, and both of the latter already carry a pseudo-metric, the result of gluing, the pseudo-bundle $V_1\cup_{\tilde{f}}V_2$ is naturally endowed with a collection of pseudo-metrics on its
fibres. If this collection turns out to depend smoothly on the point in the base, it is then a pseudo-metric on $V_1\cup_{\tilde{f}}V_2$; and indeed, this is precisely what the above composition map is.

Now, the same idea can be used in the absence of the gluing-dual commutativity, and indeed, as we say below, it yields the same end result also in that case (via an explicit construction). The main
conceptual difference lies in the fact that in the latter case the smoothness of the induced pseudo-metric depends much on the properties of the gluing diffeology (which we defined to be a rather
weak diffeology, relatively speaking).
\end{oss}

\subsubsection{Constructing a pseudo-metric on $V_1\cup_{\tilde{f}}V_2$}

Below we give full statements regarding the construction of the induced pseudo-metric on the pseudo-bundle $V_1\cup_{\tilde{f}}V_2$ obtained by gluing together of two pseudo-bundles $V_1$ and
$V_2$, each endowed with a pseudo-metric $g_1$ or $g_2$, respectively. As we have said already, the map $f$ that defines the gluing between the base spaces is assumed to be smoothly invertible,
and the pseudo-metrics $g_1$ and $g_2$ are assumed to be compatible with the gluing (see above).

\paragraph{When the gluing-dual commutativity condition is satisfied} This case has already been discussed in detail, so now we give the final statement.

\begin{thm}\label{glued:pseudometric:commutative:thm}
Let $\pi_1:V_1\to X_1$ and $\pi_2:V_2\to X_2$ be two finite-dimensional diffeological vector pseudo-bundles, and let $(\tilde{f},f)$ be a gluing between them given by a smooth invertible map
$f:X_1\supset Y\to X_2$ and its smooth fibrewise linear lift $\tilde{f}$. Let $g_i$ for $i=1,2$ be a pseudo-metric on $V_i$ such that $g_1$ and $g_2$ are compatible with the gluing along $(\tilde{f},f)$.
Finally, assume that $V_1$, $V_2$, and $(\tilde{f},f)$ satisfy the gluing-dual commutativity condition; let $\Phi_{\cup,*}$ be the corresponding commutativity diffeomorphism. Then the map
$$\tilde{g}=\left((\Phi_{\cup,*})^{-1}\otimes(\Phi_{\cup,*})^{-1}\right)\circ\Phi_{\otimes,\cup}^{V_2^*,V_1^*}\circ(g_2\cup_{(f^{-1},\tilde{f}^*\otimes\tilde{f}^*)}g_1)\circ(\varphi_{X_1\leftrightarrow X_2})$$
is a pseudo-metric on the pseudo-bundle $\pi_1\cup_{(\tilde{f},f)}\pi_2:V_1\cup_{\tilde{f}}V_2\to X_1\cup_f X_2$.
\end{thm}

A proof (a very direct one) of this statement can be found in \cite{pseudometric-pseudobundle}.

\begin{rem}
Although in the above theorem we are looking for a pseudo-metric on $V_1\cup_{\tilde{f}}V_2$, the construction applies to any pair of compatible smooth bilinear forms on $V_1$ and $V_2$ respectively.
As it follows from the construction (and the proof), for the bilinear form $\tilde{g}$ thus obtained, the rank of $\tilde{g}(x_1)$ over a point $x_1\in i_1(X_1\setminus Y)$ is equal to that of $g_1(x_1)$,
while over a point $x_2\in i_2(X_2)$ it is equal to that of $g_2(x_2)$.
\end{rem}

\paragraph{When the gluing-dual commutativity is absent} As we already mentioned, the absence of commutativity between the operation of gluing and that of taking the dual pseudo-bundle
does not necessarily preclude the existence of a natural pseudo-metric on $V_1\cup_{\tilde{f}}V_2$ induced by the existing pseudo-metrics on the factors. Rather, the flexibility of diffeology might
well allow for an \emph{ad hoc} construction of one --- it just will not be canonical, as it is in the previous case.

Note that the fibrewise construction of the induced pseudo-metric does not present difficulties, provided that the compatibility condition still holds. Indeed, if we are given two diffeological vector
pseudo-bundles $\pi_1:V_1\to X_1$ and $\pi_2:V_2\to X_2$, a gluing of the former to the latter along an appropriate pair $(\tilde{f},f)$ of maps, and two compatible pseudo-metrics $g_1$ and
$g_2$ on $V_1$ and $V_2$ respectively, then we can define a section $\tilde{g}:X_1\cup_f X_2\to(V_1\cup_{\tilde{f}}V_2)^*\otimes(V_1\cup_{\tilde{f}}V_2)^*$ by setting its value on
$x_1\in i_1(X_1\setminus Y)$ to be
$$\tilde{g}(x_1):=g_1(i_1^{-1}(x_1))\circ(j_1^{-1}\otimes j_1^{-1}),$$ and on $x_2\in i_2(X_2)$, to be
$$\tilde{g}(x_2):=g_2(i_2^{-1}(x_2))\circ(j_2^{-1}\otimes j_2^{-1}).$$ The map $\tilde{g}$ is thus well-defined and pointwise yields a pseudo-metric on the relevant fibre. The issue is why it is
smooth; this is established in the following statement, the proof of which is given in \cite{pseudometric-pseudobundle}.

\begin{thm}\label{glued:pseudometric:noncommutative:thm}
Let $\pi_1:V_1\to X_1$ and $\pi_2:V_2\to X_2$ be two finite-dimensional diffeological vector pseudo-bundles, let $(\tilde{f},f)$ be a gluing between them, and let $g_1$ and $g_2$ be
pseudo-metrics on $V_1$ and, respectively, $V_2$ compatible with respect to the gluing. Define $\tilde{g}:X_1\cup_f X_2\to(V_1\cup_{\tilde{f}}V_2)^*\otimes(V_1\cup_{\tilde{f}}V_2)^*$ by setting
$$\tilde{g}(x)=\left\{
\begin{array}{ll}
g_1(i_1^{-1}(x_1))\circ(j_1^{-1}\otimes j_1^{-1}) & \mbox{for }x\in i_1(X_1\setminus Y) \\
g_2(i_2^{-1}(x_2))\circ(j_2^{-1}\otimes j_2^{-1}) & \mbox{for }x\in i_2(X_2).
\end{array}\right.$$ Then $\tilde{g}$ is a pseudo-metric on the pseudo-bundle $\pi_1\cup_{(\tilde{f},f)}\pi_2:V_1\cup_{\tilde{f}}V_2\to X_1\cup_f X_2$.
\end{thm}

The end conclusion is that compatible pseudo-metrics on the factors of a gluing always (seem to) induce a pseudo-metric on the result, although the precise way in which it happens follows
different scenarios.

\paragraph{Example of a pseudo-metric in the noncommutative case} We now illustrate the construction that appears in Theorem \ref{glued:pseudometric:noncommutative:thm}
with the following example (\cite{clifford}).

\begin{example} \emph{(\cite{clifford})}
Consider $\pi_i:V_i\to X_i$ for $i=1,2$, where $X_1=X_2=\matR$ with standard diffeology; $V_1$ and $V_2$ also carry the standard diffeology but they have different dimensions. Specifically,
$V_1=\matR^2$ and $V_2=\matR^3$; finally, each of the maps $\pi_1,\pi_2$ is the natural projection onto the corresponding first coordinate. The gluing of these two pseudo-bundles is given by the
identification of the origins of $X_1$ and $X_2$ (so the map $f$ is obvious, $f:\{0\}\to\{0\}$); on the corresponding fibres, which are of form $\pi_1^{-1}(0)=\{(0,y)|y\in\matR\}\subset V_1$ and
$\pi_2^{-1}(0)=\{(0,y,z)|y,z\in\matR\}\subset V_2$, it is defined by $\tilde{f}:\pi_1^{-1}(0)\to\pi_2^{-1}(0)$ with $\tilde{f}(0,y)=(0,y,0)$ (it is obviously linear and smooth).

The two pseudo-bundles are endowed with pseudo-metrics, both corresponding to the canonical scalar product on the relevant Euclidean spaces ($\matR$ and $\matR^2$, which are the fibres);
we have $g_1(x)=(x,e^2\otimes e^2)$ and $g_2(x)=(x,e^2\otimes e^2+e^3\otimes e^3)$. The compatibility, which in this case means that $g_1(0)(v_1,v_2)=g_2(0)(\tilde{f}(v_1),\tilde{f}(v_2))$
for all $v_1=(0,y_1),v_2=(0,y_2)$, is obvious. ADD
\end{example}

\paragraph{Comparison of Theorem \ref{glued:pseudometric:commutative:thm} and Theorem \ref{glued:pseudometric:noncommutative:thm}} What one may naturally wonder at this point is
whether the construction of Theorem \ref{glued:pseudometric:commutative:thm} is in fact a partial case of that of Theorem \ref{glued:pseudometric:noncommutative:thm}. Let us compare the two
maps pointwise, calling the former $\tilde{g}'$ and the latter, $\tilde{g}''$.

Take $x_1\in i_1(X_1\setminus Y)$, and $v_1,v_2\in(\pi_1\cup_{(\tilde{f},f)}\pi_2)^{-1}(x_1)$; note that $v_1,v_2\in j_1(V_1\setminus\pi_1^{-1}(Y))$. Starting with $\tilde{g}'$, we obtain
$$\tilde{g}'(x_1)(v_1,v_2)=g_1(i_1^{-1}(x_1))(j_1^{-1}(v_1),j_1^{-1}(v_2))=\left(g_1(i_1^{-1}(x_1))\circ j_1^{-1}\otimes j_1^{-1}\right)(v_1,v_2)=\tilde{g}''(x_1).$$ Similarly, if $x_2\in i_2(X_2)$ and
$v_1,v_2\in(\pi_1\cup_{(\tilde{f},f)}\pi_2)^{-1}(x_2)$ then $v_1,v_2\in j_2(V_2)$, and we obtain
$$\tilde{g}'(x_2)(v_1,v_2)=g_2(i_2^{-1}(x_2))(j_2^{-1}(v_1),j_2^{-1}(v_2))=\left(g_2(i_2^{-1}(x_2))\circ j_2^{-1}\otimes j_2^{-1}\right)(v_1,v_2)=\tilde{g}''(x_2).$$ Thus, the second construction (that
of Theorem \ref{glued:pseudometric:noncommutative:thm}) does include the first one and is more general; the advantage of the first one (of Theorem \ref{glued:pseudometric:commutative:thm}) is that
it is canonically related to the map $g_2\cup_{(f^{-1},\tilde{f}^*\otimes\tilde{f}^*)}g_1$ (which in itself is obtained via a canonical construction). Another consideration is the already-mentioned one,
namely, that the smoothness of the pseudo-metric that we construct in the non-commutative case, depends very much on the structure of the gluing diffeology, specifically, on it being a very weak
diffeology (in relative terms, at least); one might wonder\footnote{We do not know the answer.} whether it is however the strongest diffeology such that the pseudo-metric thus constructed is the
strongest for which it is true. On the other hand, in the commutative case the construction is a high-level one, thus, the pseudo-metric that we obtain in this case is always smooth, presumably
even if we strengthen the diffeology.

\subsubsection{The spaces of pseudo-metrics with functional diffeology}

The above construction that assigns to two compatible pseudo-metrics on two given pseudo-bundles $\pi_1:V_1\to X_1$ and $\pi_2:V_2\to X_2$ the pseudo-metric $\tilde{g}$ on the result of their
gluing $\pi_1\cup_{(\tilde{f},f)}\pi_2:V_1\cup_{\tilde{f}}V_2\to X_1\cup_f X_2$, can be see as a map on the appropriate subset of $C^{\infty}(X_1,V_1^*\otimes V_1^*)\times C^{\infty}(X_2,V_2^*\otimes V_2^*)$
into the space $C^{\infty}(X_1\cup_f X_2,(V_1\cup_{\tilde{f}}V_2)^*\otimes(V_1\cup_{\tilde{f}}V_2)^*)$. Specifically, for any finite-dimensional diffeological vector pseudo-bundle $\pi:V\to X$, denote
by $\mathcal{G}(V,X)$ the set of all pseudo-metrics on it; endow it with the subset diffeology relative to the functional diffeology on $C^{\infty}(X,V^*\otimes V^*)$. The map just mentioned, that
we denote by $\mathcal{P}$, is a map of form
$$\mathcal{P}:\mathcal{G}(V_1,X_1)\times_{comp}\mathcal{G}(V_2,X_2)\to\mathcal{G}(V_1\cup_{\tilde{f}}V_2,X_1\cup_f X_2),$$ acting by $(g_1,g_2)\mapsto\tilde{g}$, where $g_1$ and $g_2$ are
two compatible pseudo-metrics on $\pi_1:V_1\to X_1$ and $\pi_2:V_2\to X_2$ respectively. It is defined on the subset
$\mathcal{G}(V_1,X_1)\times_{comp}\mathcal{G}(V_2,X_2)\subset\mathcal{G}(V_1,X_1)\times\mathcal{G}(V_2,X_2)$ of the direct product $\mathcal{G}(V_1,X_1)\times\mathcal{G}(V_2,X_2)$ composed
of all pairs of compatible pseudo-metrics; its diffeology is the subset diffeology relative to the product diffeology on $\mathcal{G}(V_1,X_1)\times\mathcal{G}(V_2,X_2)$. Thus, $\mathcal{P}$ acts between
diffeological spaces, and it is not hard to see that it is smooth; below we give some details.

\paragraph{The commutative case} In the case when the gluing-dual commutativity condition is satisfied, that is,
$$\mathcal{P}(g_1,g_2)=\tilde{g}=\left((\Phi_{\cup,*})^{-1}\otimes(\Phi_{\cup,*})^{-1}\right)\circ\Phi_{\otimes,\cup}^{V_2^*,V_1^*}\circ(g_2\cup_{(f^{-1},\tilde{f}^*\otimes\tilde{f}^*)}g_1)\circ
(\varphi_{X_1\leftrightarrow X_2}),$$ the map $\mathcal{P}$ is the composition of the following:\\
1) the map $\mathcal{G}(V_1,X_1)\times_{comp}\mathcal{G}(V_2,X_2)\to\mathcal{G}(V_2,X_2)\times_{comp}\mathcal{G}(V_1,X_1)$, which acts by exchanging the two factors inside the direct product
(this is induced by, or is analogous to, the switch map $\varphi_{X_1\leftrightarrow X_2}$);\\
2) the appropriate restriction of
$$\mathcal{F}_{V_2^*\otimes V_2^*,V_1^*\otimes V_1^*}:C^{\infty}(X_2,V_2^*\otimes V_2^*)\times_{comp}C^{\infty}(X_1,V_1^*\otimes V_1^*)\to C^{\infty}(X_2\cup_{f^{-1}}X_1,(V_2^*\otimes
V_2^*)\cup_{\tilde{f}^*\otimes\tilde{f}^*}(V_1^*\otimes V_1^*))$$ (this is a case of a map obtained by gluing; see above and \cite{pseudometric-pseudobundle} for details); and\\
3) the appropriate restriction of
$$C^{\infty}(X_2\cup_{f^{-1}}X_1,(V_2^*\otimes V_2^*)\cup_{\tilde{f}^*\otimes\tilde{f}^*}(V_1^*\otimes V_1^*))\to C^{\infty}(X_2\cup_{f^{-1}}X_1,(V_1\cup_{\tilde{f}^*}V_2)^*\otimes
(V_1\cup_{\tilde{f}^*}V_2)^*),$$ obtained by the post-composition with the map $(\Phi_{\cup,*})^{-1}\otimes(\Phi_{\cup,*})^{-1}$.

The first of these maps is smooth by definition of the product diffeology, while the third is so because it is a post-composition with a fixed smooth (it is a general fact, that follows from the properties
of functional diffeologies). The smoothness of the map $\mathcal{F}_{V_2^*\otimes V_2^*,V_1^*\otimes V_1^*}$ follows from Theorem 4.6 of \cite{pseudometric-pseudobundle}. Thus, we can
conclude that under the assumption of the gluing-dual commutativity condition, the corresponding map $\mathcal{P}$ is smooth.

\paragraph{The non-commutative case} The conclusion that the map $\mathcal{P}$ is smooth, is true in the non-commutative case as well (see \cite{pseudometric-pseudobundle}); although the
construction of the map $\tilde{g}$ does not fall within the standard procedure of gluing of two smooth maps, the proof is quite similar to this latter setting. In particular, the pseudo-bundle
$(V_1\cup_{\tilde{f}}V_2)^*\otimes(V_1\cup_{\tilde{f}}V_2)^*$ is disjointly covered by the sets $\left(j_1\otimes j_1\right)\left((V_1\setminus\pi_1^{-1}(Y))\otimes_{X_1\setminus Y}
(V_1\setminus\pi_1^{-1}(Y))\right)$ and $\left(j_2\otimes j_2\right)\left(V_2\otimes_{X_2} V_2\right)$, and this corresponds both to the two parts of the definition of $\tilde{g}$ and to the
presentation of plots of the appropriate gluing diffeology.

\subsection{The induced pseudo-metrics on dual pseudo-bundles}

Let us now consider the dual pseudo-metrics, meaning the ones that are --- possibly --- induced in some natural way on the corresponding dual pseudo-bundles. We use the standard pairing to define them, 
noting right away that we are only able to prove their existence under the assumptions that the initial pseudo-bundle is locally trivial,\footnote{Notice that this implies that the dual pseudo-bundle is locally trivial 
itself} although this may not constitute a significant restriction, as we are not aware of any examples of non locally trivial pseudo-bundles that admit pseudo-metrics. This exposition closely follows that in 
\cite{exterior-algebras-pseudobundles}.

\subsubsection{The induced pseudo-metric on $V^*$}

Let $\pi:V\to X$ be a locally trivial finite-dimensional diffeological vector pseudo-bundle, and let $g$ be a pseudo-metric on it. Let us first define the pseudo-bundle map $\Phi:V\to V^*$ (the meaning
of the term \emph{pseudo-bundle map} is the obvious one, $\pi=\pi^*\circ\Phi$), that fibrewise corresponds to the natural pairing map given by the pseudo-metric $g$. Specifically, we define:
$$\Phi(v)=g(\pi(v))(v,\cdot)\mbox{ for all }v\in V.$$ It is easy to show (see \cite{pseudometric} for the case of a single diffeological vector space, and then \cite{pseudometric-pseudobundle} for the case
of pseudo-bundles) that $\Phi$ is surjective, smooth, and linear on each fibre. Furthermore, although in general it is not invertible, we can still use it to correctly define a pseudo-metric on the dual
pseudo-bundle.

Indeed, the \textbf{induced pseudo-metric $g^*$} is defined by
$$g^*(x)(\Phi(v),\Phi(w)):=g(x)(v,w)\mbox{ for all }x\in X\mbox{ and for all }v,w\in V\mbox{ such that }\pi(v)=\pi(w)=x.$$ This is well-defined, because whenever $\Phi(v)=\Phi(v')$ (which obviously
can occur only for $v,v'$ belonging to the same fibre), the vectors $v$ and $v'$ differ by an element of the isotropic subspace of the fibre to which they (both) belong. We can also observe that,
since all fibres of any dual pseudo-bundle carry the standard diffeology,\footnote{Because the dual space of any finite-dimensional diffeological vector space is always standard.} $g^*(x)$ is a
scalar product.

\subsubsection{Existence of compatible pseudo-metrics on diffeological vector spaces}

What we are mostly interested in as far as the dual pseudo-metrics are concerned, is whether a pair of pseudo-metrics dual to a pair of compatible ones is in turn compatible (with what, will become
clear later). However, before considering such induced pseudo-metrics on dual pseudo-bundles, we need to consider the analogous question for the simpler case of individual diffeological vector
spaces; and this requires us to reflect some more on compatibility of pseudo-metrics in general.

Let $V$ and $W$ be finite-dimensional diffeological vector spaces, let $g_V$ be a pseudo-metric on $V$, and let $g_W$ be a pseudo-metric on $W$. Let $f:V\to W$ be a smooth linear map, with
respect to which $g_V$ and $g_W$ are compatible, $g_V(v_1,v_2)=g_W(f(v_1),f(v_2))$. As one can expect by analogy with the standard case, existence of $f$, $g_V$, and $g_W$ has
implications for the spaces $V$ and $W$ themselves, which are best explained in terms of a small preliminary notion.

\paragraph{The characteristic subspaces of $V$ and $W$} Given a pseudo-metric on a finite-dimensional diffeological vector space, the subspace generated by all the eigenvectors of this pseudo-metric
relative to the non-zero eigenvalues is a subspace whose subset diffeology is that of a standard space, and its dimension is maximal for this property. Moreover, this subspace splits off as a smooth
direct summand,\footnote{This means that the direct sum diffeology coincides with $V$'s or $W$'s own diffeology, or, alternatively, that the composition of each plot of $V$ (respectively $W$) with
the projection on $V_0$ (respectively $W_0$) is a plot of the latter.} and among all standard subspaces, it is unique with this property. Thus, the subspace in question does not actually depend on the
choice of a pseudo-metric and is an invariant of the space itself (see \cite{pseudometric}). We call this subspace the \textbf{characteristic subspace} of the diffeological vector space in question.

Let $V_0$ and $W_0$ be the characteristic subspaces of $V$ and $W$ respectively. Let also $V_1\leqslant V$ and $W_1\leqslant W$ be the isotropic subspaces relative to $g_V$ and $g_W$, so that
$V=V_0\oplus V_1$ and $W=W_0\oplus W_1$, with each decomposition being smooth. Recall also \cite{pseudometric} that $V_0$ is diffeomorphic to $V^*$ via (the restriction to $V_0$ of) the natural
pairing map $\Phi_V:v\mapsto g_V(v,\cdot)$, and likewise, $W_0$ is diffeomorphic to $W^*$ via $\Phi_W:w\mapsto g_W(w,\cdot)$. Let us consider the necessary and sufficient conditions for the
compatibility of $g_V$ and $g_W$.

\paragraph{The necessary conditions} These are easily found, also by using the standard reasoning. For instance, let $v\in V$ belong to the kernel of $f$; then by definition of compatibility we have
$$g_V(v,v')=g_W(0,f(v'))=0\mbox{ for any }v'\in V.$$ Thus, the kernel of $f$ is contained in the maximal isotropic subspace $V_1$, therefore the restriction of $f$ to $V_0$ is a bijection with its image.

A number of similar arguments easily yield the following:

\begin{lemma}\label{exist:compatible:pseudo-metrics:f:necessary:lem}
Let $V$ and $W$ be finite-dimensional diffeological vector spaces, and let $f:V\to W$ be a smooth linear map. If $V$ and $W$ admit pseudo-metrics compatible with $f$ then:
\begin{enumerate}
\item $\mbox{Ker}(f)\cap V_0=\{0\}$;
\item The subset diffeology of $f(V_0)$ is the standard one;
\item $\dim(V^*)\leqslant\dim(W^*)$.
\end{enumerate}
\end{lemma}

In particular, if $\dim(V^*)>\dim(W^*)$, then no two pseudo-metrics on $V$ and $W$ are compatible, whatever the map $f$. This is a reflection of the compatibility being an extension of the standard
situation: there is no isometry from the space of a bigger dimension to one of smaller dimension.\footnote{The choices of $f$ however could be plenty; it suffices to take $V$ the standard $\matR^n$
and $W$ any other diffeological vector space of dimension strictly smaller than $n$. Any linear map from $V$ to $W$ is then going to be smooth (see Section 3.9 in \cite{iglesiasBook}).}

\paragraph{Sufficient conditions} The same type of reasoning allows us to obtain the following statement.

\begin{thm}\label{criterio:exist:pseudometrics:vspaces:thm}
Let $V$ and $W$ be two finite-dimensional diffeological vector spaces, and let $f:V\to W$ be a smooth linear map. Then $V$ and $W$ admit compatible pseudo-metrics if and only if
$\mbox{Ker}(f)\cap V_0=\{0\}$ and $f(V_0)\leqslant W_0$.
\end{thm}

\begin{rem}
The fact that $f(V_0)\leqslant W_0$ is not entirely obvious (without there being pseudo-metrics compatible with $f$ it does not have to occur), since in general $W$ might contain many subspaces with
standard diffeology, which are not contained in its characteristic subspace. The reason why $f(V_0)$ is contained in it, follows from the fact it splits off smooth in $W$ (and this is a direct consequence
of $V$ and $W$ admitting pseudo-metrics compatible with $f$).
\end{rem}

\subsubsection{The compatibility of the induced pseudo-metrics on the duals of diffeological vector spaces}

Consider now the dual pseudo-metrics in the case of finite-dimensional diffeological vector spaces. Let $V$ be such a space.

\paragraph{The induced pseudo-metric on $V^*$} Recall (\cite{pseudometric}) that if $g$ is a pseudo-metric on $V$, the diffeological dual of $V$ carries the induced pseudo-metric $g^*$ (actually, a
scalar product, since the diffeological dual of any finite-dimensional diffeological vector space is standard) defined by
$$g^*(v_1^*,v_2^*):=g(v_1,v_2),$$ where $v_i\in V$ is any element such that $v_i^*(\cdot)=g(v_i,\cdot)$ for $i=1,2$. This is well-defined, in the sense that the result does not depend on the choice
(that in general is not unique) of $v_i$, as long as $g(v_i,\cdot)$ remains the same, and furthermore. $v_i^*$ always admits such a form.

\paragraph{The compatibility condition for $g_V^*$ and $g_W^*$} Let now $V$ and $W$ be two diffeological vector spaces, and let $g_V$ and $g_W$ be pseudo-metrics on $V$ and $W$ respectively,
compatible with respect to $f$. Let $w_1^*,w_2^*\in W^*$; then there exist $w_1,w_2\in W$, defined up to the cosets of the isotropic subspace of $g_W$, such that $w_i^*(\cdot)=g_W(w_i,\cdot)$
for $i=1,2$; for any such choice $g_W^*(w_1^*,w_2^*)=g_W(w_1,w_2)$. Furthermore, by the usual definition of the dual map $f^*(w_i^*)(\cdot)=w_i^*(f(\cdot))=g_W(w_i,f(\cdot))$. The \textbf{compatibility
condition} for $g_V^*$ and $g_W^*$ then takes the following form:
$$g_W^*(w_1^*,w_2^*)=g_V^*(f^*(w_1^*),f^*(w_2^*)).$$

\paragraph{When the dual pseudo-metrics are, or are not, compatible} Let us now consider the pseudo-metrics induced by a pair of compatible ones. \emph{A priori}, the dual pseudo-metrics may
easily not be compatible; it suffices to observe that the duals of finite-dimensional diffeological vector spaces are standard spaces, so pseudo-metrics on them are usual scalar products,  while
the notion of compatibility translates into $f^*$ being a usual isometry. This last point is a matter of additional assumptions on the original $V$, $W$, and $f$.

\begin{example}
Let $V$ be the standard $\matR^n$, with the canonical basis denoted by $e_1,\ldots,e_n$, and let $W$ be the standard $\matR^{n+k}$, with the canonical basis denoted by
$u_1,\ldots,u_n,u_{n+1},\ldots,u_{n+k}$. Let $f:V\to W$ be the embedding of $V$ via the identification of $V$ with the subspace generated by $u_1,\ldots,u_n$, given by $e_i\mapsto u_i$ for $i=1,\ldots,n$.
Let $g_V$ be any scalar product on $\matR^n$; this trivially induces a scalar product on $f(V)=\mbox{Span}(u_1,\ldots,u_n)\leqslant W_0$, and let $g_W$ be any extension of it to a scalar product
on the whole $W$.

Let us consider the dual map on the dual the standard complement of the subspace $\mbox{Span}(u_1,\ldots,u_n)$, that is, on the dual of $\mbox{Span}(u_{n+1},\ldots,u_{n+k})$. This dual is the usual
dual, so it is $\mbox{Span}(u^{n+1},\ldots,u^{n+k})$. Let $v$ be any element of $V$; since $f(v)\in\mbox{Span}(u_1,\ldots,u_n)$, we have
$$f^*(u^{n+i})(v)=u^{n+i}(f(v))=0,$$ so in the end we obtain that $\mbox{Ker}(f^*)=\mbox{Span}(u^{n+1},\ldots,u^{n+k})$.

Finally, let us consider the compatibility condition. We observe that
$$g_W^*(u^{n+i},u^{n+i})=g_W(u_{n+i},u_{n+i})>0,$$ since $g_W$ is a scalar product, while, of course,
$$g_V^*(f^*(u^{n+i}),f^*(u^{n+i}))=0.$$ Quite evidently, the compatibility condition cannot be satisfied (unless $k=0$).
\end{example}

The criterion for compatibility of the induced pseudo-metrics with the dual map has the following form.

\begin{thm}
Let $V$ and $W$ be two finite-dimensional diffeological vector spaces, and let $f:V\to W$ be a smooth linear map such that $\mbox{Ker}(f)\cap V_0=\{0\}$ and $f(V_0)\leqslant W_0$. Let $g_V$
and $g_W$ be compatible pseudo-metrics on $V$ and $W$ respectively. Then the induced pseudo-metrics $g_W^*$ and $g_V^*$ are compatible with $f^*$ if and only if $W^*$ and $V^*$ are
diffeomorphic.
\end{thm}

\begin{rem}
The condition $f(V_0)\leqslant W_0$ implies implies in particular that it is precisely the map $f^*$ that yields a diffeomorphism between $W^*$ and $V^*$. We also notice that the existence of a
diffeomorphism between $W^*$ and $V^*$ does not mean that $W$ and $V$ are themselves diffeomorphic; only their characteristic subspaces are.
\end{rem}

\subsubsection{The induced gluing and compatibility of the induced pseudo-metrics: diffeological pseudo-bundles}

We now consider the same question for diffeological pseudo-bundles. Namely, let $g_1$ and $g_2$ be pseudo-metrics on $\pi_1:V_1\to X_1$ and $\pi_2:V_2\to X_2$, compatible with respect to the
gluing along a given pair of maps $(\tilde{f},f)$, when is it true that $g_2^*$ and $g_1^*$ are compatible with the gluing of $\pi_2^*:V_2^*\to X_2$ and $\pi_1^*:V_1^*\to X_1$ along  $(\tilde{f}^*,f^{-1})$?

By the general definition, the compatibility of $g_2^*$ and $g_1^*$ means the following. Recall first that  $g_2^*:X_2\to(V_2^*)^*\otimes(V_2^*)^*$ and $g_1^*:X_1\to(V_1^*)^*\otimes(V_1^*)^*$, and
that the induced gluing of the dual pseudo-bundles produces the pseudo-bundle $\pi_2^*\cup_{(\tilde{f}^*,f^{-1})}\pi_1^*:V_2^*\cup_{\tilde{f}^*}V_1^* \to X_2\cup_{f^{-1}}X_1$. The compatibility condition
is then that there be
$$g_1^*(f^{-1}(y'))(\tilde{f}^*(v^*),\tilde{f}^*(w^*))=g_2^*(y')(v^*,w^*)$$
$$\mbox{for all }y'\in Y'=f(Y)\mbox{ and for all }v^*,w^*\in(\pi_2^{-1}(y'))^*.$$

\paragraph{The necessary condition} The compatibility between $g_1$ and $g_2$ implies in particular that for all $y\in Y$ the pseudo-metrics $g_1(y)$ and $g_2(f(y))$ are compatible with the
smooth linear map $\tilde{f}|_{\pi_1^{-1}(y)}$ between diffeological vector spaces $\pi_1^{-1}(y)$ and $\pi_2^{-1}(f(y))$. Therefore the following statement is a direct consequence of the results stated in
the previous section.

\begin{prop}
Let $\pi_1:V_1\to X_1$ and $\pi_2:V_2\to X_2$ be diffeological vector pseudo-bundles with finite-dimensional fibres, and let $(\tilde{f},f)$ be a gluing between them such that $f$ is smoothly
invertible. Let $g_1$ and $g_2$ be two pseudo-metrics on these pseudo-bundles compatible with the gluing along $(\tilde{f},f)$. If the induced pseudo-metrics $g_2^*$ and $g_1^*$ are compatible with
the gluing along $(\tilde{f}^*,f^{-1})$ then for every $y\in Y$ the restriction of $\tilde{f}$ on the fibre $\pi_1^{-1}(y)$ yields a diffeomorphism between the characteristic subspaces of
$\pi_1^{-1}(y)$ and $\pi_2^{-1}(f(y))$.
\end{prop}

Assuming furthermore that the two pseudo-bundles are locally trivially, we can then get more than just fibrewise diffeomorphism. Indeed, the collection of all the characteristic subspaces in the
total space $V$ of some pseudo-bundle (like any other collection of subspaces, one per fibre, see \cite{pseudobundle}), forms a sub-bundle of $V$, called its \textbf{characteristic sub-bundle};
$\tilde{f}$ gives a pseudo-bundle diffeomorphism between the characteristic sub-bundles of $V_1$ and $V_2$ over $Y$ and $f(Y)$.

\paragraph{Criterion of compatibility} The above statement can easily be reversed to obtain a criterion of when the induced pseudo-metrics on the dual pseudo-bundles are compatible with the dual
map $\tilde{f}^*$. This criterion is quite close to the standard one, asking for $\tilde{f}^*$ to be a usual fibrewise isometry and as close as possible to a usual smooth bundle map, although it is not one
exactly, due to the fact that it does not have to be defined on usual open sets.

\begin{thm}\label{when:dual:pseudometrics:compatible:bundles:thm}
Let $\pi_1:V_1\to X_1$ and $\pi_2:V_2\to X_2$ be two diffeological vector pseudo-bundles, locally trivial and with finite-dimensional fibres, let $(\tilde{f},f)$ be a gluing between them, and let $g_1$
and $g_2$ be pseudo-metrics on $V_1$ and $V_2$ respectively, that are compatible with the gluing along $(\tilde{f},f)$. Then the induced pseudo-metrics $g_2^*$ and $g_1^*$ on the corresponding
dual pseudo-bundles are compatible with the gluing along $(\tilde{f}^*,f^{-1})$ if and only if $\tilde{f}^*$ is a pseudo-bundle diffeomorphism of its domain with its image.
\end{thm}

Do notice that $\tilde{f}^*$ being a diffeomorphism does \emph{not} imply that $\tilde{f}$ itself is a diffeomorphism, only that its restriction to the characteristic sub-bundle is so.

\paragraph{Compatibility of $g_2^*$ and $g_1^*$ implies the gluing-dual commutativity} It follows from Theorem \ref{when:dual:pseudometrics:compatible:bundles:thm}, the remark that follows it,
and the criterion of the compatibility of pseudo-metrics $g_1$ and $g_2$, that the gluing-dual commutativity condition for $V_1$, $V_2$, and $(\tilde{f},f)$ is closely related to the compatibility of the dual
pseudo-metrics. In fact, under the assumptions we have already imposed, they are equivalent, as the next theorem shows.

\begin{thm}
Let $\pi_1:V_1\to X_1$ and $\pi_2:V_2\to X_2$ be diffeological vector pseudo-bundles, locally trivial and with finite-dimensional fibres, and let $(\tilde{f},f)$ be a gluing of $V_1$ to $V_2$, with
a smoothly invertible $f$. Suppose that $V_1$ and $V_2$ admit pseudo-metrics compatible with this gluing, and let $g_1$ and $g_2$ be a fixed choice of such pseudo-metrics. Then the induced
pseudo-metrics $g_2^*$ and $g_1^*$ on the dual pseudo-bundles $V_2^*$ and $V_1^*$ are compatible with the gluing along $(\tilde{f}^*,f^{-1})$ if and only if $V_1$, $V_2$, and $(\tilde{f},f)$ satisfy
the gluing-dual commutativity condition.
\end{thm}

We notice the \emph{only if} part of the statement uses explicitly the assumption that $V_1$ and $V_2$ admit a choice of compatible pseudo-metrics, and specifically, the implications of their
existence for the behavior of $\tilde{f}$ on the corresponding characteristic sub-bundles.

\subsubsection{The pseudo-metrics $\tilde{g}^*$ and $\widetilde{g^*}$}

Assuming the gluing-dual commutativity for $V_1$, $V_2$, and $(\tilde{f},f)$, not only implies that $(V_1\cup_{\tilde{f}}V_2)^*\cong V_2^*\cup_{\tilde{f}^*}V_1^*$; it also allows us to consider two
pseudo-metrics on it. Indeed, there is a natural pseudo-metric corresponding to the presentation of this space by the left-hand side expression, and there is one corresponding to the  right-hand side.

Specifically, the pseudo-bundle on the left carries the pseudo-metric $\tilde{g}^*$ that is induced by the pseudo-metric $\tilde{g}$. The pseudo-bundle on the right is obtained by gluing of
two pseudo-bundles carrying compatible pseudo-metrics each; it therefore carries a pseudo-metric $\widetilde{g^*}$ corresponding to this gluing. They are, respectively, maps
$$\tilde{g}^*:X_1\cup_f X_2\to(V_1\cup_{\tilde{f}}V_2)^{**}\otimes(V_1\cup_{\tilde{f}}V_2)^{**}\,\,\mbox{ and }\,\,\widetilde{g^*}:X_2\cup_{f^{-1}}X_1\to(V_2^*\cup_{\tilde{f}^*}V_1^*)^*\otimes
(V_2^*\cup_{\tilde{f}^*}V_1^*)^*.$$ It turns out that they are related by the natural diffeomorphisms between their domains and their ranges.

\paragraph{The pseudo-metric $\tilde{g}^*$} It is defined as the pseudo-metric dual to the pseudo-metric $\tilde{g}$ on $(V_1\cup_{\tilde{f}}V_2)^*$. Specifically, if
$\Psi:V_1\cup_{\tilde{f}}V_2\to(V_1\cup_{\tilde{f}}V_2)^*$ is the pairing map relative to $\tilde{g}$, that is, $\Psi(v)=\tilde{g}((\pi_1\cup_{(\tilde{f},f)}\pi_2)(v))(v,\cdot)$, then we have, for any
$x\in X_1\cup_f X_2$ and any $v^*,w^*\in((\pi_1\cup_{(\tilde{f},f)}\pi_2)^*)^{-1}(x)$, that
$$\tilde{g}^*(x)(v^*,w^*)=\tilde{g}(x)(v,w),$$ where $v,w$ are such that $\Psi(v)=v^*$ and $\Psi(w)=w^*$. We can also write in more detail that
$$\tilde{g}^*(x)(v^*,w^*)=\tilde{g}(x)(v,w)=\left\{\begin{array}{ll}
g_1((i_1^{X_1})^{-1}(x))((j_1^{V_1})^{-1}(v),(j_1^{V_1})^{-1}(w)) & \mbox{if }x\in\mbox{Range}(i_1^{X_1}),\\
g_2((i_2^{X_2})^{-1}(x))((j_2^{V_2})^{-1}(v),(j_2^{V_2})^{-1}(w)) & \mbox{if }x\in\mbox{Range}(i_2^{X_2}).
\end{array}\right.$$

\paragraph{The pseudo-metric $\widetilde{g^*}$} This one is defined on the pseudo-bundle $V_2^*\cup_{\tilde{f}^*}V_1^*$ fibrewise, by imposing it to coincide with $g_2^*$ or $g_1^*$, as appropriate.
Specifically, let $\Psi_i:V_i\to V_i^*$, for $i=1,2$, be the natural pairing maps associated to $g_1$ and $g_2$; for all $x\in X_2\cup_{f^{-1}}X_1$ and for all
$v^*,w^*\in(\pi_2^*\cup_{(\tilde{f}^*,f^{-1})}\pi_1^*)^{-1}(x)$ we have
$$\widetilde{g^*}(x)(v^*,w^*)=\left\{\begin{array}{ll}
g_2^*((i_1^{X_2})^{-1}(x))((j_1^{V_2^*})^{-1}(v^*),(j_1^{V_2^*})^{-1}(w^*))=g_2((i_1^{X_2})^{-1}(x))(v,w) & \mbox{if }x\in\mbox{Range}(i_1^{X_2}) \\
g_1^*((i_2^{X_1})^{-1}(x))((j_2^{V_1^*})^{-1}(v^*),(j_2^{V_1^*})^{-1}(w^*))=g_1((i_2^{X_1})^{-1}(x))(v,w) & \mbox{if }x\in\mbox{Range}(i_2^{X_1}),
\end{array}\right.$$ where $v,w$ are determined\footnote{Not uniquely, unless we assume to take them in the characteristic subspace of the corresponding fibre.} in the following way. If
$x\in\mbox{Range}(i_1^{X_2})$ then $v^*=j_1^{V_2^*}(\Psi_2(v))$ and $w^*=j_1^{V_2^*}(\Psi_2(w))$; and if $x\in\mbox{Range}(i_2^{X_1})$ then $v^*=j_2^{V_1^*}(\Psi_1(v))$ and
$w^*=j_2^{V_1^*}(\Psi_1(w))$.

\paragraph{Comparing $\tilde{g}^*$ and $\widetilde{g^*}$} To say that they define the same pseudo-metric means to claim the existence of a diffeomorphism
$$\Psi':(V_1\cup_{\tilde{f}}V_2)^{**}\otimes(V_1\cup_{\tilde{f}}V_2)^{**}\to(V_2^*\cup_{\tilde{f}^*}V_1^*)^*\otimes(V_2^*\cup_{\tilde{f}^*}V_1^*)^*$$ such that
$$\Psi'\circ\tilde{g}^*=\widetilde{g^*}\circ(\varphi_{X_1\leftrightarrow X_2}).$$ It suffices to take the tensor square of the inverse of the conjugate of the gluing-dual commutativity diffeomorphism, that is,
$$\left((\Phi_{\cup,*})^*\right)^{-1}:(V_1\cup_{\tilde{f}}V_2)^{**}\to(V_2^*\cup_{\tilde{f}^*}V_1^*)^*;$$ the desired equality
$$\left(\left((\Phi_{\cup,*})^*\right)^{-1}\otimes\left((\Phi_{\cup,*})^*\right)^{-1}\right)\circ\tilde{g}^*=\widetilde{g^*}\circ(\varphi_{X_1\leftrightarrow X_2})$$ follows from the above presentations of
$\tilde{g}^*$ and $\widetilde{g^*}$.

\subsection{More on gluing-dual commutativity}

We now consider the specific instances of the gluing-dual commutativity. Actually, the main result in this direction has already been stated as Theorem 5.25. Here we add some other specific
instances of when the gluing-dual commutativity is satisfied, including one that is the main auxiliary tool in proving the just-mentioned theorem.

\subsubsection{The gluing-dual commutativity and gluing along diffeomorphisms}

Gluing along a diffeomorphism is not a strictly necessary condition for the gluing-dual commutativity, but it is a sufficient one, see \cite{pseudobundle} and then \cite{exterior-algebras-pseudobundles}
for an explicit construction. That it is not necessary, is due to the fact that the dual pseudo-bundles are essentially determined by the characteristic sub-bundles; $\tilde{f}$ may behave to its liking
outside of these. On the other hand, it is sufficient for it to be a diffeomorphism of its domain of definition with its image, in order for the gluing-commutativity condition to be satisfied:

\begin{thm}\label{gluing:along:diffeo:implies:gluing:dual:commute:thm}
\emph{(\cite{exterior-algebras-pseudobundles}, Theorem 3.3)} Let $\pi_1:V_1\to X_1$ and $\pi_2:V_2\to X_2$ be two finite-dimensional diffeological vector pseudo-bundles, let $f:X_1\supseteq Y\to X_2$
be a diffeomorphism with its image, and let $\tilde{f}$ be its fibrewise linear lift that is also a diffeomorphism with its image. Then the map
$$\Phi_{\cup,*}:(V_1\cup_{\tilde{f}}V_2)^*\to V_2^*\cup_{\tilde{f}^*}V_1^*$$ defined by
$$\Phi_{\cup,*}=\left\{\begin{array}{ll}
j_2^{V_1^*}\circ(j_1^{V_1})^* & \mbox{on }((\pi_1\cup_{(\tilde{f},f)}\pi_2)^*)^{-1}(i_1^{X_1}(X_1\setminus Y))\\
j_2^{V_1^*}\circ\tilde{f}^*\circ(j_2^{V_2})^* & \mbox{on }((\pi_1\cup_{(\tilde{f},f)}\pi_2)^*)^{-1}(i_2^{X_2}(f(Y))\\
j_1^{V_2^*}\circ(j_2^{V_2})^* & \mbox{on }((\pi_1\cup_{(\tilde{f},f)}\pi_2)^*)^{-1}(i_2^{X_2}(X_2\setminus f(Y)))
\end{array}\right.$$ is a pseudo-bundle diffeomorphism covering the switch map $\varphi_{X_1\leftrightarrow X_2}$.
\end{thm}

The proof of the above-stated Theorem 5.25 is actually based on this statement, together with the implications of the existence of compatible $g_1$ and $g_2$ whose dual pseudo-metrics $g_2^*$ and
$g_1^*$ are compatible as well. We also note that the surprising thing about the above statement is not the existence itself of a bijective map $\Phi_{\cup,*}$ --- this is quite obvious from the definition
of a dual pseudo-bundle, but the fact that defining it by concatenating some rather disomogenous pieces does yield a diffeologically smooth map.

\subsubsection{The gluing-dual commutativity condition for $V_2^*$ and $V_1^*$}

Let us consider the existence of the gluing-dual commutativity condition for $V_2^*$, $V_1^*$, and $(\tilde{f}^*,f^{-1})$, under the assumption that such condition holds for $V_1$, $V_2$, and $(\tilde{f},f)$.
For the duals, this condition takes form of the existence of a diffeomorphism
$$\Phi_{\cup,*}^{(*)}:(V_2^*\cup_{\tilde{f}^*}V_1^*)^*\to V_1^{**}\cup_{\tilde{f}^{**}}V_2^{**}$$ covering the inverse of the switch map $\varphi_{X_1\leftrightarrow X_2}$. Its existence is stated in the
following theorem (see \cite{exterior-algebras-pseudobundles} for the proof) and is a direct consequence of Theorem \ref{gluing:along:diffeo:implies:gluing:dual:commute:thm}.

\begin{thm}
Let $\pi_1:V_1\to X_1$ and $\pi_2:V_2\to X_2$ be two locally trivial finite-dimensional diffeological vector pseudo-bundles, let $(\tilde{f},f)$ be a pair of smooth maps that defines a gluing of
the former pseudo-bundle to the latter, and let $\Phi_{\cup,*}:(V_1\cup_{\tilde{f}}V_2)^*\to V_2^*\cup_{\tilde{f}^*}V_1^*$ be the diffeomorphism fulfilling the gluing-dual commutativity condition.
Let $g_1$ and $g_2$ be pseudo-metrics on $V_1$ and $V_2$ respectively, compatible with respect to the gluing. Then there exists a diffeomorphism
$$\Phi_{\cup,*}^{(*)}:(V_2^*\cup_{\tilde{f}^*}V_1^*)^*\to V_1^{**}\cup_{\tilde{f}^{**}}V_2^{**}$$ covering the map $(\varphi_{X_1\leftrightarrow X_2})^{-1}:X_2\cup_{f^{-1}}X_1\to X_1\cup_f X_2$.
\end{thm}

The main point here is that the assumptions imply that $\tilde{f}^*$ is a diffeomorphism.

\section{The space of sections of a diffeological vector pseudo-bundle}

Let $\pi:V\to X$ be a finite-dimensional diffeological vector pseudo-bundle; we now consider the space $C^{\infty}(X,V)$ of its sections, under two main respects. For one thing, it is quite easy
to observe that this space may easily turn out to be (locally) infinite-dimensional, even for very simple examples of a pseudo-bundle, which of itself has finite dimension; this we illustrate
immediately via a specific example of such. After that, we turn to a more general treatment of the behavior of the spaces of sections under diffeological gluing; this implies, in particular,
that if we glue together two pseudo-bundles $\pi_1:V_1\to X_1$ and $\pi_2:V_2\to X_2$ such that $C^{\infty}(X_1,V_1)$ and $C^{\infty}(X_2,V_2)$ are (locally) finite-dimensional then the space of
sections of the resulting pseudo-bundle is finite-dimensional as well,\footnote{This is an expected finding, but given the flexibility of the diffeology with what can be considered as a smooth map,
it is not entirely trivial to establish formally.} but this is not quite \emph{vice versa}. The proofs of statements in this section appear in \cite{connections-pseudobundles}.

\subsection{A pseudo-bundle with no local basis}

We now illustrate the issue of there possibly not being a local basis of smooth sections with coefficients in $C^{\infty}(X,\matR)$. This easily occurs as soon as we have fibres with non-standard
diffeology, as it does in the following example.

\begin{example}
Let $\pi:V\to X$ be the projection of $V=\matR^3$ onto its first coordinate, so $X$ is $\matR$, which we endow with the standard diffeology. Endow $V$ with the pseudo-bundle diffeology generated
by the plot $\matR^2\ni(u,v)\mapsto(u,0,|v|)$; recall that this diffeology, already seen before, is a product diffeology for the decomposition $\matR^3=\matR\times\matR^2$ into the direct product
of the standard $\matR$ with $\matR^2$ carrying the vector space diffeology generated by the plot $v\mapsto(0,|v|)$.
\end{example}

\begin{oss}
The space $C^{\infty}(X,V)$ of smooth sections of the pseudo-bundle $\pi$ is not finitely generated over $C^{\infty}(X,\matR)$.
\end{oss}

This observation (which certainly can be obtained by some standard analytic argument) can be verified directly, recalling that any arbitrary plot of $V$ has form
$$\matR^{l+m+n}\supseteq U\ni(x,y,z)\mapsto(f_1(x),f_2(y),g_0(z)+g_1(z)|h_1(z)|+\ldots+g_k(z)|h_k(z)|),$$ where $U$ is a domain, and $f_1:\matR^l\subseteq U_x\to\matR$,
$f_2:\matR^m\supseteq U_y\to\matR$ and $g_0,g_1,\ldots,g_k,h_1,\ldots,h_k:\matR^n\supseteq U_z\to\matR$ are some ordinary smooth functions. Thus, any smooth section $s\in C^{\infty}(X,V)$ has
(at least locally) form
$$s(x)=(x,f(x),g_0(x)+g_1(x)|h_1(x)|+\ldots+g_k(x)|h_k(x)|)$$ for some ordinary smooth functions $f,g_0,g_1,\ldots,g_k,h_1,\ldots,h_k:\matR\supseteq U\to\matR$; and \emph{vice versa} every such
expression corresponds (at least, locally) to a smooth section $X\to V$. But since $g_i$ and $h_i$ are any smooth functions at all, and they can be in any finite number, for any finite arbitrarily long
collection $x_1,\ldots,x_k\in\matR$ there is a diffeologically smooth section $s$ that, seen as a usual map $\matR\to\matR^3$, is non-differentiable precisely at the points $x_1,\ldots,x_k$ (and smooth
outside of them). Thus, it is impossible that all such sections be linear combinations over $C^{\infty}(\matR,\matR)$ of the same finite set of continuous\footnote{That all sections in $C^{\infty}(X,V)$ are
continuous in the usual sense follows from their explicit description given above.} functions $\matR\to\matR^3$.

\subsection{The $(f,\tilde{f})$-invariance of a section, and compatibility of two sections}

What we obtain from the example to which the previous section is dedicated is that for a given pseudo-bundle the space of sections can \emph{a priori} be infinite-dimensional, and whether it is, or
it is not, depends on the specific pseudo-bundle at hand. So in particular, we relate the matter of (finiteness of) the dimension of $C^{\infty}(X,V)$ in terms of its interaction with the gluing procedure.
Accordingly, we concentrate on how the space of sections $C^{\infty}(X_1\cup_f X_2,V_1\cup_{\tilde{f}}V_2)$ of some pseudo-bundle obtained by gluing is related to the spaces of sections of
the factors of that gluing.

\paragraph{$(f,\tilde{f})$-invariant sections $X_1\to V_1$} Let $s_1\in C^{\infty}(X_1,V_1)$ be a section. It is said to be \textbf{$(f,\tilde{f})$-invariant} if for all $y,y'\in Y$ such that $f(y)=f(y')$ we have
$\tilde{f}(s_1(y))=\tilde{f}(s_1(y'))$. As is obvious from this definition, if the map $f$ is injective, any section is automatically $(f,\tilde{f})$-invariant.

The subset of $C^{\infty}(X_1,V_1)$ that consists of all $(f,\tilde{f})$-invariant sections is denoted by $C_{(f,\tilde{f})}^{\infty}(X_1,V_1)$. This subset is closed with respect to the summation of sections,
and with respect to the multiplication by \textbf{$f$-invariant} smooth functions, that is, functions $h:X_1\to\matR$ that possess the following property: for all $y,y'\in Y$ such that $f(y)=f(y')$ we have
$h(y)=h(y')$. Thus, if we denote the latter set of functions by $C_f^{\infty}(X_1)$ (once again, this is simply $C^{\infty}(X)$ if $f$ is injective) then $C_{(f,\tilde{f})}^{\infty}(X_1,V_1)$ is a module over the
ring $C_f^{\infty}(X)$, with the module structure obviously inherited from that of $C^{\infty}(X_1,V_1)$ as a module over $C^{\infty}(X)$.

The need for the notion of $(f,\tilde{f})$-invariance of sections will be illustrated below through the notion of compatibility of sections.

\paragraph{Compatibility of a section $s_1\in C^{\infty}(X_1,V_1)$ with a section $s_2\in C^{\infty}(X_2,V_2)$} Suppose that we have two diffeological vector pseudo-bundles $\pi_1:V_1\to X_1$ and
$\pi_2:V_2\to X_2$, a gluing between them given by maps $f:X_1\supset Y\to X_2$ and $\tilde{f}:\pi_1^{-1}(Y)\to V_2$, and a pair of sections $s_i\in C^{\infty}(X_i,V_i)$ for $i=1,2$.  We say that the sections
$s_1$ and $s_2$ are \textbf{$(f,\tilde{f})$-compatible}, or simply \textbf{compatible}, if for all $y\in Y$ we have
$$\tilde{f}(s_1(y))=s_2(f(y)).$$ It is now trivial to observe the following:

\begin{lemma}
Let $s_1\in C^{\infty}(X_1,V_1)$ be such that there exists $s_2\in C^{\infty}(X_2,V_2)$ compatible with $s_1$. Then $s_1$ is $(f,\tilde{f})$-invariant.
\end{lemma}

Thus, if we consider the subset
$$C^{\infty}(X_1,V_1)\times_{comp}C^{\infty}(X_2,V_2)\subseteq C^{\infty}(X_1,V_1)\times C^{\infty}(X_2,V_2)$$ consisting of all pairs $(s_1,s_2)$ with $s_i\in C^{\infty}(X_i,V_i)$ for $i=1,2$ such
that $s_1$ and $s_2$ are $(f,\tilde{f})$-compatible, then in fact we obtain a subset of $C_{(f,\tilde{f})}^{\infty}(X_1,V_1)\times C^{\infty}(X_2,V_2)$. Thus, we will denote the set of all pairs of
compatible sections by
$$C_{(f,\tilde{f})}^{\infty}(X_1,V_1)\times_{comp}C^{\infty}(X_2,V_2)=\{(s_1,s_2)\,|\,s_i\in C^{\infty}(X_i,V_i),\,\tilde{f}(s_1(y))=s_2(f(y))\,\forall y\in Y\}.$$

\paragraph{Compatibility of sections $s_i\in C^{\infty}(X_i,V_i)$, and sections in $C^{\infty}(X_1\cup_f X_2,V_1\cup_{\tilde{f}}V_2)$} Our motivation for introducing a separate compatibility notion
for sections of pseudo-bundles will be clarified immediately below, but here we provide an initial indication to that effect.

\begin{lemma}
Let $s\in C^{\infty}(X_1\cup_f X_2,V_1\cup_{\tilde{f}}V_2)$ be such that there exist $s_i\in C^{\infty}(X_i,V_i)$, with $i=1,2$, for which the following is true:
$$\tilde{j}_1\circ s_1=s\circ\tilde{i}_1\,\,\,\mbox{ and }\,\,\,j_2\circ s_2=s\circ i_2,$$ where $\tilde{i}_1:X_1\hookrightarrow X_1\sqcup X_2\to X_1\cup_f X_2$ is the composition of the obvious
inclusion $X_1\hookrightarrow X_1\sqcup X_2$ with the quotient projection onto $X_1\cup_f X_2$, and $\tilde{j}_1:V_1\to V_1\cup_{\tilde{f}}V_2$ is analogously defined. Then $s_1$ and $s_2$
are compatible.
\end{lemma}

Notice that if our gluing is along a pair of diffeomorphisms then $s_1$ and $s_2$ are essentially the restrictions of $s$ onto appropriate subsets of $X_1\cup_f X_2$. In the general case, $s_2$ always
exists, but $s_1$ \emph{a priori} may not, since in general $V_1$ does not inject in $V_1\cup_{\tilde{f}}V_2$, nor does $X_1$ in $X_1\cup_f X_2$. This is also the reason why the statement of
the above lemma, simple in essence, becomes rather convoluted.

\subsection{The gluing of compatible sections}

As we just stated, it is not \emph{a priori} clear (but it is true; we will establish this later) that every section $X_1\cup_f X_2\to V_1\cup_{\tilde{f}}V_2$ determines a pair of compatible ones.
The reverse, on the other hand, is easily seen; indeed, the compatibility of two section is a partial case of compatibility of two smooth maps with respect to the gluings of, respectively,
their domains and their ranges (see \cite{pseudometric-pseudobundle}).

\paragraph{From two compatible sections to a section $X_1\cup_f X_2\to V_1\cup_{\tilde{f}}V_2$} Given two compatible sections $s_1\in C^{\infty}(X_1,V_1)$ and $s_2\in C^{\infty}(X_2,V_2)$,
we define a map, that we denote by $s_1\cup_{(f,\tilde{f})}s_2$ and that is a section in $C^{\infty}(X_1\cup_f X_2,V_1\cup_{\tilde{f}}V_2)$. It is determined by the following formula:
$$(s_1\cup_{(f,\tilde{f})}s_2)(x)=\left\{\begin{array}{ll} j_1^{V_1}(s_1(x)) & \mbox{if }x\in i_1^{X_1}(X_1\setminus Y),\mbox{ and}\\
j_2^{V_2}(s_2(x)) & \mbox{if }x\in i_2^{X_2}(X_2) \end{array}\right.$$ Since $i_1^{X_1}(X_1\setminus Y)$ and $i_2^{X_2}(X_2)$ cover $X_1\cup_f X_2$ and are disjoint, $s_1\cup_{(f,\tilde{f})}s_2$
is well-defined. Finally, it follows from Proposition 4.2 of \cite{pseudometric-pseudobundle} that it is smooth, \emph{i.e.}, it does belong to $C^{\infty}(X_1\cup_f X_2,V_1\cup_{\tilde{f}}V_2)$.

\paragraph{Gluing of compatible sections and operations} The compatibility of sections is a property that is well-behaved with respect to the operations. More precisely, this means the following.

\begin{lemma}
Let $\pi_1:V_1\to X_1$ and $\pi_2:V_2\to X_2$ be two finite-dimensional diffeological vector pseudo-bundles, and let $(\tilde{f},f)$ be a gluing between them. Let $\pi_1':V_1'\to X_1$
and $\pi_2':V_2'\to X_2$ be two other pseudo-bundles over the same $X_1$ and $X_2$ respectively, and let $(\tilde{f}',f)$ be a gluing between these. Then:
\begin{enumerate}
\item If $s_1,t_1\in C^{\infty}(X_1,V_1)$ and $s_2,t_2\in C^{\infty}(X_2,V_2)$ are such that both $(s_1,s_2)$ and $(t_1,t_2)$ are $(f,\tilde{f})$-compatible pairs, then also $(s_1+t_1,s_2+t_2)$
is a $(f,\tilde{f})$-compatible pair, and
$$(s_1+t_1)\cup_{(f,\tilde{f})}(s_2+t_2)=s_1\cup_{(f,\tilde{f})}s_2+t_1\cup_{(f,\tilde{f})}t_2;$$

\item If $s_1\in C^{\infty}(X_1,V_1)$ and $s_2\in C^{\infty}(X_2,V_2)$ are $(f,\tilde{f})$-compatible sections, and $h_1\in C^{\infty}(X_1,\matR)$ and $h_2\in C^{\infty}(X_2,\matR)$ are such that
$h_2(f(y))=h_1(y)$ for all $y\in Y$, then $h_1s_1$ and $h_2s_2$ are also $(f,\tilde{f})$-compatible and
$$(h_1\cup_f h_2)\left(s_1\cup_{(f,\tilde{f})}s_2\right)=(h_1s_1)\cup_{(f,\tilde{f})}(h_2s_2),$$ where $h_1\cup_f h_2\in C^{\infty}(X_1\cup_f X_2,\matR)$ is defined by
$$(h_1\cup_f h_2)(x)=\left\{\begin{array}{ll} h_1((i_1^{X_1})^{-1}(x)) & \mbox{for }x\in i_1^{X_1}(X_1\setminus Y)\\
h_2((i_2^{X_2})^{-1}(x)) & \mbox{for }x\in i_2^{X_2}(X_2); \end{array}\right.$$

\item If $s_i\in C^{\infty}(X_i,V_i)$ for $i=1,2$ are $(f,\tilde{f})$-compatible, and $s_i'\in C^{\infty}(X_i,V_i')$ are $(f,\tilde{f}')$-compatible, then $s_1\otimes s_1'\in C^{\infty}(X_1,V_1\otimes V_1')$
and $s_2\otimes s_2'\in C^{\infty}(X_2,V_2\otimes V_2')$ are $(f,\tilde{f}\otimes\tilde{f}')$-compatible, and
$$(s_1\otimes s_1')\cup_{(f,\tilde{f}\otimes\tilde{f}')}(s_2\otimes s_2')=\left(s_1\cup_{(f,\tilde{f})}s_2\right)\otimes\left(s_1'\cup_{(f,\tilde{f}')}s_2'\right).$$
\end{enumerate}
\end{lemma}

Let us illustrate the proof (which is very simple) of the first item.

\begin{proof}
Let $y\in Y$; then
$$\tilde{f}(s_1(y)+t_1(y))=\tilde{f}(s_1(y))+\tilde{f}(t_1(y))=s_2(f(y))+t_2(f(y)),$$ so $s_1+t_1$ and $s_2+t_2$ are $(f,\tilde{f})$-compatible. Now, by definition
\begin{flushleft}
$\left((s_1+t_1)\cup_{(f,\tilde{f})}(s_2+t_2)\right)(x)=\left\{\begin{array}{l}(s_1+t_1)((i_1^{X_1})^{-1}(x))=s_1((i_1^{X_1})^{-1}(x))+t_1((i_1^{X_1})^{-1}(x))\\
(s_2+t_2)((i_2^{X_2})^{-1}(x))=s_2((i_2^{X_2})^{-1}(x))+t_2((i_2^{X_2})^{-1}(x))\end{array}\right.=$
\end{flushleft}
\begin{flushright}
$=\left\{\begin{array}{l}s_1((i_1^{X_1})^{-1}(x))\\ s_2((i_2^{X_2})^{-1}(x))\end{array}\right.
+\left\{\begin{array}{l}t_1((i_1^{X_1})^{-1}(x))\\ t_2((i_2^{X_2})^{-1}(x))\end{array}\right.=(s_1\cup_{(f,\tilde{f})}s_2)(x)+(t_1\cup_{(f,\tilde{f})}t_2)(x)$,
\end{flushright} where in each two-part formula the first line applies to $x\in i_1^{X_1}(X_1\setminus Y)$ and the second line, to $x\in i_2^{X_2}(X_2)$. The final equality that we obtain is
precisely the first item in the statement of the lemma, so we are done.
\end{proof}

\subsection{The space $C^{\infty}(X_1\cup_f X_2,V_1\cup_{\tilde{f}}V_2)$}

Let again $\pi_i:V_i\to X_i$ for $i=1,2$ be two pseudo-bundles, and let $(\tilde{f},f)$ be a gluing between them. We now consider thespaces $C^{\infty}(X_1,V_1)$, $C^{\infty}(X_2,V_2)$, and
$C^{\infty}(X_1\cup_f X_2,V_1\cup_{\tilde{f}}V_2)$, and how they are related.

\subsubsection{The induced map $\mathcal{S}:C^{\infty}(X_1,V_1)\times_{comp}C^{\infty}(X_2,V_2) \to C^{\infty}(X_1\cup_f X_2,V_1\cup_{\tilde{f}}V_2)$}

The construction of gluing of two compatible sections considered above obviously defines a map
$$\mathcal{S}:C^{\infty}(X_1,V_1)\times_{comp}C^{\infty}(X_2,V_2)\to C^{\infty}(X_1\cup_f X_2,V_1\cup_{\tilde{f}}V_2)$$ acting by
$$\mathcal{S}(s_1,s_2)=s_1\cup_{(f,\tilde{f})}s_2.$$ We have already seen that $\mathcal{S}$ is well-defined. We have also seen that
$$C^{\infty}(X_1,V_1)\times_{comp}C^{\infty}(X_2,V_2)=C_{(f,\tilde{f})}^{\infty}(X_1,V_1)\times_{comp}C^{\infty}(X_2,V_2),$$ so we will actually consider $\mathcal{S}$ as defined on
the latter space.

We now consider further properties of $\mathcal{S}$. We discover, first of all, that $\mathcal{S}$ is smooth for some natural diffeologies on its domain and its range; that it may or may not be
injective, which depends on the maps $f$ and $\tilde{f}$ being so; and finally, that it always surjective and in fact, it is a subduction.

\subsubsection{The map $\mathcal{S}$ is smooth}

The range of $\mathcal{S}$, the space $C^{\infty}(X_1\cup_f X_2,V_1\cup_{\tilde{f}}V_2)$, is endowed with its usual functional diffeology, while its domain carries the subset diffeology relative
to its inclusion into $C^{\infty}(X_1,V_1)\times C^{\infty}(X_2,V_2)$. The latter space carries the product diffeology relative to the functional diffeologies of $C^{\infty}(X_1,V_1)$ and $C^{\infty}(X_2,V_2)$.
The map $\mathcal{S}$ is smooth for these diffeologies (see \cite{pseudometric-pseudobundle}, Theorem 4.6).

\begin{thm}
The map $\mathcal{S}:C_{(f,\tilde{f})}^{\infty}(X_1,V_1)\times_{comp}C^{\infty}(X_2,V_2)\to C^{\infty}(X_1\cup_f X_2,V_1\cup_{\tilde{f}}V_2)$ is smooth.
\end{thm}

The proof of this statement is straightforward from the definitions of the diffeologies involved.

\subsubsection{$\mathcal{S}$ is not in general injective}

As can be expected, the map $\mathcal{S}$ may easily fail to be injective. Indeed, this has to do with $\tilde{f}$ having a non-trivial kernel within at least one fibre. As a trivial example,
one could consider, for $\pi_1:V_1\to X_1$ the trivial fibering of the standard $\matR^3$ over the standard $\matR$ (its $x$-axis), for $\pi_2:V_2\to X_2$, the trivial fibering of the standard $\matR^2$
over $\matR$ (also the $x$-axis), and for the gluing, the maps $f:\matR\to\matR$ which is just the identity map, and $\tilde{f}:\matR^3\to\matR^2$ acting by $\tilde{f}(x,y,z)=(x,z)$. The result of this gluing
is trivially identified with the second factor, the pseudo-bundle $\pi_2:V_2\to X_2$.

Now, if $s_1,s_1'\in C^{\infty}(X_1,V_1)$ are given by $s_1(x)=(x,0,f(x))$ and $s_1'(x)=(x,1,f(x))$ for any usual smooth $f$, and $s_2\in C^{\infty}(X_2,V_2)$ acts by $s_2(x)=(x,f(x))$ (for the same $f$)
then $s_1\cup_{(f,\tilde{f})}s_2$ and $s_1'\cup_{(f,\tilde{f})}s_2$ are well-defined and equal to each other. On the other hand, $s_1\neq s_1'$, and so the pairs $(s_1,s_2)$ and $(s_1',s_2)$ are
distinct elements of $C_{(f,\tilde{f})}^{\infty}(X_1,V_1)\times_{comp}C^{\infty}(X_2,V_2)$.

Now, the gluing just described is a degenerate case, in the sense that its result is simply the second factor. However, it can easily be extended to a non-trivial one in the following way. Consider any
other pseudo-bundle $\pi_0:V_0\to X_0$, let $x_0\in X_0$ be a point, and let $(\tilde{f}_0,f_0)$ be such that $f_0:\{x_0\}\to\{0\}\subset X_1$ and $\tilde{f}_0:\pi_0^{-1}(x_0)\to\{(0,y,z)\}$ be any linear
map. Define $V_1'=V_0\cup_{\tilde{f}_0}V_1$ and $X_1'=X_0\cup_{f_0}X_1$; then there is an obvious gluing of $\pi_0\cup_{(\tilde{f}_0,f_0)}\pi_1:V_1'\to X_1'$ to $\pi_2:V_2\to X_2$ induced by
the above maps $\tilde{f}$ and $f$. The result of this gluing has the same property, that the corresponding $\mathcal{S}$ is not injective, and it is also non-trivial, in the sense that its result (provided
that $X_0$ is not simply a one-point set) does not coincide with either of its factors.

Finally, we can obtain a more abstract result. Denote by $\mbox{Ker}(\tilde{f})$ the following subset of $V_1$:
$$\mbox{Ker}(\tilde{f})=\cup_{y\in\mbox{Domain}(f)}\mbox{ker}(\tilde{f}|_{\pi_1^{-1}(y)})\cup\mbox{Range}(s_0),$$ where $s_0:X_1\to V_1$ is the zero section.

\begin{lemma}
Let $\pi_i:V_i\to X_i$ for $i=1,2$ be two diffeological vector pseudo-bundles, and let $(\tilde{f},f)$ be a gluing between them. If $\mbox{Ker}(\tilde{f})$ admits a non-zero section and splits off as
a smooth direct summand then $\mathcal{S}$ is not injective.
\end{lemma}

Notice that, when we say that $\mbox{Ker}(\tilde{f})$ is non-trivial, we mean that it is strictly bigger than the range of the zero section.

\subsubsection{The map $\mathcal{S}$ is always surjective}

We can actually say more: $\mathcal{S}$ turns out in fact to be a subduction.

\paragraph{If $\tilde{f}$ and $f$ are diffeomorphisms, then so is $\mathcal{S}$} This is quite obvious and is due to the fact that, when the gluing is performed along a pair of diffeomorphisms, then
both $X_1$ and $X_2$ smoothly embed into $X_1\cup_f X_2$ (recall that in general, only $X_2$ does), and the same is true of $V_1$, $V_2$, and $V_1\cup_{\tilde{f}}V_2$. The embeddings of
$X_2$ into $X_1\cup_f X_2$ and of $V_2$ into $V_1\cup_{\tilde{f}}V_2$ are given by the usual inductions $i_2$ and $j_2$ respectively, while the embeddings $X_1\hookrightarrow X_1\cup_f X_2$
and $V_1\hookrightarrow V_1\cup_{\tilde{f}}V_2$, that we denote by $\tilde{i}_1$ and $\tilde{j}_1$ respectively, are defined as
$$\tilde{i}_1:X_1\hookrightarrow X_1\sqcup X_2\to X_1\cup_f X_2,$$
$$\tilde{j}_1:V_1\hookrightarrow V_1\sqcup V_2\to V_1\cup_{\tilde{f}}V_2,$$ \emph{i.e.} in the manner exactly similar to that of $i_2$ and $j_2$. Notice also that they are extensions of the always-present
inductions $i_1$ and $j_1$.

The claim made in the title of the paragraph is based on the explicit construction, for any given section $s\in C^{\infty}(X_1\cup_f X_2,V_1\cup_{\tilde{f}}V_2)$, of two compatible sections
$$s_1\in C_{(f,\tilde{f})}^{\infty}(X_1,V_1),\,\,s_2\in C^{\infty}(X_2,V_2)\,\,\,\mbox{ such that }\,\,\,\mathcal{S}(s_1,s_2)=s.$$ The two sections have the obvious definition:
$$s_1=\tilde{j}_1^{-1}\circ s\circ\tilde{i}_1,\,\,\,s_2=j_2^{-1}\circ s\circ i_2.$$ Furthermore, it is rather easy to prove the following.

\begin{lemma}
For any $s\in C^{\infty}(X_1\cup_f X_2,V_1\cup_{\tilde{f}}V_2)$ the above-defined sections $s_1$ and $s_2$ are compatible, and the assignment $s\mapsto(s_1,s_2)$ determines the smooth inverse of
$\mathcal{S}$.
\end{lemma}

\paragraph{The pseudo-bundle $\pi_1^{(\tilde{f},f)}:V_1^{\tilde{f}}\to X_1^f$ of $(f,\tilde{f})$-equivalence classes} We introduce this  construction in order to reduce the case of gluing along an
arbitrary pair of maps, to the case of gluing along two diffeomorphisms. The two spaces involved, $V_1^{\tilde{f}}$ and $X_1^f$, are similarly defined; the space $X_1^f$ (which we will also encounter
in the section dedicated to diffeological forms) is the space $X_1$ quotiented by the following equivalence relation: $y\sim_f y'$ if and only if $f(y)=f(y')$ (this applies only to points in $Y$, of course;
points outside of $Y$ are equivalent to themselves only). The space $V_1^{\tilde{f}}$ is the quotient of $V_1$ by the analogous equivalence relation, only defined with respect to the map $\tilde{f}$. Both
spaces are endowed with the respective quotient diffeologies; we will denote the two quotient projections by $\chi_1^f$ and $\chi_1^{\tilde{f}}$.

Since $\tilde{f}$ is a lift of $f$, the pseudo-bundle projection $\pi_1$ induces the map $\pi_1^{\tilde{f},f}:V_1^{\tilde{f}}\to X_1^f$ such that $\chi_1^f\circ\pi_1=\pi_1^{\tilde{f},f}\circ\chi_1^{\tilde{f}}$.
It is clear from the construction that $\pi_1^{\tilde{f},f}$ defines a pseudo-bundle, of which $V_1^{\tilde{f}}$ and $X_1^f$ are respectively the total and the base space. In particular, the vector
space structure on each fibre $(\pi_1^{\tilde{f},f})^{-1}(\chi_1^f(x))$ is inherited from such structures on the fibres of $V_1$; that it is well-defined follows from the linearity of $\tilde{f}$.

Finally, the new pseudo-bundle $\pi_1^{\tilde{f},f}:V_1^{\tilde{f}}\to X_1^f$ comes with the two induced maps that define its gluing to $\pi_2:V_2\to X_2$. These are the maps
$$f_{\sim}:\chi_1^f(Y)\to X_2\,\,\,\mbox{ and }\tilde{f}_{\sim}:\chi_1^{\sim}(\pi_1^{-1}(Y))\to V_2$$  that are determined respectively by $f=f_{\sim}\circ\chi_1^f$ and $\tilde{f}=\tilde{f}_{\sim}\circ\chi_1^{\tilde{f}}$.
The following is then an obvious consequence of the construction itself.

\begin{lemma}
If $f$ or $\tilde{f}$ is a subduction then either $f_{\sim}$ or, respectively, $\tilde{f}_{\sim}$ is a diffeomorphism. Furthermore, for any pair $(\tilde{f},f)$ we have
$$V_1\cup_{\tilde{f}}V_2\cong V_1^{\tilde{f}}\cup_{\tilde{f}_{\sim}}V_2\,\,\,\mbox{ and }\,\,\,X_1\cup_f X_2\cong X_1^f\cup_{f_{\sim}}X_2.$$
\end{lemma}

The latter two diffeomorphisms mentioned in the lemma in particular commute with the two pseudo-bundle projections, $\pi_1$ and $\pi_1^{\tilde{f},f}$, so we actually have a pseudo-bundle
diffeomorphism. The most important, at the moment, consequence of it is the following statement.

\begin{cor}
There is the following diffeomorphism:
$$C^{\infty}(X_1\cup_f X_2,V_1\cup_{\tilde{f}}V_2)\cong C^{\infty}(X_1^f\cup_{f_{\sim}}X_2,V_1^{\tilde{f}}\cup_{\tilde{f}_{\sim}}V_2).$$ Furthermore, if $f$ and $\tilde{f}$ are both subductions,
$$C^{\infty}(X_1\cup_f X_2,V_1\cup_{\tilde{f}}V_2)\cong C^{\infty}(X_1^f,V_1^{\tilde{f}}) \times_{comp}C^{\infty}(X_2,V_2),$$ where in this last case the compatibility is with respect to the maps
$(f_{\sim},\tilde{f}_{\sim})$.
\end{cor}

In particular, the second diffeomorphism filters through the first one and is due to the fact that under the assumptions made the maps $f_{\sim}$ and $\tilde{f}_{\sim}$.

\paragraph{The map $\mathcal{S}_1:C_{(f,\tilde{f})}^{\infty}(X_1,V_1)\to C_{(f_{\sim},\tilde{f}_{\sim})}^{\infty}(X_1^f,V_1^{\tilde{f}})$} The corollary stated immediately above allows for  a sort of
splitting of any section $s:X_1\cup_f X_2\to V_1\cup_{\tilde{f}}V_2$ into a section $s_1^{f,\tilde{f}}:X_1^f\to V_1^{\tilde{f}}$ and a section $s_2:X_2\to V_2$. The word \emph{splitting} means precisely
that
$$s=s_1^{f,\tilde{f}}\cup_{(f_{\sim},\tilde{f}_{\sim})}s_2.$$ What we however would like to do is to split it as a section of  $V_1$ and one of $V_2$, that is to find $s_1\in C_{(f,\tilde{f})}^{\infty}(X_1,V_1)$
and $s_2\in C_{(f,\tilde{f})}^{\infty}(X_2,V_2)$ such that $s=s_1\cup_{(f,\tilde{f})}s_2$.

For generic maps $f$ and $\tilde{f}$ the existence of such an $s_1$ is not immediately clear, therefore we need to consider first the relation between the spaces $C_{(f,\tilde{f})}^{\infty}(X_1,V_1)$
and $C_{(f_{\sim},\tilde{f}_{\sim})}^{\infty}(X_1^f,V_1^{\tilde{f}})$. To this end we define the map
$$\mathcal{S}_1:C_{(f,\tilde{f})}^{\infty}(X_1,V_1)\to C_{(f_{\sim},\tilde{f}_{\sim})}^{\infty}(X_1^f,V_1^{\tilde{f}})$$ via the condition
$$\mathcal{S}_1(s_1)\circ\chi_1^f=\chi_1^{\tilde{f}}\circ s_1\,\,\,\mbox{ for any }\,\,\,s_1\in C_{(f,\tilde{f})}^{\infty}(X_1,V_1).$$

Although this definition of $\mathcal{S}_1$ is an indirect one, it is rather easy to check that defines it univocally, and that $\mathcal{S}_1(s_1)$ is always a smooth section $X_1^f\to V_1^{\tilde{f}}$.
Furthermore, $\mathcal{S}_1$ enjoys several natural properties, that are listed in the paragraphs that follow.

\paragraph{The map $\mathcal{S}_1$ is additive and smooth} A straightforward reasoning allows first of all to show that $\mathcal{S}_1$ preserves the structure of $C_{(f,\tilde{f})}^{\infty}(X_1,V_1)$
as a module over the ring of $f$-invariant functions, as well as that of $C^{\infty}(X_1^f,V_1^{\tilde{f}})$ (which is a module over the ring of $f_{\sim}$-invariant functions). The following is shown
in \cite{connections-pseudobundles} (the proof is straightforward).

\begin{thm}
The map $\mathcal{S}_1$ is additive. Furthermore, for any $s_1\in C_{(f,\tilde{f})}^{\infty}(X_1,V_1)$ and $f$-invariant (and smooth) function $h:X_1\to\matR$ we have
$\mathcal{S}_1(hs_1)=h^f\mathcal{S}_1(s_1)$, where $h^f:X_1^f\to\matR$ is determined by $h=h^f\circ\chi_1^f$.
\end{thm}

It is also quite straightforward to show that

\begin{thm}
The map $\mathcal{S}_1$ is smooth for the functional diffeologies on its domain and its range.
\end{thm}

\paragraph{The map $\mathcal{S}_1$ has smooth right inverses} A reasoning analogous to that carried out for $\mathcal{S}$ shows that in general $\mathcal{S}_1$ is not injective. Therefore we cannot
expect it to be invertible, of course. On the other hand, it turns out that it admits right inverses, and that these inverses are smooth.

To construct one of them, recall the sub-bundle $\mbox{Ker}(\tilde{f})$ of $V_1$. Fix any decomposition $V_1=\mbox{Ker}(\tilde{f})\oplus V_1^0$ of $V_1$ into a direct sum of its sub-bundles.

Notice that over a point of $X_1\setminus Y$ the decomposition is trivial, \emph{i.e.} the fibre of $V_1^0$ coincides with that of $V_1$, while over a point $y\in Y$ it is any direct complement of
$\mbox{ker}(\tilde{f}|_{\pi_1^{-1}(y)})$. For any such choice $V_1^0$ is of course a sub-bundle (for the subset diffeology), however the resulting decomposition does not have to be smooth, and
frequently is not so, \emph{i.e.} the direct sum diffeology on $\mbox{Ker}(\tilde{f})\oplus V_1^0$ may be strictly finer than the diffeology of $V_1$. Surprisingly, the following construction
produces a right inverse of $\mathcal{S}_1$ which is smooth independently of the smoothness of the decomposition $V_1=\mbox{Ker}(\tilde{f})\oplus V_1^0$.

Let $s_1^{f,\tilde{f}}$ be any smooth section $X_1^f\to V_1^{\tilde{f}}$. Define $\mathcal{S}_1^{-1}(s_1^{f,\tilde{f}})$ by the following two conditions:
$$\left\{\begin{array}{l} \mathcal{S}_1^{-1}(s_1^{f,\tilde{f}})(x)\in V_1^0\mbox{ for all }x\in X_1,\\
\chi_1^{\tilde{f}}\circ\mathcal{S}_1^{-1}(s_1^{f,\tilde{f}})=s_1^{f,\tilde{f}}\circ\chi_1^f. \end{array}\right.$$ These two conditions guarantee, first of all, that $\mathcal{S}_1^{-1}(s_1^{f,\tilde{f}})$ is well-defined
as a map $X_1\to V_1$ (this is based simply on $\mbox{Ker}(\tilde{f})\oplus V_1^0$ being a direct sum). Furthermore, the smoothness of the section $\mathcal{S}_1^{-1}(s_1^{f,\tilde{f}})$, as a map
$X_1\to V_1$, follows from the second condition and the definition of a pushforward diffeology.

Next, it is straightforward to check that $\mathcal{S}_1^{-1}(s_1^{f,\tilde{f}})$ is $(f,\tilde{f})$-invariant. Finally, a direct calculation shows that
$$\mathcal{S}_1(\mathcal{S}_1^{-1}(s_1^{f,\tilde{f}}))=s_1^{f,\tilde{f}},$$ so indeed we have a right inverse of $\mathcal{S}_1$.

\begin{rem}
It is also clear from the construction that $\mathcal{S}_1$ admits many right inverses, one for each choice of a direct sum decomposition $V_1=\mbox{Ker}(\tilde{f})\oplus V_1^0$.
\end{rem}

\paragraph{The map $\mathcal{S}_1$ is surjective} This is a direct consequence of the existence of right inverses, so of the previous paragraph.

\begin{thm}\label{sections:first:factor:surjective:thm}
The map $\mathcal{S}_1$ is surjective as a map $$C_{(f,\tilde{f})}^{\infty}(X_1,V_1)\to C^{\infty}(X_1^f,V_1^{\tilde{f}}).$$
\end{thm}

To this we add that, if $\mbox{Ker}(\tilde{f})$ is trivial, then there exists a unique right inverse of the map $\mathcal{S}_1$, which is then a true inverse of it; it is also easy to check that in this case
$\mathcal{S}_1^{-1}$ is smooth as a map $C^{\infty}(X_1^f,V_1^{\tilde{f}})\to C_{(f,\tilde{f})}^{\infty}(X_1,V_1)$. This allows us to obtain the following statement.

\begin{prop}
If $\mbox{Ker}(\tilde{f})$ is trivial then $C_{(f,\tilde{f})}^{\infty}(X_1,V_1)$ and $C^{\infty}(X_1^f,V_1^{\tilde{f}})$ are diffeomorphic.
\end{prop}

\paragraph{$\mathcal{S}_1$ is a subduction} Even in more general case, the map $\mathcal{S}_1$ turns out to be not only surjective, but also a subduction. This follows from the existence of right
inverses, and more precisely, we can obtain the following statement.

\begin{lemma}
Let $q^{f,\tilde{f}}:U\to C^{\infty}(X_1^f,V_1^{\tilde{f}})$ be a plot of $C^{\infty}(X_1^f,V_1^{\tilde{f}})$ (for its standard functional diffeology), and let $\mathcal{S}_1^{-1}$ be any choice of a right inverse
of $\mathcal{S}_1$. Then $\mathcal{S}_1^{-1}\circ q^{f,\tilde{f}}$ is a plot of $C_{(f,\tilde{f})}^{\infty}(X_1,V_1)$.
\end{lemma}

This lemma states in a detailed form that the diffeology of $C^{\infty}(X_1^f,V_1^{\tilde{f}})$ is the pushforward, by $\mathcal{S}_1$, of the diffeology of $C_{(f,\tilde{f})}^{\infty}(X_1,V_1)$.

\paragraph{$\mathcal{S}_1$ preserves compatibility} In the next paragraph we will explain how the map $\mathcal{S}_1$ relates to the map $\mathcal{S}$; therefore we should now consider its
interaction with compatibility. More precisely, recall that $X_1$ and $V_1$ are equipped with, respectively, the maps $f$ and $\tilde{f}$, with respect to which compatibility is defined, and $X_1^f$ and
$V_1^{\tilde{f}}$ are equipped with the maps $f_{\sim}$ and $\tilde{f}_{\sim}$. The following then is true.

\begin{prop}\label{compatible:to:compatible:prop}
Let $s_1\in C_{(f,\tilde{f})}^{\infty}(X_1,V_1)$ and $s_2\in C^{\infty}(X_2,V_2)$. Then $s_1$ and $s_2$ are $(f,\tilde{f})$-compatible if and only if $\mathcal{S}_1(s_1)$ and $s_2$ are
$(f_{\sim},\tilde{f}_{\sim})$-compatible.
\end{prop}

\paragraph{The splitting of $\mathcal{S}$ as $(\mathcal{S}_1,\mbox{Id}_{C^{\infty}(X_2,V_2)})$} The proposition just stated allows thus to consider, given compatible sections
$s_1\in C_{(f,\tilde{f})}^{\infty}(X_1,V_1)$ and $s_2\in C^{\infty}(X_2,V_2)$, to consider both
$$s_1\cup_{(f,\tilde{f})}s_2\,\,\,\mbox{ and }\,\,\,\mathcal{S}_1(s_1)\cup_{(f_{\sim},\tilde{f}_{\sim})}s_2.$$ These sections are identical under the already-mentioned diffeomorphisms
$$X_1\cup_f X_2\cong X_1^f\cup_{f_{\sim}}X_2,\,\,\,V_1\cup_{\tilde{f}}V_2\cong V_1^{\tilde{f}}\cup_{\tilde{f}_{\sim}}V_2,$$ which allows us to identify the map
$$\mathcal{S}:C_{(f,\tilde{f})}^{\infty}(X_1,V_1)\times_{comp}C^{\infty}(X_2,V_2)\to C^{\infty}(X_1\cup_f X_2,V_1\cup_{\tilde{f}}V_2)$$ with the map
$$(\mathcal{S}_1,\mbox{Id}_{C^{\infty}(X_2,V_2)}):C_{(f,\tilde{f})}^{\infty}(X_1,V_1)\times_{comp}C^{\infty}(X_2,V_2)\to C^{\infty}(X_1^f\cup_{f_{\sim}}X_2,V_1^{\tilde{f}}\cup_{\tilde{f}_{\sim}}V_2).$$

It now follows from Theorem \ref{sections:first:factor:surjective:thm} and Proposition \ref{compatible:to:compatible:prop} that $(\mathcal{S}_1,\mbox{Id}_{C^{\infty}(X_2,V_2)})$ is in particular
surjective. Thus, we obtain

\begin{cor}
The map $\mathcal{S}$ is surjective.
\end{cor}

Thus, given a section $s\in C^{\infty}(X_1\cup_f X_2,V_1\cup_{\tilde{f}}V_2)$, there always exists $(s_1,s_2)\in C^{\infty}(X_1,V_1)\times_{comp}C^{\infty}(X_2,V_2)$ such that
$s=s_1\cup_{(f,\tilde{f})}s_2$. Furthermore, $s_2$ is uniquely determined by $s$, while $s_1$ is not.

\section{Pseudo-bundles of Clifford algebras and Clifford modules}

In this section we consider diffeological pseudo-bundles of Clifford algebras and those of Clifford modules, with a particular emphasis on the interactions between Clifford algebra/module structure,
and the operation of gluing. Most of these interactions do turn out in the end to be of the expected form, due to the various commutativity diffeomorphisms considered in the previous two sections. In
addition to the case of abstract Clifford modules, we consider in detail the pseudo-bundles of exterior algebras, which, as in the standard case, carry the natural Clifford action. The material of this section 
is based on \cite{clifford} and \cite{exterior-algebras-pseudobundles}, in particular, all proofs can be found therein (some bits are also cited here for illustration).

\subsection{Gluing of pseudo-bundles of Clifford algebras and those of Clifford modules}

We now recall some facts regarding diffeological gluing of two given pseudo-bundles of Clifford algebras, or two pseudo-bundles of Clifford modules.

\subsubsection{The pseudo-bundle $\cl(V,g)$}

Let $\pi:V\to X$ be a finite-dimensional diffeological vector pseudo-bundle endowed with a pseudo-metric $g:X\to V^*\otimes V^*$. The construction of the corresponding pseudo-bundle of Clifford
algebras is the immediate one, since all the operations involved (direct sum, tensor product, and taking quotients), and their relevant properties have already been described.

The \textbf{pseudo-bundle of Clifford algebras $\pi^{\cl}:\cl(V,g)\to X$} is given by
$$\cl(V,g):=\cup_{x\in X}\cl(\pi^{-1}(x),g(x))$$ endowed with the following diffeology. Consider first the \textbf{pseudo-bundle of tensor algebras} $\pi^{T(V)}:T(V)\to X$, where
$$T(V):=\cup_{x\in X}T(\pi^{-1}(x)),$$ with each $T(\pi^{-1}(x))=\bigoplus_{r}(\pi^{-1}(x))^{\otimes r}$ being the usual tensor algebra of the diffeological vector space $\pi^{-1}(x)$; the collection $T(V)$
of the tensor algebras of individual fibres is endowed with the vector space direct sum diffeology relative to the tensor product diffeology\footnote{We define the tensor product diffeology as the
quotient diffeology, with respect to the kernel of the universal map, relative to the free product diffeology. The latter in turn is the finest vector space diffeology on the free product of the factors,
containing the product diffeology.} on each factor.

By the properties of these diffeologies, the subset diffeology on each fibre of $T(V)$ is that of the tensor algebra of the diffeological vector space $\pi^{-1}(x)$. Now, in each fibre
$(\pi^{T(V)})^{-1}(x)=T(\pi^{-1}(x))$ of $T(V)$ we choose the subspace $W_x$ that is the kernel of the universal map $T(\pi^{-1}(x))\to\cl(\pi^{-1}(x),g(x))$. Then, as is generally the case,
$W=\cup_{x\in X}W_x\subset T(V)$, with the subset diffeology relative this inclusion, is a sub-bundle of $T(V)$. The corresponding quotient pseudo-bundle has $\cl(\pi^{-1}(x),g(x))$ as the fibre at $x$,
both as an algebra and from the diffeological point of view (by the properties of quotient pseudo-bundles). This quotient is $\cl(V,g)$ that we defined above and carries the quotient diffeology; by the
aforementioned properties, the subset diffeology on each fibre is the diffeology of the Clifford algebra of the corresponding fibre $\pi^{-1}(x)$ of $V$.

\begin{rem}
We denote the pseudo-bundle projection of $\cl(V,g)$ by $\pi^{\cl}$, when it is clear from the context which initial pseudo-bundle $V$ we are referring to. When dealing with more than one pseudo-bundle
at a time, we might use the extended notation $\pi^{\cl(V,g)}$ to distinguish between them.
\end{rem}

\subsubsection{The pseudo-bundle $\cl(V_1\cup_{\tilde{f}}V_2,\tilde{g})$ as the result of a gluing}

The main result, that we immediately state and that appears in \cite{clifford}, is the following one.

\begin{thm}
Let $\pi_1:V_1\to X_1$ and $\pi_2:V_2\to X_2$ be two finite-dimensional diffeological vector pseudo-bundles, let $(\tilde{f},f)$ be a pair of smooth maps, each of which is a diffeomorphism, defining
a gluing between them, let $Y\subseteq X_1$ be the domain of definition of $f$, and let $g_1$ and $g_2$ be pseudo-metrics on $V_1$ and $V_2$, compatible with the gluing along
$(\tilde{f},f)$. Then there exists a map
$$\tilde{F}^{\cl}:(\pi^{\cl(V_1,g_1)})^{-1}(Y)\to\cl(V_2,g_2)$$ such that $(\tilde{F}^{\cl},f)$ defines a gluing of $\cl(V_1,g_1)$ to $\cl(V_2,g_2)$, and a diffeomorphism
$$\Phi^{\cl}:\cl(V_1,g_1)\cup_{\tilde{F}^{\cl}}\cl(V_2,g_2)\to\cl(V_1\cup_{\tilde{f}}V_2,\tilde{g})$$ that covers the identity map on $X_1\cup_f X_2$.
\end{thm}

The construction of the map $\tilde{F}^{\cl}$ is the immediately obvious one. It is defined on each fibre over a point $y\in Y$ as the map
$$\cl(\pi_1^{-1}(y),g_1|_{\pi_1^{-1}(y)})\to\cl(\pi_2^{-1}(f(y)),g_2|_{\pi_2^{-1}(f(y))})$$ induced by $\tilde{f}$ via the universal property of Clifford algebras. That this is well-defined follows from
the compatibility of pseudo-metrics $g_1$ and $g_2$.\footnote{As we have said in the previous section, the compatibility of pseudo-metrics is an extension of the concept of an isometry map.} The
diffeomorphism $\Phi^{\cl}$ is then the natural identification; essentially, it follows from the gluing construction that over a point of form $x=i_1^{X_1}(x_1)$ the fibres of both
$\cl(V_1,g_1)\cup_{\tilde{F}^{\cl}}\cl(V_2,g_2)$ and $\cl(V_1\cup_{\tilde{f}}V_2,\tilde{g})$ are naturally identified with $\cl(\pi_1^{-1}(x),g_1|_{\pi_1^{-1}(x)})$, while over any point of
form $x=i_2(x_2)$ these fibres are identified with $\cl(\pi_2^{-1}(x),g_2|_{\pi_2^{-1}(x)})$.

It is also quite clear that this theorem naturally extends to any finite sequence of consecutive gluings, as long as they are all done along diffeomorphisms. The main point is that the Clifford algebras'
pseudo-bundle of the result is obtained by some natural gluing of the initial pseudo-bundles; and if the latter are standard, the end result is also a gluing of some standard bundles of Clifford
algebras, with each fibre of this result (which might well be non-standard itself) inherited from one of the factors.

\subsubsection{Gluing of Clifford modules $E_1$ and $E_2$}

Let us now consider the behavior of the pseudo-bundles of Clifford modules under the operation of gluing. In what immediately follows, we consider a gluing of two abstract pseudo-bundles of
Clifford modules over two given pseudo-bundles of Clifford algebras whose gluing is also fixed, and define what it means for the two actions to be compatible with respect to the gluing of the Clifford
modules. There is then a natural induced action of the result-of-gluing (of algebras) on the result-of-gluing (of modules), which turns out to be smooth.

\paragraph{The setting} Let $X_1$ and $X_2$ be two diffeological spaces, and let $f:X_1\supset Y\to X_2$ be a smooth map. Consider two finite-dimensional diffeological vector pseudo-bundles
over them, $\pi_1:V_1\to X_1$ and $\pi_2:V_2\to X_2$, and a smooth fibrewise-linear lift $\tilde{f}$ of $f$ to a map $\pi_1^{-1}(Y)\to\pi_2^{-1}(f(Y))$. Suppose that each of these pseudo-bundles
carries a pseudo-metric, $g_1$ and $g_2$ respectively, and suppose that these pseudo-metrics are compatible for the gluing along $(\tilde{f},f)$; consider the corresponding pseudo-bundles
of Clifford algebras, $\pi_1^{\cl}:\cl(V_1,g_1)\to X_1$ and $\pi_2^{\cl}:\cl(V_2,g_2)\to X_2$, their gluing along the map $$\tilde{F}^{\cl}:(\pi^{\cl(V_1,g_1)})^{-1}(Y)\to\cl(V_2,g_2),$$as well as the
resulting pseudo-bundle
$$\cl(V_1,g_1)\cup_{\tilde{F}^{\cl}}\cl(V_2,g_2)=\cl(V_1\cup_{\tilde{f}}V_2,\tilde{g})$$ over the space $X_1\cup_f X_2$; the above equality actually stands for the diffeomorphism $\Phi^{\cl}$.

Assume also that we are given two pseudo-bundles of Clifford modules, $\chi_1:E_1\to X_1$ and $\chi_2:E_2\to X_2$, over $\cl(V_1,g_1)$ and $\cl(V_2,g_2)$ respectively; this means there is a
smooth pseudo-bundle map $c_i:\cl(V_i,g_i)\to\mathcal{L}(E_i,E_i)$ that covers the identity on the bases. Suppose further that there is a smooth fibrewise linear map
$\tilde{f}':\chi_1^{-1}(Y)\to\chi_2^{-1}(f(Y))$ that covers $f$. We wish to specify under which conditions $E_1\cup_{\tilde{f}'}E_2$ is a Clifford module over $\cl(V_1\cup_{\tilde{f}}V_2,\tilde{g})$, via
an action induced by $c_1$ and $c_2$.

Notice that we will avail ourselves of the extended notation for the standard inductions $j_*$, while still writing $i_1$ for $i_1^{X_1}$, and $i_2$ for $i_2^{X_2}$; the base space is the same
for all pseudo-bundles throughout the section, so this shall not create confusion, while allowing for simpler formulae. We will also use, whenever it is reasonable to do so (but not at the expense
of clarity), the notation $j_*^{\cl}$, rather than the full form such as, for example, $j_1^{\cl(V_1,g_1)}$.

\paragraph{Compatibility of $c_1$ and $c_2$} Let $y\in Y$, and let $v\in(\pi_1^{\cl})^{-1}(y)$. Consider
$$\Phi^{\cl}(j_2^{\cl}(\tilde{F}^{\cl}(v)))\in((\pi_1\cup_{(\tilde{f},f)}\pi_2)^{\cl})^{-1}(i_2(f(y)))\subseteq\cl(V_1\cup_{\tilde{f}}V_2,\tilde{g}).$$ Compare the following two:
$$c_1(v):\chi_1^{-1}(y)\to\chi_1^{-1}(y),\,\,\,\mbox{ and }\,\,\,c_2(\tilde{F}^{\cl}(v)):\chi_2^{-1}(f(y))\to\chi_2^{-1}(f(y));$$ compare also
$$\tilde{f}'\circ(c_1(v))\,\,\,\mbox{ and }\,\,\,(c_2(\tilde{F}^{\cl}(v)))\circ\tilde{f}'.$$

In order to define the induced action on $E_1\cup_{\tilde{f}'}E_2$, we essentially need to specify it for elements of form $\Phi^{\cl}(j_2^{\cl(V_2,g_2)}(\tilde{F}^{\cl}(v)))$. An element of the
latter form acts on the fibre $(\chi_1\cup_{(\tilde{f'},f)}\chi_2)^{-1}(i_2(f(y))$. Its action therefore could be described by
$$\Phi^{\cl}(j_2^{\cl(V_2,g_2)}(\tilde{F}^{\cl}(v)))(\tilde{f}'(e_1))=\tilde{f}'(c_1(v)(e_1))\,\,\,\mbox{ for an arbitrary }\,\,\,e_1\in\chi_1^{-1}(y).$$ On the other hand, for any given element $e_2\in\chi_2^{-1}(f(y))$
(whether it does or does not belong to the image of $\tilde{f}'$) we might have
$$\Phi^{\cl}(j_2^{\cl}(\tilde{F}^{\cl}(v)))(e_2)=c_2(\tilde{F}^{\cl}(v))(e_2).$$ In order to obtain a smooth induced action, we wish to ensure that these expressions are compatible with each other. We thus
obtain the following notion.

\begin{defn}
The actions $c_1$ and $c_2$ are \textbf{compatible} if for all $y\in Y$, for all $v\in(\pi_1^{\cl})^{-1}(y)$, and for all $e_1\in\chi_1^{-1}(y)$ we have
$$\tilde{f}'(c_1(v)(e_1))=c_2(\tilde{F}^{\cl}(v))(\tilde{f}'(e_1)).$$
\end{defn}

We remark that compatibility of two Clifford actions in the sense just stated is \emph{not} an instance of $(f,g)$-compatibility of smooth maps; see \cite{clifford} for explanation.

\paragraph{The induced action} Assuming that $c_1$ and $c_2$ are compatible in the above sense allows us to define the corresponding induced action, which first of all is a homomorphism
$$c:\cl(V_1\cup_{\tilde{f}}V_2,\tilde{g})\to\mathcal{L}(E_1\cup_{\tilde{f}'}E_2,E_1\cup_{\tilde{f}'}E_2).$$ Using the diffeomorphism
$(\Phi^{\cl})^{-1}:\cl(V_1\cup_{\tilde{f}}V_2,\tilde{g})=\cl(V_1,g_1)\cup_{\tilde{F}^{\cl}}\cl(V_2,g_2)$, we can describe the action $c$ as:
\begin{flushleft}
$c(v)(e)=$
\end{flushleft}
$$\left\{\begin{array}{ll}
j_1^{E_1}\left(c_1((j_1^{\cl(V_1,g_1)})^{-1}((\Phi^{\cl})^{-1}(v)))((j_1^{E_1})^{-1}(e))\right) & \mbox{if }(\Phi^{\cl})^{-1}(v)\in\mbox{Im}(j_1^{\cl(V_1,g_1)})\Rightarrow e\in\mbox{Im}(j_1^{E_1}), \\
j_2^{E_2}\left(c_2((j_2^{\cl(V_2,g_2)})^{-1}((\Phi^{\cl})^{-1}(v)))((j_2^{E_2})^{-1}(e))\right) & \mbox{if }(\Phi^{\cl})^{-1}(v)\in\mbox{Im}(j_2^{\cl(V_2,g_2)})\Rightarrow e\in\mbox{Im}(j_2^{E_2}).
\end{array}\right. $$
Since the images of the inductions $\Phi^{\cl}\circ j_1^{\cl(V_1,g_1)}$, $\Phi^{\cl}\circ j_2^{\cl(V_2,g_2)}$ are disjoint and cover $\cl(V_1\cup_{\tilde{f}}V_2,\tilde{g})$, and those of
$j_1^{E_1},j_2^{E_2}$ cover $E_1\cup_{\tilde{f}'}E_2$ (and are disjoint as well), this action is well-defined. Furthermore, each $c(v)$ is an endomorphism of the corresponding fibre, because both
$c_1((j_1^{\cl})^{-1}((\Phi^{\cl})^{-1}(v)))$ and $c_2((j_2^{\cl})^{-1}((\Phi^{\cl})^{-1}(v)))$ (whichever is relevant) are so. Furthermore, the following is true (see \cite{clifford}).

\begin{thm}
The action $c$ is smooth as a map $\cl(V_1\cup_{\tilde{f}}V_2,\tilde{g})\to\mathcal{L}(E_1\cup_{\tilde{f}'}E_2,E_1\cup_{\tilde{f}'}E_2)$.
\end{thm}

\subsubsection{Unitary Clifford modules}

The definition of a unitary Clifford module in the diffeological context is just a \emph{verbatim} extension of the usual one. Let $\pi:V\to X$ be a diffeological vector pseudo-bundle endowed with a
pseudo-metric $g$, let $\pi^{\cl}:\cl(V,g)\to X$ be the corresponding pseudo-bundle of Clifford algebras, and let $\chi:E\to X$ be a finite-dimensional diffeological vector pseudo-bundle endowed with
a pseudo-metric $g_E$ and such that each fibre $\chi^{-1}(x)$ carries the standard diffeology.

\begin{defn}
The pseudo-bundle $\chi:E\to X$ is said to be a \textbf{unitary Clifford module} over $V$ if there exists a smooth pseudo-bundle map $c:\cl(V,g)\to\mathcal{L}(E,E)$ such that its restriction onto each
fibre is an algebra homomorphism and for all $x\in X$, for all unitary $v_1,v_2\in\pi^{-1}(x)$, and for all $w_1,w_2\in\chi^{-1}(x)$, we have
$$g_E(x)(c(v_1)(w_1),c(v_2)(w_2))=g_E(x)(w_1,w_2).$$
\end{defn}

We note that this definition makes sense when applied to dual pseudo-bundles (whose fibres are all standard, and so pseudo-metrics give scalar products on them), not so much for arbitrary
pseudo-bundle.

\subsection{Gluing of pseudo-bundles of exterior algebras: contravariant version}

In this section we consider the contravariant\footnote{To justify the distinction from the covariant case, recall that for diffeological vector spaces there almost never is an isomorphism
between the space itself and its dual.} version of the exterior algebra (first of a diffeological vector space, then of a diffeological vector pseudo-bundle), by which we mean the following.
Let $V$ be a finite-dimensional diffeological vector space; for each tensor degree of $V$ consider the usual alternating operator
$$\mbox{Alt}:\underbrace{V\otimes\ldots\otimes V}_n\to\underbrace{V\otimes\ldots\otimes V}_n,\mbox{ where }\mbox{Alt}(v_1\otimes\ldots\otimes)=
\frac{1}{n!}\sum_{\sigma}(-1)^{\mbox{sgn}(\sigma)}v_{\sigma(1)}\otimes\ldots\otimes v_{\sigma(n)}$$ is extended by linearity. The contravariant $n$-th exterior algebra of $V$ is the image
$$\bigwedge_n(V)=\mbox{Alt}(\underbrace{V\otimes\ldots\otimes V}_n)\mbox{ with }\bigwedge_0(V):=\matR;$$ the whole exterior algebra $\bigwedge_*(V)$ is the direct sum of all
such terms,
$$\bigwedge_*V=\bigoplus_{n\geqslant 0}\bigwedge_n(V).$$ We obtain a pseudo-bundle of exterior algebras by employing the same operations in the pseudo-bundle version, and defining
the alternating operator fibrewise.

What we consider in this section is the behavior of such objects under gluing. Specifically, having assumed that we are, as usual, given two pseudo-bundles $\pi_1:V_1\to X_1$ and $\pi_2:V_2\to X_2$,
and a gluing of the former to the latter along the maps $(\tilde{f},f)$, we extend this gluing to one of the pseudo-bundle $\bigwedge_*(V_1)$ to the pseudo-bundle $\bigwedge_*(V_2)$, along the
natural induced map $\tilde{f}_*^{\bigwedge}$ (and the same map $f$ on the bases); and then show that the result is diffeomorphic to $\bigwedge(V_1\cup_{\tilde{f}}V_2)$.

\subsubsection{Gluing of $\bigwedge_*(V_1)$ and $\bigwedge_*(V_2)$ along the induced map $\tilde{f}_*^{\bigwedge}$}

If $V_1$ and $V_2$ are just two diffeological vector spaces, and $\tilde{f}_*:V_1\to V_2$ is any smooth linear map between them, then it extends, by linearity and tensor product multiplicativity, to a
smooth linear map between the respective tensor degrees of these spaces; thus, to the smooth linear map
$(\tilde{f}_*)^{\otimes n}:\underbrace{V_1\otimes\ldots\otimes V_1}_n\to\underbrace{V_2\otimes\ldots\otimes V_2}_n$. Such extension commutes with the corresponding alternating operators,
that is,
$$\mbox{Alt}_2^{(n)}\circ(\tilde{f}_*)^{\otimes n}=(\tilde{f}_*)^{\otimes n}\circ\mbox{Alt}_1^{(n)},$$ where $\mbox{Alt}_i^{(n)}$ is the $n$-th degree alternating operator on the space $V_i$. This yields
a smooth linear map $\tilde{f}_*^{\bigwedge}:\bigwedge_*(V_1)\to\bigwedge_*(V_2)$.

Let now $\pi_1:V_1\to X_1$ and $\pi_2:V_2\to X_2$ be two vector pseudo-bundles, and let $(\tilde{f}_*,f)$ be a pair of maps defining a gluing between them; let $Y\subset X_1$ be the domain of
definition of $f$. Then the corresponding collection of maps
$$\bigcup_{y\in Y}\left(\tilde{f}_*|_{\pi_1^{-1}(y)}\right)^{\bigwedge},$$ yields the smooth and fibrewise linear map $\tilde{f}_*^{\bigwedge}$ between the corresponding subsets of $\bigwedge_*(V_1)$
and $\bigwedge_*(V_2)$. Thus, it defines a gluing between the corresponding pseudo-bundles of contravariant exterior algebras $\pi_1^{\bigwedge_*}:\bigwedge_*(V_1)\to X_1$ and
$\pi_2^{\bigwedge_*}:\bigwedge_*(V_2)\to X_2$, with the gluing on the base spaces given by the same map $f$.\footnote{To go into a bit more detail, for each of $T(V_1)$, $T(V_2)$ there is the
fibrewise-defined alternating operator $\mbox{Alt}_i$, for $i=1,2$; it is a map $T(V_i)\to T(V_i)$ that covers the identity on the base space $X_i$ and that is defined, on each fibre, as the usual
alternating operator associated to the fibre. The image $\bigwedge_*(V_i)$ of each $\mbox{Alt}_i$ is a sub-bundle of $T(V_i)$, consisting of fibres of form $\bigwedge_*(\pi_i^{-1}(x))$; it has both
the sub-bundle diffeology (the usual subset diffeology) and the pushforward diffeology (relative to $\mbox{Alt}_i$), with the two diffeologies easily shown to coincide. Between these two
pseudo-bundles, $\bigwedge_*(V_1)$ and $\bigwedge_*(V_2)$, there is the map $\tilde{f}_*^{\bigwedge}$ just-mentioned; the result of gluing of
$\bigwedge_*(V_1)$ to $\bigwedge_*(V_2)$ along it is the pseudo-bundle $\bigwedge_*(V_1)\cup_{\tilde{f}_*^{\bigwedge}}\bigwedge_*(V_2)$.}

\subsubsection{The pseudo-bundles $\bigwedge_*(V_1\cup_{\tilde{f}_*}V_2)$ and $\bigwedge_*(V_1)\cup_{\tilde{f}_*^{\bigwedge}}\bigwedge_*(V_2)$}

We have just seen that a given gluing map $\tilde{f}_*$ extends to a map $\tilde{f}_*^{\bigwedge}$ that defines a gluing between $\bigwedge_*(V_1)$ and $\bigwedge_*(V_2)$; on the hand, the
same map $\tilde{f}_*$ can be used to first perform the gluing of $V_1$ to $V_2$, and then construct the contravariant exterior algebra of the resulting pseudo-bundle $V_1\cup_{\tilde{f}_*}V_2$.

Indeed, as any diffeological vector pseudo-bundle, $V_1\cup_{\tilde{f}_*}V_2$ has its own alternating operator $\mbox{Alt}$, whose image is the pseudo-bundle $\bigwedge_*(V_1\cup_{\tilde{f}_*}V_2)$.
Since each fibre of the pseudo-bundle of tensor algebras $T(V_1\cup_{\tilde{f}_*}V_2)$ coincides with either a fibre of $T(V_1)$ or one of $T(V_2)$, and fibrewise each alternating operator is the
usual one of a (diffeological) vector space, it makes sense to wonder whether the pseudo-bundles $\bigwedge_*(V_1\cup_{\tilde{f}_*}V_2)$ and
$\bigwedge_*(V_1)\cup_{\tilde{f}_*^{\bigwedge}}\bigwedge_*(V_2)$ are diffeomorphic in a canonical way.

\paragraph{The alternating operators $\mbox{Alt}$, $\mbox{Alt}_1$, and $\mbox{Alt}_2$} The fact that there indeed is such a diffeomorphism, follows essentially from the commutativity of gluing
with the operations of tensor product and the direct sum, as well as the definition of the operator $\mbox{Alt}$, and more precisely, the fact that its restriction to any given fibre coincides with either
$\mbox{Alt}_1$ or $\mbox{Alt}_2$. Indeed, it is a matter of a technicality to observe that there is the following relation between the $n$-th components of these three operators:
$$\mbox{Alt}^{(n)}=\Phi_{\otimes,\cup}^{(\otimes n)}\circ\left(\mbox{Alt}_1^{(n)}\cup_{\left(\tilde{f}^{\otimes n},\tilde{f}^{\otimes}\right)}\mbox{Alt}_2^{(n)}\right)\circ\Phi_{\cup,\otimes}^{(\otimes n)},$$
where $\Phi_{\cup,\otimes}^{(\otimes n)}$ and $\Phi_{\otimes,\cup}^{(\otimes n)}$ are the two mutually inverse commutativity diffeomorphisms for the operations of gluing and tensor
product.\footnote{Informally we could just say that $\mbox{Alt}$ is obtained by, or that it splits as, gluing together $\mbox{Alt}_1$ and $\mbox{Alt}_2$.} More precisely,
$$\Phi_{\cup,\otimes}^{(\otimes n)}:(V_1\cup_{\tilde{f}_*}V_2)^{\otimes n}\to(V_1)^{\otimes n}\cup_{(\tilde{f}_*)^{\otimes n}}(V_2)^{\otimes n}$$ and $\Phi_{\otimes,\cup}^{(\otimes n)}$ is its inverse.
Do note that the notation $\Phi_{\cup,\otimes}^{(\otimes n)}$ might be misleading, since we are \emph{not} referring to the $n$-th tensor degree of the diffeomorphism $\Phi_{\cup,\otimes}$, but rather
a new diffeomorphism that is defined from the beginning on the $n$-th tensor degree of $V_1\cup_{\tilde{f}}V_2$ (we avoid giving further details here, but see \cite{exterior-algebras-pseudobundles}).
For the moment, we just rewrite the same expression as
$$\Phi_{\cup,\otimes}^{(\otimes n)}\circ\mbox{Alt}^{(n)}=\left(\mbox{Alt}_1^{(n)}\cup_{\left(\tilde{f}^{\otimes n},\tilde{f}^{\otimes}\right)}\mbox{Alt}_2^{(n)}\right)\circ\Phi_{\cup,\otimes}^{(\otimes n)}.$$

\paragraph{The diffeomorphism $\Phi^{\bigwedge_*}:\bigwedge_*(V_1\cup_{\tilde{f}_*}V_2)\to\bigwedge_*(V_1)\cup_{\tilde{f}_*^{\bigwedge}}\bigwedge_*(V_2)$} Having essentially defined the map
$\Phi^{\bigwedge_*}$ on each tensor degree, we now set
$$\Phi^{\bigwedge_*}=\bigoplus_n\Phi_{\cup,\otimes}^{(\otimes n)}\mid_{\bigwedge_*(V_1\cup_{\tilde{f}_*}V_2)}.$$ This expression is actually abbreviated, since the commutativity of gluing with the
direct sum is only implicit therein. Indeed, $\Phi^{\bigwedge_*}$ is defined on the subspace of antisymmetric tensors in $\bigoplus_n\left(V_1\cup_{\tilde{f}_*}V_2\right)^{\otimes n}$, that
is, on $\bigwedge_*(V_1\cup_{\tilde{f}_*}V_2)$ (as wanted), but it takes values in the space $\bigoplus_n\left(V_1^{\otimes n}\cup_{\tilde{f}_*^{\otimes n}}V_2^{\otimes n}\right)$, whereas we
need it to take values in $\left(\oplus_nV_1^{\otimes n}\right)\cup_{\oplus_n\tilde{f}_*^{\otimes  n}}\left(\oplus_nV_2^{\otimes n}\right)$. For this to happen, $\Phi^{\bigwedge_*}$ must be
post-composed with the appropriate diffeomorphism between the latter two pseudo-bundles, specifically with the diffeomorphism
$$\Phi_{\cup,\oplus}:\bigoplus_n\left(V_1^{\otimes n}\cup_{\tilde{f}_*^{\otimes n}}V_2^{\otimes n}\right)\to\left(\oplus_nV_1^{\otimes n}\right)\cup_{\oplus_n\tilde{f}_*^{\otimes
n}}\left(\oplus_nV_2^{\otimes n}\right).$$ The full form of $\Phi^{\bigwedge_*}$ therefore is
$$\Phi^{\bigwedge_*}=\Phi_{\cup,\oplus}\circ\left(\bigoplus_n\Phi_{\cup,\otimes}^{(\otimes n)}\mid_{\bigwedge_*(V_1\cup_{\tilde{f}_*}V_2)}\right);$$ its inverse is given by
$$\left(\Phi^{\bigwedge_*}\right)^{-1}=\bigoplus_n\Phi_{\otimes,\cup}^{(\otimes n)}\mid_{\bigwedge_*(V_1)\cup_{\tilde{f}_*^{\bigwedge}}\bigwedge_*(V_2)}$$ (in the abbreviated form) and by
$$\left(\Phi^{\bigwedge_*}\right)^{-1}=\left(\bigoplus_n\Phi_{\otimes,\cup}^{(\otimes n)}\right)\circ\left(\Phi_{\oplus,\cup}\mid_{\bigwedge_*(V_1)\cup_{\tilde{f}_*^{\bigwedge}}\bigwedge_*(V_2)}\right),$$
where $\Phi_{\oplus,\cup}$ is the inverse of $\Phi_{\cup,\oplus}$.

\subsection{Gluing pseudo-bundles of covariant exterior algebras}

We now turn to the (more usual) covariant version of the exterior algebra. This case is somewhat trickier than the contravariant one, due to a complicated behavior of the gluing operation with respect
to taking dual pseudo-bundles.

\subsubsection{The covariant exterior algebra}

Let $V$ be a diffeological vector space; the covariant alternating operator $\mbox{Alt}$ is defined just as the contravariant one, but it acts on each space $\underbrace{V^*\otimes\ldots\otimes V^*}_n$;
by definition, $\bigwedge^n(V)$ is $\mbox{Alt}(\underbrace{V^*\otimes\ldots\otimes V^*}_n)$, and the direct sum of all $\bigwedge^n(V)$ is the exterior algebra $\bigwedge(V)$, with respect to the exterior product. It is smooth for the pushforward diffeology by $\mbox{Alt}$, so $\bigwedge(V)$ is a diffeological algebra.

If $V_1$ and $V_2$ are two diffeological vector spaces and $\tilde{f}:V_1\to V_2$ is a smooth linear map, there is a natural  induced map $\tilde{f}^{\bigwedge}:\bigwedge(V_2)\to\bigwedge(V_1)$,
which is smooth and linear. On each space $\underbrace{V_2^*\otimes\ldots\otimes V_2^*}_n$ it acts as $(\tilde{f}^*)^{\otimes n}$. This commutes with the alternating operator, in the sense that
$\mbox{Alt}_1\circ(\tilde{f}^*)^{\otimes n}=(\tilde{f}^*)^{\otimes n}\circ\mbox{Alt}_2$, where $\mbox{Alt}_1$ is the alternating operator for the space $V_1$, and $\mbox{Alt}_2$ is one for the space $V_2$.
Hence the direct sum of all maps of form $(\tilde{f}^*)^{\otimes n}$ is a well-defined map between $\bigwedge(V_2)$ and $\bigwedge(V_1)$.

Let now $\pi:V\to X$ be a finite-dimensional diffeological vector  pseudo-bundle. The collection $\mbox{Alt}_V$ of the covariant alternating operators associated to each fibre yields a pseudo-bundle map
of $T(V^*)$ into itself, whose image is, by definition, the \textbf{pseudo-bundle of covariant exterior algebras}, which we denote by $\bigwedge(V)$, with the corresponding pseudo-bundle projection denote
by $\pi^{\bigwedge}:\bigwedge(V)\to X$. (Obviously, the construction of $\mbox{Alt}_V$ is that the contravariant case, it is just applied to the dual pseudo-bundle $\pi^*:V^*\to X$). Also in this case,
$\bigwedge(V)$ carries \emph{a priori} two natural diffeologies: one as a subset of $T(V^*)$, the other obtained by pushing forward the diffeology of $T(V^*)$ by the alternating operator $\mbox{Alt}_V$.
Once again, these two diffeologies coincide.

\subsubsection{The induced gluing of $\bigwedge(V_2)$ to $\bigwedge(V_1)$}

Let now $\pi_1:V_1\to X_1$ and $\pi_2:V_2\to X_2$ be two finite-dimensional diffeological vector pseudo-bundles, and let $(\tilde{f},f)$ be a gluing between them such that $f$ is smoothly
invertible; let $Y\subset X_1$ be the domain of definition of $f$, and let $Y'$ stand for $f(Y)$. There is then the natural induced map
$$\tilde{f}^{\bigwedge}:\bigwedge(V_2)\supset\bigwedge(\pi_2^{-1}(Y'))\to\bigwedge(\pi_1^{-1}(Y))\subset\bigwedge(V_1),$$ which is defined on each fibre by taking the already-defined map
$\bigwedge(\pi_2^{-1}(y'))\to\bigwedge(\pi_1^{-1}(f^{-1}(y')))$ between the exterior algebras of diffeological vector spaces. This map is smooth for the subset diffeologies on $\bigwedge(\pi_2^{-1}(Y'))$
and $\bigwedge(\pi_1^{-1}(Y))$, and, together with the map $f^{-1}$, it defines a gluing of $\bigwedge(V_2)$ to $\bigwedge(V_1)$, whose result is the pseudo-bundle
$$\pi_2^{\bigwedge}\cup_{(\tilde{f}^{\bigwedge},f^{-1})}\pi_1^{\bigwedge}:\bigwedge(V_2)\cup_{\tilde{f}^{\bigwedge}}\bigwedge(V_1)\to X_2\cup_{f^{-1}}X_1.$$

\subsubsection{Comparison of $\bigwedge(V_1\cup_{\tilde{f}}V_2)$ with $\bigwedge(V_2)\cup_{\tilde{f}^{\bigwedge}}\bigwedge(V_1)$}

There is another possibility for the interplay between the operation of gluing and one of building the pseudo-bundle of exterior algebras, and specifically, that of first gluing the pseudo-bundle
$V_1$ to $V_2$ along $(\tilde{f},f)$, thus obtaining
$$\pi_1\cup_{(\tilde{f},f)}\pi_2:V_1\cup_{\tilde{f}}V_2\to X_1\cup_f X_2,$$ and then considering the corresponding pseudo-bundle
$$(\pi_1\cup_{(\tilde{f},f)}\pi_2)^{\Lambda}:\Lambda(V_1\cup_{\tilde{f}}V_2)\to X_1\cup_f X_2.$$ It is quite natural then to ask under which assumptions the two pseudo-bundles thus obtained
($\bigwedge(V_1\cup_{\tilde{f}}V_2)$ and $\bigwedge(V_2)\cup_{\tilde{f}^{\bigwedge}}\bigwedge(V_1)$) are diffeomorphic, and it is also quite clear that the necessary conditions should include
the gluing-dual commutativity; this turns out to be a sufficient condition as well.

Let $\Phi_{\cup,*}:(V_1\cup_{\tilde{f}}V_2)^*\to V_2^*\cup_{\tilde{f}^*}V_1^*$ be the gluing-dual commutativity diffeomorphism. It can be extended to a diffeomorphism
$$\Phi^{\bigwedge}:\bigwedge(V_1\cup_{\tilde{f}}V_2)\to\bigwedge(V_2)\cup_{\tilde{f}^{\bigwedge}}\bigwedge(V_1)$$ that covers the switch map, is linear on the fibres, and is smooth, in the following
way. If we omit the pre- and post-compositions with the gluing-direct sum and the gluing-tensor product commutativity diffeomorphisms, $\Phi^{\bigwedge}$ is simply
$$\Phi^{\bigwedge}=\left(\bigoplus_n\Phi_{\cup,*}^{\otimes n}\right)\mid_{\bigwedge(V_1\cup_{\tilde{f}}V_2)}.$$ Note that that this time, by $\Phi_{\cup,*}^{\otimes n}$ we do mean the $n$-th tensor
degree of $\Phi_{\cup,*}$.

\paragraph{Adding $\Phi_{\cup,\otimes}$} We need the $n$-th tensor degree component of $\Phi^{\bigwedge}$ to be a map of form
$$\left((V_1\cup_{\tilde{f}}V_2)^*\right)^{\otimes n}\to(V_2^*)^{\otimes n}\cup_{(\tilde{f}^*)^{\otimes n}}(V_1^*)^{\otimes n},$$ while each $\Phi_{\cup,*}^{\otimes n}$ has form
$$\Phi_{\cup,*}^{\otimes n}:\left((V_1\cup_{\tilde{f}}V_2)^*\right)^{\otimes n}\to\left(V_2^*\cup_{\tilde{f}^*}V_1^*\right)^{\otimes n}.$$ The diffeomorphism that we need to add to this  component is
$$\Phi_{\cup,\otimes}^{(\otimes n)}:\left(V_2^*\cup_{\tilde{f}^*}V_1^*\right)^{\otimes n}\to(V_2^*)^{\otimes n}\cup_{(\tilde{f}^*)^{\otimes n}}(V_1^*)^{\otimes n},$$ the already-mentioned extension of
the gluing-tensor product commutativity diffeomorphism to the case of $n$ factors.\footnote{Recall that it is \emph{not} the $n$-th tensor degree.} Thus, the full form of the $n$-th degree component of
$\Phi^{\bigwedge}$ is
$$\left(\Phi_{\cup,\otimes}^{\otimes n}\right)\circ\left(\Phi_{\cup,*}^{\otimes n}\right)\mid_{\bigwedge(V_1\cup_{\tilde{f}}V_2)}.$$

\paragraph{Adding $\Phi_{\cup,\oplus}$} It now suffices to add the gluing-direct sum commutativity diffeomorphism, that is, the map
$$\Phi_{\cup,\oplus}:\bigoplus_n\left((V_2^*)^{\otimes n}\cup_{(\tilde{f}^*)^{\otimes n}}(V_1^*)^{\otimes n}\right)\to\left(\bigoplus_n(V_2^*)^{\otimes n}\right)\cup_{\bigoplus_n(\tilde{f}^*)^{\otimes
n}}\left(\sum_n(V_1^*)^{\otimes n}\right).$$ Thus, the entire diffeomorphism $\Phi^{\bigwedge}$ is the following map:
$$\Phi^{\bigwedge}=\Phi_{\cup,\oplus}\circ\bigoplus_n\left(\left(\Phi_{\cup,\otimes}^{\otimes n}\right)\circ\left(\Phi_{\cup,*}^{\otimes n}\right)\mid_{\bigwedge(V_1\cup_{\tilde{f}}V_2)}\right).$$
Its inverse is obtained by taking the inverse of $\Phi_{\cup,\oplus}$ and then inverting (separately) each component under the sum.

\paragraph{The range of $\Phi^{\bigwedge}$} Finally, we mention why the range of $\Phi^{\bigwedge}$ is indeed the space $\bigwedge(V_2)\cup_{\tilde{f}^{\Lambda}}\bigwedge(V_1)$. This follows
from the properties of alternating operators
$$\mbox{Alt}_{\cup,*}:\left((V_1\cup_{\tilde{f}}V_2)^*\right)^{\otimes n}\to\left((V_1\cup_{\tilde{f}}V_2)^*\right)^{\otimes n},\,\,\,\mbox{Alt}_2:(V_2^*)^{\otimes n}\to(V_2^*)^{\otimes n},\,\,\,
\mbox{Alt}_1:(V_1^*)^{\otimes n}\to(V_1^*)^{\otimes n};$$ in particular, up to adding the appropriate commutativity diffeomorphisms, we have
$$\Phi_{\cup,*}^{\otimes n}\circ\mbox{Alt}_{\cup,*}=\left(\mbox{Alt}_2\cup_{\left((\tilde{f}^*)^{\otimes n},(\tilde{f}^*)^{\otimes n}\right)}\mbox{Alt}_1\right)\circ\Phi_{\cup,*}^{\otimes n},$$ where
$$\mbox{Alt}_2\cup_{\left((\tilde{f}^*)^{\otimes n},(\tilde{f}^*)^{\otimes n}\right)}\mbox{Alt}_1$$ is the result of the gluing of maps $\mbox{Alt}_2$ and $\mbox{Alt}_1$ along the pair
$\left((\tilde{f}^*)^{\otimes n},(\tilde{f}^*)^{\otimes n}\right)$. Since it has range $\bigwedge(V_2)\cup_{\tilde{f}^{\bigwedge}}\bigwedge(V_1)$, so does $\mbox{Alt}_{\cup,*}$.

\subsection{The Clifford actions on $\bigwedge_*(V_1\cup_{\tilde{f}}V_2)$ and on $\bigwedge(V_1\cup_{\tilde{f}}V_2)$}

Let $\pi_1:V_1\to X_1$ and $\pi_2:V_2\to X_2$ be two diffeological vector pseudo-bundles, and let $(\tilde{f},f)$ be a gluing between them. Suppose that $V_1$ and $V_2$ are equipped with
pseudo-metrics $g_1$ and $g_2$ respectively, compatible with this gluing; let $\tilde{g}$ stand for the pseudo-metric obtained from the gluing of $g_1$ and $g_2$. In this case each fibre of
$\bigwedge_*(V_1\cup_{\tilde{f}}V_2)$ at a point $x\in X_1\cup_f X_2$ is naturally a Clifford module over $\cl((\pi_1\cup_{(\tilde{f},f)}\pi_2)^{-1}(x),\tilde{g}(x))$, and the same is true for fibres
$\bigwedge_*(V_i)$ at $x\in X_i$ and $\cl(\pi_i^{-1}(x),g_i(x))$ for $i=1,2$. Thus, we have three pseudo-bundles of Clifford modules; indeed, for all the three fibres usual action, which is smooth
on each fibre, turns out to be smooth across the fibres (this is something that is true in its maximal generality). On the other hand, we have seen that the pseudo-bundle
$\bigwedge_*(V_1\cup_{\tilde{f}}V_2)$ is obtained by gluing $\bigwedge_*(V_1)$ to $\bigwedge_*(V_2)$, so next, we consider the interaction of the Clifford action with this gluing, showing that
the natural induced action on $\bigwedge_*(V_1\cup_{\tilde{f}}V_2)$ coincides with the standard one.

All the same is true also in the covariant case, although, as we have already noted (in the case of the dual pseudo-metrics especially), it requires more intricate assumptions.

\subsubsection{The Clifford action of $\cl(V,g)$ on $\bigwedge_*(V)$ is smooth}

Let $\pi:V\to X$ be a finite-dimensional diffeological vector pseudo-bundle that admits a pseudo-metric $g$, let $\pi^{\cl}:\cl(V,g)\to X$ be the corresponding pseudo-bundle of Clifford algebras,
and let $\pi^{\bigwedge_*}:\bigwedge_*(V)\to X$ be the corresponding pseudo-bundle of contravariant exterior algebras. The standard Clifford action $c$ of $\cl(V,g)$ on $\bigwedge_*(V)$
(see, for instance, \cite{heat-kernel}, Ch. 3.1) is defined by setting, for all $x\in X$ and for all $v\in\pi^{-1}(x)$, that $c(v)=\varepsilon(v)-i(v)\in\mbox{End}(\pi^{-1}(x))$, where
$$\varepsilon(v)(v_1\wedge\ldots\wedge v_k)=v\wedge v_1\wedge\ldots\wedge v_k\mbox{ and }i(v)(v_1\wedge\ldots\wedge v_k)=\sum_{j=1}^k(-1)^{j+1}v_1\wedge\ldots\wedge g(x)(v,v_j)\wedge
\ldots\wedge v_k;$$ this extends to the rest of the Clifford algebra by linearity and substituting the tensor product with the composition. Considered on each single fibre, that is, on a
finite-dimensional diffeological vector space, this fibrewise action is smooth (see \cite{clifford}). This quite easily extends to the case of locally trivial pseudo-bundles:

\begin{prop}
Let $\pi:V\to X$ be a locally trivial finite-dimensional diffeological vector pseudo-bundle that admits a pseudo-metric $g$. Then the fibrewise Clifford action $c$ is smooth as a map
$\cl(V,g)\to\mathcal{L}(\bigwedge_*(V),\bigwedge_*(V))$.
\end{prop}

\subsubsection{The case of $\bigwedge_*(V_1\cup_{\tilde{f}}V_2)=\bigwedge_*(V_1)\cup_{\tilde{f}_*^{\bigwedge}}\bigwedge_*(V_2)$}

As we already indicated, in the case where we are given two pseudo-bundles $\pi_1:V_1\to X_1$ and $\pi_2:V_2\to X_2$, and a gluing $(\tilde{f}_*,f)$ between them,\footnote{Suppose for
simplicity that both are diffeomorphisms of their respective domains with their images.} there are essentially two ways of seeing the same pseudo-bundle
$\bigwedge_*(V_1\cup_{\tilde{f}}V_2)=\bigwedge_*(V_1)\cup_{\tilde{f}_*^{\bigwedge}}\bigwedge_*(V_2)$. In particular, the left-hand side of this expression is by construction a pseudo-bundle of
exterior algebras and comes immediately with the standard Clifford action of $\cl(V_1\cup_{\tilde{f}}V_2,\tilde{g})$, while the right-hand side is the result of gluing of two pseudo-bundles of exterior
algebras,\footnote{Which at the moment we consider as its primary structure, while its identification via the diffeomorphism $\Phi^{\bigwedge_*}$ with the left-hand side pseudo-bundle is
secondary to that.} each of which comes with its own smooth Clifford action, $c_1:\cl(V_1,g_1)\to\mathcal{L}(\bigwedge_*(V_1),\bigwedge_*(V_1))$ and
$c_2:\cl(V_2,g_2)\to\mathcal{L}(\bigwedge_*(V_2),\bigwedge_*(V_2))$, respectively. It thus suffices to show that they are compatible, in order to obtain the induced Clifford action on
$\bigwedge_*(V_1\cup_{\tilde{f}}V_2)$; which can then be compared to the standard Clifford action $c:\cl(V_1\cup_{\tilde{f}}V_2,\tilde{g})\to\mathcal{L}(\bigwedge_*(V_1\cup_{\tilde{f}}V_2),
\bigwedge_*(V_1\cup_{\tilde{f}}V_2))$.

\paragraph{Compatibility of the two actions} It is of course sufficient to check the compatibility condition for elements $v\in V_i\subset\cl(V_i,g_i)$ and for individual exterior products
$v_1\wedge\ldots\wedge v_k$ in $\bigwedge_*(\pi_1^{-1}(y))$ and their images in $\bigwedge_*(\pi_2^{-1}(f(y)))$. For them, the condition is
$$\tilde{f}_*^{\bigwedge}(c_1(v)(v_1\wedge\ldots\wedge v_k))=c_2(\tilde{f}(v))(\tilde{f}_*^{\bigwedge}(v_1\wedge\ldots\wedge v_k)),$$ where $y\in Y$ is any element in the domain of gluing,
$v\in\pi_1^{-1}(y)\subset\cl(V_1,g_1)$, and $v_1\wedge\ldots\wedge v_k\in(\pi_1^{\bigwedge_*})^{-1}(y)$ (in particular, $v_1,\ldots,v_k\in\pi_1^{-1}(y)$).

Let us consider the left-hand part of the smoothness condition. We have, first of all,
$$c_1(v)(v_1\wedge\ldots\wedge v_k)=v\wedge v_1\wedge\ldots\wedge v_k-\sum_{j=1}^k(-1)^{j+1}v_1\wedge\ldots\wedge g_1(v,v_j)\wedge\ldots\wedge v_k,$$ therefore by the definition of
$\tilde{f}_*^{\bigwedge}$ and by the linearity of it we have
$$\tilde{f}_*^{\bigwedge}(c_1(v)(v_1\wedge\ldots\wedge v_k))=$$
$$=\tilde{f}(v)\wedge\tilde{f}(v_1)\wedge\ldots\wedge\tilde{f}(v_k)-
\sum_{j=1}^k(-1)^{j+1}g_1(y)(v,v_j)\tilde{f}(v_1)\wedge\ldots\wedge\tilde{f}(v_{j-1})\wedge\tilde{f}(v_{j+1})\wedge\ldots\wedge\tilde{f}(v_k).$$ Furthermore, and again by the definition of
$\tilde{f}_*^{\bigwedge}$, the right-hand side of the compatibility condition is
\begin{flushleft}
$c_2(\tilde{f}(v))(\tilde{f}_*^{\bigwedge}(v_1\wedge\ldots\wedge v_k))=c_2(\tilde{f}(v))(\tilde{f}(v_1)\wedge\ldots\wedge\tilde{f}(v_k))=$
\end{flushleft}
$$=\tilde{f}(v)\wedge\tilde{f}(v_1)\wedge\ldots\wedge\tilde{f}(v_k)-
\sum_{j=1}^k(-1)^{j+1}g_2(f(y))(\tilde{f}(v),\tilde{f}(v_j))\tilde{f}(v_1)\wedge\ldots\wedge\tilde{f}(v_{j-1})\wedge\tilde{f}(v_{j+1})\wedge\ldots\wedge\tilde{f}(v_k).$$ The two expressions clearly
coincide, since the compatibility of pseudo-metrics $g_1$ and $g_2$ means precisely that $g_1(y)(v,v_j)=g_2(f(y))(\tilde{f}(v),\tilde{f}(v_j))$ for all $y$, $v$, and $v_j$. We therefore conclude
that the standard Clifford actions $c_1$ and $c_2$ of $\cl(V_1,g_1)$ and $\cl(V_2,g_2)$ on $\bigwedge_*(V_1)$ and $\bigwedge_*(V_2)$, respectively, are compatible the gluing along
$(\tilde{f}_*^{\bigwedge},f)$.

\paragraph{The induced action on $\bigwedge_*(V_1\cup_{\tilde{f}}V_2)$} It remains to comment on the action of $\cl(V_1\cup_{\tilde{f}}V_2,\tilde{g})$ on $\bigwedge_*(V_1\cup_{\tilde{f}}V_2)$.
This is an instance of Theorem 6.3 (more precisely, of the induced action $c$ described immediately prior to its statement), and once again, it is sufficient to specify this action for
$v\in V_1\cup_{\tilde{f}}V_2$ and $v_1\wedge\ldots\wedge v_k\in\bigwedge_*(V_1\cup_{\tilde{f}}V_2)$ with $(\pi_1\cup_{\tilde{f}}\pi_2)(v)=\pi^{\bigwedge_*}(v_1\wedge\ldots\wedge v_k)$.
Furthermore, the formulae are the already-seen ones, and we just add the appropriate standard inclusions, obtaining
\begin{flushleft}
$c(v)(v_1\wedge\ldots\wedge v_k)=$
\end{flushleft}
\begin{flushright}
$=j_1^{\bigwedge_*(V_1\cup_{\tilde{f}}V_2)}\left(c_1((j_1^{V_1})^{-1}(v))((j_1^{V_1})^{-1}(v_1)\wedge\ldots\wedge(j_1^{V_1})^{-1}(v_k))\right)$,
if $(\pi_1\cup_{\tilde{f}}\pi_2)(v)\in\mbox{Range}(i_1^{X_1})$,
\end{flushright}
\begin{flushleft}
$c(v)(v_1\wedge\ldots\wedge v_k)=$
\end{flushleft}
\begin{flushright}
$=j_2^{\bigwedge_*(V_1\cup_{\tilde{f}}V_2)}\left(c_2((j_2^{V_2})^{-1}(v))((j_2^{V_2})^{-1}(v_1)\wedge\ldots\wedge(j_2^{V_2})^{-1}(v_k))\right)$,
if $(\pi_1\cup_{\tilde{f}}\pi_2)(v)\in\mbox{Range}(i_2^{X_2})$.
\end{flushright}

\subsubsection{The action of $\cl(V_2^*,g_2^*)\cup_{(\tilde{F}^*)^{\cl}}\cl(V_1^*,g_1^*)$ on $\bigwedge(V_2)\cup_{\tilde{f}^{\bigwedge}}\bigwedge(V_1)$: compatibility of the actions
on the two factors}

We now consider the covariant case, which is a more complicated one, mainly due to the existence of various ways of presenting the pseudo-bundles involved (that of Clifford algebras and that of
exterior algebras), although, due to the diffeomorphisms described in the previous sections, essentially there is only one of each. In this section we treat the two respective presentations given by
gluing.

\paragraph{The Clifford actions $c_i$ and $c_i^*$} Let now $c_2^*:\cl(V_2^*,g_2^*)\to\mathcal{L}(\bigwedge(V_2),\bigwedge(V_2))$ and
$c_1^*:\cl(V_1^*,g_1^*)\to\mathcal{L}(\bigwedge(V_1),\bigwedge(V_1))$ be the standard Clifford actions.\footnote{The $*$ in $c_i^*$ is just a choice of notation; obviously, we do not mean the map dual
to$c_i$.} Then for $v^*\in V_i^*$ and for $v^1\wedge\ldots\wedge v^k\in(\pi_i^{\bigwedge})^{-1}(\pi_i^*(v^*))$ we have
$$c_i^*(v^*)(v^1\wedge\ldots\wedge v^k)=v^*\wedge v^1\wedge\ldots\wedge v^k-\sum_{j=1}^k(-1)^{j+1}v^1\wedge\ldots\wedge g_i^*(\pi_i^*(v^*))(v^*,v_j^*)\wedge\ldots\wedge v^k.$$
This action turns out to be closely related to $c_i$ and the natural pairing map $\Phi_i:V_i\to V_i^*$, that acts by $v\mapsto g_i(\pi_i(v))(v,\cdot)$. The map $\Phi_i$ is always smooth and (in the
finite-dimensional case) surjective; since the pseudo-bundles we are considering are assumed to be locally trivial, it also has a smooth inverse onto the characteristic sub-bundle of $V_i$. Finally,
each $\Phi_i$ extends to a smooth fibrewise linear map $\bigwedge_*(V_i)\to\bigwedge_*(V_i^*)=\bigwedge(V_i)$, that is well-behaved with respect to the exterior product:
$$\Phi_i(v_1\wedge\ldots\wedge v_k)=\Phi_i(v_1)\wedge\ldots\wedge\Phi_i(v_k).$$ The natural implication is the following statement.

\begin{prop}
For $i=1,2$ and for all $v,v_1,\ldots,v_k\in V_i$ such that $\pi_i(v)=\pi_i(v_1)=\ldots=\pi_i(v_k)$ we have
$$\Phi_i(c_i(v)(v_1\wedge\ldots\wedge v_k))=c_i^*(\Phi_i(v))(\Phi_i(v_1)\wedge\ldots\wedge\Phi_i(v_k)).$$
\end{prop}

In other words, there is a natural commutativity between each $c_i$, $c_i^*$, and the corresponding $\Phi_i$ (which is actually two-way if restricted to the characteristic subspace).

\paragraph{The actions of $\cl(V_2^*,g_2^*)$ and $\cl(V_1^*,g_1^*)$ on $\bigwedge(V_2)$ and $\bigwedge(V_1)$ are compatible} The above allows to easily check that the standard actions
$c_2^*$ and $c_1^*$ are compatible with the respect to the gluing of $\cl(V_2^*,g_2^*)$ to $\cl(V_1^*,g_1^*)$, which is along $((\tilde{F}^*)^{\cl},f^{-1})$, and the gluing of $\bigwedge(V_2)$ to
$\bigwedge(V_1)$, that is along $(\tilde{f}^{\bigwedge},f^{-1})$. Specifically, this means the following.

\begin{prop}
For all $v^*,v^1,\ldots,v^k\in V_2^*$ such that $\pi_2^*(v^*)=\pi_2^*(v^1)=\ldots=\pi_2^*(v^k)\in f(Y)$ we have
$$\tilde{f}^{\bigwedge}(c_2^*(v^*)(v^1\wedge\ldots\wedge v^k))=c_1^*((\tilde{F}^*)^{Cl}(v^*))(\tilde{f}^{\Lambda}(v^1\wedge\ldots\wedge v^k))=
c_1^*(\tilde{f}^*(v^*))(\tilde{f}^*(v^1)\wedge\ldots\wedge\tilde{f}^*(v^k)).$$
\end{prop}

The corollary of this is that there is the induced action $c_{*,\cup}$ of $\cl(V_2^*,g_2^*)\cup_{(\tilde{F}^*)^{\cl}}\cl(V_1^*,g_1^*)$ on $\bigwedge(V_2)\cup_{\tilde{f}^{\bigwedge}}\bigwedge(V_1)$.
Since the presentation $\cl(V_2^*,g_2^*)\cup_{(\tilde{F}^*)^{\cl}}\cl(V_1^*,g_1^*)$ does not automatically imply that it is a pseudo-bundle of Clifford algebras, we cannot yet say that
$c_{\cup,*}$ is a Clifford action; although it is one, as we explain in the section that follows.

\subsubsection{The diffeomorphism $\bigwedge(V_1\cup_{\tilde{f}}V_2)\cong\bigwedge(V_2)\cup_{\tilde{f}^{\bigwedge}}\bigwedge(V_1)$: comparing the Clifford actions}

The summary of the situation as it appears now, is that we have three pseudo-bundles of Clifford algebras
$$\cl((V_1\cup_{\tilde{f}}V_2)^*,\tilde{g}^*)\cong\cl(V_2^*\cup_{\tilde{f}^*}V_1^*,\widetilde{g^*})\cong\cl(V_2^*,g_2^*)\cup_{(\tilde{F}^*)^{\cl}}\cl(V_1^*,g_1^*)$$ and two pseudo-bundles of
exterior algebras
$$\bigwedge(V_1\cup_{\tilde{f}}V_2)\cong\bigwedge(V_2)\cup_{\tilde{f}^{\bigwedge}}\bigwedge(V_1),$$ with various identifications and Clifford-type actions between them. All of this
therefore reduces to a just one pseudo-bundle of Clifford algebras acting on just one pseudo-bundle of exterior algebras; let us see how exactly this happens for each case.

\paragraph{The Clifford algebra pseudo-bundle $\cl(V_2^*\cup_{\tilde{f}^*}V_1^*,\widetilde{g^*})$} The pseudo-metric $\widetilde{g^*}$ was described in Section 5.5.5; essentially,
its meaning is that on a fibre over a point in $i_1^{X_2}(X_2\setminus f(Y))$ it coincides with the pseudo-metric $g_2^*$, while elsewhere it coincides with the pseudo-metric $g_1^*$. The
construction of $\cl(V_2^*\cup_{\tilde{f}^*}V_1^*,\widetilde{g^*})$ allows for a ready identification of it with $\cl(V_2^*,g_2^*)\cup_{(\tilde{F}^*)^{\cl}}\cl(V_1^*,g_1^*)$.

Furthermore, $\widetilde{g^*}$ is related to $\tilde{g}^*$ in the way described in the same section, that is, by the formula
$$\left(\left((\Phi_{\cup,*})^*\right)^{-1}\otimes \left((\Phi_{\cup,*})^*\right)^{-1}\right)\circ\tilde{g}^*=\widetilde{g^*}\circ(\varphi_{X_1\leftrightarrow X_2});$$ this allows us to see that
$\cl(V_2^*\cup_{\tilde{f}^*}V_1^*,\widetilde{g^*})\cong\cl((V_1\cup_{\tilde{f}}V_2)^*,\tilde{g}^*)$ in a natural way.

\paragraph{Summary of diffeomorphisms: Clifford algebras} In this section we have mostly discussed diffeomorphisms for the pseudo-bundles of exterior algebras, mentioning only briefly
(Theorem 6.2) the diffeomorphism
$$\Phi^{\cl}:\cl(V_1,g_1)\cup_{\tilde{F}^{\cl}}\cl(V_2,g_2)\to\cl(V_1\cup_{\tilde{f}}V_2,\tilde{g}).$$ Furthermore, written in this  form it appears to refer to the contravariant case; of course, it
suffices to substitute $(V_2^*,g_2^*)$ for $(V_1,g_1)$, and $(V_1^*,g_1^*)$ for $(V_2,g_2)$, to obtain the diffeomorphism, that we denote by $\Phi^{\cl(*)}$, and that goes
$$\Phi^{\cl(*)}:\cl(V_2^*,g_2^*)\cup_{(\tilde{F}^*)^{\cl}}\cl(V_1^*,g_1^*)\to\cl(V_2^*\cup_{\tilde{f}^*}V_1^*,\widetilde{g^*}),$$ one of the diffeomorphisms that we referred to at the beginning of
this section.

What we need now is a diffeomorphism
$$\cl((V_1\cup_{\tilde{f}}V_2)^*,\tilde{g}^*)\cong\cl(V_2^*\cup_{\tilde{f}^*}V_1^*,\widetilde{g^*}).$$ This is essentially recovered from the gluing-dual commutativity diffeomorphism
$\Phi_{\cup,*}:(V_1\cup_{\tilde{f}}V_2)^*\to V_2^*\cup_{\tilde{f}^*}V_1^*$, in a way somewhat similar to how it was used in the case exterior algebras. Specifically, we first extend it by the
tensor product multiplicativity to the tensor algebras pseudo-bundle and then take its pushforward along the two projections onto the respective Clifford algebras pseudo-bundles.
That this is well-defined (that is, that the defining relation of a Clifford algebra is preserved by $\Phi_{\cup,*}$) easily follows from the above formula that relates $\tilde{g}^*$ and
$\widetilde{g^*}$. We introduce the following notation for the diffeomorphism thus obtained:
$$\Phi_{\cup,*}^{\cl}:\cl((V_1\cup_{\tilde{f}}V_2)^*,\tilde{g}^*)\to\cl(V_2^*\cup_{\tilde{f}^*}V_1^*,\widetilde{g^*}).$$

We note that possibly the main advantage that comes from the discussion carried out so far is possibly the existence of the composite diffeomorphism
$$\left(\Phi^{\cl(*)}\right)^{-1}\circ\Phi_{\cup,*}^{\cl}:\cl((V_1\cup_{\tilde{f}}V_2)^*,\tilde{g}^*)\to\cl(V_2^*,g_2^*)\cup_{(\tilde{F}^*)^{\cl}}\cl(V_1^*,g_1^*),$$ which allows to view the (covariant version
of the) Clifford algebra of a pseudo-bundle obtained by gluing as itself being the result of gluing of the Clifford algebras of the factors. Let us finally give a unique list of the diffeomorphisms for
Clifford algebras:
\begin{itemize}
\item $\Phi_{\cup,*}^{\cl}:\cl((V_1\cup_{\tilde{f}}V_2)^*,\tilde{g}^*)\to\cl(V_2^*\cup_{\tilde{f}^*}V_1^*,\widetilde{g^*})$;

\item $\Phi^{\cl(*)}:\cl(V_2^*,g_2^*)\cup_{(\tilde{F}^*)^{\cl}}\cl(V_1^*,g_1^*)\to\cl(V_2^*\cup_{\tilde{f}^*}V_1^*,\widetilde{g^*})$;

\item $\left(\Phi^{\cl(*)}\right)^{-1}\circ\Phi_{\cup,*}^{\cl}:\cl((V_1\cup_{\tilde{f}}V_2)^*,\tilde{g}^*)\to\cl(V_2^*,g_2^*)\cup_{(\tilde{F}^*)^{\cl}}\cl(V_1^*,g_1^*)$.
\end{itemize}

\paragraph{Summary of diffeomorphisms: covariant exterior algebras} For pseudo-bundles of exterior algebras, there are only two options,
$$\bigwedge(V_1\cup_{\tilde{f}}V_2)\cong\bigwedge(V_2)\cup_{\tilde{f}^{\bigwedge}}\bigwedge(V_1).$$ As we have seen shortly before (Section 6.3.3), there is a natural diffeomorphism
between them:
$$\Phi^{\bigwedge}:\bigwedge(V_1\cup_{\tilde{f}}V_2)\to\bigwedge(V_2)\cup_{\tilde{f}^{\bigwedge}}\bigwedge(V_1).$$

\paragraph{Summary of actions} Turning now to the Clifford actions, we outline first which Clifford algebra (or the result of gluing of such) has natural action on which pseudo-bundle of exterior
algebras:
\begin{itemize}
\item $\cl((V_1\cup_{\tilde{f}}V_2)^*,\tilde{g}^*)$ acts on $\bigwedge(V_1\cup_{\tilde{f}}V_2)$ via the standard Clifford action $c$;

\item $\cl(V_2^*,g_2^*)\cup_{(\tilde{F}^*)^{\cl}}\cl(V_1^*,g_1^*)$ acts on $\bigwedge(V_2)\cup_{\tilde{f}^{\bigwedge}}\bigwedge(V_1)$ via the action $c_{\cup,*}$ (see Proposition 6.8) induced by
the standard Clifford actions $c_2^*$ and $c_1^*$ of $\cl(V_2^*,g_2^*)$ and $\cl(V_1^*,g_1^*)$ on $\bigwedge(V_2)$ and $\bigwedge(V_1)$ respectively;

\item $\cl(V_2^*\cup_{\tilde{f}^*}V_1^*,\widetilde{g^*})$ has, again, the standard Clifford action, which we have not mentioned yet and which we now denote by $c_{*,\cup}$, on
$\bigwedge_*(V_2^*\cup_{\tilde{f}^*}V_1^*)$. Notice that the latter is the contravariant exterior algebra of the pseudo-bundle $V_2^*\cup_{\tilde{f}^*}V_1^*$; it is naturally diffeomorphic to
$\bigwedge(V_1\cup_{\tilde{f}}V_2)$ via the obvious extension $\Phi_{\cup,*}^{\bigwedge}$ of the gluing-dual commutativity diffeomorphism; this diffeomorphism goes
$$\Phi_{\cup,*}^{\bigwedge}:\bigwedge(V_1\cup_{\tilde{f}}V_2)\to\bigwedge_*(V_2^*\cup_{\tilde{f}^*}V_1^*).$$
\end{itemize}

\paragraph{The equivalence of actions} Let us now specify how the diffeomorphisms $\Phi^{\cl(*)}$, $\Phi_{\cup,*}^{\cl}$, and $\left(\Phi^{\cl(*)}\right)^{-1}\circ\Phi_{\cup,*}^{\cl}$, as well as
the actions $c$, $c_{\cup,*}$, and $c_{*,\cup}$, are related to each other (it is quite clear that they are). The most natural, or the most immediate, relations to check are:
\begin{itemize}
\item the action $c$ of $\cl((V_1\cup_{\tilde{f}}V_2)^*,\tilde{g}^*)$ on $\bigwedge(V_1\cup_{\tilde{f}}V_2)$ should be compared to the action $c_{*,\cup}$ of $\cl(V_2^*\cup_{\tilde{f}^*}V_1^*,\widetilde{g^*})$
on $\bigwedge_*(V_2^*\cup_{\tilde{f}^*}V_1^*)$, with respect to the diffeomorphisms $\Phi_{\cup,*}^{\cl}$ and $\Phi_{\cup,*}^{\bigwedge}$;

\item the action $c_{\cup,*}$ of $\cl(V_2^*,g_2^*)\cup_{(\tilde{F}^*)^{\cl}}\cl(V_1^*,g_1^*)$ on $\bigwedge(V_2)\cup_{\tilde{f}^{\bigwedge}}\bigwedge(V_1)$ should be compared, again, to the action
$c_{*,\cup}$ of $\cl(V_2^*\cup_{\tilde{f}^*}V_1^*,\widetilde{g^*})$ on $\bigwedge_*(V_2^*\cup_{\tilde{f}^*}V_1^*)$, with respect to the diffeomorphisms $\Phi^{\cl(*)}$ and
$\Phi_{\cup,*}^{\bigwedge}\circ(\Phi^{\bigwedge})^{-1}$.
\end{itemize}
Notice that these equivalences imply the equivalence of $c$ and $c_{\cup,*}$ automatically.

Let us consider first the actions $c$ and $c_{*,\cup}$. We give a down-to-earth description of their interrelation, which is as follows. Let $v\in\cl((V_1\cup_{\tilde{f}}V_2)^*,\tilde{g}^*)$, and
let $e\in\bigwedge(V_1\cup_{\tilde{f}}V_2)$ be such that $\pi^{\bigwedge}(e)=\pi^{\cl}(v)$ (that is, $e$ belongs to the fibre on which $c(v)$ acts). Then we have
$$\Phi_{\cup,*}^{\bigwedge}(c(v)(e))=c_{*,\cup}(\Phi_{\cup,*}^{\cl}(v))(\Phi_{\cup,*}^{\bigwedge}(e)).$$

Let us now consider $c_{\cup,*}$ and $c_{*,\cup}$. The same kind of relation holds, \emph{i.e.}, if $v\in\cl(V_2^*,g_2^*)\cup_{(\tilde{F}^*)^{\cl}}\cl(V_1^*,g_1^*)$ and
$e\in\bigwedge(V_2)\cup_{\tilde{f}^{\bigwedge}}\bigwedge(V_1)$ are such that $c_{\cup,*}(v)(e)$ is well-defined (since the base space is the same for both pseudo-bundles, this means that they are
taken in the two fibres over the same point), we have the following:
$$\left(\Phi_{\cup,*}^{\bigwedge}\circ(\Phi^{\bigwedge})^{-1}\right)(c_{\cup,*}(v)(e))=c_{*,\cup}\left(\Phi^{\cl(*)}(v)\right)\left((\Phi_{\cup,*}^{\bigwedge}\circ(\Phi^{\bigwedge})^{-1})(e)\right).$$

Finally, for completeness we give the equivalence formula for the action $c$ (of $\cl((V_1\cup_{\tilde{f}}V_2)^*,\tilde{g}^*)$ on $\bigwedge(V_1\cup_{\tilde{f}}V_2)$) and $c_{\cup,*}$
(of $\cl(V_2^*,g_2^*)\cup_{(\tilde{F}^*)^{\cl}}\cl(V_1^*,g_1^*)$ on $\bigwedge(V_2)\cup_{\tilde{f}^{\bigwedge}}\bigwedge(V_1)$), with respect to $\left(\Phi^{\cl(*)}\right)^{-1}\circ\Phi_{\cup,*}^{\cl}$
and $\Phi^{\bigwedge}$. For $v\in\cl((V_1\cup_{\tilde{f}}V_2)^*,\tilde{g}^*)$ and $e\in\bigwedge(V_1\cup_{\tilde{f}}V_2)$ that project to the same point in the base space $X_1\cup_f X_2$, we have
$$\Phi^{\bigwedge}(c(v)(e))=c_{\cup,*}(((\Phi^{\cl(*)})^{-1}\circ\Phi_{\cup,*}^{\cl})(v))(\Phi^{\bigwedge}(e)).$$

\subsection{Some examples}

In selecting examples to illustrate the constructions described in the present section, there are two main considerations to keep in mind. First of all, most of our constructions ask for $\tilde{f}^*$
to be a diffeomorphism; we note that this does not imply that $\tilde{f}$ itself is so, only that its restriction on the characteristic sub-bundle should be one. The distinction is particularly important
in the covariant case. Notice also that in this case the purely diffeological side of matters has less to do with some unusual \emph{function} being regarded as smooth (all fibres having standard
diffeologies), and more to do with some unusual \emph{spaces} being looked at, as if they were smooth manifold.

\subsubsection{The by-now-classic: two planes over a cross}

We dedicate this section to considering a rather simple, but not entirely trivial, example that illustrates the above abstract constructions. The basic object for it is the trivial fibering of the standard
$\matR^2$ over the standard $\matR$ (via the projection onto the first coordinate). The easiest way to glue together two copies of such, is to take the wedge of the two copies of the base
$\matR$ at their respective origins, and then identify the lines over them in some obvious fashion.

\paragraph{The pseudo-bundles $\pi_1:V_1\to X_1$ and $\pi_2:V_2\to X_2$, the gluing $(\tilde{f},f)$, and the pseudo-metrics $g_1$ and $g_2$} Thus, as we just said, we have $V_1=V_2=\matR^2$
with its standard diffeology, $X_1=X_2=\matR$, also standard, and the two standard projections on the $x$-axis, $\pi_1:V_1\ni(x,y)\mapsto x\in\matR=X_1$ and $\pi_2:V_2\ni(x,y)\mapsto x\in\matR=X_2$.
The pseudo-bundle structure is given by imposing on each fibre $(x,y_1)+(x,y_2)\mapsto(x,y_1+y_2)$ and $\lambda(x,y_1)\mapsto(x,\lambda y_1)$ (recall once again that this is different from
the standard operations on $\matR^2$). Next, we have the gluing of these two pseudo-bundles along $(\tilde{f},f)$, where $f$ is defined on the origin $\{0\}\subset X_1$ of $\matR$ via
$f(0)=0\in X_2=\matR$; its lift $\tilde{f}$ acts on the fibre $\{(0,y)\}\subset V_1=\matR^2$, mapping it to the fibre $\{(0,y)\}\subset V_2=\matR^2$ via a usual linear transformation. Thus, $\tilde{f}$ is
uniquely defined by some constant $a\in\matR$ via the rule
$$\tilde{f}(0,1)=(0,a)\in V_2=\matR^2.$$

Notice that since all fibres have standard diffeology, the characteristic subspace of any fibre coincides with the fibre itself; the implication of this, that is relevant at the moment, is that, for there to exist
compatible pseudo-metrics $g_1$ and $g_2$, we must have $a\neq 0$. Then, said in simpler terms, any pseudo-metric $g_1$ on $V_1$ is determined by a usual smooth and everywhere positive
function $f_1:\matR\to\matR$ by setting
$$g_1(x)(v,w)=f_1(x)\cdot e^2(v)\cdot e^2(w),$$ where $e^2$ is the second element of the canonical dual basis (in other words, $e^2(v)$ is just the $y$-coordinate of $v$). Likewise, $g_2$ can be
written as
$$g_2(x)(v,w)=f_2(x)\cdot e^2(v)\cdot e^2(w),$$ and the compatibility means that $f_1(0)=a^2f_2(0)$. Indeed, the compatibility condition means that we must have
$$g_1(0)\left((0,y_1),(0,y_2)\right)=g_2(0)\left(\tilde{f}(0,y_1),\tilde{f}(0,y_2)\right)\Leftrightarrow f_1(0)y_1y_2=g_2(0)\left((0,ay_1),(0,ay_2)\right)=a^2f_2(0)y_1y_2.$$

\paragraph{The result of gluing} The pseudo-bundle that results from the gluing described in the previous paragraph can be described by representing $V_1\cup_{\tilde{f}}V_2$ as the union
$\{(x,0,z)\}\cup\{(0,y,z)\}$ of two planes in $\matR^3$, and, accordingly, $X_1\cup_f X_2$ as the union $\{(x,0,0)\}\cup\{(0,y,0)\}$ of the two axes, with the projection $\pi_1\cup_{(\tilde{f},f)}\pi_2$
acting by $(x,0,z)\mapsto(x,0,0)$, $(0,y,z)\mapsto(0,y,0)$. Notice that this is more of a topological representation than a diffeological one, in the sense that the subset diffeology on the two subsets
relative to the standard one on $\matR^3$ is coarser than the gluing diffeology (see \cite{watts}, Example 2.61).

The pseudo-metric $\tilde{g}$ (recall that in the gluing-dual commutative case, such as the one we are considering at the moment, it does not depend of the specific way of constructing it) on
$V_1\cup_{\tilde{f}}V_2$ thus presented can be described as
$$\tilde{g}(x,y,0)=\left\{\begin{array}{ll} f_1(x)dz^2 & \mbox{if }y=0\mbox{ and }x\neq 0,\\ f_2(y)dz^2 & \mbox{if }x=0. \end{array}\right.$$

\paragraph{The pairing maps $\Psi_{g_1}$, $\Psi_{g_2}$, and $\Psi_{\tilde{g}}$} Since all fibres are standard, the characteristic sub-bundles coincide with the pseudo-bundles themselves,
so all three maps are automatically invertible. Written explicitly, they act by:
$$\Psi_{g_1}(x,y)=f_1(x)ye^2,\,\,\,\Psi_{g_2}(x,y)=f_2(x)ye^2,$$ and then, using the just-mentioned presentation of  $V_1\cup_{\tilde{f}}V_2$ as the subset of $\matR^3$ given by the
equation $xy=0$, we have
$$\Psi_{\tilde{g}}(x,y,z)=\left\{\begin{array}{ll} f_1(x)zdz & \mbox{if }y=0,\\ f_2(y)zdz & \mbox{if }x=0 \end{array}\right.$$

\paragraph{The dual pseudo-metrics} The dual pseudo-metrics are therefore described in the same manner as $g_1$ and $g_2$, but the coefficients are inverted:
$$g_2^*(x)(v^*,w^*)=\frac{1}{f_2(x)}\cdot v^*(e_2)\cdot w^*(e_2)\,\,\,\mbox{ and }\,\,\,g_1^*(x)(v^*,w^*)=\frac{1}{f_1(x)}\cdot v^*(e_2)\cdot w^*(e_2).$$ Indeed, let $v^*=(x,y^ve^2)$ and $w^*=(x,y^we^2)$;
then
$$g_2^*(x)(v^*,w^*)=g_2(x)(\Psi_{g_2}^{-1}(v^*),\Psi_{g_2}^{-1}(w^*))=f_2(x)e^2(\Psi_{g_2}^{-1}(v^*))e^2(\Psi_{g_2}^{-1}(w^*))$$ by definition. If $(x,y_v)=\Psi_{g_2}^{-1}(v^*)$, we have
$y^ve^2=g_2(x)(v,\cdot)=f_2(x)y_ve^2$, so in the end
$$e^2(\Psi_{g_2}^{-1}(v^*))=\frac{1}{f_2(x)}v^*(e_2),\,\,\,\,e^2(\Psi_{g_2}^{-1}(w^*))=\frac{1}{f_2(x)}w^*(e_2)\Rightarrow$$
$$\Rightarrow g_2^*(x)(v^*,w^*)=f_2(x)\cdot\frac{1}{f_2(x)}v^*(e_2)\cdot\frac{1}{f_2(x)}w^*(e_2),$$ as we claimed.

Notice also that by definition of the dual map $\tilde{f}^*$ we have $\tilde{f}^*(0,e^2)=(0,ae^2)$. By direct calculation we obtain
$$g_2^*(0)(e^2,e^2)=\frac{1}{f_2(0)},\,\,g_1^*(0)(\tilde{f}(e^2),\tilde{f}(e^2))=\frac{a^2}{f_1(0)};$$ since the assumption that we have made already, that of $f_1(0)=a^2f_2(0)$, is equivalent to
$\frac{1}{f_2(0)}=\frac{a^2}{f_1(0)}$, we indeed obtain that the compatibility of $g_1$ with $g_2$ implies the compatibility of $g_2^*$ with $g_1^*$, by the sole assumption that $\tilde{f}=\tilde{f}_0$ be
a diffeomorphism (as it should be by one of the results cited in this section).

Finally, we can write the pseudo-metric $\tilde{g}^*\equiv\widetilde{g^*}$ as
$$\tilde{g}^*(x',y',0)=\left\{\begin{array}{ll}
\frac{1}{f_1(x')}\frac{\partial}{\partial z}\otimes\frac{\partial}{\partial z} & \mbox{if }y'=0,\\
\frac{1}{f_2(x')}\frac{\partial}{\partial z}\otimes\frac{\partial}{\partial z} & \mbox{if }x'=0.
\end{array}\right.$$ This is of course a rather artificial choice, that we make in order to stay as close as possible to the standard notation (in particular, when writing $\frac{\partial}{\partial z}$ we implicitly
identify the second dual of $\matR^3$, with $\matR^3$ itself; we have already done in describing the dual pseudo-metrics).

\paragraph{The pseudo-bundles of Clifford algebras} All fibres in our case are $1$-dimensional, so as a vector space, the Clifford algebra of any of them coincides with the direct of the fibre itself
with a copy of $\matR$, which is the span of the unit. Let $x_i\in X_i$ be a fixed point; then the fibre $\pi_i^{-1}(x_i)$ is the set $\{(x_i,y)\}$. The Clifford relation is then $(x_i,1)\otimes(x_i,1)=-f_i(x_i)$,
so the multiplication in the Clifford algebra $\cl(\pi_i^{-1}(x_i),g_i(x_i))$ is given by
$$(x_i,y_1)\cdot_{\cl}(x_i,y_2)=-f_i(x_i)y_1y_2.$$ The gluing of $\cl(\pi_1^{-1}(0),g_1(0))$ to $\cl(\pi_2^{-1}(0),g_2(0))$ is given by the direct sum of $\tilde{f}$ with the identity on the scalar part,
$\tilde{F}^{\cl}=\tilde{f}\oplus\mbox{Id}_{\matR}$, and we have
$$\tilde{F}^{\cl}\left((0,y_1)\cdot_{\cl}(0,y_2)\right)=-f_1(0)y_1y_2\mbox{ and }$$
$$\tilde{f}(0,y_1)\cdot_{\cl}\tilde{f}(0,y_2)=(0,ay_1)\cdot_{\cl}(0,ay_2)=-a^2f_2(0)y_1y_2=-f_1(0)y_1y_2.$$

Each of the two individual Clifford algebras' pseudo-bundles is thus a trivial fibering of $\matR^3$ over $\matR$; the result of their gluing can be described as the subset in $\matR^4$ given by the
equation $xy=0$, so that
$$\cl(V_1,g_1)\cup_{\tilde{F}^{\cl}}\cl(V_2,g_2)=\{(x,0,z,w),\mbox{ where }x\neq 0\}\cup\{(0,y,z,w)\},$$ with the pseudo-bundle projection given by
$$(\pi_1^{\cl}\cup_{(\tilde{F}^{\cl},f)}\pi_2^{\cl})(x,y,z,w)=\left\{\begin{array}{ll} (x,0,0,0), & \mbox{if }y=0\mbox{ and }x\neq 0\\
(0,y,0,0), & \mbox{if }x=0, \end{array}\right.$$ and the Clifford multiplication being defined by
$$(x,0,z_1,w_1)\cdot_{\cl}(x,0,z_2,w_2)=(x,0,z_1w_2+z_2w_1,-f_1(x)z_1z_2+w_1w_2),$$
$$(0,y,z_1,w_1)\cdot_{\cl}(0,y,z_2,w_2)=(0,y,z_1w_2+z_2w_1,-f_2(y)z_1z_2+w_1w_2).$$

From this, it is also quite evident that the result trivially coincides with $\cl(V_1\cup_{\tilde{f}}V_2,\tilde{g})$, so much in fact, that we can only distinguish between the two by choosing two
slightly different forms of designating the same subset in $\matR^4$. Specifically, in the case of $\cl(V_1\cup_{\tilde{f}}V_2,\tilde{g})$ we describe the set of its points as
$$\{(x,y,z,w),\mbox{ where }xy=0\}.$$ Obviously, this is the same set as we described as the set of points of $\cl(V_1,g_1)\cup_{\tilde{F}^{\cl}}\cl(V_2,g_2)$; the chosen
presentation of the latter emphasizes its structure as the result of a gluing.

Notice also that, viewing $V_1\cup_{\tilde{f}}V_2$ as the subset of  $\matR^3$ given by the equation $xy=0$, the Clifford relation, for elements of $V_1\cup_{\tilde{f}}V_2$, yields
$$\left\{\begin{array}{l}
(x,0,z_1)\cdot_{\cl}(x,0,z_2)=-f_1(x)z_1z_2\\
(0,y,z_1)\cdot_{\cl}(0,y,z_2)=-f_2(y)z_1z_2,
\end{array}\right.$$ the Clifford multiplication in $\cl(V_1\cup_{\tilde{f}}V_2,\tilde{g})$ (obtained by adding the fourth coordinate) is given by precisely the same formula as in the case of
$\cl(V_1,g_1)\cup_{\tilde{F}^{\cl}}\cl(V_2,g_2)$.

\paragraph{The pseudo-bundles of covariant Clifford algebras} The case of the Clifford algebras associated to the dual pseudo-bundles of $V_1^*$, $V_2^*$, and $(V_1\cup_{\tilde{f}}V_2)^*$ is
the same as that in the contravariant case. There is one formal distinction that we can make in order to stress the difference between $(V_1\cup_{\tilde{f}}V_2)^*$ and $V_2^*\cup_{\tilde{f}^*}V_1^*$,
which we indicate immediately for the Clifford algebras as a whole.

Specifically, consider again the subset in $\matR^4$ given by the equation $xy=0$. This is the subset that is identified with all three (shapes of) the Clifford algebra, in accordance with the fact
that all three are diffeomorphic. For all three possibilities, we identify the copy of $V_1^*$ contained in either of them,\footnote{Recall that our gluing is along a diffeomorphism, so all our
pseudo-bundles admit natural inclusions of both $V_1^*$ and $V_2^*$; in general, this would not be the case for one of them.} with the hyperplane $\{(x,0,z,0)\}$, and the copy of $V_2^*$, with
the hyperplane $\{(0,y,z,0)\}$ (once again, the fourth coordinate $w$ corresponds to the scalar part of the Clifford algebra of a fibre). The distinction mentioned above consists in the following.
When this subset is viewed as $\cl((V_1\cup_{\tilde{f}}V_2)^*,\tilde{g}^*)$, we describe the Clifford multiplication as
$$\left\{\begin{array}{l}
(x,0,z_1,w_1)\cdot_{\cl}(x,0,z_2,w_2)=(x,0,z_1w_2+z_2w_1,-\frac{1}{f_1(x)}z_1z_2+w_1w_2)\mbox{ for }x\neq 0,\\
(0,y,z_1,w_1)\cdot_{\cl}(0,y,z_2,w_2)=(0,y,z_1w_2+z_2w_1,-\frac{1}{f_2(2)}z_1z_2+w_1w_2)\mbox{ otherwise}.
\end{array}\right.$$ On the other hand, when we view the same subset as either $\cl(V_2^*,g_2^*)\cup_{(\tilde{F}^*)^{\cl}}\cl(V_1^*,g_1^*)$ or $\cl(V_2^*\cup_{\tilde{f}^*}V_1^*,\widetilde{g^*})$,
we describe the corresponding product by
$$\left\{\begin{array}{l}
(x,0,z_1,w_1)\cdot_{\cl}(x,0,z_2,w_2)=(x,0,z_1w_2+z_2w_1,-\frac{1}{f_1(x)}z_1z_2+w_1w_2)\mbox{ for all }x,\\
(0,y,z_1,w_1)\cdot_{\cl}(0,y,z_2,w_2)=(0,y,z_1w_2+z_2w_1,-\frac{1}{f_2(2)}z_1z_2+w_1w_2)\mbox{ for }y\neq 0.
\end{array}\right.$$ As we see the difference, albeit minimal, is found precisely over the domain of gluing (a single point in our case). It is also clear how the compatibility condition $f_1(0)=a^2f_2(0)$
ensures that the two expressions coincide, so that the presentation reflects the diffeomorphism of the two objects.

\paragraph{The pseudo-bundles of exterior algebras} These can be presented in exactly the same way as those of Clifford algebras (but the multiplication works differently). Namely, in the
contravariant case we have a unique presentation immediately, which is, again, as a subset of $\matR^4$ given by the equation $xy=0$, with the exterior product
$$\left\{\begin{array}{l}
(x,0,z_1,w_1)\wedge(x,0,z_2,w_2)=(x,0,z_1w_2+z_2w_1,w_1w_2),\\
(0,y,z_1,w_1)\wedge(0,y,z_2,w_2)=(0,y,z_1w_2+z_2w_1,w_1w_2)
\end{array}\right.$$ (this is simply because all fibres are $1$-dimensional). Notice that if the exterior algebras $\bigwedge_*(V_1)$ and $\bigwedge_*(V_2)$ are seen as the hyperplanes of the
equations $y=0$ and $x=0$, the gluing map $\tilde{f}_*^{\bigwedge}$ acts by
$$\tilde{f}_*^{\bigwedge}:\{y=0\}\ni(0,0,z,w)\mapsto(0,0,az,w).$$

The case of the three exterior algebras relative to the dual pseudo-bundles is analogous, and the expression that defines the exterior product is exactly the same. There is a slight formal
difference in describing the gluing map $(\tilde{f}^*)^{\bigwedge}$, which acts by
$$\{x=0\}\ni(0,0,z,w)\mapsto(0,0,az,w);$$ the formula is exactly the same, and the difference consists in considering its domain of definition (which remains the same) as a subset of the hyperplane
$\{x=0\}$ rather than one of $\{y=0\}$.

\paragraph{The Clifford actions} Finally, we use the same presentations to describe the Clifford actions. These are of course standard, since all fibres are standard, so we shall see them only
over the domain of gluing.

In the contravariant case, we have two exterior algebras, $\bigwedge_*(V_1\cup_{\tilde{f}}V_2)$ and $\bigwedge_*(V_1)\cup_{\tilde{f}_*^{\bigwedge}}\bigwedge_*(V_2)$, with the actions $c$ and
$\tilde{c}$ of, respectively, $\cl(V_1\cup_{\tilde{f}}V_2,\tilde{g})$ and $\cl(V_1,g_1)\cup_{\tilde{F}^{\cl}}\cl(V_2,g_2)$. In the former case, we have
\begin{flushleft}
$c((0,0,z_2,w_2))(0,0,z,w)=(0,0,z_2,w_2)\wedge(0,0,z,w)-\left(0,0,0,\tilde{g}(0,0,0)((0,0,z_2),(0,0,z))\right)=$
\end{flushleft}
\begin{flushright}
$=(0,0,z_2w+w_2z,w_2w+f_2(0)z_2z)$;
\end{flushright} in the latter case, the only thing that changes with respect to the formula just given, is that the term $\tilde{g}(0,0,0)((0,0,z_2),(0,0,z)$ is replaced by the term $g_2(0,0)((0,z_2),(0,z))$,
whose value however is exactly the same.

The covariant case is in fact analogous, although in principle we have three exterior algebras, $\bigwedge(V_1\cup_{\tilde{f}}V_2)$, $\bigwedge_*(V_2^*\cup_{\tilde{f}^*}V_1^*)$, and
$\bigwedge(V_2)\cup_{(\tilde{f}^*)^{\bigwedge}}\bigwedge(V_1)$, with the actions $c$, $c_{*,\cup}$, and $c_{\cup,*}$ of, respectively, $\cl((V_1\cup_{\tilde{f}}V_2)^*,\tilde{g}^*)$,
$Cl(V_2^*\cup_{\tilde{f}^*}V_1^*,\widetilde{g^*})$, and $\cl(V_2^*,g_2^*)\cup_{(\tilde{F}^*)^{\cl}}\cl(V_1^*,g_1^*)$. Once again, these actions have the same form everywhere except over the
point of gluing (the origin), where we would formally write the formulae for $c((0,0,z_2,w_2))(0,0,z,w)$, $c_{*,\cup}((0,0,z_2,w_2))(0,0,z,w)$, and $c_{\cup,*}((0,0,z_2,w_2))(0,0,z,w)$ with respect to
$\tilde{g}^*$, $\widetilde{g^*}$, or $g_1^*$, respectively, with the compatibility condition ensuring the same result in all three cases. Thus, we have
\begin{flushleft}
$c((0,0,z_2,w_2))(0,0,z,w)=(0,0,z_2,w_2)\wedge(0,0,z,w)-\left(0,0,0,\tilde{g}^*(0,0,0)((0,0,z_2),(0,0,z))\right)=$
\end{flushleft}
\begin{flushright}
$=(0,0,z_2w+w_2z,w_2w+\frac{1}{f_2(0)}z_2z)$,
\end{flushright}
\begin{flushleft}
$c_{*,\cup}((0,0,z_2,w_2))(0,0,z,w)=(0,0,z_2,w_2)\wedge(0,0,z,w)-\left(0,0,0,\widetilde{g^*}(0,0,0)((0,0,z_2),(0,0,z))\right)=$
\end{flushleft}
\begin{flushright}
$=(0,0,z_2w+w_2z,w_2w-g_1^*(0,0)((0,z_2),(0,z)))=(0,0,z_2w+w_2z,w_2w+\frac{1}{f_1(0)}z_2z)$,
\end{flushright}
\begin{flushleft}
$c_{\cup,*}((0,0,z_2,w_2))(0,0,z,w)=(0,0,z_2,w_2)\wedge(0,0,z,w)-\left(0,0,0,g_1^*(0,0)((0,z_2),(0,z))\right)=$
\end{flushleft}
\begin{flushright}
$=(0,0,z_2w+w_2z,w_2w+\frac{1}{f_1(0)}z_2z)$.
\end{flushright} In particular, we see from the last two expressions that even written with extreme formality, the difference between the latter two pseudo-bundles is very slight, reflecting the fact that
the diffeomorphisms involved are very natural indeed.

\subsubsection{An example when $\tilde{f}$ is not a diffeomorphism, while $\tilde{f}^*$ is}

As we have noted quite a few times already, this possibility is a peculiarity of diffeology. It can be easily illustrated by extending the example considered in the previous section, in the following way.

Let $\pi_2:V_2\to X_2$ be the same, \emph{i.e.}, the standard projection $\matR^2\to\matR$; define $\pi_1:V_1\to X_1$ to be the projection of $V_1=\matR^3$ to $X_1=\matR$, where $X_1$ carries
the standard diffeology, and $V_1=\matR\times\matR\times\matR$ carries the product diffeology relative to the standard diffeologies on the first two factors and the vector space diffeology generated
by the plot $\matR\ni x\mapsto|x|$ on the third factor.\footnote{In fact, any non-standard vector space diffeology would be sufficient for our purposes.} The projection $\pi_1$ is just the projection onto
the first factor. The gluing map $f$ for the bases is the same, $\{0\}\to\{0\}$, and the one for the total spaces is almost the  same, specifically, $\tilde{f}(0,y,z)=(0,ay)$ with $a\neq 0$ (again, notice that
zeroing out the third coordinate is necessary for $\tilde{f}$ to be smooth). The pseudo-bundle $\pi_2:V_2\to X_2$ carries the same pseudo-metric $g_2$ as in the previous example, while the
pseudo-metric $g_1$ on $\pi_1:V_1\to X_1$ extends the previous one in a trivial manner:
$$g_1(x)((x,y_1,z_1),(x,y_2,z_2))=f_1(x)y_1y_2.$$ The compatibility condition remains the same.

It should also be noted right away that the entire covariant case coincides with that of the example treated in the previous section. We only consider the pseudo-bundle $V_1\cup_{\tilde{f}}V_2$ and
the corresponding contravariant constructions.

\paragraph{The pseudo-bundle $V_1\cup_{\tilde{f}}V_2$} We represent it as a subset in $\matR^4$, specifically as the union of the plane given by the equations $x=0$ and $w=0$ (the part corresponding
to $V_2$), and of the set $\{y=0\}\setminus\{x=0,y=0,w=0\}$; this is the part corresponding to $V_1$, where excising the line $\{x=0,y=0,w=0\}$ reflects how $V_1\cup_{\tilde{f}}V_2$ contains
$V_1\setminus\pi_1^{-1}(Y)$, and not the entire $V_1$. Thus, the entire set can be described as
$$\left\{\begin{array}{ll} (x,0,z,w) & \mbox{except the points }(0,0,z,0)\\ (0,y,z,0) & \mbox{for all }y,z. \end{array}\right.$$

\paragraph{The two Clifford algebras} The Clifford algebra of $V_2$ is the already seen one; relative to the presentation of $V_1\cup_{\tilde{f}}V_2$ given above, we could describe it as a
subset of $\matR^5$, adding the 5th coordinate $u_1$ for the scalar part of $\cl(V_2,g_2)\cong\matR\oplus V_2$. Thus,
$$\cl(V_2,g_2)=\{(0,y,z,0,u_1)\},$$ with the Clifford multiplication given by
$$(0,y,z',0,u_1')\cdot_{\cl}(0,y,z'',0,u_1'')=(0,y,u_1''z'+u_1'z'',0,u_1'u_1''-f_2(y)z'z'').$$

The Clifford algebra $\cl(V_1,g_1)$ is bigger; since the fibres of $V_1$ have dimension $2$, each fibre of $\cl(V_1,g_1)$ has dimension $4$. Thus, we represent it as a subset in $\matR^6$, by adding
the coordinates $u_1,u_2$, where $u_1$ corresponds to the scalar part and $u_2$ corresponds to the degree $2$ vector part. Thus,
$$\cl(V_1,g_1)=\{(x,0,z,w,u_1,u_2)\},$$ with the Clifford multiplication given by
\begin{flushleft}
$(x,0,z',w',u_1',u_2')\cdot_{\cl}(x,0,z'',w'',u_1'',u_2'')=$
\end{flushleft}
\begin{flushright}
$=(x,0,z'u_1''+z''u_1',w'u_1''+w''u_1'+f_1(x)w'',u_1'u_1''-f_1(x)z'z'',u_1'u_2''+u_1''u_2')$.
\end{flushright} Finally, $\cl(V_1\cup_{\tilde{f}}V_2,\tilde{g})$ can be described as the following subset in $\matR^6$:
$$\{(x,y,z,w,u_1,u_2)\mbox{ such that }xy=0,\,x=0\Rightarrow w=u_2=0\},$$ while $\cl(V_1,g_1)\cup_{\tilde{F}^{\cl}}\cl(V_2,g_2)$ is presented as the subset in $\matR^6$ of the following form:
$$\{(x,0,z,w,u_1,u_2)\mbox{ such that }x\neq 0\}\cup\{(0,y,z,0,u_1,0)\mbox{ for all }y,z\}.$$ The fibrewise multiplication is given by
\begin{flushleft}
$(x,0,z',w',u_1',u_2')\cdot_{Cl}(x,0,z'',w'',u_1'',u_2'')=$
\end{flushleft}
\begin{flushright}
$=(x,0,z'u_1''+z''u_1',w'u_1''+w''u_1'+f_1(x)w'',u_1'u_1''-f_1(x)z'z'',u_1'u_2''+u_1''u_2')$,
\end{flushright}
\begin{flushleft}
$(0,y,z',0,u_1',0)\cdot_{\cl}(0,y,z'',0,u_1'',0)=(0,y,u_1''z'+u_1'z'',0,u_1'u_1''-f_2(y)z'z'',0)$.
\end{flushleft} The distinction between the two shapes of the Clifford algebra could be made by referring to $\tilde{g}$ in the former case, and to $g_1$ and $g_2$ in the latter case; the end result
would immediately be the same, specifically the one just indicated.

\paragraph{The contravariant exterior algebras} Likewise, the exterior algebras $\bigwedge_*(V_1)$ and $\bigwedge_*(V_2)$ are given by the same sets. Both of these we immediately represent as
subsets of $\matR^6$, with the $5$-th coordinate being the scalar part and the $6$-th coordinate being the exterior product corresponding to the exterior product relative to the $3$-rd and the $4$-th
coordinates; in the case of $V_2$, this part is obviously trivial. Thus, we have
$$\bigwedge_*(V_1)=\{(x,0,z,w,u_1,u_2)\},\,\,\,\bigwedge_*(V_2)=\{(0,y,z,0,u_1,0)\},$$ with the exterior product given by
\begin{flushleft}
$(x,0,z',w',u_1',u_2')\wedge(x,0,z'',w'',u_1'',u_2'')=$
\end{flushleft}
\begin{flushright}
$=(x,0,u_1''z'+u_1'z'',u_1''w'+u_1'w'',u_1'u_1'',u_1''u_2'+u_1'u_2''+z'w''-z''w')$,
\end{flushright}
\begin{flushleft}
$(0,y,z',0,u_1',0)\wedge(0,y,z'',0,u_1'',0)=(0,y,u_1''z'+u_1'z'',0,u_1'u_1'',0)$.
\end{flushleft}
The exterior algebras $\bigwedge_*(V_1\cup_{\tilde{f}}V_2)$ and $\bigwedge_*(V_1)\cup_{\tilde{f}^{\bigwedge_*}}\bigwedge_*(V_2)$ are then represented respectively by the sets
$$\bigwedge_*(V_1\cup_{\tilde{f}}V_2)=\{(x,y,z,w,u_1,u_2),\mbox{ where }xy=0,\,x=0\Rightarrow w=u_2=0\},$$
$$\bigwedge_*(V_1)\cup_{\tilde{f}^{\bigwedge_*}}\bigwedge_*(V_2)=\{(x,0,z,w,u_1,u_2)\mbox{ such that }x\neq 0\}\cup\{(0,y,z,0,u_1,0)\}.$$ It is obvious that the two presentations determine
the same set, with the second one possibly giving a better idea of the structure of the set, and the first one allowing for the uniform description of the exterior product, in the following way:
\begin{flushleft}
$(x,y,z',w',u_1',u_2')\wedge(x,y,z'',w'',u_1'',u_2'')=$
\end{flushleft}
\begin{flushright}
$=(x,y,u_1''z'+u_1'z'',u_1''w'+u_1'w'',u_1'u_1'',u_1''u_2'+u_1'u_2''+z'w''-z''w')$.
\end{flushright}

\paragraph{The Clifford actions} It remains to describe the corresponding Clifford actions. As is standard, in the case of $\cl(V_1,g_1)$, it suffices to consider the action of elements of
form $(x,0,z,0,0,0)$ and $(x,0,0,w,0,0)$ on elements of form $(x,0,z,0,0,0)$, $(x,0,0,w,0,0)$, $(x,0,0,0,u_1,0)$, and $(x,0,0,0,0,u_2)$.

For these elements the multiplication is determined as follows
$$\left\{\begin{array}{l}
c_1(x,0,z,0,0,0)(x,0,z',0,0,0)=(x,0,0,0,-f_1(x)z^2,0)\\
c_1(x,0,z,0,0,0)(x,0,0,w,0,0)=(x,0,0,0,0,zw)\\
c_1(x,0,z,0,0,0)(x,0,0,0,u_1,0)=(x,0,u_1z,0,0,0)\\
c_1(x,0,z,0,0,0)(x,0,0,0,0,u_2)=(x,0,0,-u_2f_1(x)z,0,0)\\
c_1(x,0,0,w,0,0)(x,0,z,0,0,0)=(x,0,0,0,0,-zw)\\
c_1(x,0,0,w,0,0)(x,0,0,w',0,0)=(x,0,0,0,-f_1(x)ww',0)\\
c_1(x,0,0,w,0,0)(x,0,0,0,u_1,0)=(x,0,0,u_1w,0,0)\\
c_1(x,0,0,w,0,0)(x,0,0,0,0,u_2)=(x,0,0,0,0,0).
\end{array}\right.$$ In the case of $\cl(V_2,g_2)$, it suffices to consider the action of $(0,y,z,0,0,0)$ on elements of form $(0,y,z,0,0,0)$ and $(0,y,0,0,u_1,0)$, and we have
$$\left\{\begin{array}{l} c_2(0,y,z,0,0,0)(0,y,z',0,0,0)=(0,y,0,0,-f_2(y)zz',0) \\
c_2(0,y,z,0,0,0)(0,y,0,0,u_1,0)=(0,y,u_1z,0,0,0) \end{array}\right.$$
Finally, the Clifford action on both $\bigwedge_*(V_1\cup_{\tilde{f}}V_2)$ and $\bigwedge_*(V_1)\cup_{\tilde{f}^{\bigwedge_*}}\bigwedge_*(V_2)$ is obtained by concatenating the two lists;
the difference between the two pseudo-bundles is not seen on the level of defining the action, but rather in how we determine the two sets of points (as already been indicated above), underlying
the commutativity between the gluing and the exterior product.

\subsubsection{Remarks on other examples}

The two examples considered in the previous sections were necessarily (for reasons of length) among the simplest possible without being entirely trivial. Many similar ones can be constructed
by taking other pairs of standard bundles; we note that the difference would be in length and not in substance. On the other hand, a substantially difference might be obtained by considering a
non-simply-connected domain of gluing, more interestingly, one with a non-trivial homotopy (in the sense of the standard topology). Presumably, this would influence the possibility of finding
compatible pseudo-metrics; we make no further comments on this, concluding the section at this point.

\section{The pseudo-bundle of diffeological $1$-forms $\Lambda^1(X_1\cup_f X_2)$}

In Section 3 we recalled the abstract definition of the pseudo-bundle $\Lambda^1(X)$ of diffeological $1$-forms on a given diffeological space $X$. We now consider how it behaves with respect to gluing. 
Namely, given two diffeological spaces $X_1$ and $X_2$, and a gluing map $f:X_1\supset Y\to X_2$ between them, there are three pseudo-bundles of $1$-forms, $\Lambda^1(X_1)$, $\Lambda^1(X_2)$, 
and $\Lambda^1(X_1\cup_f X_2)$; we wonder how the latter pseudo-bundle is related to the former two. It turns out that it is not the result of any kind of gluing between $\Lambda^1(X_1)$ and
$\Lambda^1(X_2)$, but rather a partially defined direct sum of them. To illustrate what is meant by this,  take a one-point gluing, \emph{i.e.} a wedge of $X_1$ and $X_2$ at some $x_0$. Then the fibre over 
$x_0$ of the pseudo-bundle $\Lambda^1(X_1\vee_{x_0}X_2)$ is the direct sum of $\Lambda_{x_0}^1(X_1)$ with $\Lambda_{x_0}^1(X_2)$, while elsewhere the fibre is inherited from one of them, as 
appropriate. For instance, if $X_1$ and $X_2$ are the coordinate axes in $\matR^2$, then $\Lambda_{x_0}^1(X_1\vee_{x_0}X_2)$ has fibre $\matR$ everywhere except at the origin, where it is $\matR^2$ 
(disregarding for the moment its diffeology). Note also that it is not even locally trivial, although the two initial ones are so. The section is based on the results of \cite{differential-forms-gluing}.

\subsection{The space $\Omega^1(X_1\cup_f X_2)$}

The diffeological vector space $\Omega^1(X_1\cup_f X_2)$, which is the main precursor to the pseudo-bundle $\Lambda^1(X_1\cup_f X_2)$, has a rather simple description in terms of $\Omega^1(X_1)$ 
and $\Omega^1(X_2)$. It is based on the description of the image of the pullback map
$$\pi^*:\Omega^1(X_1\cup_f X_2)\to\Omega^1(X_1\sqcup X_2)\cong\Omega^1(X_1)\times\Omega^1(X_2),$$ where $\pi:X_1\sqcup X_2\to X_1\cup_f X_2$ is the quotient projection that appears in
the definition of diffeological gluing. The diffeomorphism $\Omega^1(X_1\sqcup X_2)\cong\Omega^1(X_1)\times\Omega^1(X_2)$ is based on the following property of the disjoint union diffeology: for every 
plot $p:U\to X_1\sqcup X_2$ there is a disjoint union decomposition $U=U_1\sqcup U_2$, where each $U_i$ is either empty or a domain, and if $U_i\neq\emptyset$ then $p_i=p|_{U_i}$ is a plot of $X_1$.
The pullback map itself is almost never surjective; its image, under the assumption that $f$ is a subduction, can be determined using the fact that the space $X_1\cup_f X_2$ admits an alternative 
representation, which is in terms of gluing along a diffeomorphism. This is considered in the section immediately following.

\subsubsection{The space $\Omega^1(X_1\cup_f X_2)$ as $\Omega_f^1(X_1)\times_{comp}\Omega^1(X_2)$}

The set $\Omega_f^1(X_1)\times_{comp}\Omega^1(X_2)$ appearing in the title of this is a (generally proper) subset of $\Omega^1(X_1)\times\Omega^1(X_2)$. If the latter is given the structure of the direct 
sum $\Omega^1(X_1)\oplus\Omega^1(X_2)$, which is consistent with its identification with $\Omega^1(X_1\sqcup X_2)$, this set is also a vector subspace. Its composition is determined by the map $f$.

\paragraph{The space $\Omega_f^1(X_1)$ of $f$-invariant forms} Let $p_1,p_1':U\to X_1$ be two plots of $X_1$; we say that they are \textbf{$f$-equivalent} if for any $u\in U$ such that $p_1(u)\neq p_1'(u)$ 
we have that $p_1(u),p_1'(u)\in Y$ and $f(p_1(u))=f(p_1'(u))$. A form $\omega_1\in\Omega^1(X_1)$ is said to be \textbf{$f$-invariant} if for any two $f$-equivalent plots $p_1,p_1'$ we have 
$\omega_1(p_1)=\omega_1(p_1')$. The subset $\Omega_f^1(X_1)$ of $\Omega^1(X_1)$ that consists of all $f$-invariant forms is obviously a vector subspace of $\Omega^1(X_1)$. It is easy to observe that 
if we denote $\tilde{i}_1:X_1\hookrightarrow(X_1\sqcup X_2)\to X_1\cup_f X_2$ then the image of the pullback map $\tilde{i}_1^*$ is contained in $\Omega_f^1(X_1)$; therefore the image of the pullback 
map $\pi^*$ is contained in $\Omega_f^1(X_1)\times\Omega^1(X_2)$, which in general is a proper subset of $\Omega^1(X_1)\times\Omega^1(X_2)$. However, the image of $\pi^*$ is still smaller than 
$\Omega_f^1(X_1)\times\Omega^1(X_2)$, as we explain immediately below.

\paragraph{Compatibility of a form $\omega_1\in\Omega^1(X_1)$ with a form $\omega_2\in\Omega^1(X_2)$} It is rather easy to find that a pair
$(\omega_1,\omega_2)\in\Omega^1(X_1)\times\Omega^1(X_2)$ that belongs to the image of $\pi^*$ must satisfy the following condition: for every plot $p_1$ of the subset diffeology on $Y$
(the domain of gluing), we have
$$\omega_1(p_1)=\omega_2(f\circ p_1).$$ Two forms $\omega_1$ and $\omega_2$ that satisfy this condition are said to be \textbf{compatible} (the full term would be,
\emph{compatible with $f$}; we omit indicating the gluing map whenever it is clear from the context).

\paragraph{The image of the pullback map} Denote
$$\Omega_f^1(X_1)\times_{comp}\Omega^1(X_2)=\{(\omega_1,\omega_2)\,|\,\omega_1\in\Omega_f^1(X_1),\,\omega_2\in\Omega^1(X_2),\,\omega_1\mbox{ and }\omega_2\mbox{ are compatible}\}.$$
It is relatively easy to show (see \cite{differential-forms-gluing}) that $\pi^*$ is a diffeomorphism
$$\pi^*:\Omega^1(X_1\cup_f X_2)\to\Omega_f^1(X_1)\times_{comp}\Omega^1(X_2).$$ Its inverse is given by the following formula:
$$(\omega_1,\omega_2)\mapsto\omega_1\cup_f\omega_2\mbox{ such that }(\omega_1\cup_f\omega_2)(p)=\left\{\begin{array}{ll}
\omega_1(p_1) & \mbox{if }p\mbox{ lifts to a plot }p_1\mbox{ of }X_1\mbox{ and}\\
\omega_2(p_2) & \mbox{if }p\mbox{ lifts to a plot }p_2\mbox{ of }X_2 \end{array}\right.$$ for every plot $p$ of $X_1\cup_f X_2$ that has a connected domain of definition. Obviously, the values
of $\omega_1\cup_f\omega_2$ on plots with connected domains uniquely determine it. The form $\omega_1\cup_f\omega_2$ is well-defined, since any plot of $X_1\cup_f X_2$ has at most one lift
to $X_2$; whenever it has more than one lift to $X_1$, all such lifts are $f$-equivalent, and finally, if a given $p$ has lifts $p_1$ and $p_2$ to both $X_1$ and $X_2$, the equality
$\omega_1(p_1)=\omega_2(p_2)$ follows from the compatibility condition.

\paragraph{A criterion for compatibility of forms} Let $i:Y\hookrightarrow X_1$ and $j:f(Y)\hookrightarrow X_2$ be the natural inclusions, and let 
$$i^*:\Omega^1(X_1)\to\Omega^1(Y)\,\,\,\mbox{ and }\,\,\,j^*:\Omega^1(X_2)\to\Omega^1(f(Y))$$  be the corresponding pullback maps. It is quite easy to show the following.

\begin{lemma}\label{compatible:forms:in:omega:lem}
Two forms $\omega_1\in\Omega^1(X_1)$ and $\omega_2\in\Omega^1(X_2)$ are compatible if and only if 
$$i^*\omega_1=f^*(j^*\omega_2). $$ 
\end{lemma}

\subsubsection{The natural projections $\Omega^1(X_1\cup_f X_2)\to\Omega^1(X_1)$ and $\Omega^1(X_1\cup_f X_2)\to\Omega^1(X_2)$}

The two projections are the pullback maps associated to the compositions $X_i\hookrightarrow X_1\sqcup X_2\to X_1\cup_f X_2$ of  the natural inclusions $X_i\hookrightarrow X_1\sqcup X_2$ with
the quotient projection $\pi$. 

\paragraph{The reduced space $X_1^f$} The construction that we are about to describe allows us to consider, instead of an arbitrary gluing map, only the case when $f$ is a diffeomorphism. This is 
achieved by first replacing $X_1$ by its reduction by $f$-equivalence, that is, by the space
$$X_1^f:=X_1/\sim,\mbox{ where }y_1\sim y_2\Leftrightarrow f(y_1)=f(y_2);$$ $X_1^f$ is endowed with the quotient diffeology, as well as with the quotient projection $\pi_1^f:X_1\to X_1^f$ and the 
pushforward $f_{\sim}:\pi_1^f(Y)\to X_2$ of the map $f$. The following properties then hold:
\begin{itemize}
\item $\Omega^1(X_1^f)\cong\Omega_f^1(X_1)$ via the pullback map $(\pi_1^f)^*$;

\item $\pi_1^f$ preserves the compatibility, in the sense that if $\omega_1\in\Omega_f^1(X_1)$ is $f$-compatible with some $\omega_2\in\Omega^1(X_2)$ then $((\pi_1^f)^*)^{-1}(\omega_1)$ is
$f_{\sim}$-compatible with the same $\omega_2$;

\item if $f$ is a subduction then $f_{\sim}$ is a diffeomorphism of its domain with its image;

\item there is a diffeomorphism $X_1^f\cup_{f_{\sim}}X_2\cong X_1\cup_f X_2$, that commutes with the natural inductions.
\end{itemize}
The properties just listed allow in many cases to consider, instead of the gluing of $X_1$ to $X_2$ along an arbitrary smooth map, a gluing  of $X_1^f$ to $X_2$ along a diffeomorphism, with the obvious
advantages of the latter. 

\paragraph{The images of the two projections} Assume, by the reasoning just made, that $f$ is a diffeomorphism. Let us consider the images of the two projections 
$$\mbox{pr}_1:\Omega^1(X_1)\times_{comp}\Omega^1(X_2)\to\Omega^1(X_1)\,\,\mbox{ and }\,\,\mbox{pr}_2:\Omega^1(X_1)\times_{comp}\Omega^1(X_2)\to\Omega^1(X_2);$$ since $f$ is a diffeomorphism, 
the two cases are symmetric.

By definition of $\Omega^1(X_1)\times_{comp}\Omega^1(X_2)$, we have
$$\mbox{Im}(\mbox{pr}_1)=\{\omega_1\in\Omega^1(X_1)\,|\,\mbox{there exists }\omega_2\in\Omega^1(X_2)\mbox{ s. t. }\omega_1,\omega_2\mbox{ are compatible}\}.$$ By Lemma 
\ref{compatible:forms:in:omega:lem} this is equivalent to
$$\mbox{Im}(\mbox{pr}_1)=\{\omega_1\in\Omega^1(X_1)\,|\,i^*\omega_1\in\mbox{Im}(f^*j^*)\}.$$ We thus obtain that 
$$\mbox{Im}(\mbox{pr}_1)=(i^*)^{-1}(\mbox{Im}(f^*j^*))=(i^*)^{-1}\left(f^*j^*(\Omega^1(X_2))\right).$$ Likewise,
$$\mbox{Im}(\mbox{pr}_2)=(j^*)^{-1}\left((f^*)^{-1}i^*(\Omega^1(X_1))\right).$$

\paragraph{The surjectivity of $\mbox{pr}_1$ and $\mbox{pr}_2$} We will mostly treat the case when the two projections are surjective, the condition that can be expressed as, 
$$i^*(\Omega^1(X_1))=(f^*j^*)(\Omega^1(X_2)).$$ We will usually put it in as an assumption, noting that it is quite frequently satisfied (such as for gluings along one-point sets and usual open domains). For 
the rest, we shall avoid discussing when it is, or is not satisfied; given the breadth of what can be considered a diffeological space (pretty much anything, in relative terms), leaving it as an assumption just 
stated seems a reasonable thing to do.

\subsection{The fibres of the pseudo-bundle $\Lambda^1(X_1\cup_f X_2)$}

In this section and further on, we mostly assume that the gluing map $f$ is a diffeomorphism with its image, although some of the statements can be given without this assumption.

\subsubsection{Preliminary considerations on $\Lambda^1(X)$: compatible elements, and pullback maps}

We start our consideration of the pseudo-bundle $\Lambda^1(X_1\cup_f X_2)$ by collecting in this subsection the preliminary notions, regarding the pseudo-bundle $\Lambda^1(X)$, that are relevant to its 
behavior under gluing. Apart from stating our main viewpoint on $\Lambda^1(X)$ as a quotient pseudo-bundle, two main items that we consider are the compatibility notion and a version of the pullback map.

\paragraph{The pseudo-bundle $\Lambda^1(X)$ as a quotient pseudo-bundle} Let $X$ be any diffeological space. By the original definition, each fibre of $\Lambda^1(X)$ is a diffeological quotient of form 
$\Omega^1(X)/\Omega_x^1(X)$, where $x\in X$ is an arbitrary point and $\Omega_x^1(X)$ is the subspace of all $1$-forms on $X$ vanishing at $x$ (see Section 3). 

\paragraph{Compatibility of elements in $\Lambda^1(X_1)$ with those in $\Lambda^1(X_2)$} The compatibility notion for elements of $\Lambda^1(X_1)$ and $\Lambda^1(X_2)$ is almost immediate from 
that of compatible forms in $\Omega^1(X_1)$ and $\Omega^1(X_2)$ and is as follows (it does not require any particular assumption on $f$).

\begin{defn}\label{compatible:elements:of:lambda:defn}
Let $y\in Y$, and let $\alpha_1=\omega_1+\Omega_y^1(X_1)\in\Lambda_y^1(X_1)$ and $\alpha_2=\omega_2+\Omega_{f(y)}^1(X_2)\in\Lambda_{f(y)}^1(X_2)$. We say that $\alpha_1$ and $\alpha_2$ are 
\textbf{compatible} if for every $\omega_1'\in\alpha_1$ and for every $\omega_2'\in\alpha_2$ the forms $\omega_1'$ and $\omega_2'$ are compatible.
\end{defn}

This definition is consistent with the alternative definition of $\Lambda_y^1(X_1)\oplus_{comp}\Lambda_{f(y)}^1(X_2)$ as the image of $\Omega^1(X_1)\oplus_{comp}\Omega^1(X_2)$ in the quotient 
$$\left(\Omega^1(X_1)\oplus\Omega^1(X_2)\right)/\left(\Omega_y^1(X_1)\oplus\Omega_{f(y)}^1(X_2)\right).$$ We will use the notation $\Lambda^1(X_1)\oplus_{comp}\Lambda^1(X_2)$ to mean the 
collection $\bigcup_{y\in Y}\left(\Lambda_y^1(X_1)\oplus_{comp}\Lambda_{f(y)}^1(X_2)\right)$ of all such fibres, for all $y\in Y$. This can also be described as the image of 
$Y\times\left(\Omega^1(X_1)\oplus_{comp}\Omega^1(X_2)\right)$ in the quotient
$$\left(Y\times(\Omega^1(X_1)\oplus\Omega^1(X_2))\right)/\bigcup_{y\in Y}\left(\{y\}\times(\Omega_y^1(X_1)\oplus\Omega_{f(y)}^1(X_2))\right).$$

\paragraph{The pseudo-bundle version of a pullback map} The pullback map can also be defined for elements of pseudo-bundles of form $\Lambda^1(X)$, although the construction that we are about to 
describe is not always applicable and, when it comes to smooth maps between proper subsets, it gets somewhat cumbersome. It is however sufficient for the uses that we will make of it. 

Let first $f:X_1\to X_2$ be a diffeomorphism between two diffeological spaces. Then the map $(f^{-1},f^*):X_2\times\Omega^1(X_2)\to X_1\times\Omega^1(X_1)$, where $f^*:\Omega^1(X_2)\to\Omega^1(X_1)$ 
is the already-seen pullback map, descends to a well-defined map $f_{\Lambda}^*:\Lambda^1(X_2)\to\Lambda^1(X_1)$. In particular, if two diffeological spaces $X_1$ and $X_2$ are glued along a 
diffeomorphism $f:X_1\supseteq Y\to X_2$, then there is a pullback map 
$$f_{\Lambda}^*:\Lambda^1(f(Y))\to\Lambda^1(Y).$$ However, $\Lambda^1(Y)$ and $\Lambda^1(f(Y))$ do not, in general, embed in, respectively, $\Lambda^1(X_1)$ and $\Lambda^1(X_2)$; the best that 
we can do is the following.

Let
$$\pi_Y^{\Omega,\Lambda}:Y\times\Omega^1(Y)\to\Lambda^1(Y)\,\,\,\mbox{ and }\,\,\,\pi_{f(Y)}^{\Omega,\Lambda}:f(Y)\times\Omega^1(f(Y))\to\Lambda^1(f(Y))$$ be the defining projections of $\Lambda^1(Y)$ 
and $\Lambda^1(f(Y))$ respectively. Then it is easy to obtain the following statement.

\begin{lemma}
The following is true:
\begin{enumerate}
\item The map $(i^{-1},i^*):i(Y)\times\Omega^1(X_1)\to Y\times\Omega^1(Y)$ descends to a well-defined map
$$i_{\Lambda}^*:\Lambda^1(X_1)\supset(\pi_1^{\Lambda})^{-1}(Y)\to\Lambda^1(Y)\,\,\mbox{ such that }\,
\pi_Y^{\Omega,\Lambda}\circ(i^{-1},i^*)=i_{\Lambda}^*\circ\pi_1^{\Omega,\Lambda}|_{i(Y)\times\Omega^1(X_1)};$$
\item The map $(j^{-1},j^*):j(f(Y))\times\Omega^1(X_2)\to f(Y)\times\Omega^1(f(Y))$ descends to a well-defined map
$$j_{\Lambda}^*:\Lambda^1(X_2)\supset(\pi_2^{\Lambda})^{-1}(f(Y))\to\Lambda^1(f(Y))\,\,\mbox{ such that }\,
\pi_{f(Y)}^{\Omega,\Lambda}\circ(j^{-1},j^*)=j_{\Lambda}^*\circ\pi_2^{\Omega,\Lambda}|_{j(f(Y))\times\Omega^1(X_2)}.$$
\end{enumerate}
\end{lemma}

\paragraph{Compatibility in terms of pullback maps} We have already related the compatibility notion for forms $\omega_1\in\Omega^1(X_1)$ and $\omega_2\in\Omega^1(X_2)$ to the pullback maps $i^*$ 
and $j^*$, corresponding to the natural inclusions $i:Y\hookrightarrow X_1$ and $j:f(Y)\hookrightarrow X_2$ (Lemma \ref{compatible:forms:in:omega:lem}).  The analogous statement is also true for elements 
of the pseudo-bundles $\Lambda^1(X_1)$ and $\Lambda^1(X_2)$.

\begin{prop}\label{compatible:in:lambda:via:pullback:maps:prop}
Two elements $\alpha_1\in\Lambda^1(X_1)$ and $\alpha_2\in\Lambda^1(X_2)$ are compatible if and only if $\pi_1^{\Lambda}(\alpha_1)\in Y$, $\pi_2^{\Lambda}(\alpha_2)=f(\pi_1^{\Lambda}(\alpha_1))$, 
and 
$$i_{\Lambda}^*\alpha_1=f_{\Lambda}^*(j_{\Lambda}^*\alpha_2).$$
\end{prop}

\paragraph{The conditions $i^*(\Omega^1(X_1))=(f^*j^*)(\Omega^1(X_2))$ and $i_{\Lambda}^*((\pi_1^{\Lambda})^{-1}(Y))=(f_{\Lambda^*j_{\Lambda^*}})((\pi_2^{\Lambda})^{-1}(f(Y)))$} As we previously 
said, from this section onwards we will carry forward also the assumption that the two direct product projections $\mbox{pr}_1:\Omega^1(X_1)\times_{comp}\Omega^1(X_2)\to\Omega^1(X_1)$ and 
$\mbox{pr}_2:\Omega^1(X_1)\times_{comp}\Omega^1(X_2)\to\Omega^1(X_2)$, expressed equivalently as the equality $i^*(\Omega^1(X_1))=(f^*j^*)(\Omega^1(X_2))$. The basic meaning of this condition 
is that for every form $\omega_1\in\Omega^1(X_1)$ there is at least one form $\omega_2\in\Omega^1(X_2)$ such that $\omega_1$ and $\omega_2$ are compatible, and \emph{vice versa}. The obvious 
question then is whether this assumption implies the analogous equality for the pseudo-bundles $\Lambda^1(X_1)$ and $\Lambda^1(X_2)$, that is, is it true that 
$$i^*(\Omega^1(X_1))=(f^*j^*)(\Omega^1(X_2))\,\Rightarrow\,i_{\Lambda}^*((\pi_1^{\Lambda})^{-1}(Y))=(f_{\Lambda^*j_{\Lambda^*}})((\pi_2^{\Lambda})^{-1}(f(Y)))?$$
That this is indeed the case easily follows from the definition of the map $f_{\Lambda}^*$.

\subsubsection{The characteristic maps $\tilde{\rho}_1^{\Lambda}$ and $\tilde{\rho}_2^{\Lambda}$: definition}

As we will see in the sections that follow, the pseudo-bundle $\Lambda^1(X_1\cup_f X_2)$ does not admit any description in terms of standard constructions applied to $\Lambda^1(X_1)$ and 
$\Lambda^1(X_2)$, not even via the gluing operation. On the other hand, it is of course strongly related to them; to describe this relation, we need the two auxiliary maps $\tilde{\rho}_1^{\Lambda}$ and 
$\tilde{\rho}_2^{\Lambda}$ defined in this section.

As any pseudo-bundle of diffeological 1-forms, $\Lambda^1(X_1\cup_f X_2)$ is defined as a specific pseudo-bundle quotient of
$$(X_1\cup_f X_2)\times\left(\Omega^1(X_1)\times_{comp}\Omega^1(X_2)\right).$$ Since $X_1\cup_f X_2$ is a diffeological quotient of $X_1\sqcup X_2$,  $\Lambda^1(X_1\cup_f X_2)$ is also a 
pseudo-bundle quotient of 
\begin{flushleft}
$(X_1\sqcup X_2)\times\left(\Omega^1(X_1)\times_{comp}\Omega^1(X_2)\right)\cong$
\end{flushleft}
\begin{flushright}
$\cong\left(X_1\times\left(\Omega^1(X_1)\times_{comp}\Omega^1(X_2)\right)\right)\sqcup\left(X_2\times\left(\Omega^1(X_1)\times_{comp}\Omega^1(X_2)\right)\right)$.
\end{flushright} Let now
\begin{flushleft}
$\rho_1:X_1\times\left(\Omega^1(X_1)\times_{comp}\Omega^1(X_2)\right)\to X_1\times\Omega^1(X_1)$ and 
\end{flushleft}
\begin{flushright}
$\rho_2:X_2\times\left(\Omega^1(X_1)\times_{comp}\Omega^1(X_2)\right)\to X_2\times\Omega^1(X_2)$
\end{flushright} be the maps acting by identity on $X_1$ or $X_2$, as appropriate, and by the projection on the first, respectively, the second factor on $\Omega^1(X_1)\times_{comp}\Omega^1(X_2)$. It is 
rather easy to find that these maps preserve the vanishing of $1$-forms and therefore descend to well-defined and smooth maps
\begin{flushleft}
$\tilde{\rho}_1^{\Lambda}:\Lambda^1(X_1\cup_f X_2)\supset(\pi^{\Lambda})^{-1}(i_1(X_1)\cup i_2(f(Y)))\to\Lambda^1(X_1)$ and
\end{flushleft}
\begin{flushright}
$\tilde{\rho}_2^{\Lambda}:\Lambda^1(X_1\cup_f X_2)\supset(\pi^{\Lambda})^{-1}(i_2(X_2))\to\Lambda^1(X_2)$.
\end{flushright}

\subsubsection{The fibrewise structure of $\Lambda^1(X_1\cup_f X_2)$}

We shall now consider the fibres of $\Lambda^1(X_1\cup_f X_2)$.

\begin{thm}\label{fibres:of:lambda:thm}
Let $X_1$ and $X_2$ be two diffeological spaces, and let $f:X_1\supseteq Y\to X_2$ be a gluing diffeomorphism such that $i^*(\Omega^1(X_1))=(f^*j^*)(\Omega^1(X_2))$, where $i:Y\hookrightarrow X_1$ 
and $j:f(Y)\hookrightarrow X_2$ are the natural inclusions. Let $x\in X_1\cup_f X_2$, Then: 
\begin{enumerate}
\item If $x\in i_1(X_1\setminus Y)$ then $\Lambda_x^1(X_1\cup_f X_2)\cong\Lambda_{\tilde{x}}^1(X_1)$, where $\tilde{x}=i_1^{-1}(x)$;
\item If $x\in i_2(X_2\setminus f(Y))$ then $\Lambda_x^1(X_1\cup_f X_2)\cong\Lambda_{\tilde{x}}^1(X_2)$, where $\tilde{x}=i_2^{-1}(x)$;
\item If $x\in i_2(f(Y))$ then $\Lambda_x^1(X_1\cup_f X_2)\cong\Lambda_y^1(X_1)\times_{comp}\Lambda_{f(y)}^1(X_2)$, where $y=(f^{-1}\circ i_2^{-1})(x)$.
\end{enumerate}
\end{thm}

\begin{rem}
If $f$ is not a diffeomorphism with its image, the first two items above may hold still provided that the equality $i^*(\Omega^1(X_1))=(f^*j^*)(\Omega^1(X_2))$ is maintained; if it is not, they are probably not 
true. 
\end{rem}

\subsection{The pseudo-bundle decomposition of $\Lambda^1(X_1\cup_f X_2)$}

Throughout this section we again assume that the gluing map $f$ is a diffeomorphism and is such that $i^*(\Omega^1(X_1))=(f^*j^*)(\Omega^1(X_2))$. Then Theorem \ref{fibres:of:lambda:thm} applies, 
suggesting the following decomposition of $\Lambda^1(X_1\cup_f X_2)$.

\subsubsection{The main statement}

The following is, together with Theorem \ref{diffeology:lambda:in:terms:of:rho:thm}, the best description that we can give of $\Lambda^1(X_1\cup_f X_2)$ as a diffeological pseudo-bundle (in addition to 
some concrete observations regarding its plots, immediately after).

\begin{thm}\label{three:piece:diffeology:in:lambda:thm}
Let $X_1$, $X_2$, and the gluing diffeomorphism $f$ be such that $i^*(\Omega^1(X_1))=(f^*j^*)(\Omega^1(X_2))$. Then the following is true:
\begin{enumerate}
\item The map 
$$\tilde{\rho}_1^{\Lambda}|_{(\pi^{\Lambda})^{-1}(i_1(X_1\setminus Y))}:\Lambda^1(X_1\cup_f X_2)\supseteq(\pi^{\Lambda})^{-1}(i_1(X_1\setminus Y))\to
(\pi_1^{\Lambda})^{-1}(X_1\setminus Y)\subseteq\Lambda^1(X_1)$$ is a diffeomorphism for the subset diffeologies on $(\pi^{\Lambda})^{-1}(i_1(X_1\setminus Y))\subseteq\Lambda^1(X_1\cup_f X_2)$ and 
$(\pi_1^{\Lambda})^{-1}(X_1\setminus Y)\subseteq\Lambda^1(X_1)$;
\item The map 
$$\tilde{\rho}_2^{\Lambda}|_{(\pi^{\Lambda})^{-1}(i_2(X_2\setminus f(Y)))}:\Lambda^1(X_1\cup_f X_2)\supseteq(\pi^{\Lambda})^{-1}(i_2(X_2\setminus f(Y)))\to
(\pi_2^{\Lambda})^{-1}(X_2\setminus f(Y))\subseteq\Lambda^1(X_2)$$ is a diffeomorphism for the subset diffeologies on $(\pi^{\Lambda})^{-1}(i_2(X_2\setminus f(Y)))\subseteq\Lambda^1(X_1\cup_f X_2)$ 
and $(\pi_2^{\Lambda})^{-1}(X_2\setminus f(Y))\subseteq\Lambda^1(X_2)$;
\item The map 
$$\tilde{\rho}_1^{\Lambda}|_{(\pi^{\Lambda})^{-1}(i_2(f(Y)))}\oplus\tilde{\rho}_2^{\Lambda}|_{(\pi^{\Lambda})^{-1}(i_2(f(Y)))}:(\pi^{\Lambda})^{-1}(i_2(f(Y)))\to
(\pi_1^{\Lambda})^{-1}(i_2(f(Y)))\oplus_{comp}(\pi_2^{\Lambda})^{-1}(i_2(f(Y)))$$ is a diffeomorphism for the subset diffeology on $(\pi^{\Lambda})^{-1}(i_2(f(Y)))\subseteq\Lambda^1(X_1\cup_f X_2)$ and 
the subset diffeology on $(\pi_1^{\Lambda})^{-1}(i_2(f(Y)))\oplus_{comp}(\pi_2^{\Lambda})^{-1}(i_2(f(Y)))$ relative to the direct sum diffeology on the direct sum of $(\pi_1^{\Lambda})^{-1}(i_2(f(Y)))$ and 
$(\pi_2^{\Lambda})^{-1}(i_2(f(Y)))$ considered as pseudo-bundles over $Y\cong f(Y)$.
\end{enumerate}
\end{thm}

Theorem \ref{three:piece:diffeology:in:lambda:thm} provides a partial description of diffeology of $\Lambda^1(X_1\cup_f X_2)$. It is more direct than that given by Theorem 
\ref{diffeology:lambda:in:terms:of:rho:thm}, but it has the disadvantage of not explaining how the diffeology behaves in between the different pieces. Below we make some comments to that effect.

\subsubsection{General observations}

Every plot $p:U\to\Lambda^1(X_1\cup_f X_2)$ of $\Lambda^1(X_1\cup_f X_2)$ locally lifts to a plot $\tilde{p}$ of $(X_1\cup_f X_2)\times\Omega^1(X_1\cup_f X_2)$. Furthermore, we can assume
right away that $U$ is small enough so that $\tilde{p}$ has form $(p_{\cup},p^{\Omega})$, where $p_{\cup}$ is a plot of $X_1\cup_f X_2$ and $p^{\Omega}$ is a plot of $\Omega^1(X_1\cup_f X_2)$.
Assuming in addition that $U$ is connected, we obtain that $p_{\cup}$ lifts to either a plot $p_1$ of $X_1$ or a plot $p_2$ of $X_2$, so that it has one of the following forms:
$$p_{\cup}=\left\{\begin{array}{ll}i_1\circ p_1 & \mbox{on }p_1^{-1}(X_1\setminus Y) \\ i_2\circ f\circ p_1 & \mbox{on }p_1^{-1}(Y) \end{array}\right.\,\,\, \mbox{ or }\,\,\,p_{\cup}=i_2\circ p_2.$$
Restricting $U$ even further, if necessary, we obtain that the composition $\pi^*\circ p^{\Omega}$ (has form $\pi^*\circ p^{\Omega}=(p_1^{\Omega},p_2^{\Omega})$, where each $p_i^{\Omega}$ is a plot 
of $\Omega^1(X_i)$, for $i=1,2$.

Thus, we can summarize the discussion carried out so far by saying that, for every plot $p:U\to\Lambda^1(X_1\cup_f X_2)$ of $\Lambda^1(X_1\cup_f X_2)$ there exists a sub-domain $U'$
such that the lift $\tilde{p}$ of $p|_{U'}$ to $(X_1\cup_f X_2)\times\left(\Omega^1(X_1)\times_{comp}\Omega^1(X_2)\right)$ has either form
$$\tilde{p}=\left\{\begin{array}{ll}
(i_1\circ p_1,(p_1^{\Omega},p_2^{\Omega})) & \mbox{on }(\pi^{\Lambda}\circ p|_{U'})^{-1}(X_1\setminus Y)\\
(i_2\circ f\circ p_1,(p_1^{\Omega},p_2^{\Omega})) & \mbox{on }(\pi^{\Lambda}\circ p|_{U'})^{-1}(Y)
\end{array}\right.$$ or form
$$\tilde{p}=(i_2\circ p_2,(p_1^{\Omega},p_2^{\Omega}))\mbox{ if }\mbox{Range}(\pi^{\Lambda}\circ p|_{U'})\subseteq i_2(X_2).$$

\subsubsection{The two pushforward diffeologies on $\Omega^1(Y)$}

Let $\calD_1^{\Omega}$ be the diffeology on $\Omega^1(Y)$ that is the pushforward of the diffeology on $\Omega^1(X_1)$ by the map $i^*$. Similarly, let $\calD_2^{\Omega}$ be the diffeology
on $\Omega^1(Y)$ that is the pushforward of the diffeology on $\Omega^1(X_2)$ by the map $f^*\circ j^*$. The two diffeologies $\calD_1^{\Omega}$ and $\calD_2^{\Omega}$ are \emph{a
priori} different and, since all pullback maps are smooth, both are contained in the standard functional diffeology of $\Omega^1(Y)$.

As we have said before, if $p_1^{\Omega}:U\to\Omega^1(X_1)$ is a plot of $\Omega^1(X_1)$ and $p_2^{\Omega}:U'\to\Omega^1(X_2)$ is any plot of $\Omega^1(X_2)$, then the pair
$(p_1^{\Omega}(u),p_2^{\Omega}(u'))$ is compatible if and only if $i^*(p_1^{\Omega}(u))=(f^*j^*)(p_2^{\Omega}(u'))$. For simplicity we will consider such pairs for plots of $\Omega^1(X_i)$ having 
the same domain of definition $U$, as this can be done without loss of generality. The following two statements are then a direct consequence of the definition of a pushforward diffeology:
\begin{itemize}
\item $\calD_1^{\Omega}\subseteq\calD_2^{\Omega}$ $\Leftrightarrow$ for any plot $p_1^{\Omega}:U\to\Omega^1(X_1)$ there exists a plot $p_2^{\Omega}:U\to\Omega^1(X_2)$ such that 
$p_1^{\Omega}(u)$ and $p_2^{\Omega}(u)$ are compatible for all $u\in U$, and

\item $\calD_2^{\Omega}\subseteq\calD_1^{\Omega}$ $\Leftrightarrow$ for any plot $p_2^{\Omega}:U\to\Omega^1(X_2)$ there exists a plot $p_1^{\Omega}:U\to\Omega^1(X_1)$ such that 
$p_1^{\Omega}(u)$ and $p_2^{\Omega}(u)$ are compatible for all $u\in U$.
\end{itemize}

\subsubsection{The plots of $\Lambda^1(X_1\cup_f X_2)$ and those of $\Lambda^1(X_1)$}

As follows from the definition of gluing diffeology (and as has been said above), the local shapes of plots of $\Lambda^1(X_1\cup_f X_2)$ naturally separate into two groups: those whose compositions
with the projection $\pi^{\Lambda}:\Lambda^1(X_1\cup_f X_2)\to X_1\cup_f X_2$ lift to plots of $X_1$, and those for which such compositions lift to plots of $X_2$. In this section we consider plots from
the first group.

\paragraph{From a plot of $\Lambda^1(X_1\cup_f X_2)$ to one of $\Lambda^1(X_1)$} Let $p:U\to\Lambda^1(X_1\cup_f X_2)$ be a plot with $U$ connected and small enough so that $p$ lifts to a plot
$\tilde{p}=(p_{\cup},p^{\Omega}):U\to(X_1\cup_f X_2)\times\left(\Omega^1(X_1)\times_{comp}\Omega^1(X_2)\right)$ of $(X_1\cup_f X_2)\times\left(\Omega^1(X_1)\times_{comp}\Omega^1(X_2)\right)$, 
the component $p_{\cup}=\pi^{\Lambda}\circ p$ lifts to a plot $p_1$ of $X_1$, and $p^{\Omega}$ has form $p^{\Omega}=(p_1^{\Omega},p_2^{\Omega})$, where $p_1^{\Omega}$ and $p_2^{\Omega}$ 
are plots of $\Omega^1(X_1)$ and $\Omega^1(X_2)$ respectively. The fact that $p$ is a plot of $\Lambda^1(X_1\cup_f X_2)$ means, in particular, that $p_1^{\Omega}(u)$ and $p_2^{\Omega}(u)$ are 
compatible for any $u\in U$. Moreover, $p_{\cup}=\pi^{\Lambda}\circ p$ is a plot of $X_1\cup_f X_2$, and by the assumption that it lifts to a plot of $X_1$, it has form 
$p_{\cup}=\left\{\begin{array}{ll}i_1\circ p_1 & \mbox{on }p_1^{-1}(X_1\setminus Y) \\ i_2\circ f\circ p_1 & \mbox{on }p_1^{-1}(Y) \end{array}\right.$, whereas $p=\pi^{\Omega,\Lambda}\circ\tilde{p}$,
and by definition
$$p(u)=(p_1^{\Omega}(u),p_2^{\Omega}(u))+\pi^*\left(\Omega_{p_{\cup}(u)}^1(X_1\cup_f X_2)\right).$$ Altogether we have:
$$p(u)=\left\{\begin{array}{ll}
(p_1^{\Omega}(u),p_2^{\Omega}(u))+\left(\Omega_{p_1(u)}^1(X_1)\times_{comp}\Omega^1(X_2)\right) & \mbox{if }\pi^{\Lambda}\circ p=i_1\circ p_1 \\
(p_1^{\Omega}(u),p_2^{\Omega}(u))+\left(\Omega_{p_1(u)}^1(X_1)\times_{comp}\Omega_{f(p_1(u))}^1(X_2)\right) & \mbox{if }\pi^{\Lambda}\circ p=i_2\circ f\circ p_1.
\end{array}\right.$$
In particular, there is a plot $p^1$ of $\Lambda^1(X_1)$ naturally associated to $p$, that is given by,
$$p^1(u)=p_1^{\Omega}(u)+\Omega_{p_1(u)}^1(X_1)\,\mbox{ for all }\,u\in U.$$ It is also clear from this form that the same $p^1$ may correspond to many different $p$'s, as will also be evidenced by the 
discussion in the next paragraph. 

\paragraph{From plots of $\Lambda^1(X_1)$ to those of $\Lambda^1(X_1\cup_f X_2)$} Let $p:U\to\Lambda^1(X_1)$ be a plot of $\Lambda^1(X_1)$, and let $\tilde{p}=(p_1,p_1^{\Omega})$
be its lift to a plot of $X_1\times\Omega^1(X_1)$. For $p$ to extend to a plot of $\Lambda^1(X_1\cup_f X_2)$ it is necessary and sufficient that there exist a plot $p_2^{\Omega}:U\to\Omega^1(X_2)$ of 
$\Omega^1(X_2)$ such that $p_1^{\Omega}(u)$ and $p_2^{\Omega}(u)$ are compatible for all $u\in U$, \emph{i.e.} such that $i^*p_1^{\Omega}(u)=(f^*j^*)(p_2^{\Omega}(u))$. In other words, 
$i^*\circ p_1^{\Omega}$ must be a plot of $\calD_2^{\Omega}$. This leads to the following statement.

\begin{lemma}
Let $p:U\to\Lambda^1(X_1)$ be a plot of $\Lambda^1(X_1)$, and let $\tilde{p}_1=(p_1,p_1^{\Omega})$ be its lift to a plot of $X_1\times\Omega^1(X_1)$. If $\calD_1^{\Omega}\subseteq\calD_2^{\Omega}$
then there exists a lift
$$\tilde{p}:U\to(X_1\cup_f X_2)\times\left(\Omega^1(X_1)\times_{comp}\Omega^1(X_2)\right)$$ of a plot of $\Lambda^1(X_1\cup_f X_2)$, that has form
$$\tilde{p}=\left\{\begin{array}{l} \left(i_1\circ p_1,(p_1^{\Omega},p_2^{\Omega})\right),\\ \left(i_2\circ f\circ p_1,(p_1^{\Omega},p_2^{\Omega})\right) \end{array}\right.$$
for an appropriate plot $p_2^{\Omega}$ of $\Omega^1(X_2)$.
\end{lemma}

This means that $p=\tilde{\rho}_1^{\Lambda}\circ(\pi^{\Omega,\Lambda}\circ\tilde{p})$, that is, $p$ belongs to the pushforward by $\tilde{\rho}_1^{\Lambda}$ of the subset diffeology on 
$(\pi^{\Lambda})^{-1}(i_1(X_1\setminus Y)\cup i_2(f(Y)))\subseteq\Lambda^1(X_1\cup_f X_2)$.

\paragraph{The map $\tilde{\rho}_1^{\Lambda}$ as a subduction} We now describe under which conditions $\tilde{\rho}_1^{\Lambda}$ is a subduction.  

\begin{prop}\label{tilde-rho-1:as:a:subduction:prop}
Let $X_1$ and $X_2$ be two diffeological spaces, and let $f:X_1\supseteq Y\to X_2$ be a gluing diffeomorphism. The map 
$$\tilde{\rho}_1^{\Lambda}:\Lambda^1(X_1\cup_f X_2)\supseteq (\pi^{\Lambda})^{-1}(i_1(X_1\setminus Y)\cup i_2(f(Y)))\to\Lambda^1(X_1)$$ is a subduction onto its range if and 
only if $\calD_1^{\Omega}\subseteq\calD_2^{\Omega}$. In particular, if $i^*(\Omega^1(X_1))=(f^*j^*)(\Omega^1(X_2))$ and $\calD_1^{\Omega}=\calD_2^{\Omega}$ then $\tilde{\rho}_1^{\Lambda}$ is 
surjective and a subduction onto $\Lambda^1(X_1)$.
\end{prop}

The \emph{vice versa} of the second statement of this Proposition is also true: if $\tilde{\rho}_1^{\Lambda}$ is a subduction onto $\Lambda^1(X_1)$ then both $i^*(\Omega^1(X_1))=(f^*j^*)(\Omega^1(X_2))$ 
and $\calD_1^{\Omega}=\calD_2^{\Omega}$ are satisfied. Observe also that, since each point of $\Omega^1(X_1)$ can be (non uniquely) represented by a constant plot, elements of $i^*(\Omega^1(X_1))$ 
can be seen as forming a subset of $\calD_1^{\Omega}$, and similarly there is an inclusion $(f^*j^*)(\Omega^1(X_2))\subset\calD_2^{\Omega}$. 

\begin{rem}
The equality $\calD_1^{\Omega}=\calD_2^{\Omega}$ implies $i^*(\Omega^1(X_1))=(f^*j^*)(\Omega^1(X_2))$. Indeed, let $i^*(\omega_1)\in i^*(\Omega^1(X_1))$, where $\omega_1\in\Omega^1(X_1)$; 
choose any constant map $p_1^{\Omega}:U\to\{\omega_1\}\subset\Omega^1(X_1)$, defined on a domain $U$ in some $\matR$. This is a plot of $\Omega^1(X_1)$ since all constant maps are so. Thus, 
$i^*\circ p_1^{\Omega}\in\calD_1^{\Omega}=\calD_2^{\Omega}$. Since $\calD_2^{\Omega}$ is defined as the pushforward of the diffeology of $\Omega^1(X_2)$ by the map $f^*j^*$, there exists a plot 
$p_2^{\Omega}:U\to\Omega^1(X_2)$ of $\Omega^1(X_2)$ such that $i^*\circ p_1^{\Omega}=(f^*j^*)\circ p_2^{\Omega}$. Let $\omega_2\in\Omega^1(X_2)$ be any form in the range of $p_2^{\Omega}$.
Then $i^*(\omega_1)=(f^*j^*)(\omega_2)$. Since $\omega_1\in\Omega^1(X_1)$ is arbitrary, this means that $i^*(\Omega^1(X_1))\subseteq(f^*j^*)(\Omega^1(X_2))$. The reverse inclusion is proved in 
exactly the same way.
\end{rem}

\subsubsection{The plots of $\Lambda^1(X_1\cup_f X_2)$ and those of $\Lambda^1(X_2)$}

The consideration of this case is entirely analogous to the previous one. Apart from the fact that gluing along a diffeomorphism is by nature symmetric, this is also due to the symmetric behavior 
of $\Lambda^1(X_1\cup_f X_2)$ over the domain of gluing.

\paragraph{From a plot of $\Lambda^1(X_1\cup_f X_2)$ to a plot of $\Lambda^1(X_2)$} As in the case of the factor $X_1$, a plot $p:U\to\Lambda^1(X_1\cup_f X_2)$ whose range lies over $i_2(X_2)$, is 
described by
$$p(u)=\left\{\begin{array}{ll}
(p_1^{\Omega}(u),p_2^{\Omega}(u))+\left(\Omega^1(X_1)\times_{comp}\Omega_{p_2(u)}^1(X_2)\right), & \mbox{for }u\mbox{ s. t. }\pi^{\Lambda}(p(u))\in i_2(X_2\setminus f(Y)), \\
(p_1^{\Omega}(u),p_2^{\Omega}(u))+\left(\Omega_{f^{-1}(p_2(u))}^1(X_1)\times_{comp}\Omega_{p_2(u)}^1(X_2)\right) & \mbox{for }u\in\mbox{ s. t. }\pi^{\Lambda}(p(u))\in i_2(f(Y)),
\end{array} \right.$$ assuming that $U$ is small enough so that there is a lift of $p$ to a plot of $(X_1\cup_f X_2)\times\left(\Omega^1(X_1)\times_{comp}\Omega^1(X_2)\right)$, and that this lift has form 
$(i_2\circ p_2,(p_1^{\Omega},p_2^{\Omega}))$, where $p_2$ is a plot of $X_2$ and $(p_1^{\Omega},p_2^{\Omega})$ is a pair of plots of $\Omega^1(X_1)$ and $\Omega^1(X_2)$ respectively. These plots 
are again such that $p_1^{\Omega}(u)$ and $p_2^{\Omega}(u)$ are compatible for all $u\in U$, that is,
$$i^*(p_1^{\Omega}(u))=(f^*j^*)(p_2^{\Omega}(u))\,\,\,\mbox{ for all }u\in U.$$ In particular, the pair $(p_2,p_2^{\Omega})$ is a plot of $X_2\times\Omega^1(X_2)$, and therefore descends to a plot of 
$\Lambda^1(X_2)$.

\paragraph{The \emph{vice versa}: from a plot of $\Lambda^1(X_2)$ to a plot of $\Lambda^1(X_1\cup_f X_2)$} By exactly the same reasoning as in the previous section (\emph{i.e.}, in the case of a plot of 
$\Lambda^1(X_1)$), we obtain the following statement.

\begin{lemma}
Let $p:U\to\Lambda^1(X_2)$ be a plot of $\Lambda^1(X_2)$, and let $\tilde{p}_2=(p_2,p_2^{\Omega})$ be its lift to a plot of $X_2\times\Omega^1(X_2)$. If $f$ is such that
$\calD_1^{\Omega}=\calD_2^{\Omega}$ then there exists a lift $\tilde{p}:U\to(X_1\cup_f X_2)\times\left(\Omega^1(X_1)\times_{comp}\Omega^1(X_2)\right)$ of some plot of $\Lambda^1(X_1\cup_f X_2)$
that has form $\tilde{p}=(i_2\circ p_2,(p_1^{\Omega},p_2^{\Omega}))$, where $p_1^{\Omega}$ is a plot of $\Omega^1(X_1)$.
\end{lemma}

\paragraph{The map $\tilde{\rho}_2^{\Lambda}$ is a subduction} This obviously holds under the same conditions as Proposition \ref{tilde-rho-1:as:a:subduction:prop}, and the claim is fully analogous. 

\begin{prop}\label{tilde-rho-2:as:a:subduction:prop}
Let $X_1$ and $X_2$ be two diffeological spaces, and let $f:X_1\supseteq Y\to X_2$ be a diffeomorphism of its domain with its image. The map 
$$\tilde{\rho}_2^{\Lambda}:\Lambda^1(X_1\cup_f X_2)\supseteq(\pi^{\Lambda})^{-1}(i_2(X_2))\to\Lambda^1(X_2)$$ is a subduction onto its range if and only if 
$\calD_2^{\Omega}\subseteq\calD_1^{\Omega}$. In particular, $\tilde{\rho}_2^{\Lambda}$ is a subduction onto $\Lambda^1(X_2)$ if and only if $\calD_1^{\Omega}=\calD_2^{\Omega}$.
\end{prop}

\subsubsection{Characterizing the diffeology of $\Lambda^1(X_1\cup_f X_2)$ via the maps $\tilde{\rho}_1^{\Lambda}$ and $\tilde{\rho}_2^{\Lambda}$}

These maps allow for the following characterization of the diffeology of $\Lambda^1(X_1\cup_f X_2)$.

\begin{thm}\label{diffeology:lambda:in:terms:of:rho:thm}
Let $X_1$ and $X_2$ be two diffeological spaces, and let $f:X_1\supseteq Y\to X_2$ be a gluing diffeomorphism such that $\calD_1^{\Omega}=\calD_2^{\Omega}$. Then the diffeology of 
$\Lambda^1(X_1\cup_f X_2)$ is the coarsest one such that both $\tilde{\rho}_1^{\Lambda}$ and $\tilde{\rho}_2^{\Lambda}$ are smooth.
\end{thm}

\subsection{Endowing $\Lambda^1(X_1\cup_f X_2)$ with a pseudo-metric}

As we already mentioned in the dedicated section, many pseudo-bundles do not admit pseudo-metrics. We therefore need to consider whether requiring a pseudo-bundle of form $\Lambda^1(X)$ to carry 
one is a sensible assumption. The specific approach to this question, that we follow in this section, is to assume that $\Lambda^1(X_1)$ and $\Lambda^1(X_2)$ do admit pseudo-metrics and to introduce 
conditions allowing to obtain out of them a pseudo-metric on $\Lambda^1(X_1\cup_f X_2)$.

\subsubsection{The starting point: a straightforward construction}

The most obvious way to go about constructing a pseudo-metric on $\Lambda^1(X_1\cup_f X_2)$ is the following one. Assume that $f$ is a diffeomorphism with its image, let it, for simplicity, be such that 
$\calD_1^{\Omega}=\calD_2^{\Omega}$, and suppose that $\Lambda^1(X_1)$ and $\Lambda^1(X_2)$ admit pseudo-metrics; denote them by $g_1^{\Lambda}$ and $g_2^{\Lambda}$ respectively.

All fibres of $\Lambda^1(X_1\cup_f X_2)$ coincide with either a fibre of $\Lambda^1(X_1)$ or $\Lambda^1(X_2)$, or with a subset of their direct sum. Thus, the most obvious way to define a (prospective)
pseudo-metric $g^{\Lambda}$ on $\Lambda^1(X_1\cup_f X_2)$ is to set that:
\begin{itemize}
\item on $(\pi^{\Lambda})^{-1}(i_1(X_1\setminus Y))$, $g^{\Lambda}$ coincides with $g_1^{\Lambda}$;
\item on $(\pi^{\Lambda})^{-1}(i_2(X_2\setminus f(Y)))$, $g^{\Lambda}$ coincides with $g_2^{\Lambda}$; 
\item for any given point $x\in i_2(f(Y))$, $g^{\Lambda}(x)$ should be a bilinear form on a subspace of the direct sum $\Lambda_{f^{-1}(i_2^{-1}(x))}^1(X_1)\oplus\Lambda_{i_2^{-1}(x)}^1(X_2)$ of two vector 
spaces, each of which is already endowed with a pseudo-metric, $g_1^{\Lambda}(f^{-1}(i_2^{-1}(x)))$ and $g_2^{\Lambda}(i_2^{-1}(x))$ respectively. The definition of $g^{\Lambda}(x)$ is then a standard 
construction, carried out by requiring the two direct summands to be orthogonal and the restriction of $g^{\Lambda}(x)$ to either of them to coincide with the already existing pseudo-metric on that summand. 
\end{itemize}
This construction is not applicable always; the two pseudo-metrics should be well-behaved with respect to each other over the domain of gluing (this is still another version of compatibility, described 
immediately below). Even for a well-chosen pair of pseudo-metrics $g_1^{\Lambda}$ and $g_2^{\Lambda}$, it needs to be adjusted somewhat in order to ensure the smoothness across fibres.

\subsubsection{The compatibility for $g_1^{\Lambda}$ and $g_2^{\Lambda}$}

Two pseudo-metrics on $\Lambda^1(X_1)$ and $\Lambda^1(X_2)$ are defined to be compatible (with a given gluing of the base spaces $X_1$ and $X_2$) if they behave in the same way on pairs of 
compatible forms. The precise definition is as follows. 

\begin{defn}\label{compatible:pseudo-metrics:on:lambda:defn}
Let $X_1$ and $X_2$ be two diffeological spaces, and let $f:X_1\supseteq Y\to X_2$ be a smooth map. Let $g_1^{\Lambda}$ and $g_2^{\Lambda}$ be pseudo-metrics on $\Lambda^1(X_1)$ and 
$\Lambda^1(X_2)$ respectively. We say that $g_1^{\Lambda}$ and $g_2^{\Lambda}$ are \textbf{compatible}, for any $y\in Y$ and for any two compatible pairs $(\omega',\omega'')$ and
$(\mu',\mu'')$, where $\omega',\mu'\in\Lambda^1(X_1)$ and $\omega'',\mu''\in\Lambda^1(X_2)$, we have
$$g_1^{\Lambda}(y)(\omega',\mu')=g_2^{\Lambda}(f(y))(\omega'',\mu'').$$
\end{defn}

The above notion is stated for an arbitrary smooth gluing map, although we will only use it in the case when this map is a diffeomorphism, usually satisfying one of our extra conditions. To apply it to a 
construction of a pseudo-metric on $\Lambda^1(X_1\cup_f X_2)$, we observe the following.

\begin{prop}\label{compatible:pseudo-metrics:on:lambda:and:maps:rho:prop}
Let $X_1$ and $X_2$ be diffeological spaces, let $f$ be a gluing map, and let $g_1^{\Lambda}$ and $g_2^{\Lambda}$ be pseudo-metrics on $\Lambda^1(X_1)$ and $\Lambda^1(X_2)$. If $g_1^{\Lambda}$ 
and $g_2^{\Lambda}$ are compatible then if for all $y\in Y$ and for all $\omega,\mu\in(\pi^{\Lambda})^{-1}(i_2(f(y)))$ we have
$$g_1^{\Lambda}(y)(\tilde{\rho}_1^{\Lambda}(\omega),\tilde{\rho}_1^{\Lambda}(\mu))=g_2^{\Lambda}(f(y))(\tilde{\rho}_2^{\Lambda}(\omega),\tilde{\rho}_2^{\Lambda}(\mu)).$$
If $f$ is a diffeomorphism such that $\calD_1^{\Omega}=\calD_2^{\Omega}$ then this is an if and only if condition.
\end{prop}

\begin{proof}
Indeed, for any $\omega\in\Lambda^1(X_1\cup_f X_2)$ the forms $\tilde{\rho}_1^{\Lambda}(\omega)$ and $\tilde{\rho}_2^{\Lambda}(\omega)$ are compatible by construction. To establish the equivalence 
with the extra assumption of $\calD_1^{\Omega}=\calD_2^{\Omega}$, it suffices to observe that for any pair of compatible forms $(\omega_1,\omega_2)$, where $\omega_i\in\Lambda^1(X_i)$, there exists 
$\omega\in\Lambda^1(X_1\cup_f X_2)$ such that $\tilde{\rho}_i^{\Lambda}(\omega)=\omega_i$, which follows from Propositions \ref{tilde-rho-1:as:a:subduction:prop} and 
\ref{tilde-rho-2:as:a:subduction:prop}.
\end{proof}

\subsubsection{The definition of $g^{\Lambda}$}

Given two compatible pseudo-metrics $g_1^{\Lambda}$ and $g_2^{\Lambda}$, we can now combine them in the following manner:
$$g^{\Lambda}(x)(\cdot,\cdot)=\left\{\begin{array}{ll}
g_1^{\Lambda}(i_1^{-1}(x))(\tilde{\rho}_1^{\Lambda}(\cdot),\tilde{\rho}_1^{\Lambda}(\cdot)), & \mbox{if }x\in i_1(X_1\setminus Y),\\
\frac12\left(g_1^{\Lambda}(f^{-1}(i_2^{-1}(x)))(\tilde{\rho}_1^{\Lambda}(\cdot),\tilde{\rho}_1^{\Lambda}(\cdot))+g_2^{\Lambda}(i_2^{-1}(x))(\tilde{\rho}_2^{\Lambda}(\cdot),\tilde{\rho}_2^{\Lambda}(\cdot))\right),
& \mbox{if }x\in i_2(f(Y)), \\ 
g_2^{\Lambda}(i_2^{-1}(x))(\tilde{\rho}_2^{\Lambda}(\cdot),\tilde{\rho}_2^{\Lambda}(\cdot)), & \mbox{if }x\in i_2(X_2\setminus f(Y)).
\end{array}\right.$$ It is not hard to see that the result is indeed a pseudo-metric on $\Lambda^1(X_1\cup_f X_2)$; the coefficient $\frac12$ is added to ensure the smoothness, as 
discussed below.

Observe that the compatibility condition ensures (and it was chosen precisely for this reason) that
\begin{flushleft}
$\frac12\left(g_1^{\Lambda}(f^{-1}(i_2^{-1}(x)))(\tilde{\rho}_1^{\Lambda}(\cdot),\tilde{\rho}_1^{\Lambda}(\cdot))+
g_2^{\Lambda}(i_2^{-1}(x))(\tilde{\rho}_2^{\Lambda}(\cdot),\tilde{\rho}_2^{\Lambda}(\cdot))\right)=$
\end{flushleft}
\begin{flushright}
$=\frac12\left(g_2^{\Lambda}(i_2^{-1}(x))(\tilde{\rho}_2^{\Lambda}(\cdot),\tilde{\rho}_2^{\Lambda}(\cdot))+g_2^{\Lambda}(i_2^{-1}(x))(\tilde{\rho}_2^{\Lambda}(\cdot),\tilde{\rho}_2^{\Lambda}(\cdot))\right)=
g_2^{\Lambda}(i_2^{-1}(x))(\tilde{\rho}_2^{\Lambda}(\cdot),\tilde{\rho}_2^{\Lambda}(\cdot))$.
\end{flushright} This occurs over every point in $i_2(f(Y))$ and would have happened if the two summands were given different coefficients, as long as these coefficients sum up to $1$. We choose $\frac12$ 
because the gluing is symmetric in the present case.

We also mention why we do not define, over such points, $g^{\Lambda}$ to simply be equal to $g_2^{\Lambda}(i_2^{-1}(x))(\tilde{\rho}_2^{\Lambda}(\cdot),\tilde{\rho}_2^{\Lambda}(\cdot))$ from the start (as 
is the usual way of going about maps defined on spaces obtained by gluing). This is simply to stress the symmetricity between the two factors, that exists for $\Lambda^1(X_1\cup_f X_2)$ (such symmetricity 
usually is absent from the gluing diffeology) and, although the two-part (rather than three-part) definition would emphasize the smoothness property of $g^{\Lambda}$, it would make it harder to see why it 
has the desired rank.

\subsubsection{Evaluating $g^{\Lambda}$ on plots}\label{g-lambda:is:smooth:sect}

In general, the smoothness of a prospective pseudo-metric $g$ on some pseudo-bundle $V\to X$ is verified by considering its evaluation on an arbitrary plot of $X$ and two plots of $V$. Let us consider such 
evaluation for the prospective pseudo-metric $g^{\Lambda}$. Let $p:U\to X_1\cup_f X_2$ be a plot of $X_1\cup_f X_2$, and let $q,s:U'\to\Lambda^1(X_1\cup_f X_2)$ be some plots of 
$\Lambda^1(X_1\cup_f X_2)$ (we can, without loss of generality, assume them to be defined on the same domain, and also assume that $U$ and $U'$ are connected). Consider the evaluation map
$$(u,u')\mapsto g^{\Lambda}(p(u))(q(u'),s(u'))\in\matR,$$ defined on the set of $(u,u')$ such that $\pi^{\Lambda}(q(u'))=\pi^{\Lambda}(s(u'))=p(u)$. 

Since $U$ is connected, $p$ lifts to either a plot $p_1$ of $X_1$ or a plot $p_2$ of $X_2$. Restrict, if necessary, $U'$ so that each of $q,s$ lifts to a plot of 
$(X_1\cup_f X_2)\times\left(\Omega^1(X_1)\times_{comp}\Omega^1(X_2)\right)$, and that these lifts have form
$$\tilde{q}=(q_{\cup},(q_1^{\Omega},q_2^{\Omega})),\,\,\,\tilde{s}=(s_{\cup},(s_1^{\Omega},s_2^{\Omega})).$$ Notice that $q_{\cup},s_{\cup}$ are both plots of $X_1\cup_f X_2$, and, since by assumption they
have the same domain of definition, $q_{\cup}\equiv s_{\cup}$ on the entire $U'$. Moreover, for all $(u,u')$ in the domain of definition of the evaluation map we have $q_{\cup}(u')=s_{\cup}(u')=p(u)$. Thus, if 
$p$ lifts to a plot of $X_1$ then so do $q_{\cup}$ and $s_{\cup}$, and the same occurs in the case when $p$ lifts to a plot of $X_2$.

If $p$ lifts to a plot $p_1$ of $X_1$, we shall have
\begin{flushleft}
$g^{\Lambda}(p(u))(q(u'),s(u'))=$
\end{flushleft}
\begin{flushright}
$=\left\{\begin{array}{ll}
g_1^{\Lambda}(p_1(u))(\tilde{\rho}_1^{\Lambda}(q(u')),\tilde{\rho}_1^{\Lambda}(s(u'))) & \mbox{on }p_1^{-1}(X_1\setminus Y) \\
\frac12\left(g_1^{\Lambda}(p_1(u))(\tilde{\rho}_1^{\Lambda}(q(u')),\tilde{\rho}_1^{\Lambda}(s(u')))+g_2^{\Lambda}(f(p_1(u)))(\tilde{\rho}_2^{\Lambda}(q(u')),\tilde{\rho}_2^{\Lambda}(s(u')))\right) & 
\mbox{on }p_1^{-1}(Y).
\end{array}\right.$
\end{flushright} In the case when $p$ lifts to a plot $p_2$ of $X_2$, we obtain 
\begin{flushleft}
$g^{\Lambda}(p(u))(q(u'),s(u'))=$
\end{flushleft}
\begin{flushright}
$=\left\{\begin{array}{ll}
g_2^{\Lambda}(p_2(u))(\tilde{\rho}_2^{\Lambda}(q(u')),\tilde{\rho}_2^{\Lambda}(s(u'))) & \mbox{on }p_2^{-1}(X_2\setminus f(Y)) \\
\frac12\left(g_1^{\Lambda}(f^{-1}(p_2(u)))(\tilde{\rho}_1^{\Lambda}(q(u')),\tilde{\rho}_1^{\Lambda}(s(u')))+g_2^{\Lambda}(p_2(u))(\tilde{\rho}_2^{\Lambda}(q(u')),\tilde{\rho}_2^{\Lambda}(s(u')))\right) & 
\mbox{on }p_2^{-1}(f(Y)).
\end{array}\right.$
\end{flushright}
It follows then from Proposition \ref{compatible:pseudo-metrics:on:lambda:and:maps:rho:prop} that for $p$ that lifts to a plot of $X_1$ the evaluation function coincides with one for $g_1^{\Lambda}$, while 
if $p$ lifts to a plot of $p_2$, it coincides with one for $g_2^{\Lambda}$. All such evaluation functions are smooth, because $g_1^{\Lambda}$ and $g_2^{\Lambda}$ are pseudo-metrics to begin with, and 
therefore so is any arbitrary evaluation function for $g^{\Lambda}$.

\subsubsection{Proving that $g^{\Lambda}$ is a pseudo-metric on $\Lambda^1(X_1\cup_f X_2)$}

We can now collect everything together to justify the claim in the title. We have already explained in Section \ref{compatible:pseudo-metrics:on:lambda:and:maps:rho:prop} why $g^{\Lambda}$ is smooth, 
so it remains to consider its rank, \emph{i.e.}, to show that it is the largest possible for any given fibre. This is entirely obvious for points not in $i_2(f(Y))$, so let $x\in i_2(f(Y))$. 

The fibre $\Lambda_x^1(X_1\cup_f X_2)$ has form $\Lambda_{f^{-1}(i_2^{-1}(x))}^1(X_1)\times_{comp}\Lambda_{i_2^{-1}(x)}^1(X_2)$, which in particular is a subspace in 
$\Lambda_{f^{-1}(i_2^{-1}(x))}^1(X_1)\times\Lambda_{i_2^{-1}(x)}^1(X_2)$. The definition of $g^{\Lambda}$ obviously extends to that of a bilinear form on the latter space; furthermore, this extension is a 
multiple, with a constant and positive coefficient, to the standard direct sum bilinear form on a direct sum of vector spaces. It follows that the extension itself has the maximal rank possible, and therefore so do 
its restrictions to all vector subspaces, of which $\Lambda_{f^{-1}(i_2^{-1}(x))}^1(X_1)\times_{comp}\Lambda_{i_2^{-1}(x)}^1(X_2)$ is an instance.

\begin{thm}
Let $X_1$ and $X_2$ be two diffeological spaces such that $\Lambda^1(X_1)$ and $\Lambda^1(X_2)$ admit pseudo-metrics, and let $f:X_1\supseteq Y\to X_2$ be a gluing map such that 
$\calD_1^{\Omega}=\calD_2^{\Omega}$. Let $g_1^{\Lambda}$ and $g_2^{\Lambda}$ be pseudo-metrics on $\Lambda^1(X_1)$ and $\Lambda^1(X_2)$ respectively compatible in the sense of 
Definition \ref{compatible:pseudo-metrics:on:lambda:defn}. Then the corresponding $g^{\Lambda}$ is a pseudo-metric on $\Lambda^1(X_1\cup_f X_2)$.
\end{thm}

\subsection{The case of a conical gluing}

We now consider a specific case when $f$ is not a diffeomorphism, namely that of an arbitrary subset $Y\subseteq X_1$ being glued to a one-point space $X_2$. Let $x\in X_1\cup_f X_2$ be the image 
$i_2(X_2)$; we will refer to $x$ as the conical point. As always, the fibre at $x$ of $\Lambda^1(X_1\cup_f X_2)$ is the quotient
$$\Omega^1(X_1\cup_f X_2)/\Omega_x^1(X_1\cup_f X_2);$$ since any diffeological form assigns the zero form to any constant plot, $\Omega^1(X_2)$ is the zero space, and the compatibility condition is 
empty. Therefore $\Omega^1(X_1\cup_f X_2)=\Omega_f^1(X_1)$ as a set (the space of all $f$-invariant forms, see Section 8.1.1), and it can be checked that
$$\Omega_x^1(X_1\cup_f X_2)=\bigcap_{y\in Y}(\Omega_f^1)_y(X_1);$$ see Proposition 6.8 in \cite{differential-forms-gluing}.

We therefore conclude that the fibres of $\Lambda^1(X_1\cup_f X_2)$ outside of the conical point coincide with those of $\Lambda^1(X_1)$. The fibre at the conical point has form
$$\Lambda_x^1(X_1\cup_f X_2)=\Omega_f^1(X_1)/\bigcap_{y\in Y}(\Omega_f^1)_y(X_1).$$ It is thus distinct from a fibre of  $\Lambda^1(X_1)$, unless $Y$ is a one-point set. It furthermore admits a 
surjection on any fibre of form $\Lambda_y^1(X_1)$.

Let us consider the plots of $X_1\cup_f X_2$ and those of $\Lambda^1(X_1\cup_f X_2)$ in the vicinity of (the fibre over) the conical point $x$. Let first $p:U\to X_1\cup_f X_2$ be a plot whose range contains 
$x$. By definition of the gluing diffeology, this means that there exists a plot $p_1$ of $X_1$ such that 
$p(u)=\left\{\begin{array}{ll} i_1(p_1(u)) & \mbox{for }u\in U\setminus p^{-1}(x) \\ x & \mbox{for }u\in p^{-1}(x) \end{array}\right.$. This means that $p$ is essentially a plot of the reduced space $X_1^f$ (see 
Section 8.1.2), so we have 
$$\Omega^1(X_1\cup_f X_2)\cong\Omega^1(X_1^f)\cong\Omega_f^1(X_1).$$ The definition of $f$-invariance in the present case takes form $\omega_1(p_1)=\omega_1(p_1')$ for all $f$-equivalent plots 
$p_1$ and $p_1'$ of $X_1$, and the $f$-equivalence in this case means that $p_1$ and $p_1'$ have the same domain of definition $U$, and for all $u$ such that $p_1(u)\neq p_1'(u)$, we have that 
$p_1(u),p_1'(u)\in Y$. Since $X_1$ and $Y$ can be anything, this is quite likely a non-trivial condition. 

Let us now consider the plots of the pseudo-bundle $\Lambda^1(X_1\cup_f X_2)$ in the vicinity of the fibre over the conical point. At first approximation at least, these plots can be characterized by their 
lifts to plots $(X_1\cup_f X_2)\times\Omega^1(X_1\cup_f X_2)$ (this is the general case), and thus in our specific case it lifts to a plot of $X_1^f\times\Omega^1(X_1^f)$. This in fact is true of any plot of 
$\Lambda^1(X_1\cup_f X_2)$, whether its range intersects $Y$ or not. The final conclusion (and this is likely the shortest way to put it) that we thus can draw is that 
$$\Lambda^1(X_1\cup_f X_2)\cong\Lambda^1(X_1^f),$$ with a somewhat more explicit description being as follows: every plot is (locally) a projection of a map of form $(p_1,p_1^{\Omega})$, where $p_1$ 
can be taken to be any plot of $X_1$ and $p_1^{\Omega}$ is a plot of $\Omega^1(X_1)$ such its value at every point is an $f$-invariant form.

\section{The dual pseudo-bundle $(\Lambda^1(X))^*$}

There is not yet a fully established standard version of the tangent pseudo-bundle of a diffeological space. The most promising, and prominent, version appears at the moment to be that of the \emph{internal 
tangent bundle} (see \cite{CWtangent}); a number of other constructions have been proposed over time, such as the \emph{external tangent bundle}, see again \cite{CWtangent}, and the pseudo-bundle
$T^1(X)$ of $1$-tangent vectors, see \cite{iglesiasBook}, Chapter 6. In the present work we adopt probably a simpler version with respect to those, by using the dual pseudo-bundle $(\Lambda^1(X))^*$ of the 
pseudo-bundle $\Lambda^1(X)$ of diffeological $1$-forms. The reason for this is entirely obvious; indeed, the elements of $(\Lambda^1(X))^*$ have a natural pairing with elements of $\Lambda^1(X)$, and 
this is sufficient within the scope of the present work.

Throughout this section we will carry on the assumptions from the previous section, that:
\begin{itemize}
\item the gluing map $f$ is a diffeomorphism of its domain with its image and is such that $\calD_1^{\Omega}=\calD_2^{\Omega}$, and (this may not be always needed)
\item the pseudo-bundles $\Lambda^1(X_1)$ and $\Lambda^1(X_2)$ have only finite-dimensional fibres.
\end{itemize}
The second assumption, and an application of results based on the first one, imply in particular that also $\Lambda^1(X_1\cup_f X_2)$ has finite-dimensional fibres, and therefore, by the definition of
the dual pseudo-bundle, so does $(\Lambda^1(X_1\cup_f X_2))^*$.

\subsection{General considerations on $(\Lambda^1(X))^*$}

We first make several observations regarding the pseudo-bundle $(\Lambda^1(X))^*$ on its own; some of them are not specific to it, but rather apply to any dual pseudo-bundle $V^*$ where $V$ is a 
pseudo-bundle with finite-dimensional fibres that admits a pseudo-metric.

\subsubsection{Embedding $(\Lambda^1(X))^*$ in $X\times(\Omega^1(X))^*$}\label{dual:lambda:embeds:into:dual:omega:times:X:sect}

Let $\pi:V\to X$ be a diffeological vector pseudo-bundle. We can describe $(\Lambda^1(X))^*$ in the following way. By definition, $\Lambda^1(X)$ is a quotient pseudo-bundle, and more precisely the 
following one:
$$\Lambda^1(X):=(X\times\Omega^1(X))/\left(\bigcup_{x\in X}\{x\}\times\Omega_x^1(X)\right).$$ We can view $X\times\Omega^1(X)$ as the total space of the trivial pseudo-bundle over $X$ with fibre 
$\Omega^1(X)$. Then $\bigcup_{x\in X}\{x\}\times\Omega_x^1(X)$ is a sub-bundle of it,and $\Lambda^1(X)$ is the corresponding quotient pseudo-bundle.

Consider the quotient projection 
$$\pi^{\Omega,\Lambda}:X\times\Omega^1(X)\to\Lambda^1(X).$$ The corresponding dual map
$$(\pi^{\Omega,\Lambda})^*:(\Lambda^1(X))^*\to X\times(\Omega^1(X))^*$$ between the duals of the pseudo-bundles $\Lambda^1(X)$ and $X\times\Omega^1(X)$ is of course smooth. Furthermore,
since $\pi^{\Lambda}$is surjective, the map dual to it is injective. Moreover, it is an induction, so on $(\pi^{\Lambda})^*((\Lambda^1(X))^*)$ it has asmooth inverse. Thus, $(\pi^{\Lambda})^*$ is a
natural diffeomorphism of $(\Lambda^1(X))^*$ with a subset of $X\times(\Omega^1(X))^*$.

\begin{rem}
Since $(\pi^{\Lambda})^*:(\Lambda^1(X))^*\to X\times(\Omega^1(X))^*$ is an induction, $(\Lambda^1(X))^*$ can be identified with a sub-bundle of the trivial pseudo-bundle $X\times(\Omega^1(X))^*$.
Notice that in the diffeological context this does \emph{not} imply that $(\Lambda^1(X))^*$ is trivial, or even locally trivial, itself. Indeed, every collection of vector subspaces of fibres of a diffeological 
pseudo-bundle, one per fibre, determines a sub-bundle, so any trivial diffeological bundle contains many non-trivial sub-bundles. This is specifically the case of pseudo-bundles of form $\Lambda^1(X)$, 
as it follows from the description of $(\Lambda^1(X_1\cup_f X_2))^*$ that we develop in this section.
\end{rem}

\subsubsection{If $\Lambda^1(X)$ admits pseudo-metrics then $(\Lambda^1(X))^*\cong\Lambda^1(X)$}

The reasoning below applies to any finite-dimensional diffeological vector pseudo-bundle $\pi:V\to X$ that admits a pseudo-metric. For such a pseudo-bundle and for the chosen pseudo-metric $g$ on it, 
recall the corresponding pairing map $\Phi_g:V\to V^*$ given by
$$\Phi_g(v)(\cdot)=g(\pi(v))(v,\cdot).$$ It is a general consequence of the definition of the dual pseudo-bundle diffeology that $\Phi_g$ is a subduction. Furthermore, it can be rather easily deduced from 
the basic construction of the pseudo-bundle $\Lambda^1(X)$ that the subset diffeology on any finite-dimensional fibre is standard, so restricted to that fibre, $\Phi_g$ is a diffeomorphism. In particular, we
have the following.

\begin{prop}\label{fin-dim:lambda:implies:standard:fibres:prop}
If $\Lambda^1(X)$ has only finite-dimensional fibres then the subset diffeology on each fibre is the standard one. In particular, if $\Lambda^1(X)$ admits a pseudo-metric $g^{\Lambda}$ then 
$\Phi_{g^{\Lambda}}$ is a diffeomorphism $\Lambda^1(X)\to(\Lambda^1(X))^*$.
\end{prop}

\subsection{The compatibility notion for elements of $(\Lambda^1(X_1))^*$ and $(\Lambda^1(X_2))^*$}

As in the case of elements of $\Lambda^1(X_1)$ and $\Lambda^1(X_2)$, the compatibility notion for $(\Lambda^1(X_1))^*$ and $(\Lambda^1(X_2))^*$ is relevant only for fibres over the domain of gluing, 
and only over the points related by $f$. That is to say, for $\alpha_1^*\in(\Lambda^1(X_1))^*$ and $\alpha_2^*\in(\Lambda^1(X_2))^*$ to be compatible it is necessary that 
$\alpha_1^*\in(\Lambda^1(X_1))^*_y$ and $\alpha_2^*\in(\Lambda^1(X_2))^*_{f(y)}$ for some $y\in Y$. Furthermore, this compatibility notion explicitly refers to a choice of compatible pseudo-metrics 
on (fibres of) $(\Lambda^1(X_1))^*$ and $(\Lambda^1(X_2))^*$, in addition to its dependance on the gluing map $f$.

Let $y\in Y$, and let $g_1^{\Lambda}(y)$ and $g_2^{\Lambda}(f(y))$ be compatible pseudo-metrics on $\Lambda_y^1(X_1)$ and $\Lambda_{f(y)}^1(X_2)$ respectively (presumably coming from 
pseudo-metrics $g_1^{\Lambda}$ and $g_2^{\Lambda}$ on the entire pseudo-bundles $\Lambda^1(X_1)$ and $\Lambda^1(X_2)$, although this is not essential at the moment). Since by assumption 
$\Lambda_y^1(X_1)$ and $\Lambda_{f(y)}^1(X_2)$ are just standard spaces, $g_1^{\Lambda}(y)$ and $g_2^{\Lambda}(f(y))$ are scalar product. Let $\pi_{1,\Lambda}^{ort,y}$ be the orthogonal 
projection of $\Lambda_y^1(X_1)$ onto the orthogonal complement of $\Lambda_y^1(X_1)\cap\mbox{Ker}(i_{\Lambda}^*)$,
$$\pi_{1,\Lambda}^{ort,y}:\Lambda_y^1(X_1)\to\left(\Lambda_y^1(X_1)\cap\mbox{Ker}(i_{\Lambda}^*)\right)^{\perp}.$$ Let $\pi_{2,\Lambda}^{ort,f(y)}$ be the analogous orthogonal projection of 
$\Lambda_{f(y)}^1(X_2)$ onto the orthogonal complement of $\Lambda_{f(y)}^1(X_2)\cap\mbox{Ker}(j_{\Lambda}^*)$,
$$\pi_{2,\Lambda}^{ort,f(y)}:\Lambda_{f(y)}^1(X_2)\to\left(\Lambda_{f(y)}^1(X_2)\cap\mbox{Ker}(j_{\Lambda}^*)\right)^{\perp}.$$

\begin{defn}\label{compatible:elements:of:dual:lambda:defn}
Two elements $\alpha_1^*\in(\Lambda_y^1(X_1))^*$ and $\alpha_2^*\in(\Lambda_{f(y)}^1(X_2))^*$ are \textbf{compatible} (with respect to $f$, $g_1^{\Lambda}(y)$, and $g_2^{\Lambda}(f(y))$) if for every 
compatible pair $\beta_1\in\Lambda_y^1(X_1)$ and $\beta_2\in\Lambda_{f(y)}^1(X_2)$ we have that
$$\alpha_1^*(\pi_{1,\Lambda}^{ort,y}(\beta_1))=\alpha_2^*(\pi_{2,\Lambda}^{ort,f(y)}(\beta_2)).$$
\end{defn}

Assuming that the two respective fibres of $\Lambda^1(X_1)$ and $\Lambda^1(X_2)$, \emph{i.e.} those of form $\Lambda_y^1(X_1)$ and $\Lambda_{f(y)}^1(X_2)$ for $y\in Y$, admit compatible 
pseudo-metrics $g_1^{\Lambda}(y)$ and $g_2^{\Lambda}(f(y))$ (which is not strictly necessary for stating the above definition), there is a natural correspondence between pairs of compatible functions 
in $(\Lambda_y^1(X_1))^*$ and $(\Lambda_{f(y)}^1(X_2))^*$, and pairs of compatible forms in $\Lambda_y^1(X_1)$ and $\Lambda_{f(y)}^1(X_2)$. It is expressed by the following statement.

\begin{lemma}\label{comp:functions:correspond:comp:forms:lem}
Let $X_1$ and $X_2$ be two diffeological spaces, let $f:X_1\supseteq Y\to X_2$ be a diffeomorphism such that $i^*(\Omega^1(X_1))=(f^*j^*)(\Omega^1(X_2))$, and let $y\in Y$ be a point. Suppose 
furthermore that $\Lambda_y^1(X_1)$ and $\Lambda_{f(y)}^1(X_2)$ are finite-dimensional and admit compatible pseudo-metrics $g_1^{\Lambda}(y)$ and $g_2^{\Lambda}(f(y))$ respectively. Then for any 
$\alpha_1\in\Lambda_y^1(X_1)$ and $\alpha_2\in\Lambda_{f(y)}^1(X_2)$ the functions $\alpha_1^*(\cdot):=g_1^{\Lambda}(y)(\alpha_1,\cdot)\in(\Lambda_y^1(X_1))^*$ and 
$\alpha_2^*(\cdot):=g_2^{\Lambda}(f(y))(\alpha_2,\cdot)\in(\Lambda_{f(y)}^1(X_2))^*$ are compatible if and only if $\alpha_1$ and $\alpha_2$ are compatible as elements of $\Lambda_y^1(X_1)$ and 
$\Lambda_{f(y)}^1(X_2)$.
\end{lemma}

\begin{proof}
Let first $\alpha_1$ and $\alpha_2$ be compatible, and let us show that the corresponding $\alpha_1^*$ and $\alpha_2^*$ are compatible in the sense of Definition 
\ref{compatible:elements:of:dual:lambda:defn}. Let $\beta_1$ and $\beta_2$ be any compatible pair, and let $z_1:=\pi_{1,\Lambda}^{ort,y}(\beta_1)$ and $z_2:=\pi_{2,\Lambda}^{ort,f(y)}(\beta_2)$, so that 
in particular, $\beta_1=w_1+z_1$ and $\beta_2=w_2+z_2$ with $w_1\in\mbox{Ker}(i_{\Lambda}^*)$ and $w_2\in\mbox{Ker}(j_{\Lambda}^*)$. Thus, $z_1$ and $z_2$ are compatible as well, and the desired 
equality
$$\alpha_1^*(z_1)=g_1^{\Lambda}(y)(\alpha_1,z_1)=g_2^{\Lambda}(f(y))(\alpha_2,z_2)=\alpha_2^*(z_2)$$ follows from the compatibility of the pseudo-metrics $g_1^{\Lambda}(y)$ and $g_2^{\Lambda}(f(y))$. 

\emph{Vice versa}, let $\alpha_1^*\in(\Lambda_y^1(X_1))^*$ and $\alpha_2^*\in(\Lambda_{f(y)}^1(X_2))^*$ be compatible, and let $\alpha_1\in\Lambda_y^1(X_1)$ and $\alpha_2\in\Lambda_{f(y)}^1(X_2)$ 
be such that $\alpha_1^*(\cdot)=g_1^{\Lambda}(y)(\alpha_1,\cdot)$ and $\alpha_2^*(\cdot)=g_2^{\Lambda}(f(y))(\alpha_2,\cdot)$; we need to show that $\alpha_1$ and $\alpha_2$ are compatible, which 
is equivalent to $i_{\Lambda}^*\alpha_1=f_{\Lambda}^*(j_{\Lambda}^*\alpha_2)\Leftrightarrow i_{\Lambda}^*z_1=f_{\Lambda}^*(j_{\Lambda}^*z_2)$, where $z_1=\pi_{1,\Lambda}^{ort,y}(\alpha_1)$ and 
$z_2=\pi_{2,\Lambda}^{ort,f(y)}(\alpha_2)$.

Observe that the restriction of $i_{\Lambda}^*$ to $\left(\Lambda_y^1(X_1)\cap\mbox{Ker}(i_{\Lambda}^*)\right)^{\perp}$ is obviously invertible; for brevity we denote this inverse simply by 
$(i_{\Lambda}^*)^{-1}$ (since $\Lambda_y^1(X_1)$ has standard diffeology, this inverse is smooth). Let now $u_1,\ldots,u_k$ be an orthonormal basis of 
$\left(\Lambda_{f(y)}^1(X_2)\cap\mbox{Ker}(j_{\Lambda}^*)\right)^{\perp}$. We claim that its image under $(i_{\Lambda}^*)^{-1}\circ f_{\Lambda}^*\circ j_{\Lambda}^*$ is an orthonormal basis of 
$\left(\Lambda_y^1(X_1)\cap\mbox{Ker}(i_{\Lambda}^*)\right)^{\perp}$. To show this, it suffices to observe that $(i_{\Lambda}^*)^{-1}\circ f_{\Lambda}^*\circ j_{\Lambda}^*$ is an isomorphism , which follows 
from the condition $i^*(\Omega^1(X_1))=(f^*j^*)(\Omega^1(X_2))$ (that implies the corresponding equality $i_{\Lambda}^*(\Lambda_y^1(X_1))=(f_{\Lambda}^*j_{\Lambda}^*)(\Lambda_{f(y)}^1(X_2))$), and 
also an isometry, which is implied by the compatibility of $g_1^{\Lambda}(y)$ and $g_2^{\Lambda}(f(y))$. 

Let now $v_i:=((i_{\Lambda}^*)^{-1}\circ f_{\Lambda}^*\circ j_{\Lambda}^*)(u_i)$ for $i=1,\ldots,k$; write $z_1=\sum_{i=1}^ka_iv_i$ and $z_2=\sum_{i=1}^kb_iu_i$. It suffices to show that $a_i=b_i$ for all 
$i=1,\ldots,k$. Indeed, $a_i=\alpha_1^*(v_i)=\alpha_2^*(u_i)=b_i$; the middle equality is true because $v_i$ and $u_i$ are compatible by construction and are their own orthogonal projections on 
$\left(\Lambda_y^1(X_1)\cap\mbox{Ker}(i_{\Lambda}^*)\right)^{\perp}$ and $\left(\Lambda_{f(y)}^1(X_2)\cap\mbox{Ker}(j_{\Lambda}^*)\right)^{\perp}$ respectively, so the equality follows from compatibility 
of $\alpha_1^*$ and $\alpha_2^*$. 
\end{proof}

Recall, as has already been observed, that whenever $\Lambda_y^1(X_1)$ has finite dimension, for any pseudo-metric $g_1^{\Lambda}(y)$ on it and for any $\alpha_1^*\in(\Lambda_y^1(X_1))^*$ there 
exists $\alpha_1\in\Lambda_y^1(X_1)$ such that $\alpha_1^*(\cdot)=g_1^{\Lambda}(y)(\alpha_1,\cdot)$ (that is, $\alpha_1^*=\Phi_{g_1^{\Lambda}(y)}(\alpha_1)$), and the analogous statement is, of course, 
true for $\Lambda_{f(y)}^1(X_2)$. We therefore conclude the following.

\begin{cor}
Let the gluing diffeomorphism $f$ be such that $i^*(\Omega^1(X_1))=(f^*j^*)(\Omega^1(X_2))$. Then the set of all $\alpha_1^*\in(\Lambda_y^1(X_1))^*$ such that there exists at least one 
$\alpha_2^*\in(\Lambda_{f(y)}^1(X_2))^*$ compatible with $\alpha_1^*$ coincides with $(\Lambda_y^1(X_1))^*$. Similarly, for every $\alpha_2^*\in(\Lambda_{f(y)}^1(X_2))^*$ there exists 
$\alpha_1^*\in(\Lambda_y^1(X_1))^*$ compatible with $\alpha_2^*$.
\end{cor}

\subsection{The fibres of $(\Lambda^1(X_1\cup_f X_2))^*$}

In this section we state and prove the following result.

\begin{thm}\label{dual:lambda:fibres:thm}
Let $X_1$ and $X_2$ be two diffeological spaces such that both $\Lambda^1(X_1)$ and $\Lambda^1(X_2)$ admit pseudo-metrics, and let $f:X_1\supseteq Y\to X_2$ be a diffeomorphism such that 
$i^*(\Omega^1(X_1))=(f^*j^*)(\Omega^1(X_2))$. Assume that there exists a choice of pseudo-metrics on $\Lambda^1(X_1)$ and $\Lambda^1(X_2)$ that are compatible with respect to $f$. Then the fibre of 
$(\Lambda^1(X_1\cup_f X_2))^*$ at any arbitrary $x\in X_1\cup_f X_2$ has the following form:
$$(\Lambda^1(X_1\cup_f X_2))^*\cong\left\{\begin{array}{cl} 
(\Lambda^1(X_1))^*_{i_1^{-1}(x)} & \mbox{if }x\in i_1(X_1\setminus Y), \\
(\Lambda^1(X_1))^*_{f^{-1}(i_2^{-1}(x))}\oplus_{comp}(\Lambda^1(X_2))^*_{i_2^{-1}(x)} & \mbox{if }x\in i_2(f(Y)), \\
(\Lambda^1(X_2))^*_{i_2^{-1}(x)} & \mbox{if }x\in i_2(X_2\setminus f(Y)),
\end{array}\right.$$ where $(\Lambda^1(X_1))^*_{f^{-1}(i_2^{-1}(x))}\oplus_{comp}(\Lambda^1(X_2))^*_{i_2^{-1}(x)}\subseteq(\Lambda^1(X_1))^*_{f^{-1}(i_2^{-1}(x))}\oplus(\Lambda^1(X_2))^*_{i_2^{-1}(x)}$ 
is the subset of all compatible pairs in $(\Lambda^1(X_1))^*_{f^{-1}(i_2^{-1}(x))}\oplus(\Lambda^1(X_2))^*_{i_2^{-1}(x)}$.
\end{thm}

\begin{proof}
The claim is entirely obvious for fibres over points outside of the domain of gluing, \emph{i.e.}, points in $i_1(X_1\setminus Y)$ or in $i_2(X_2\setminus f(Y))$; in fact, it is a direct consequence of the definition 
of a dual pseudo-bundle and of Theorem 1.6. On the contrary, it is not as obvious for fibres over points of form $i_2(f(y))$, \emph{i.e.}, those in the domain of gluing.

Let $y\in Y$ be a point. By Theorem 1.6, we have that $\Lambda_{i_2(f(y))}^1(X_1\cup_f X_2)$ can be identified with $\Lambda_y^1(X_1)\oplus_{comp}\Lambda_{f(y)}^1(X_2)$. We need to prove that there is 
an (automatically smooth, since all diffeologies are standard) isomorphism
$$\left(\Lambda_y^1(X_1)\oplus_{comp}\Lambda_{f(y)}^1(X_2)\right)^*\cong(\Lambda_y^1(X_1))^*\oplus_{comp}(\Lambda_{f(y)}^1(X_2))^*.$$ Defining such a prospective isomorphism is straightforward. 
Let $\alpha_1^*+\alpha_2^*\in(\Lambda_y^1(X_1))^*\oplus_{comp}(\Lambda_{f(y)}^1(X_2))^*$; define $\alpha^*\in\left(\Lambda_y^1(X_1)\oplus_{comp}\Lambda_{f(y)}^1(X_2)\right)^*$ by setting 
$\alpha^*(\alpha_1+\alpha_2):=\alpha_1^*(\alpha_1)+\alpha_2^*(\alpha_2)$. In order to show that this is an isomorphism, we construct its inverse. 

To describe the inverse, we essentially need to define a way to split a given function 
$$\alpha^*\in\left(\Lambda_y^1(X_1)\oplus_{comp}\Lambda_{f(y)}^1(X_2)\right)^*$$ into the sum $\alpha^*=\alpha_1^*+\alpha_2^*$ of two compatible functions $\alpha_1^*\in(\Lambda_y^1(X_1))^*$ and 
$\alpha_2^*\in(\Lambda_{f(y)}^1(X_2))^*$. We do this by first extending $\alpha^*$ to a function $\tilde{\alpha}^*$ on the entire direct sum $\Lambda_y^1(X_1)\oplus\Lambda_{f(y)}^1(X_2)$, splitting the 
function thus obtained into the sum of \emph{some} functions on $\Lambda_y^1(X_1)$ and $\Lambda_{f(y)}^1(X_2)$ (the obvious standard step), and then showing that the two resulting functions are 
compatible in the sense of Definition \ref{compatible:elements:of:dual:lambda:defn}. 

In order to extend $\alpha^*$ to a linear function on $\Lambda_y^1(X_1)\oplus\Lambda_{f(y)}^1(X_2)$, let us write, as in the proof of Lemma \ref{comp:functions:correspond:comp:forms:lem}, 
$$\Lambda_y^1(X_1)=\left(\Lambda_y^1(X_1)\cap\mbox{Ker}(i_{\Lambda}^*)\right)^{\perp}\oplus\left(\Lambda_y^1(X_1)\cap\mbox{Ker}(i_{\Lambda}^*)\right),$$ 
$$\Lambda_{f(y)}^1(X_2)=\left(\Lambda_{f(y)}^1(X_2)\cap\mbox{Ker}(j_{\Lambda}^*)\right)^{\perp}\oplus\left(\Lambda_{f(y)}^1(X_2)\cap\mbox{Ker}(j_{\Lambda}^*)\right),$$ where the orthogonal complements 
are with respect to the pseudo-metrics $g_1^{\Lambda}(y)$ and $g_2^{\Lambda}(f(y))$. Thus write their direct sum as
$$\left(\Lambda_y^1(X_1)\cap\mbox{Ker}(i_{\Lambda}^*)\right)^{\perp}\oplus\left(\Lambda_y^1(X_1)\cap\mbox{Ker}(i_{\Lambda}^*)\right)\oplus
\left(\Lambda_{f(y)}^1(X_2)\cap\mbox{Ker}(j_{\Lambda}^*)\right)\oplus \left(\Lambda_{f(y)}^1(X_2)\cap\mbox{Ker}(j_{\Lambda}^*)\right)^{\perp},$$ and observe that for this presentation 
$\Lambda_y^1(X_1)\oplus_{comp}\Lambda_{f(y)}^1(X_2)$ becomes the subset of all elements of form 
$$\left(((i_{\Lambda}^*)^{-1}f_{\Lambda}^*j_{\Lambda}^*)(\beta_2'),\beta_1,\beta_2,\beta_2'\right),$$ where $\alpha_1\in\Lambda_y^1(X_1)\cap\mbox{Ker}(i_{\Lambda}^*)$, 
$\alpha_2\in\Lambda_{f(y)}^1(X_2)\cap\mbox{Ker}(j_{\Lambda}^*)$, and $\beta_2'\in\left(\Lambda_{f(y)}^1(X_2)\cap\mbox{Ker}(j_{\Lambda}^*)\right)^{\perp}$ are all arbitrary. 

For this four-term decomposition of the direct sum $\Lambda_y^1(X_1)\oplus\Lambda_{f(y)}^1(X_2)$, let $p_1,p_2,p_3,p_4$ be the projections onto the respective terms. We define $\tilde{\alpha}^*$ by 
setting, for any arbitrary $\alpha\in\Lambda_y^1(X_1)\oplus\Lambda_{f(y)}^1(X_2)$, that 
$$\tilde{\alpha}^*(\alpha)=\frac12\alpha^*(((i_{\Lambda}^*)^{-1}f_{\Lambda}^*j_{\Lambda}^*)(p_4(\alpha))+p_4(\alpha))+
\frac12\alpha^*(p_1(\alpha)+((j_{\Lambda}^*)^{-1}(f_{\Lambda}^*)^{-1}i_{\Lambda}^*)(p_1(\alpha)))+\alpha^*(p_2(\alpha)+p_3(\alpha)).$$ It is clear in particular that if 
$\alpha\in\Lambda_y^1(X_1)\oplus\Lambda_{f(y)}^1(X_2)$ then $\tilde{\alpha}^*(\alpha)=\alpha^*(\alpha)$.

The natural presentation of $\tilde{\alpha}^*$ as an element of $(\Lambda_y^1(X_1))^*\oplus(\Lambda_{f(y)}^1(X_2))^*$ is by the sum $\alpha_1^*+\alpha_2^*$, where $\alpha_1^*\in(\Lambda_y^1(X_1))^*$ 
acts on an arbitrary $\alpha_1\in\Lambda_y^1(X_1)$ by 
$$\alpha_1^*(\alpha_1)=\frac12\alpha^*(p_1(\alpha_1)+((j_{\Lambda}^*)^{-1}(f_{\Lambda}^*)^{-1}i_{\Lambda}^*)(p_1(\alpha_1)))+\alpha^*(p_2(\alpha_1)),$$ and $\alpha_2^*\in(\Lambda_{f(y)}^1(X_2))^*$ 
acts on $\alpha_2\in\Lambda_{f(y)}^1(X_2)$ by 
$$\alpha_2^*(\alpha_2)=\frac12\alpha^*(((i_{\Lambda}^*)^{-1}f_{\Lambda}^*j_{\Lambda}^*)(p_4(\alpha_2))+p_4(\alpha_2))+\alpha^*(p_3(\alpha_2)).$$ It now remains to check that $\alpha_1^*$ and 
$\alpha_2^*$ are compatible. Since $\alpha_1$ and $\alpha_2$ are compatible if and only if $p_1(\alpha_1)=(i_{\Lambda}^*)^{-1}f_{\Lambda}^*j_{\Lambda}^*)(p_4(\alpha_2))$, the compatibility of 
$\alpha_1^*$ is the direct consequence of their construction. It is then obvious that the assignment $\alpha^*\mapsto\alpha_1^*+\alpha_2^*$ is the desired inverse, so we obtain the final claim.
\end{proof}

The diffeology of $(\Lambda^1(X_1\cup_f X_2))^*$ can be described via the pseudo-metric $g^{\Lambda}$ on $\Lambda^1(X_1\cup_f X_2)$ induced by $g_1^{\Lambda}$ and $g_2^{\Lambda}$, and the 
corresponding pairing map. The end description can be summarized as follows.

\begin{oss}
Let $p:U\to(\Lambda^1(X_1\cup_f X_2))^*$ be a plot of $(\Lambda^1(X_1\cup_f X_2))^*$. Then, up to replacing $U$ with its smaller sub-domain, the following is true: either there exists a plot 
$q_1:U\to\Lambda^1(X_1)$ of $\Lambda^1(X_1)$ such that
$$p(u)(\cdot)=g_1^{\Lambda}(\pi_1^{\Lambda}(q_1(u)))(q_1(u),\tilde{\rho}_1^{\Lambda}(\Phi_{g^{\Lambda}}^{-1}(\cdot)))\,\,\mbox{ for all }u\in U,$$ or else there exists a plot $q_2:U\to\Lambda^1(X_2)$ of 
$\Lambda^1(X_2)$ such that
$$p(u)(\cdot)=g_2^{\Lambda}(\pi_2^{\Lambda}(q_2(u)))(q_2(u),\tilde{\rho}_1^{\Lambda}(\Phi_{g^{\Lambda}}^{-1}(\cdot)))\,\,\mbox{ for all }u\in U.$$
\end{oss}

Finally, the following statement allows to obtain an overview of the pseudo-bundle $(\Lambda^1(X_1\cup_f X_2))^*$ as a whole.

\begin{thm}\label{three:piece:deco:of:dual:lambda:thm}
Let $f$ be such that $\calD_1^{\Omega}=\calD_2^{\Omega}$. Then there are the following diffeomorphisms:
\begin{enumerate}
\item $((\pi^{\Lambda})^*)^{-1}(i_1(X_1\setminus Y))\cong((\pi_1^{\Lambda})^*)^{-1}(X_1\setminus Y)$ via the restriction of the map $(\tilde{\rho}_1^{\Lambda})^*$ to 
$((\pi_1^{\Lambda})^*)^{-1}(X_1\setminus Y)\subset(\Lambda^1(X_1))^*$;
\item $((\pi^{\Lambda})^*)^{-1}(i_2(f(Y)))\cong((\pi_1^{\Lambda})^*)^{-1}(Y)\oplus_{comp}((\pi_2^{\Lambda})^*)^{-1}(f(Y))$ via the diffeomorphism defined by the appropriately restricted direct sum of the dual 
maps $(\tilde{\rho}_1^{\Lambda})^*$ and $(\tilde{\rho}_2^{\Lambda})^*$;
\item $((\pi^{\Lambda})^*)^{-1}(i_2(X_2\setminus f(Y)))\cong((\pi_2^{\Lambda})^*)^{-1}(X_2\setminus f(Y))$ via the restriction of the map $(\tilde{\rho}_2^{\Lambda})^*$ to 
$((\pi_2^{\Lambda})^*)^{-1}(X_2\setminus f(Y))\subset(\Lambda^1(X_2))^*$.
\end{enumerate} 
\end{thm}

\begin{proof}
This is a direct consequence of the properties of the maps $\tilde{\rho}_1^{\Lambda}$ and $\tilde{\rho}_2^{\Lambda}$.
\end{proof}

The spaces 
\begin{flushleft}
$((\pi^{\Lambda})^*)^{-1}(i_1(X_1\setminus Y))\subset(\Lambda^1(X_1\cup_f X_2))^*$, $((\pi^{\Lambda})^*)^{-1}(i_2(f(Y)))\subset(\Lambda^1(X_1\cup_f X_2))^*$, 
\end{flushleft}
\begin{flushright}
$((\pi_1^{\Lambda})^*)^{-1}(X_1\setminus Y)\subset(\Lambda^1(X_1))^*$, and $((\pi_2^{\Lambda})^*)^{-1}(X_2\setminus f(Y))\subset(\Lambda^1(X_2))^*$ 
\end{flushright}
are considered with the corresponding subset diffeologies. The space 
$$((\pi_1^{\Lambda})^*)^{-1}(Y)\oplus_{comp}((\pi_2^{\Lambda})^*)^{-1}(f(Y))\subset((\pi_1^{\Lambda})^*)^{-1}(Y)\oplus((\pi_2^{\Lambda})^*)^{-1}(f(Y))$$ is considered with the subset diffeology relative to 
the direct sum diffeology, in the sense of pseudo-bundles, on $((\pi_1^{\Lambda})^*)^{-1}(Y)\oplus((\pi_2^{\Lambda})^*)^{-1}(f(Y))$ (the latter being considered as a pseudo-bundle over $Y$, or $f(Y)$, in the 
obvious way).

\subsection{Endowing $(\Lambda^1(X_1\cup_f X_2))^*$ with an induced pseudo-metric}

We are actually interested here in pseudo-metrics on $(\Lambda^1(X_1\cup_f X_2))^*$ induced by those on $\Lambda^1(X_1\cup_f X_2)$, and among the latter, in those that come from two compatible 
pseudo-metrics $g_1^{\Lambda}$ and $g_2^{\Lambda}$ on $\Lambda^1(X_1)$ and $\Lambda^1(X_2)$. Given such choice, the assumption of compatibility provides us, on one hand, with the pseudo-metric 
$g^{\Lambda}$ on $\Lambda^1(X_1\cup_f X_2)$, and on the other hand, with the dual pseudo-metrics $(g_1^{\Lambda})^*$ and $(g_2^{\Lambda})^*$ on $(\Lambda^1(X_1))^*$ and $(\Lambda^1(X_2))^*$. 
These dual pseudo-metrics are also compatible in the sense of a notion that mimics that of compatible pseudo-metrics on $\Lambda^1(X_1)$ and $\Lambda^1(X_2)$ and in a similar manner they provide us 
with a direct construction of a certain pseudo-metric $(g_1^{\Lambda})^*+(g_2^{\Lambda})^*$ on $(\Lambda^1(X_1\cup_f X_2))^*$. The latter actually coincides with the dual $(g^{\Lambda})^*$ of the 
pseudo-metric $g^{\Lambda}$ on $\Lambda^1(X_1\cup_f X_2)$ induced by $g_1^{\Lambda}$ and $g_2^{\Lambda}$ (see Section 8).

\subsubsection{If $g_1^{\Lambda}$ and $g_2^{\Lambda}$ are compatible then so are $(g_1^{\Lambda})^*$ and $(g_2^{\Lambda})^*$}\label{duals:of:comp:p-metrics:are:comp:sect}

Let $g_1^{\Lambda}$ and $g_2^{\Lambda}$ be pseudo-metrics on $\Lambda^1(X_1)$ and $\Lambda^1(X_2)$, and let $f:X_1\supseteq Y\to X_2$ be a gluing diffeomorphism. Assume that $g_1^{\Lambda}$ 
and $g_2^{\Lambda}$ are compatible. There is a natural induced notion of compatibility for pseudo-metrics on $(\Lambda^1(X_1))^*$ and $(\Lambda^1(X_2))^*$.

\begin{defn}
Let $g_1^{\Lambda^*}$ be some pseudo-metric on $(\Lambda^1(X_1))^*$ (not necessarily coinciding with the dual pseudo-metric $(g_1^{\Lambda})^*$), and let $g_2^{\Lambda}$ be a pseudo-metric on 
$(\Lambda^1(X_2))^*$. The pseudo-metrics $g_1^{\Lambda^*}$ and $g_2^{\Lambda^*}$ are \textbf{compatible} if for every compatible pair $\alpha_1^*\in(\Lambda^1(X_1))^*$, 
$\alpha_2^*\in(\Lambda^1(X_2))^*$ we have that
$$g_1^{\Lambda^*}((\pi_1^{\Lambda})^*)(\alpha_1^*,\alpha_1^*)=g_2^{\Lambda^*}((\pi_2^{\Lambda})^*)(\alpha_2^*,\alpha_2^*).$$
\end{defn}

The above definition is stated for some arbitrary pseudo-metrics on $(\Lambda^1(X_1))^*$ and $(\Lambda^1(X_2))^*$ but we indeed are mostly interested in the case of the dual pseudo-metrics 
$(g_1^{\Lambda})^*$ and $(g_2^{\Lambda})^*$. 

\begin{lemma}\label{duals:of:comp:pseudo-metrics:are:comp:lem}
Let $X_1$ and $X_2$ be two diffeological spaces such that $(\Lambda^1(X_1))^*$ and $(\Lambda^1(X_2))^*$ have only finite-dimensional fibres, and let $f:X_1\supseteq Y\to X_2$ be a gluing map that is 
a diffeomorphism and is such that $\calD_1^{\Omega}=\calD_2^{\Omega}$. Let $g_1^{\Lambda}$ and $g_2^{\Lambda}$ be compatible pseudo-metrics on $\Lambda^1(X_1)$ and $\Lambda^1(X_2)$. Then 
the dual pseud-metrics $(g_1^{\Lambda})^*$ and $(g_2^{\Lambda})^*$ are compatible as well.
\end{lemma}

\begin{proof}
This is a direct consequence of the two compatibility notions.
\end{proof}

\subsubsection{The pseudo-metric $(g_1^{\Lambda})^*+(g_2^{\Lambda})^*$}

We have just seen in Lemma \ref{duals:of:comp:pseudo-metrics:are:comp:lem} that the duals of compatible pseudo-metrics are themselves compatible. Thus, it should be possible to define a pseudo-metric 
$(g_1^{\Lambda})^*+(g_2^{\Lambda})^*$ on $(\Lambda^1(X_1\cup_f X_2))^*$ induced by them, in the same manner that it was done for compatible pseudo-metrics on $\Lambda^1(X_1)$ and 
$\Lambda^1(X_2)$. To define this pseudo-metric, we first introduce the following notation (used also in Section 11):
$$\chi_1^*=\Phi_{g_1^{\Lambda}}\circ\tilde{\rho}_1^{\Lambda}\circ(\Phi_{g^{\Lambda}})^{-1}:(\Lambda^1(X_1\cup_f X_2))^*\supseteq((\pi^{\Lambda})^*)^{-1}(i_1(X_1\setminus Y)\cup i_2(f(Y)))
\to(\Lambda^1(X_1))^*\,\,\mbox{ and }$$
$$\chi_2^*=\Phi_{g_2^{\Lambda}}\circ\tilde{\rho}_2^{\Lambda}\circ(\Phi_{g^{\Lambda}})^{-1}:(\Lambda^1(X_1\cup_f X_2))^*\supseteq((\pi^{\Lambda})^*)^{-1}(i_2(X_2))\to(\Lambda^1(X_2))^*.$$
The pseudo-metric $(g_1^{\Lambda})^*+(g_2^{\Lambda})^*$ is then defined as follows:
$$\left((g_1^{\Lambda})^*+(g_2^{\Lambda})^*\right)((\pi^{\Lambda})^*(\alpha^*))(\alpha^*,\alpha^*)=$$
$$=\left\{\begin{array}{cl} 
(g_1^{\Lambda})^*(i_1^{-1}((\pi^{\Lambda})^*(\alpha^*)))\left(\chi_1^*(\alpha^*),\chi_1^*(\alpha^*)\right) & \mbox{if }(\pi^{\Lambda})^*(\alpha^*)\in i_1(X_1\setminus Y), \\
\frac12(g_1^{\Lambda})^*(f^{-1}(i_2^{-1}((\pi^{\Lambda})^*(\alpha^*))))\left(\chi_1^*(\alpha^*),\chi_1^*(\alpha^*)\right)+ & \\
+\frac12(g_2^{\Lambda})^*(i_2^{-1}((\pi^{\Lambda})^*(\alpha^*)))\left(\chi_2^*(\alpha^*),\chi_2^*(\alpha^*)\right) & \mbox{if }(\pi^{\Lambda})^*(\alpha^*)\in i_2(f(Y)), \\
(g_2^{\Lambda})^*(i_2^{-1}((\pi^{\Lambda})^*(\alpha^*)))\left(\chi_2^*(\alpha^*),\chi_2^*(\alpha^*)\right) & \mbox{if }(\pi^{\Lambda})^*(\alpha^*)\in i_2(X_2\setminus f(Y))
\end{array}\right.$$ for all $\alpha^*\in(\Lambda^1(X_1\cup_f X_2))^*$. Observe that, since for any $\alpha^*$ such that $(\pi^{\Lambda})^*(\alpha^*)\in i_2(f(Y))$ the images $\chi_1^*(\alpha^*)$ and 
$\chi_2^*(\alpha^*)$ are compatible, over $i_2(f(Y))$ the pseudo-metric $(g_1^{\Lambda})^*+(g_2^{\Lambda})^*$ can be identified with $g_1^{\Lambda}$ (as well as with $g_2^{\Lambda}$), which 
guarantees its smoothness.

\subsubsection{Comparison of $(g_1^{\Lambda})^*+(g_2^{\Lambda})^*$ with $(g^{\Lambda})^*$}

To put everything together, observe that on $(\Lambda^1(X_1\cup_f X_2))^*$, there is the pseudo-metric $(g_1^{\Lambda})^*+(g_2^{\Lambda})^*$ obtained by combining the dual pseudo-metrics 
$(g_1^{\Lambda})^*$ and $(g_2^{\Lambda})^*$; indeed, by Lemma \ref{duals:of:comp:pseudo-metrics:are:comp:lem} we know that these pseudo-metrics are compatible. On the other hand, since 
$g_1^{\Lambda}$ and $g_2^{\Lambda}$ are compatible from the start, there is also the dual pseudo-metric $(g^{\Lambda})^*$.

\begin{thm}
Let $X_1$ and $X_2$ be two diffeological spaces such that $\Lambda^1(X_1)$ and $\Lambda^1(X_2)$ have only finite-dimensional fibres, and let $f:X_1\supseteq Y\to X_2$ be such that 
$\calD_1^{\Omega}=\calD_2^{\Omega}$. Let $g_1^{\Lambda}$ and $g_2^{\Lambda}$ be compatible pseudo-metrics  on $\Lambda^1(X_1)$ and $\Lambda^1(X_2)$ respectively. Then the pseudo-metrics 
$(g_1^{\Lambda})^*+(g_2^{\Lambda})^*$ and $(g^{\Lambda})^*$ coincide.
\end{thm}

\begin{proof}
The proof is by direct calculation of which we omit the details.
\end{proof}

\section{Connections on diffeological vector pseudo-bundles}

We now turn to considering a diffeological version for the notion of a connection. A certain (preliminary, by the author's own admittance) version of this notion appears in \cite{iglesiasBook}, Section 8.32 (it is 
the one we recalled in Section 2). The version in the above source however appears to be more in the spirit of algebraic topology; in the present section we give a separate exposition in the form which seems 
more suitable for our purposes. The covariant derivatives are defined with respect to sections of the dual pseudo-bundle $(\Lambda^1(X))^*$ (and using pairing maps can be similarly defined with respect to 
sections of $\Lambda^1(X_1)$ itself), which thus play the role of smooth vector fields. The proofs of the results stated can be found in \cite{connections-pseudobundles}.

\subsection{What is a diffeological connection}

The \emph{verbatim} extension of the usual definition of a connection on a vector bundle $E\to M$ as a linear operator $C^{\infty}(M,E)\to C^{\infty}(M,T^*M\otimes E)$, to the case of a diffeological vector 
pseudo-bundle $\pi:V\to X$, is as follows.

\begin{defn}\label{diffeological:connection:defn}
Let $\pi:V\to X$ be a finite-dimensional diffeological vector pseudo-bundle, and let $C^{\infty}(X,V)$ be the space of its smooth sections. A \textbf{connection} on this pseudo-bundle
is a smooth linear operator
$$\nabla:C^{\infty}(X,V)\to C^{\infty}(X,\Lambda^1(X)\otimes V),$$
which satisfies the Leibnitz rule: for every function $f\in C^{\infty}(X,\matR)$ and for every section $s\in C^{\infty}(X,V)$ there is the equality:
$$\nabla(fs)=df\otimes s+f\nabla s.$$
\end{defn}

Notice that on the left-hand side of the equality we have a term that, by explicit definition, is a section of $\Lambda^1(X)\otimes V$, whereas the term on the right-hand side includes the differential of a smooth 
function on $X$, that so far has been defined as an element of $\Omega^1(X)$. Thus, taken literally, the definition is not well-stated. What is being meant, however, is that there is a section of $\Lambda^1(X)$ 
that is associated to $df$ in a natural way. This is the section
$$df:x\mapsto\pi^{\Omega,\Lambda}(x,df),\,\,\,df\in C^{\infty}(X,\Lambda^1(X)),$$ where $\pi^{\Omega,\Lambda}$ is the quotient projection that defines $\Lambda^1(X)$. It should always be clear from the 
context whether $df$ stands for an element of the vector space $\Omega^1(X)$ or for a section of $\Lambda^1(X)$, so we use the same notation for both.

A usual connection on a smooth manifold is of course a diffeological connection in the above sense. We first illustrate the definition with a non-standard example; after that, we define covariant derivatives
along smooth sections of $(\Lambda^1(X))^*$, and consider the behavior of this construction under the operation of diffeological gluing.

\begin{example}
Let $X$ be the wedge at the origin of two copies of $\matR$; denote these two copies by $X_1$ and $X_2$, so that $X=X_1\vee_0 X_2$. Consider $X$ as the subset $\{xy=0,\,z=0\}\subset\matR^3$,
identifying $X_1$ with the $x$-axis $y=z=0$ and $X_2$, with the $y$-axis $x=z=0$. Let $V$ be the union of the $xz$-coordinate plane $\{y=0\}$ with the $yz$-coordinate plane $\{x=0\}$, and let
$\pi:V\to X$ be the restriction to $V$ of the standard projection of $\matR^3$ onto the $xy$-coordinate plane. The pre-image $\pi^{-1}(x,y,0)=\{(x,y,z)\,|\,z\in\matR\}$ of any point $(x,y,0)\in X$ has a
natural structure of a vector space given by the operations on the third coordinate, keeping the first two fixed (so, for instance, the sum of the vectors $(x,y,z_1)$ and $(x,y,z_2)$ is the vector
$(x,y,z_1+z_2)$).
\end{example}

Below we describe in a generic diffeological connection on this pseudo-bundle.

\subsubsection{The gluing presentation of $\pi:V\to X$ and a choice of pseudo-metric on it}

We consider $X$ as the result of gluing of $X_1$ to $X_2$ along the obvious origin-to-origin map $f$, and $V$, as a result of gluing of $V_1=\{y=0\}$ to $V_2=\{x=0\}$ along the identity map $\tilde{f}$
on the line $\{(0,0,z)\,|\,z\in\matR\}$. Let $\pi_1:V_1\to X_1$ and $\pi_2:V_2\to X_2$ be the corresponding restrictions of $\pi$. Consider the following pseudo-metrics on $V_1$ and $V_2$:
$g_1(x,0,0)=h_1(x)dz^2$ and $g_2(0,y,0)=h_2(y)dz^2$, where $h_1,h_2:\matR\to\matR$ are two usual smooth everywhere positive functions. Assume also that $h_1(0)=h_2(0)$; apart from these
conditions, $h_1$ and $h_2$ can be any. We then endow $V$ with the pseudo-metric $\tilde{g}$ obtained by the usual gluing of $g_1$ and $g_2$ (see Section 5); indeed, their compatibility in the
required sense is reflected by the condition $h_1(0)=h_2(0)$. This means that
$$\tilde{g}(x,y,0)=\left\{\begin{array}{ll} h_1(x)dz^2 & \mbox{if }y=0 \\ h_2(y)dz^2 & \mbox{if }x=0 \end{array}\right.$$ Equivalently, we can also write that
$$\tilde{g}=(h_1\cup_f h_2)dz^2,$$ where $h_1\cup_f h_2$ is the function on $X$ obtained by the usual gluing of $h_1$ and $h_2$.

\subsubsection{The standard connections $\nabla^1$ and $\nabla^2$ on the factors of gluing}

Since both $\pi_1$ and $\pi_2$ are, on their own, standard bundles $\matR^2\to\matR$, each can be endowed with a usual connection; and since both $g_1$ and $g_2$ are usual scalar products
(Riemannian metrics), there are connections compatible (in the standard sense) with them. Let us choose the specific functions $h_1$ and $h_2$ (for instance, $h_1(x)=e^x$ and $h_2(y)=e^{-y}$)
and calculate the corresponding Christoffel symbol of each. We shall have, for $g_1$, that
$$\Gamma_{11}^1(g_1)=\frac12 g^{11}\frac{dg_{11}}{dx}=\frac{h_1'(x)}{2h_1(x)};$$ let us denote by $\nabla^1$ the corresponding connection on the tangent bundle $T(\matR)$ of the real line. This is
the Levi-Civita connection on this tangent bundle corresponding to the Riemannian structure given by $g_1$. There is a natural identification of the bundle $T(\matR)\to\matR$ with $V_1\to X_1$ that
acts by $(x,s_1(x)\frac{d}{dx})\mapsto(x,0,s_1(x))$. Likewise, we identify $V_2\to X_2$ with another copy of $T(\matR)\to\matR$, via $(y,s_2(y)\frac{d}{dy})\mapsto(0,y,s_2(y))$. This allows to endow
$V_2$ with the connection $\nabla^2$ that corresponds to the Levi-Civita connection on $T(\matR)$ with the Riemannian structure given by $g_2$, with the Christoffel symbol is
$$\Gamma_{11}^1(g_2)=\frac{h_2'(x)}{2h_2(x)}.$$ The two connections can be described by
$$\nabla^1\frac{d}{dx}=\frac{h_1'(x)}{2h_1(x)}dx\otimes\frac{d}{dx},\,\,\,\nabla^2\frac{d}{dy}=\frac{h_2'(x)}{2h_2(x)}dy\otimes\frac{d}{dy}.$$ We can also put them in our coordinates $x,y,z$, so obtaining
$$\nabla^1(x,0,1)=\frac{h_1'(x)}{2h_1(x)}dx\otimes(x,0,1),\,\,\,\nabla^2(0,y,1)=\frac{h_2'(x)}{2h_2(x)}dy\otimes(0,y,1),$$ where respectively $dx$ and $dy$ are the standard sections of $\Lambda^1(X_1)$
and $\Lambda^1(X_2)$. More generally, by the Leibnitz rule we have
\begin{flushleft}
$\nabla^1s_1=\frac{h_1'(x)}{2h_1(x)}ds_1\otimes(x,0,1)+\frac{h_1'(x)}{2h_1(x)}s_1dx\otimes(x,0,1)=$
\end{flushleft}
\begin{flushright}
$=\frac{h_1'(x)}{2h_1(x)}(ds_1+s_1dx)\otimes(x,0,1)=\frac{h_1'(x)(s_1'(x)+s_1(x))}{2h_1(x)}dx\otimes(x,0,1)$,
\end{flushright}
\begin{flushleft}
$\nabla^2s_2=\frac{h_2'(y)}{2h_2(y)}ds_2\otimes(0,y,1)+\frac{h_2'(y)}{2h_2(y)}s_2dy\otimes(0,y,1)=$
\end{flushleft}
\begin{flushright}
$=\frac{h_2'(y)}{2h_2(y)}(ds_2+s_2dy)\otimes(0,y,1)=\frac{h_2'(y)(s_2'(y)+s_2(y))}{2h_2(y)}dy\otimes(0,y,1)$.
\end{flushright}

\subsubsection{Towards a connection on $V$}

The total space $V$ of $\pi$ is a union of $V_1$ and $V_2$ along the line $(0,0,z)$. The rough idea is that outside of this line the prospective connection $\nabla$ on $V$ should coincide with either
$\nabla^1$ or $\nabla^2$, as appropriate. Let us consider how it should behave on this line (more precisely, in a neighborhood of it). Since a connection is an operator
$C^{\infty}(X,V)\to C^{\infty}(X,\Lambda^1(X)\otimes V)$, consider a section $s:X\to V$. By the results of Section 6, $s$ has form $s=s_1\cup_{(f,\tilde{f})}s_2$ for some compatible sections
$s_1\in C^{\infty}(X_1,V_1)$ and $s_2\in C^{\infty}(X_2,V_2)$. These sections can be written as
$$(x,0,0)\mapsto(x,0,s_1(x))\,\,\,\mbox{ and }\,\,\,(0,y,0)\mapsto(0,y,s_2(y))\,\,\,\mbox{ such that }\,\,\,s_1(0)=s_2(0),$$ the last equality corresponding to the compatibility condition. We can thus consider
the assignment
$$s\,\,\mapsto\,\,\nabla^1s_1\in C^{\infty}(X_1,\Lambda^1(X_1)\otimes V_1),\,\,\nabla^2s_2\in C^{\infty}(X_2,\Lambda^1(X_2)\otimes V_2).$$

In order to assign to $\nabla^1s_1$ and $\nabla^2s_2$ an appropriately defined section $\nabla s$ in $C^{\infty}(X,\Lambda^1(X)\otimes V)$, we first consider their values at the origin. The value of
$\nabla^1s_1$ at the origin is
$$\nabla^1s_1|_{x=0}=\frac{h_1'(0)(s_1'(0)+s_1(0))}{2h_1(0)}dx\otimes(0,0,1),$$ while the value of $\nabla^2s_2$ is
$$\nabla^2s_2|_{y=0}=\frac{h_2'(0)(s_2'(0)+s_2(0))}{2h_2(0)}dy\otimes(0,0,1).$$ Let us consider their sum:\footnote{We are making a minor omission in describing the passage to the formula that follows,
which we will return to later on.}
$$\left(\frac{h_1'(0)(s_1'(0)+s_1(0))}{2h_1(0)}dx+\frac{h_2'(0)(s_2'(0)+s_2(0))}{2h_2(0)}dy\right)\otimes(0,0,1),$$ which we want to be an element of $\Lambda^1(X)\otimes V$, and more precisely, of its
fibre $\Lambda_0^1(X)\otimes\pi^{-1}(0)$ at $0$.

Recall that this fibre is $\left(\Lambda_0^1(X_1)\oplus_{comp}\Lambda_0^1(X_2)\right)\otimes\pi_2^{-1}(0)$. Thus, in order to view the above sum as an element of the fibre, we essentially need the
compatibility (with respect to $f$) of the forms $\frac{h_1'(0)(s_1'(0)+s_1(0))}{2h_1(0)}dx$ and $\frac{h_2'(0)(s_2'(0)+s_2(0))}{2h_2(0)}dy$. In this specific example the compatibility condition for $1$-forms
is empty, since $f$ is defined on a one-point set, and so the above sum is indeed an element of $\Lambda^1(X)\otimes V$. In a more general situation, there would be some non-trivial identity to be
satisfied involving the two forms; such a purported identity would be a condition (akin to the compatibility condition for pseudo-metrics, etc.) indicating which pairs $(\nabla^1,\nabla^2)$ of connections
on $V_1$ and $V_2$ respectively give rise to a well-defined connection on $V_1\cup_{\tilde{f}}V_2$. Likewise, in our case the gluing of $V_1$ to $V_2$ is given by the identity map on the line $x=y=0$;
more generally (this is the omission mentioned in the footnote above), we should consider the formal sum
$$\frac{h_1'(0)(s_1'(0)+s_1(0))}{2h_1(0)}dx\otimes\tilde{f}(0,0,1)+\frac{h_2'(0)(s_2'(0)+s_2(0))}{2h_2(0)}dy\otimes(0,0,1).$$

\subsubsection{Defining $\nabla$ on $V$}

We can now summarize all of the above reasoning by defining $\nabla$ to be $\nabla^1\oplus_Y\nabla^2$ (with $Y=\{0\}$), where the meaning of the latter symbol is the following one. Let
$s\in C^{\infty}(X,V)$ be a section, and let $s=s_1\cup_{(f,\tilde{f})}s_2$ for $s_1\in C^{\infty}(X_1,V_1)$ and $s_2\in C^{\infty}(X_2,V_2)$; recall that under the assumption that $f$ is a diffeomorphism,
$s_1$ and $s_2$ are uniquely defined by $s$ (in general, only $s_2$ is so). Thus, $\nabla^1s_1$ and $\nabla^2s_2$ are uniquely defined as well, and we can define $\nabla s$ as follows:
$$\nabla s(x)=\left\{\begin{array}{ll}
((\tilde{\rho}_1^{\Lambda})^{-1}\otimes j_1)\circ\nabla^1s_1(x) & \mbox{if }x\in X_1\setminus\{0\}, \\
((\tilde{\rho}_2^{\Lambda})^{-1}\otimes j_2)\circ\nabla^2s_2(x) & \mbox{if } x\in X_2\setminus\{0\}, \\
(i_1^{\Lambda}\otimes(j_2\circ\tilde{f}))\circ\nabla^1s_1(x)+(i_2^{\Lambda}\otimes j_2)\circ\nabla^2s_2(x) & \mbox{if }x=0,
\end{array} \right.$$ where $i_1^{\Lambda}:\Lambda_0^1(X_1)\to\Lambda_0^1(X_1)\oplus\Lambda_0^1(X_2)$ and $i_2^{\Lambda}:\Lambda_0^1(X_2)\to\Lambda_0^1(X_1)\oplus\Lambda_0^1(X_2)$
are the standard injections of the factors of the direct sum (\emph{i.e.}, they are the trivial identifications $\Lambda_0^1(X_1)\cong\Lambda_0^1(X_1)\oplus\{0\}$ and
$\Lambda_0^1(X_2)\cong\{0\}\oplus\Lambda_0^1(X_2)$), and $j_2$ is the standard induction $V_2\hookrightarrow V_1\cup_{\tilde{f}}V_2$ (notice also that the inverses of $\tilde{\rho}_1^{\Lambda}$ and
$\tilde{\rho}_2^{\Lambda}$ are not defined in general, but they are at the points where we consider them). The term $(i_1^{\Lambda}\otimes(j_2\circ\tilde{f}))\circ\nabla^1s_1(0)$ is thus an element of
$\left(\Lambda_0^1(X_1)\oplus\{0\}\right)\otimes(\pi_1\cup_{(\tilde{f},f)}\pi_2)^{-1}(0)$, and the term $(i_2^{\Lambda}\otimes j_2)\circ\nabla^2s_2(0)$ is an element of
$\left(\{0\}\oplus\Lambda_0^1(X_2)\right)\otimes(\pi_1\cup_{(\tilde{f},f)}\pi_2)^{-1}(0)$; the sum of these terms is taken in
$\left(\Lambda_0^1(X_1)\oplus\Lambda_0^1(X_2)\right)\otimes(\pi_1\cup_{(\tilde{f},f)}\pi_2)^{-1}(0)\subset\Lambda^1(X)\otimes V$.

As we have observed, $s_1$ and $s_2$ are uniquely defined by $s$, therefore $\nabla s$ is well-defined as a map $X\to\Lambda^1(X)\otimes V$; we need to check however that it is smooth, that is, that
$\nabla$ is indeed a map $C^{\infty}(X,V)\to C^{\infty}(X,\Lambda^1(X)\otimes V)$. We will then need to check, furthermore, that $\nabla$ is a smooth map for the functional diffeologies on these two spaces
of sections.

\subsubsection{$\nabla$ is well-defined as a map $C^{\infty}(X,V)\to C^{\infty}(X,\Lambda^1(X)\otimes V)$}

As we indicated above, we need to check that, for any arbitrary $s\in C^{\infty}(X,V)$ the image $\nabla s$ is a smooth map. Let $p:U\to X$ be a plot of $X$; assume from the start that $U$ is connected, so,
as is the case of all gluing diffeologies, $p$ lifts to either a plot $p_1$ of $X_1$ or a plot $p_2$ of $X_2$. By assumption, $s\circ p$ is a plot of $V$; furthermore, if $p$ lifts to a plot $p_1$ then $s\circ p$ lifts
to a plot $q_1$ of $V_1$, and if $p$ lifts to a plot $p_2$ then $s\circ p$ lifts to a plot $q_2$ of $V_2$.

We now need to show that $(\nabla s)\circ p$ is a plot of $\Lambda^1(X)\otimes V$. We shall assume that the image of $p$ contains the origin; if it does not then $(\nabla s)\circ p$ coincides up to
appropriate smooth identifications with either $(\nabla^1s_1)\circ p_1$ or $(\nabla^2s_2)\circ p_2$, so there would be nothing to prove. It is furthermore sufficient to consider only the case when $p$ lifts to
a plot $p_1$ of $X_1$, since the case when it lifts to a plot of $X_2$ is treated identically.

If $p$ lifts to $p_1$ then we have, by direct calculation,
\begin{flushleft}
$\nabla s(p(u))=\left\{\begin{array}{ll}
(\tilde{\rho}_1^{\Lambda})^{-1}\otimes j_1)\circ\nabla^1s_1(p_1(u) & \mbox{for }u\mbox{ such that }p_1(u)\neq 0, \\
(i_1^{\Lambda}\otimes(j_2\circ\tilde{f}))\circ\nabla^1s_1(p_1(u))+(i_2^{\Lambda}\otimes j_2)\circ\nabla^2s_2(f(p_1(u))) & \mbox{for }u\mbox{ such that }p_1(u)=0 \end{array}\right.=$
\end{flushleft}
\begin{flushright}
$=\left\{\begin{array}{ll}
\frac{h_1'(p_1(u))(s_1'(p_1(u))+s_1(p_1(u)))}{2h_1(p_1(u))}dx\otimes(p_1(u),0,1) & \mbox{for }u\mbox{ such that }p_1(u)\neq 0, \\
\left(\frac{h_1'(0)(s_1'(0)+s_1(0))}{2h_1(0)}dx+\frac{h_2'(0)(s_2'(0)+s_2(0))}{2h_2(0)}dy\right)\otimes(0,0,1) & \mbox{for }u\mbox{ such that }p_1(u)=0 \end{array}\right.$
\end{flushright}
We need this to be a plot of $\Lambda^1(X)\otimes V$. As follows from the definition of the tensor product diffeology, it suffices to ensure that the projection to $\Lambda^1(X)$ is a plot of it.

More precisely, consider the following auxiliary map $\varphi_{s_1,p_1}:U\to\Lambda^1(X)$:
$$\varphi_{s_1,p_1}(u)=\left\{\begin{array}{ll}
\frac{h_1'(p_1(u))(s_1'(p_1(u))+s_1(p_1(u)))}{2h_1(p_1(u))}dx & \mbox{for }u\mbox{ such that }p_1(u)\neq 0, \\
\frac{h_1'(0)(s_1'(0)+s_1(0))}{2h_1(0)}dx+\frac{h_2'(0)(s_2'(0)+s_2(0))}{2h_2(0)}dy & \mbox{for }u\mbox{ such that }p_1(u)=0.
\end{array}\right.$$ Since $p_1$ is already a plot (of the standard $\matR$), it is sufficient to show that the following map
$$u\mapsto\left\{\begin{array}{ll}
\frac{h_1'(u)(s_1'(u)+s_1(u))}{2h_1(u)}dx & \mbox{for }u\neq 0, \\
\frac{h_1'(0)(s_1'(0)+s_1(0))}{2h_1(0)}dx+\frac{h_2'(0)(s_2'(0)+s_2(0))}{2h_2(0)}dy & \mbox{for }u\mbox{ such that }p_1(u)=0
\end{array}\right.$$ is a plot of $\Lambda^1(X_1\cup_f X_2)$ defined on the interval $(-\varepsilon,\varepsilon)$ for a sufficiently small $\varepsilon>0$. The general description of plots of
$\Lambda^1(X_1\cup_f X_2)$ is that they are given by a pair of plots $(q_1(u),q_2(u))$ (formally, plots of $\Omega^1(X_1)$ and $\Omega^1(X_2)$ respectively) whose values are compatible for all $u$.
Finding such a pair is trivial: we take $u\mapsto\frac{h_1'(u)(s_1'(u)+s_1(u))}{2s_1(u)}dx$ for $q_1$ and the constant plot $u\mapsto\frac{h_2'(0)(s_2'(0)+s_2(0))}{2h_2(0)}dy$ for $q_2$. We can therefore
conclude that $\varphi_{s_1,p_1}$ is a plot of $\Lambda^1(X)$, and therefore $(\nabla s)\circ p$ is a plot of $\Lambda^1(X)\otimes V$. It then remains to observe that the case when $p$ lifts to $p_2$ is
entirely analogous, to obtain the following.

\begin{lemma}
For every section $s\in C^{\infty}(X,V)$ the image $\nabla s$ belongs to $C^{\infty}(X,\Lambda^1(X)\otimes V)$.
\end{lemma}

Thus, $\nabla$ is well-defined. It is clear that it has the linearity property and satisfies the Leibnitz rule, since both of these conditions are fibrewise.

\subsubsection{$\nabla$ is smooth for the functional diffeologies on $C^{\infty}(X,V)$ and $C^{\infty}(X,\Lambda^1(X)\otimes V)$}

Let $q:U\to C^{\infty}(X,V)$ be a plot of $C^{\infty}(X,V)$; assume that the domain $U$ is connected. Represent, for each $u$, the section $q(u)$ as $q_1(u)\cup_{(f,\tilde{f})}q_2(u)$; it is easy to see that
each $q_i$ thus defined is a plot of $C^{\infty}(X_i,V_i)$.

We now need to show $\nabla\circ q$ is a plot of $C^{\infty}(X,\Lambda^1(X)\otimes V)$; to do so, consider a plot $p:U'\to X$ of $X$. Again, it is sufficient to assume that $U'$ is connected, so that $p$
lifts to either a plot $p_1$ of $X_1$ or a plot $p_2$ of $X_2$ (these two cases are essentially symmetric, since the gluing is along a diffeomorphism, although there is a formal difference: the fibre at the
origin is one of $V_2$). We need to consider the evaluation map for $\nabla\circ q$ and $p$, which is the map $(u,u')\mapsto(\nabla\circ q)(u)(p(u'))$. This is a map $U\times U'\to\Lambda^1(X)\otimes V$,
of which we need to show that it is a plot of $\Lambda^1(X)\otimes V$.

Suppose that $p$ lifts to a plot $p_1$; as before, we have
\begin{flushleft}
$(\nabla\circ q)(u)(p(u'))=$
\end{flushleft}
\begin{flushright}
$\left\{\begin{array}{ll}
\frac{h_1'(p_1(u'))((q_1(u))'(p_1(u'))+q_1(u)(p_1(u')))}{2h_1(p_1(u'))}dx\otimes(p_1(u'),0,1) & \mbox{for }u'\mbox{ such that }p_1(u')\neq 0, \\
\left(\frac{h_1'(0)((q_1(u))'(0)+q_1(u)(0))}{2h_1(0)}dx+\frac{h_2'(0)((q_2(u))'(0)+q_2(u)(0))}{2h_2(0)}dy\right)\otimes(0,0,1) & \mbox{for }u'\mbox{ such that }p_1(u')=0.
\end{array}\right.$
\end{flushright} This turns out to be the plot of $\Lambda^1(X)\otimes V$ for all the same reasons as before. Indeed, as before, $p_1$ is already a plot (an ordinary smooth function), and it suffices
to consider the auxiliary map into $\Lambda^1(X)$
$$(u,u'')\mapsto\left\{\begin{array}{ll}
\frac{h_1'(u'')((q_1(u))'(u'')+q_1(u)(u''))}{2h_1(u'')}dx & \mbox{for }u''\neq 0, \\
\frac{h_1'(0)((q_1(u))'(0)+q_1(u)(0))}{2h_1(0)}dx+\frac{h_2'(0)((q_2(u))'(0)+q_2(u)(0))}{2h_2(0)}dy & \mbox{for }u\mbox{ such that }p_1(u'')=0,
\end{array}\right.$$ of which we need to show that it is a plot of $\Lambda^1(X)$. Since by assumption $u\mapsto q_1(u)$ and $u\mapsto q_2(u)$ are ordinary smooth functions, we can apply
the same reasoning as in the previous section (in fact, the latter is a partial case of the former). We thus obtain the following conclusion.

\begin{prop}
The operator $\nabla$ is a diffeological connection on $V$.
\end{prop}

We thus conclude our discussion of this simple example of a diffeological connection on the result of a gluing induced by two given connections on the factors. We stress again that, since we chose
the simplest possible gluing, it does not fully reflect the general situation; indeed, in our example no role was played by the potential compatibility condition for the connections on the factors (we also
avoided specifying $\tilde{f}$ and discussing issues related to it). Later on we will consider this more general situation.

\subsection{Covariant derivatives}

The notion of the covariant derivatives along a smooth section of $(\Lambda^1(X))^*$ is the obvious one.

\begin{defn}
Let $X$ be a diffeological space, let $\pi:V\to X$ be a diffeological vector pseudo-bundle, let $\nabla:C^{\infty}(X,V)\to C^{\infty}(X,\Lambda^1(X)\otimes V)$ be a diffeological connection, and let
$t:X\to(\Lambda^1(X))^*$ be a smooth section of $(\Lambda^1(X))^*$. Let $s\in C^{\infty}(X,V)$. The \textbf{covariant derivative} of $s$ along $t$ is defined as $\nabla_s(t):=\nabla_t s$. 
\end{defn}

Written explicitly for $s$ of form $s=\sum\alpha^i\otimes v^i$, where $\alpha^i$ are some local sections of $\Lambda^1(X)$ and $v^i$ are some local sections of $V$, we have 
$(\nabla_t s)(x)=\sum t(x)(\alpha^i(x))v^i$. This local shape allows to see that $\nabla_t s$ is an element of $C^{\infty}(X,V)$. Indeed, the diffeology on $(\Lambda^1(X))^*$, as on any dual pseudo-bundle in
general, is defined in such a way that any evaluation map $x\mapsto t(x)(\alpha^i(x))$ be smooth. Moreover, it is straightforward to check that for any fixed $t$ the map $\nabla_t:C^{\infty}(X,V)\to C^{\infty}(X,V)$ 
is smooth for the functional diffeology on $C^{\infty}(X,V)$.

Later on we will also make use of covariant derivatives with respect to sections of $\Lambda^1(X)$, in the case when $\Lambda^1(X)$ is endowed with a pseudo-metric. 

\begin{defn}
Let $X$ be a diffeological space, let $\pi:V\to X$ be a diffeological vector pseudo-bundle, and let $\nabla:C^{\infty}(X,V)\to C^{\infty}(X,\Lambda^1(X)\otimes V)$ be a diffeological connection. Let $g^{\Lambda}$ 
be a pseudo-metric on $\Lambda^1(X)$, let $t\in C^{\infty}(X,\Lambda^1(X))$, and let $s\in C^{\infty}(X,V)$. Then the \textbf{covariant derivative of $s$ along $t$} is the covariant derivative of $s$ along 
$\Phi_{g^{\Lambda}}\circ t$, 
$$\nabla_{t}s:=\nabla_{(\Phi_{g^{\Lambda}})\circ t}s,$$ where $\Phi_{g^{\Lambda}}$ is the pairing map corresponding to the pseudo-metric $g^{\Lambda}$. 
\end{defn}

\subsection{Diffeological connections and gluing of pseudo-bundles}

We now consider the behavior of diffeological connections under gluing. Let $\pi_1\cup_{(\tilde{f},f)}\pi_2:V_1\cup_{\tilde{f}}V_2\to X_1\cup_f X_2$ be a pseudo-bundle obtained by gluing. Suppose that 
$V_1$ and $V_2$ are endowed with diffeological connections, $\nabla^1$ and $\nabla^2$ respectively. We might expect that, as it happens for all other objects that we have considered, under 
appropriate compatibility conditions, these two connections induce one on the pseudo-bundle obtained by gluing, \emph{i.e.} $V_1\cup_{\tilde{f}}V_2$; indeed, this is what happens in the example 
considered in Section 10.1. Below we describe how this can be done for abstract pseudo-bundles and connections on them, starting with the appropriate compatibility notion for connections.

\subsubsection{The compatibility notion for connections}

Recall the criterion of compatibility of elements of $\Lambda^1(X_1)$ and $\Lambda^1(X_2)$ in terms of the three pullback maps $f^*$, $i_{\Lambda}^*$, and $j_{\Lambda}^*$ (Proposition 8.4): for any 
$y\in Y$, two elements $\alpha_1\in\Lambda_y^1(X_1)$ and $\alpha_2\in\Lambda_{f(y)}^1(X_2)$ are compatible if and only if $i_{\Lambda}^*\alpha_1=f^*(j_{\Lambda}^*\alpha_2)$. This criterion allows us 
to give the following definition of the compatibility of a connection on a pseudo-bundle $V_1$ over $X_1$ with a connection on a pseudo-bundle $V_2$ over $X_2$.

\begin{defn}\label{compatibility:of:connections:defn}
Let $\pi_1:V_1\to X_1$ and $\pi_2:V_2\to X_2$ be two diffeological vector pseudo-bundles, let $f$ and $\tilde{f}$ be maps defining a gluing of the former to the latter, each of which is a diffeomorphism
of its domain with its image, and let $Y$ be the domain of definition of $f$. Let $\nabla^1$ be a connection on $V_1$, and let $\nabla^2$ be a connection on $V_2$. We say that $\nabla^1$ and 
$\nabla^2$ are \textbf{compatible} if for any pair $s_1\in C^{\infty}(X_1,V_1)$ and $s_2\in C^{\infty}(X_2,V_2)$ of compatible sections, and for any $y\in Y$, we have 
$$\left((i_{\Lambda}^*\otimes\tilde{f})\circ(\nabla^1s_1)\right)(y)=\left(((f^*j_{\Lambda}^*)\otimes\mbox{Id}_{V_2})\circ(\nabla^2s_2)\right)(f(y)).$$
\end{defn}

\subsubsection{Obtaining a connection on $V_1\cup_{\tilde{f}}V_2$ out of compatible $\nabla^1$ and $\nabla^2$ on $V_1$ and $V_2$}\label{induced:connection:on:glued:secn}

Let $\pi_1:V_1\to X_1$ and $\pi_2:V_2\to X_2$ be two pseudo-bundles, let $(\tilde{f},f)$ be a gluing of the former to the latter such that $f:X_1\supseteq Y\to X_2$ and $\tilde{f}:\pi_1^{-1}(Y)\to V_2$ are 
diffeomorphisms of their domains with their images, and $f$ is such that $\calD_1^{\Omega}=\calD_2^{\Omega}$. Let $\nabla^1$ and $\nabla^2$ be connections on $V_1$ and $V_2$ respectively, 
compatible in the sense of Definition \ref{compatibility:of:connections:defn}. We now define an induced connection $\nabla^{\cup}$ on $V_1\cup_{\tilde{f}}V_2$. 

As any connection on $V_1\cup_{\tilde{f}}V_2$, the operator $\nabla^{\cup}$ is a map 
$$C^{\infty}(X_1\cup_f X_2,V_1\cup_{\tilde{f}}V_2)\to C^{\infty}(X_1\cup_f X_2,\Lambda^1(X_1\cup_f X_2)\otimes(V_1\cup_{\tilde{f}}V_2)),$$ which is defined as follows. Let $s$ be any section in 
$C^{\infty}(X_1\cup_f X_2,V_1\cup_{\tilde{f}}V_2)$. Since $f$ and $\tilde{f}$ are diffeomorphisms, $s$ has a unique presentation in the form $s=s_1\cup_{(f,\tilde{f})}s_2$ for certain compatible sections
 $s_1\in C^{\infty}(X_1,V_1)$ and $s_2\in C^{\infty}(X_2,V_2)$. Thus, there is a well-defined assignment  
$$C^{\infty}(X_1\cup_f X_2,V_1\cup_{\tilde{f}}V_2)\ni s\mapsto(\nabla^1s_1,\nabla^2s_2)$$
$$\mbox{ for }\,\,\nabla^1s_1\in C^{\infty}(X_1,\Lambda^1(X_1)\otimes V_1)\mbox{ and }\nabla^2s_2\in C^{\infty}(X_2,\Lambda^1(X_2)\otimes V_2).$$ To these, we now assign a section 
$$\nabla^{\cup}s:X_1\cup_f X_2\to\Lambda^1(X_1\cup_f X_2)\otimes(V_1\cup_{\tilde{f}}V_2),$$ whose value at any given $x\in X_1\cup_f X_2$ is determined as follows:
$$(\nabla^{\cup}s)(x)=$$
$$\left\{\begin{array}{cl}
((\tilde{\rho}_1^{\Lambda})^{-1}\otimes j_1)((\nabla^1s_1)(i_1^{-1}(x))) & \mbox{for }x\in i_1(X_1\setminus Y), \\
((\tilde{\rho}_2^{\Lambda})^{-1}\otimes j_2)((\nabla^2s_2)(i_2^{-1}(x))) & \mbox{for }x\in i_2(X_2\setminus f(Y)), \\
((\tilde{\rho}_1^{\Lambda}\oplus\tilde{\rho}_2^{\Lambda})^{-1}\otimes\mbox{Id}_{V_1\cup_{\tilde{f}}V_2})
((\mbox{Incl}_{\Lambda_{f^{-1}(i_2^{-1}(x))}^1(X_1)}\otimes(j_2\circ\tilde{f}))\left(\nabla^1s_1(f^{-1}(i_2^{-1}(x)))\right)+ & \\
+(\mbox{Incl}_{\Lambda_{i_2^{-1}(x)}^1(X_2)}\otimes j_2)\left(\nabla^2s_2(i_2^{-1}(x))\right) & \mbox{ for }x\in i_2(f(Y))),
\end{array}\right.$$
where for $x\in i_2(f(Y))$ we have
\begin{itemize}
\item $\mbox{Incl}_{\Lambda_{f^{-1}(i_2^{-1}(x))}^1(X_1)}$ is the standard inclusion 
$$\Lambda_{f^{-1}(i_2^{-1}(x))}^1(X_1)\cong\Lambda_{f^{-1}(i_2^{-1}(x))}^1(X_1)\oplus\{0\}\hookrightarrow\Lambda_{f^{-1}(i_2^{-1}(x))}^1(X_1)\oplus\Lambda_{i_2^{-1}(x)}^1(X_2),$$
\item $\mbox{Incl}_{\Lambda_{i_2^{-1}(x)}^1(X_2)}$ is likewise the standard inclusion
$$\Lambda_{i_2^{-1}(x)}^1(X_2)\cong\{0\}\oplus\Lambda_{i_2^{-1}(x)}^1(X_2)\hookrightarrow\Lambda_{f^{-1}(i_2^{-1}(x))}^1(X_1)\oplus\Lambda_{i_2^{-1}(x)}^1(X_2),$$ and
\item $(\tilde{\rho}_1^{\Lambda}\oplus\tilde{\rho}_2^{\Lambda})^{-1}\otimes\mbox{Id}_{V_1\cup_{\tilde{f}}V_2}$ acts
$$(\Lambda_{f^{-1}(i_2^{-1}(x))}^1(X_1)\oplus\Lambda_{i_2^{-1}(x)}^1(X_2))\otimes(V_1\cup_{\tilde{f}}V_2)\to\Lambda^1(X_1\cup_f X_2)\otimes(V_1\cup_{\tilde{f}}V_2).$$
\end{itemize}

By the assumption of the compatibility of the connections $\nabla^1$ and $\nabla^2$, the resulting $\nabla^{\cup}s$ is well-defined as a map 
$X_1\cup_f X_2\to\Lambda^1(X_1\cup_f X_2)\otimes(V_1\cup_{\tilde{f}}V_2)$. Furthermore, the following statement is a matter of a simple (even if lengthy) check.

\begin{prop}
Let $f$ be a diffeomorphism and such that $\calD_1^{\Omega}=\calD_2^{\Omega}$. Then for every smooth section $s\in C^{\infty}(X_1\cup_f X_2,V_1\cup_{\tilde{f}}V_2)$ we have
$$\nabla^{\cup}s\in C^{\infty}(X_1\cup_f X_2,\Lambda^1(X_1\cup_f X_2)\otimes(V_1\cup_{\tilde{f}}V_2)),$$ that is, $\nabla^{\cup}$ is a map 
$C^{\infty}(X_1\cup_f X_2,V_1\cup_{\tilde{f}}V_2)\to C^{\infty}(X_1\cup_f X_2,\Lambda^1(X_1\cup_f X_2)\otimes(V_1\cup_{\tilde{f}}V_2))$.
\end{prop}

It is more or less clear from the construction that $\nabla^{\cup}$ is linear and satisfies the Leibnitz rule (the latter check is based on the description of the behavior of the differential under gluing, see 
Section 10.3.3 immediately below). It can also be verified that it is smooth for the functional diffeologies on $C^{\infty}(X_1\cup_f X_2,V_1\cup_{\tilde{f}}V_2)$ and on 
$C^{\infty}(X_1\cup_f X_2,\Lambda^1(X_1\cup_f X_2)\otimes(V_1\cup_{\tilde{f}}V_2))$. There is therefore the following statement.

\begin{thm}\label{induced:operator:on:glued:pseudo-bundle:is:connection:thm}
Let $\pi_1:V_1\to X_1$ and $\pi_2:V_2\to X_2$ be two diffeological vector pseudo-bundles, and let $(\tilde{f},f)$ be a gluing between them such that both $\tilde{f}$ and $f$ are diffeomorphisms, and $f$ has 
the property that $\calD_1^{\Omega}=\calD_2^{\Omega}$. Then the operator $\nabla^{\cup}$ is a connection on $V_1\cup_{\tilde{f}}V_2$.
\end{thm}

\subsubsection{The differentials and gluing}

Let $X_1$ and $X_2$ be two diffeological spaces, let $f:X_1\supseteq Y\to X_2$ be a gluing diffeomorphism, and let $h:X_1\cup_f X_2\to\matR$ be a smooth function (for the standard diffeology on $\matR$). 
Then, as in the case of smooth sections of pseudo-bundles, there is a presentation of $h$ in the form $h=h_1\cup_f h_2$, where $h_1\in C^{\infty}(X_1,\matR)$ and $h_2\in C^{\infty}(X_2,\matR)$. 

As we have explained already, the three differentials $dh$, $dh_1$, and $dh_2$, defined originally as elements of $\Omega^1(X_1\cup_f X_2)$, $\Omega^1(X_1)$, and $\Omega^1(X_2)$, are also naturally 
seen as smooth sections of, respectively, $\Lambda^1(X_1\cup_f X_2)$, $\Lambda^1(X_1)$, and $\Lambda^1(X_2)$. Between them, there is the following relation.

\begin{prop}
Let diffeological spaces $X_1$ and $X_2$ and the gluing diffeomorphism $f:X_1\supseteq Y\to X_2$ be such that $\calD_1^{\Omega}=\calD_2^{\Omega}$, and let $h=h_1\cup_f h_2:X_1\cup_f X_2\to\matR$ 
be a smooth function. Then the following is true:
$$dh(x)=\left\{\begin{array}{cl}
(\tilde{\rho}_1^{\Lambda})^{-1}(dh_1(i_1^{-1}(x))) & \mbox{if }x\in i_1(X_1\setminus Y), \\
(\tilde{\rho}_1^{\Lambda}\oplus\tilde{\rho}_2^{\Lambda})^{-1}\left(dh_1(f^{-1}(i_2^{-1}(x)))\oplus dh_2(i_2^{-1}(x))\right) & \mbox{if }x\in i_2(f(Y)), \\
(\tilde{\rho}_2^{\Lambda})^{-1}(dh_2(i_2^{-1}(x))) & \mbox{if }x\in i_2(X_2\setminus f(Y)).
\end{array}\right.$$
\end{prop}

It is worth noting that for points outside of the domain of gluing we have the following expected identities:
$$\tilde{\rho}_1^{\Lambda}(dh(i_1(x)))=dh_1(x)\,\mbox{ for all }x\in X_1\setminus Y\,\,\mbox{ and }\,\,\tilde{\rho}_2^{\Lambda}(dh(i_2(x)))=dh_2(x)\,\mbox{ for all }x\in X_2\setminus f(Y).$$

\subsection{The operations on diffeological connections}

For usual connections on smooth vector bundles there are standard ways of obtaining induced connections on direct sums, tensor products, and dual bundles. We now comment on how this carries over to 
the context of diffeological connections. For direct sums and tensor products the result is quite analogous to the standard case, although, since we avoid using local bases (diffeological vector pseudo-bundles 
may easily not have them) and therefore do not make recourse to the local matrix of $1$-forms, it is achieved a bit differently. For dual pseudo-bundles we do not claim any definite answer, limiting ourselves 
to pointing out the potential differences; these, however, regard the methods and do not necessarily preclude reaching the same conclusion. This is a question that we leave in the open.

\subsubsection{The direct sum}\label{direct:sum:of:connections:secn}

Let $\pi_1:V_1\to X$ and $\pi_2:V_2\to X$ be two diffeological vector pseudo-bundles over the same base space $X$; consider their direct sum, the pseudo-bundle 
$$\pi_1\oplus\pi_2:V_1\oplus V_2\to X.$$ Denote by $\mbox{pr}^{V_1}:V_1\oplus V_2\to V_1$ and $\mbox{pr}^{V_2}:V_1\oplus V_2\to V_2$ the obvious projections of the direct sum onto its summands 
(on each fibre these are the standard projections associated to the decomposition of a vector space into a direct sum). Such projections are always smooth, by the definition of the direct sum diffeology. 
Let
$$\mbox{Incl}_{V_1}:V_1\cong V_1\oplus\{0\}\hookrightarrow V_1\oplus V_2,\,\,\,\mbox{Incl}_{V_2}:V_2\cong\{0\}\oplus V_2\hookrightarrow V_1\oplus V_2,$$ and let 
$\nabla^1:C^{\infty}(X,V_1)\to C^{\infty}(X,\Lambda^1(X)\otimes V_1)$ be a connection on $V_1$, and let $\nabla^2:C^{\infty}(X,V_2)\to C^{\infty}(X,\Lambda^1(X)\otimes V_2)$ be a 
connection on $V_2$. 

We define the following connection $\nabla^1\oplus\nabla^2:C^{\infty}(X,V_1\oplus V_2)\to C^{\infty}(X,\Lambda^1(X)\otimes(V_1\oplus V_2))$ on $V_1\oplus V_2$. Let $s\in C^{\infty}(X,V_1\oplus V_2)$. 
Denote
$$s_1:=\mbox{pr}^{V_1}\circ s\,\,\,\mbox{ and }\,\,\,s_2:=\mbox{pr}^{V_2}\circ s,$$ and set 
$$(\nabla^1\oplus\nabla^2)s=(\mbox{Id}_{\Lambda^1(X)}\otimes\mbox{Incl}_{V_1})\circ(\nabla^1s_1)\oplus(\mbox{Id}_{\Lambda^1(X)}\otimes\mbox{Incl}_{V_2})\circ(\nabla^2s_2).$$
Although the formal description is different from the standard smooth case, the essence of the construction is the same.

\subsubsection{The tensor product}

This is analogous to the case of the direct sum. Let $\pi_1:V_1\to X$ and $\pi_2:V_2\to X$ be two diffeological vector pseudo-bundles, and consider $\pi_1\otimes\pi_2:V_1\otimes V_2\to X$.
Let $\nabla^1:C^{\infty}(X,V_1)\to C^{\infty}(X,\Lambda^1(X)\otimes V_1)$ be a connection on $V_1$, and let $\nabla^2:C^{\infty}(X,V_2)\to C^{\infty}(X,\Lambda^1(X)\otimes V_2)$ be a connection on 
$V_2$; the corresponding connection $\nabla^{\otimes}:C^{\infty}(X,V_1\otimes V_2)\to C^{\infty}(X,\Lambda^1(X)\otimes(V_1\otimes V_2))$ can be defined in a way that mimics the standard construction, 
that is, by setting 
$$\nabla^{\otimes}:=\nabla^1\otimes\mbox{Id}_{C^{\infty}(X,V_2)}+\mbox{Id}_{C^{\infty}(X,V_1)}\otimes\nabla^2.$$

\subsubsection{The dual pseudo-bundle} 

Let us now discuss the possibility of carrying over the standard construction of the induced connection on the dual bundle to the diffeological context. Let $\pi:V\to X$ be a diffeological vector pseudo-bundle, 
endowed with a connection $\nabla$. The standard construction requires a choice of a local basis $\{s_1,\ldots,s_n\}$, so we must assume that there exists one. This is the first main difference, since 
diffeological pseudo-bundles are not locally trivial, so they are not required to admit any.

Now, as we have seen in the case of $(\Lambda^1(X))^*$, if $V$ admits a pseudo-metric (also not guaranteed in general) and has only standard fibres, then $V\cong V^*$ via the corresponding pairing map. 
Thus, of course, there is the obvious dual connection. 

Finally, if $V$, admitting a pseudo-metric, has some non-standard fibres then the pairing map $\Phi_g$ (where $g$ is the pseudo-metric chosen) is still a subduction. Suppose that $V$ admits a (local) basis 
$s_1,\ldots,s_n$ of smooth sections, hence $\Phi_g\circ s_1,\ldots,\Phi_g\circ s_n$ is a local generating set. This might be used to define the dual connection via the standard rule, with the issue being whether 
the resulting purported connection is well-defined (we do not follow through on this, as we are only going to consider a kind of dual connection in the case of $\Lambda^1(X)$ with standard fibres).

\subsection{Compatibility with pseudo-metrics and gluing}

The notion of compatibility of a given diffeological connection $\nabla$ on a pseudo-bundle $\pi:V\to X$ with a given pseudo-metric $g$ on $V$ mimics the standard one.

\begin{defn}\label{connection:compatible:with:pseudo-metric:defn}
Let $\pi:V\to X$ be a diffeological vector pseudo-bundle with finite-dimensional fibres, let $g$ be a pseudo-metric on $V$, and let $\nabla$ be a diffeological connection on this pseudo-bundle. The 
connection $\nabla$ is said to be \textbf{compatible with the pseudo-metric $g$} if for every two smooth sections $s,t:X\to V$ we have that
$$d(g(s,t))=g(\nabla s,t)+g(s,\nabla t),$$ where for every $1$-form $\omega\in\Lambda^1(X)$ we set by definition $g(\omega\otimes s,t)=g(s,\omega\otimes t):=\omega\cdot g(s,t)$.
\end{defn}

Let now $\pi_1:V_1\to X_1$ and $\pi_2:V_2\to X_2$ be two diffeological vector pseudo-bundles, and let $(\tilde{f},f)$ be a gluing between them. Suppose that these pseudo-bundles can be endowed with 
compatible pseudo-metrics $g_1$ and $g_2$ respectively, and that they also admit connections $\nabla^1$ and $\nabla^2$ that satisfy the compatibility condition relative to the gluing along $(\tilde{f},f)$.
  
Recall from Section 5 that, given a choice of $g_1$ and $g_2$, the pseudo-bundle $V_1\cup_{\tilde{f}}V_2$ carries the induced pseudo-metric $\tilde{g}$. Assume now that the gluing maps $\tilde{f}$ and 
$f$ are such that, given a choice of $\nabla^1$ and $\nabla^2$, the induced connection $\nabla^{\cup}$ is well-defined, \emph{i.e.}, both maps are diffeomorphisms and $f$ is such that 
$\calD_1^{\Omega}=\calD_2^{\Omega}$. Then it is natural to ask whether the assumption that $\nabla^1$ and $\nabla^2$ are compatible with, respectively, $g_1$ and $g_2$ is sufficient to ensure that 
$\nabla^{\cup}$ is compatible with $\tilde{g}$; it turns out that this is the case.

\begin{thm}\label{induced:connection:is:compatible:with:induced:pseudo-metric:thm}
Let $\pi_1:V_1\to X_1$ and $\pi_2:V_2\to X_2$ be two diffeological vector pseudo-bundles, and let $(\tilde{f},f)$ be a gluing between them such that both $\tilde{f}$ and $f$ are diffeomorphisms, and $f$ is 
such that $\calD_1^{\Omega}=\calD_2^{\Omega}$. Assume that $V_1$ admits a pseudo-metric $g_1$ and a connection $\nabla^1$ compatible with $g_1$, and likewise, that $V_2$ admits a pseudo-metric 
$g_2$ and a connection $\nabla^2$ compatible with $g_2$. Assume finally that $g_1$ is compatible with $g_2$, and that $\nabla^1$ is compatible with $\nabla^2$ (both in terms of the gluing along 
$(\tilde{f},f)$). Then the induced connection $\nabla^{\cup}$ is compatible with the induced pseudo-metric $\tilde{g}$.
\end{thm}

The proof of this statement is quite straightforward and, in addition to the definition of the induced connection, uses the above description of the behavior of the differential under gluing.

\section{Diffeological Levi-Civita connections on $X$}

In this section we consider diffeological Levi-Civita connections. There are two sorts of them, one is defined on $(\Lambda^1(X))^*$, the other on $\Lambda^1(X)$; the two versions are related by the 
pairing map diffeomorphism. The notion itself mimics the standard one. All statements appearing below were proved in \cite{connections-LC}.

\subsection{Levi-Civita connections on $(\Lambda^1(X))^*$}

Let $X$ be a diffeological space such that $(\Lambda^1(X))^*$ admits a pseudo-metric $g^{\Lambda^*}$.

\begin{defn}\label{lie:bracket:on:dual:lambda:defn}
Let $f:X\to\matR$ be a smooth function, and let $t,t_1,t_2\in C^{\infty}(X,(\Lambda^1(X))^*)$. The \textbf{action of $t$ on $f$} is then defined by 
$$t(f):X\ni x\mapsto t(df)=t([f])\in C^{\infty}(X,\matR).$$ The \textbf{Lie bracket} $[t_1,t_2]\in C^{\infty}(X,(\Lambda^1(X))^*)$ is defined by 
$$[t_1,t_2](s)=t_1(t_2(s))-t_2(t_1(s))\mbox{ for any }s\in\Lambda^1(X).$$ In particular, 
$$[t_1,t_2](f)=t_1(t_2(f))-t_2(t_1(f))\in C^{\infty}(X,\matR).$$
\end{defn}

Observe that this Lie bracket is antisymmetric (this is obvious from the formula), bilinear, and satisfies the Jacobi identity (all of these hold for the same reason that they do in the standard case). We can 
now define the torsion, whose definition is fully analogous to the standard one.

\begin{defn}\label{torsion:of:a:connection:defn}
Let $X$ be a diffeological space, and let $\nabla$ be a connection on $(\Lambda^1(X))^*$. The \textbf{torsion} $T$ of $\nabla$ on $(\Lambda^1(X))^*$ is defined by setting, for all 
$t_1,t_2\in C^{\infty}(X,(\Lambda^1(X))^*)$,
$$T(t_1,t_2)=\nabla_{t_1}t_2-\nabla_{t_2}t_1-[t_1,t_2]\in C^{\infty}(X,(\Lambda^1(X))^*).$$ Since $T$ is clearly bilinear, it is a map 
$T:C^{\infty}(X,(\Lambda^1(X))^*)\otimes C^{\infty}(X,(\Lambda^1(X))^*)\to C^{\infty}(X,(\Lambda^1(X))^*)$. The connection $\nabla$ is called \textbf{symmetric} if $T$ is the zero tensor:
$$\nabla_{t_1}t_2-\nabla_{t_2}t_1=[t_1,t_2] \mbox{ for all }t_1,t_2\in C^{\infty}(X,(\Lambda^1(X))^*).$$
\end{defn}

The definition of the Levi-Civita connection is then identical to the standard one.

\begin{defn}
Let $X$ be a diffeological space such that $(\Lambda^1(X))^*$ is endowed with a pseudo-metric $g^{\Lambda^*}$, and let $\nabla$ be a connection on $(\Lambda^1(X))^*$. $\nabla$ is called a 
\textbf{Levi-Civita} connection if it is symmetric and compatible with the pseudo-metric $g^{\Lambda^*}$. 
\end{defn}

Any $((\Lambda^1(X))^*,g^{\Lambda^*})$ admits at most one Levi-Civita connection, for reasons that are essentially the same as in the standard case. Indeed, the standard formula, which also yields 
the uniqueness of the connection, that is,
\begin{flushleft}
$t_1(g^{\Lambda^*}(t_2,t_3))+t_2(g^{\Lambda^*}(t_3,t_1))-t_3(g^{\Lambda^*}(t_1,t_2))+$
\end{flushleft}
\begin{flushright}
$+g^{\Lambda^*}([t_1,t_2],t_3)+g^{\Lambda^*}([t_3,t_1],t_2)-g^{\Lambda^*}([t_2,t_3],t_1)=2g^{\Lambda^*}(\nabla_{t_1}t_2,t_3)$,
\end{flushright}
with $t_1,t_2,t_3\in C^{\infty}(X,(\Lambda^1(X))^*)$ arbitrary, holds in the present context as well. On the other hand, we cannot make the same claim regarding the existence (notice that in general, 
we have avoided the existence questions for connections, or even pseudo-mterics).

\subsection{Pushforward and pullback connections}

As we have seen in Section 9, if we assume that $\Lambda^1(X)$ admits a pseudo-metric $g^{\Lambda}$ (so in particular, it has only finite-dimensional fibres) then the corresponding pairing 
$\Phi_{g^{\Lambda}}$ is a diffeomorphism $\Lambda^1(X)\to(\Lambda^1(X))^*$. Therefore all constructions carry over from one to the other, in particular, a connection $\nabla$ on $\Lambda^1(X)$ yields a 
connection $\nabla^*$ on $(\Lambda^1(X))^*$, and \emph{vice versa} a connection $\nabla^*$ on $(\Lambda^1(X))^*$ induces a specific connection $\nabla$ on $\Lambda^1(X)$. We say that $\nabla^*$ is 
the \textbf{pushforward of the connection $\nabla$ by the pairing map $\Phi_{g^{\Lambda}}$}, and that $\nabla$ is the \textbf{pullback of $\nabla^*$}. 

The explicit relation between the two connections is as follows:
$$\nabla^*t=(\mbox{Id}_{\Lambda^1(X)}\otimes\Phi_{g^{\Lambda}})\circ\nabla(\Phi_{g^{\Lambda}}^{-1}\circ t)\,\,\mbox{ for any }\,\,t\in C^{\infty}(X,(\Lambda^1(X))^*),$$
$$\nabla s=(\mbox{Id}_{\Lambda^1(X)}\otimes\Phi_{g^{\Lambda}}^{-1})\circ\nabla^*(\Phi_{g^{\Lambda}}\circ s)\,\,\mbox{ for any }\,\,s\in C^{\infty}(X,\Lambda^1(X)).$$ This identification trivially preserves 
covariant derivatives. It is also easy to verify that it preserves compatibility with pseudo-metrics, in the sense that if $\nabla$ is compatible with the given $g^{\Lambda}$ then $\nabla^*$ is compatible with the 
dual pseudo-metric $(g^{\Lambda})^*$, and \emph{vice versa}.

\subsection{Levi-Civita connections on $\Lambda^1(X)$}

Roughly speaking, a Levi-Civita connection on $\Lambda^1(X)$ is the pullback of the Levi-Civita connection on $(\Lambda^1(X))^*$; the pullback connection is the Levi-Civita one for the pullback 
pseudo-metric. Its definition can also be stated separately, as that of a symmetric connection compatible with the given pseudo-metric $g^{\Lambda}$. Since we have already considered covariant derivatives 
along sections of $\Lambda^1(X)$, as well as the compatibility with pseudo-metrics, it now suffices to define the Lie bracket. This is also done through the pairing map, and in an obvious way:
$$[s_1,s_2]=\Phi_{g^{\Lambda}}^{-1}\circ[\Phi_{g^{\Lambda}}\circ s_1,\Phi_{g^{\Lambda}}\circ s_2],$$ using the already-given definition of the Lie bracket on $(\Lambda^1(X))^*$. A connection $\nabla$ on 
$(\Lambda^1(X),g^{\Lambda})$, where $g^{\Lambda}$ is a pseudo-metric is a \textbf{Levi-Civita connection} if it is compatible with $g^{\Lambda}$ and is symmetric,
$$\nabla_{s_1}s_2-\nabla_{s_2}s_1=[s_1,s_2].$$

\subsection{Compatible connections on $\Lambda^1(X_1)$ and $\Lambda^1(X_2)$, and the induced connection on $\Lambda^1(X_1\cup_f X_2)$}

We now turn to considering the interactions between the connections on $\Lambda^1(X_1)$ and $\Lambda^1(X_2)$, and those on $\Lambda^1(X_1\cup_f X_2)$. More precisely, we show that certain pairs 
of connections on $\Lambda^1(X_1)$ and $\Lambda^1(X_2)$ induce a well-defined connection on $\Lambda^1(X_1\cup_f X_2)$; these pairs are determined by an appropriate compatibility notion. In 
particular, the Levi-Civita connections on $\Lambda^1(X_1)$ and $\Lambda^1(X_2)$ defined with respect to compatible pseudo-metrics determine the Levi-Civita connection on $\Lambda^1(X_1\cup_f X_2)$.

\subsubsection{The compatibility notion for connections on $\Lambda^1(X_1)$ and $\Lambda^1(X_2)$}

We now define the compatibility for connections on $\Lambda^1(X_1)$ and $\Lambda^1(X_2)$. Two sections $s_1\in C^{\infty}(X_1,\Lambda^1(X_1))$ and $s_2\in C^{\infty}(X_2,\Lambda^1(X_2))$ are 
\textbf{compatible} if for all $y\in Y$ the images $s_1(y)$ and $s_2(f(y))$ are compatible elements of $\Lambda_y^1(X_1)$ and $\Lambda_{f(y)}^1(X_2)$ respectively. 

\begin{defn}\label{compatibility:of:connections:on:lambda:defn}
Let $\nabla^1$ be a connection on $\Lambda^1(X_1)$ and let $\nabla^2$ be a connection on $\Lambda^1(X_2)$. We say that $\nabla^1$ and $\nabla^2$ are \textbf{compatible} if for every 
$y\in Y$ and for every two compatible sections $s_1\in C^{\infty}(X_1,(\Lambda^1(X_1))^*)$ and $s_2\in C^{\infty}(X_2,(\Lambda^1(X_2))^*)$ we have that 
$$\left((i_{\Lambda}^*\otimes i_{\Lambda}^*)\circ(\nabla^1s_1)\right)(y)=\left(((f_{\Lambda}^*j_{\Lambda}^*)\otimes(f_{\Lambda}^*j_{\Lambda}^*))\circ(\nabla^2s_2)\right)(f(y)).$$
\end{defn}

The aim of this definition is to ensure that for every $y\in Y$ and for every pair of compatible sections $s_1\in C^{\infty}(X_1,\Lambda^1(X_1))$, $s_2\in C^{\infty}(X_2,\Lambda^1(X_2))$ the sum 
$(\nabla^1s_1)(y)\oplus(\nabla^2s_2)(f(y))$, which in general is an element of 
$$\left(\Lambda_y^1(X_1)\otimes\Lambda_y^1(X_1)\right)\,\oplus\,\left(\Lambda_{f(y)}^1(X_2)\otimes\Lambda_{f(y)}^1(X_2)\right),$$ be, in a natural sense, an element of 
$$\left(\Lambda_y^1(X_1)\oplus_{comp}\Lambda_{f(y)}^1(X_2)\right)\otimes\left(\Lambda_y^1(X_1)\oplus_{comp}\Lambda_{f(y)}^1(X_2)\right).$$

\subsubsection{The connection on $\Lambda^1(X_1\cup_f X_2)$ induced by two compatible ones}

Two compatible connections on $\Lambda^1(X_1)$ and $\Lambda^1(X_2)$ naturally induce a connection on $\Lambda^1(X_1\cup_f X_2)$. To describe this induced connection, consider first the following. 
Let $s\in C^{\infty}(X_1\cup_f X_2,\Lambda^1(X_1\cup_f X_2))$. Define
$$s_1:=\tilde{\rho}_1^{\Lambda}\circ s\circ \tilde{i}_1\in C^{\infty}(X_1,\Lambda^1(X_1)),\,\,\,s_2:=\tilde{\rho}_2^{\Lambda}\circ s\circ i_2\in C^{\infty}(X_2,\Lambda^1(X_2)).$$

\begin{defn}\label{induced:connection:on:lambda:defn}
Let $X_1$ and $X_2$ be two diffeological spaces, let $f:X_1\supseteq Y\to X_2$ be a map that is a diffeomorphism of its domain with its image such that $\calD_1^{\Omega}=\calD_2^{\Omega}$, and let 
$\nabla^1$ and $\nabla^2$ be connections on $\Lambda^1(X_1)$ and $\Lambda^1(X_2)$ respectively, compatible in the sense of Definition \ref{compatibility:of:connections:on:lambda:defn}. Let 
$s:X_1\cup_f X_2\to(\Lambda^1(X_1\cup_f X_2))^*$ be a smooth section. The \textbf{induced connection $\nabla_{\cup}$} on $\Lambda^1(X_1\cup_f X_2)$ is defined by setting
\begin{flushleft}
$\left(\nabla_{\cup}s\right)(x)=$
\end{flushleft}
$$\left\{\begin{array}{cl}
((\tilde{\rho}_1^{\Lambda})^{-1}\otimes(\tilde{\rho}_1^{\Lambda})^{-1})((\nabla^1s_1)(i_1^{-1}(x))) & \mbox{for }x\in i_1(X_1\setminus Y), \\
((\tilde{\rho}_2^{\Lambda})^{-1}\otimes(\tilde{\rho}_2^{\Lambda})^{-1})((\nabla^2s_2)(i_2^{-1}(x))) & \mbox{for }x\in i_2(X_2\setminus f(Y)), \\
((\tilde{\rho}_1^{\Lambda}\oplus\tilde{\rho}_2^{\Lambda})^{-1}\otimes(\tilde{\rho}_1^{\Lambda}\oplus\tilde{\rho}_2^{\Lambda})^{-1})\left((\nabla^1s_1)(f^{-1}(i_2^{-1}(x)))\oplus(\nabla^2s_2)(i_2^{-1}(x))\right) 
& \mbox{for }x\in i_2(f(Y)).
\end{array}\right.$$
\end{defn}

The compatibility notion for connections (Definition \ref{compatibility:of:connections:on:lambda:defn}) ensures that $\nabla^{\cup}$ is well-defined, in the sense that 
$(\nabla^1s_1)(f^{-1}(i_2^{-1}(x)))\oplus(\nabla^2s_2)(i_2^{-1}(x))$ always belongs to the range of 
$(\tilde{\rho}_1^{\Lambda}\oplus\tilde{\rho}_2^{\Lambda})\otimes(\tilde{\rho}_1^{\Lambda}\oplus\tilde{\rho}_2^{\Lambda})$. Moreover, the following is true.

\begin{thm}
The operator $\nabla^{\cup}$ is a diffeological connection on $\Lambda^1(X_1\cup_f X_2)$.
\end{thm}

\subsubsection{Compatibility of the Levi-Civita connections on $\Lambda^1(X_1)$ and $\Lambda^1(X_2)$}

Since any Levi-Civita connection is determined by the pseudo-metric, we might expect that those corresponding to compatible pseudo-metrics might also be compatible. Of course, this is not completely 
immediate, since the two compatibility notions are not completely analogous. 

\begin{prop}\label{levi-civita:connections:wrt:compatible:pseudo-metrics:are:compatible:prop}
Let $g_1^{\Lambda}$ and $g_2^{\Lambda}$ be compatible pseudo-metrics on $\Lambda^1(X_1)$ and $\Lambda^1(X_2)$ respectively. Let $\nabla^1$ be the Levi-Civita connection on 
$(\Lambda^1(X_1),g_1^{\Lambda})$ (which we mean, of course, that $\nabla^1$ is in particular compatible with $g_1^{\Lambda}$), and let $\nabla^2$ be the Levi-Civita connection on 
$(\Lambda^1(X_2),g_2^{\Lambda})$. Then $\nabla^1$ and $\nabla^2$ are compatible.
\end{prop}

We thus obtain that the two Levi-Civita connections defined with respect to compatible pseudo-metrics always give rise to the induced connection $\nabla^{\cup}$ on $\Lambda^1(X_1\cup_f X_2)$. As we 
see in the next section, a stronger statement is actually true: for the appropriate pseudo-metric, $\nabla^{\cup}$ is itself the Levi-Civita connection.

\subsection{The Levi-Civita connection on $\Lambda^1(X_1\cup_f X_2)$}

Let $\nabla^1$ and $\nabla^2$ be the Levi-Civita connections on $(\Lambda^1(X_1),g_1^{\Lambda})$ and $(\Lambda^1(X_2),g_2^{\Lambda})$, where $g_1^{\Lambda}$ and $g_2^{\Lambda}$ are 
compatible. As follows from Proposition \ref{levi-civita:connections:wrt:compatible:pseudo-metrics:are:compatible:prop}, $\Lambda^1(X_1\cup_f X_2)$ comes endowed both with the induced connection 
$\nabla^{\cup}$ and the induced pseudo-metric $g^{\Lambda}$. It is more generally true that the induced connection is compatible with the induced pseudo-metric. It remains to check that also the 
symmetricity is inherited, to ensure that $\nabla^{\cup}$ is the Levi-Civita connection in turn.

\subsubsection{Compatibility with pseudo-metrics and gluing} 

Let $X_1$ and $X_2$ be diffeological spaces, let $f:X_1\supseteq Y\to X_2$ be a diffeomorphism such that $\calD_1^{\Omega}=\calD_2^{\Omega}$, and let $g_1^{\Lambda}$ and $g_2^{\Lambda}$ be 
pseudo-metrics on $\Lambda^1(X_1)$ and $\Lambda^1(X_2)$ respectively, that are compatible with respect to $f$. Let also $\nabla^1$ and $\nabla^2$ be diffeological connections on $\Lambda^1(X_1)$ 
and $\Lambda^1(X_2)$ that are compatible with $f$ in the sense of Definition \ref{compatibility:of:connections:on:lambda:defn}. Assume also that each of them is compatible with the corresponding 
pseudo-metric ($\nabla^1$ is compatible with $g_1^{\Lambda}$ and $\nabla^2$ is compatible with $g_2^{\Lambda}$) in the sense of Definition \ref{connection:compatible:with:pseudo-metric:defn}. The 
following is then established by direct calculation.

\begin{prop}\label{induced:connection:on:lambda:is:compatible:prop}
The induced connection $\nabla^{\cup}$ is compatible with the pseudo-metric $g^{\Lambda}$. 
\end{prop}

\subsubsection{Symmetric connections and gluing} 

The analogue of Proposition \ref{induced:connection:on:lambda:is:compatible:prop} is true also for the symmetricity property. Specifically, we have 
the following statement.

\begin{prop}\label{induced:connection:on:dual:lambda:is:symmetric:prop}
Let $X_1$ and $X_2$ be diffeological spaces, let $f:X_1\supseteq Y\to X_2$ be a diffeomorphism such that $\calD_1^{\Omega}=\calD_2^{\Omega}$, and let $\nabla^1$ and $\nabla^2$ be connections on 
$\Lambda^1(X_1)$ and $\Lambda^1(X_2)$ respectively, compatible in the sense of Definition \ref{compatibility:of:connections:on:lambda:defn}. If both $\nabla^1$ and $\nabla^2$ are symmetric then 
the induced connection $\nabla^{\cup}$ is symmetric as well.
\end{prop}

This statement is based on two lemmas describing the behavior of covariant derivatives and the Lie bracket under gluing.

\begin{lemma}\label{covariant:derivatives:and:gluing:lem}
Let $s,t\in C^{\infty}(X_1\cup_f X_2,\Lambda^1(X_1\cup_f X_2))$, and let $x\in X_1\cup_f X_2$. Denote $s_1:=\tilde{\rho}_1^{\Lambda}\circ s\circ\tilde{i}_1$, $s_2:=\tilde{\rho}_2^{\Lambda}\circ s\circ i_2$, 
$t_1:=\tilde{\rho}_1^{\Lambda}\circ t\circ\tilde{i}_1$, and $t_2:=\tilde{\rho}_2^{\Lambda}\circ t\circ i_2$. Then:
$$(\nabla^{\cup}_t s)(x)=\left\{\begin{array}{cl}
(\tilde{\rho}_1^{\Lambda})^{-1}((\nabla^1_{t_1}s_1)(i_1^{-1}(x))) & \mbox{if }x\in i_1(X_1\setminus Y),\\
(\tilde{\rho}_1^{\Lambda}\oplus\tilde{\rho}_2^{\Lambda})^{-1}((\nabla^1_{t_1}s_1)(\tilde{i}_1^{-1}(x))\oplus(\nabla^2_{t_2}s_2)(i_2^{-1}(x))) & \mbox{if }x\in i_2(f(Y)), \\
(\tilde{\rho}_2^{\Lambda})^{-1}((\nabla^2_{t_2}s_2)(i_2^{-1}(x))) & \mbox{if }x\in i_2(X_2\setminus f(Y)).
\end{array}\right.$$
\end{lemma}

The statement just made is obtained again by direct calculation, as is the statement below.

\begin{lemma}\label{lie:bracket:and:gluing:lem}
Let $s,t,r\in C^{\infty}(X_1\cup_f X_2,\Lambda^1(X_1\cup_f X_2))$, let $s_1,s_2,t_1,t_2$ be as above, let $r_1,r_2$ be similarly defined, and let $x\in X_1\cup_f X_2$. Then
$$[s,t](r)(x)=\left\{\begin{array}{cl}
(\tilde{\rho}_1^{\Lambda})^{-1}([s_1,t_1](r_1)(i_1^{-1}(x))) & \mbox{if }x\in i_1(X_1\setminus Y),\\
(\tilde{\rho}_1^{\Lambda}\oplus\tilde{\rho}_2^{\Lambda})^{-1}([s_1,t_1](r_1)(\tilde{i}_1^{-1}(x))\oplus[s_2,t_2](r_2)(i_2^{-1}(x))) & \mbox{if }x\in i_2(f(Y)), \\
(\tilde{\rho}_2^{\Lambda})^{-1}([s_2,t_2](r_2)(i_2^{-1}(x))) & \mbox{if }x\in i_2(X_2\setminus f(Y)).
\end{array}\right.$$
\end{lemma}

\subsubsection{The final statement} 

All the above yields the following result. 

\begin{thm}\label{induced:connection:on:lambda:is:levi-civita:thm}
Let $X_1$ and $X_2$ be two diffeological spaces, let $f:X_1\supseteq Y\to X_2$ be a diffeomorphism such that $\calD_1^{\Omega}=\calD_2^{\Omega}$, and let $g_1^{\Lambda^*}$ and $g_2^{\Lambda^*}$
be compatible pseudo-metrics on $\Lambda^1(X_1)$ and $\Lambda^1(X_2)$ respectively. Let $\nabla^1$ and $\nabla^2$ be the Levi-Civita connections on $(\Lambda^1(X_1),g_1^{\Lambda})$ and 
$(\Lambda^1(X_2),g_2^{\Lambda})$. Then $\nabla_{\cup}$ is the Levi-Civita connection on $(\Lambda^1(X_1\cup_f X_2),g_1^{\Lambda})$.
\end{thm}

Notice that, due to the assumption that $\Lambda^1(X_1)$ and $\Lambda^1(X_2)$, and as a consequence $\Lambda^1(X_1\cup_f X_2)$, have finite-dimensional fibres, the pairing maps 
$\Phi_{g_1^{\Lambda}}$, $\Phi_{g_2^{\Lambda}}$, and $\Phi_{g^{\Lambda}}$ are all diffeomorphisms onto $(\Lambda^1(X_1))^*$, $(\Lambda^1(X_2))^*$, and $(\Lambda^1(X_1\cup_f X_2))^*$, so the above 
statements hold for these dual pseudo-bundles as well.

\section{Clifford connections}

The diffeological counterpart of the notion of a Clifford connection is obtained by straightforward extension of the standard notion. The results of this section are original to the present manuscript and come 
with complete proofs.

\subsection{A diffeological Clifford connection}

Let $X$ be a diffeological space such that $\Lambda^1(X)$ has only finite-dimensional fibres and is endowed with a pseudo-metric $g^{\Lambda}$. Let $\pi:V\to X$ be a pseudo-bundle of Clifford modules 
over $\cl(\Lambda^1(X),g^{\Lambda})$; we could for instance have $V=\bigwedge(\Lambda^1(X))$. The standard notion of a Clifford connection is a connection $\nabla^E$ on a smooth vector bundle $E$ of 
Clifford modules over the cotangent bundle $T^*M$ of a Riemannian manifold $(M,g)$, such that for every vector field $X\in C^{\infty}(M,TM)$, every smooth $1$-form $\phi\in C^{\infty}(M,T^*M)$, and every 
section $s\in C^{\infty}(M,E)$ the following equality is satisfied:
$$\nabla_X^E(c(\phi)s)=c(\nabla_X^{LC}\phi)(s)+c(\phi)(\nabla_X^E s),$$ where $c$ is the Clifford action (of $T^*M$ on $E$) and $\nabla^{LC}$ is the Levi-Civita connection on the cotangent bundle.

\subsubsection{Definition}

The diffeological notion uses $\Lambda^1(X)$ for the cotangent bundle, and sections of its dual pseudo-bundle $(\Lambda^1(X))^*$ as smooth vector fields, leading to the following definition.

\begin{defn}
Let $X$ be a diffeological space such that $\Lambda^1(X)$ admits pseudo-metrics, let $g^{\Lambda}$ be a pseudo-metric on $\Lambda^1(X)$, and let $\nabla^{\Lambda}$ be the Levi-Civita connection 
on $(\Lambda^1(X),g^{\Lambda})$. Let $\chi:E\to X$ be a pseudo-bundle of Clifford modules over $\cl(\Lambda^1(X),g^{\Lambda})$ with Clifford action $c$, and let $\nabla^E$ be a diffeological connection 
on it. Then $\nabla^E$ is a \textbf{Clifford connection} if for every $t\in C^{\infty}(X,(\Lambda^1(X))^*)$, for every $s\in C^{\infty}(X,\Lambda^1(X))$, and for every $r\in C^{\infty}(X,E)$ we have
$$\nabla^E_t(c(s)r)=c(\nabla^{\Lambda}_t s)(r)+c(s)(\nabla^E_t r).$$
\end{defn}

\subsubsection{Example}

Let us consider briefly the construction of Section 10.1. Recall that the base space $X$ is the union of the coordinate axes in $\matR^2$, and the two pseudo-bundles $\pi_1:V_1\to X_1$ and 
$\pi_2:V_2\to X_2$ are naturally identified with two copies of the tangent bundle $T\matR^1\to\matR^1$ (endowed with two different pseudo-metrics), so we can also view them as the diffeological cotangent 
bundles $\Lambda^1(X_1)$ and $\Lambda^1(X_2)$. Due to the fact that the gluing is along a single point subspace and the choice of $\tilde{f}$, the pseudo-bundle $V_1\cup_{\tilde{f}}V_2$ coincides then 
with $\Lambda^1(X_1\cup_f X_2)$. The connections that we considered on $V_1$ and $V_2$ are actually the Levi-Civita connections corresponding to the chosen pseudo-metrics, and the resulting 
connection on $V_1\cup_{\tilde{f}}V_2$ is the induced connection $\nabla^{\cup}$. Thus, by Theorem \ref{induced:connection:on:lambda:is:levi-civita:thm} it is the Levi-Civita connection on 
$V_1\cup_{\tilde{f}}V_2\cong\Lambda^1(X_1\cup_f X_2)$. As a matter of standard reasoning (see, for instance, \cite{heat-kernel}, Section 3.6), it yields a Clifford connection on 
$\bigwedge(\Lambda^1(X_1\cup_f X_2))$, seen as a pseudo-bundle of Clifford module over $\cl((\Lambda^1(X_1\cup_f X_2),g)$ with the usual Clifford action.

\subsection{Gluing of Clifford modules over $\Lambda^1(X_1)$ and $\Lambda^1(X_2)$}

We have already considered gluing of Clifford modules in Section 7, with the conclusion that the result of gluing is again a Clifford module over an appropriate Clifford algebra, which itself is the result of 
gluing. This situation does not have an automatic carry-over to the case of Clifford modules over $\Lambda^1(X_1)$ and $\Lambda^1(X_2)$. Indeed, in the latter context we want the result to be a Clifford 
module over $\Lambda^1(X_1\cup_f X_2)$ (more precisely, over $\cl(\Lambda^1(X_1\cup_f X_2),g^{\Lambda})$, where $g^{\Lambda}$ is induced by the pseudo-metrics $g_1^{\Lambda}$ and 
$g_2^{\Lambda}$ on $\Lambda^1(X_1)$ and $\Lambda^1(X_2)$ respectively). Since $\Lambda^1(X_1\cup_f X_2)$ is not the result of any gluing of $\Lambda^1(X_1)$ and $\Lambda^1(X_2)$, we cannot 
use the same strategy as in Section 7; in this section we show that a certain induced action can be obtained, but it is defined differently from the induced action considered in Section 7.

\subsubsection{Notation and approach}

Let $X_1$ and $X_2$ be two diffeological spaces, let $f:X_1\supseteq Y\to X_2$ be a diffeomorphism such that $\calD_1^{\Omega}=\calD_2^{\Omega}$, and suppose that $\Lambda^1(X_1)$ and 
$\Lambda^1(X_2)$ admit compatible pseudo-metrics $g_1^{\Lambda}$ and $g_2^{\Lambda}$. 

Let $\chi_1:E_1\to X_1$ and $\chi_2:E_2\to X_2$ be pseudo-bundles of Clifford modules with Clifford actions $c_1$ and $c_2$ by $\cl(\Lambda^1(X_1),g_1^{\Lambda})$ and 
$\cl(\Lambda^1(X_2),g_2^{\Lambda})$ respectively, and let $\tilde{f}':\chi_1^{-1}(Y)\to\chi_2^{-1}(f(Y))$ be a smooth fibrewise linear map that covers $f$. If $c_1$ and $c_2$ are compatible actions (Section 7) 
then there is a well-defined gluing of these Clifford modules, with the result $\chi_1\cup_{(\tilde{f}',f)}\chi_2:E_1\cup_{\tilde{f}'}E_2\to X_1\cup_f X_2$, with each fibre inheriting the Clifford module structure over 
either $\cl(\Lambda^1(X_1),g_1^{\Lambda})$ or $\cl(\Lambda^1(X_2),g_2^{\Lambda})$. 

These structures endow $E_1\cup_{\tilde{f}'}E_2$ with a certain structure resembling that of a Clifford module. However, it is not a Clifford module over $\cl(\Lambda^1(X_1\cup_f X_2),g^{\Lambda})$; in fact, 
even the action of $\Lambda^1(X_1\cup_f X_2)$ on it is not automatic. This is because $\Lambda^1(X_1\cup_f X_2)$ is not the result of any gluing between $\Lambda^1(X_1)$ and $\Lambda^1(X_2)$, since 
it has fibres that do not coincide with any of the fibres of either $\Lambda^1(X_1)$ and $\Lambda^1(X_2)$ (these are fibres over the domain of gluing). This situation is therefore different from the one 
considered in Section 7, where $E_1\cup_{(\tilde{f}',f)}E_2$, obtained by gluing the Clifford modules $E_1$ and $E_2$ over certain $\cl(V_1,g_1)$ and $\cl(V_2,g_2)$, inherited under certain assumptions the 
Clifford action of the appropriate $\cl(V_1\cup_{(\tilde{f},f)}V_2,g_1\cup_{(f,\tilde{f})}g_2)\cong\cl(V_1,g_1)\cup_{(\tilde{f}^{\cl},f)}\cl(V_2,g_2)$.

Below we consider what natural action $\cl(\Lambda^1(X_1\cup_f X_2),g^{\Lambda})$ might inherit from $\cl(\Lambda^1(X_1),g_1^{\Lambda})$ and $\cl(\Lambda^1(X_2),g_2^{\Lambda})$. The construction 
is based on using the partially defined maps $\tilde{\rho}_1^{\Lambda}$ and $\tilde{\rho}_2^{\Lambda}$, and the universal property of Clifford algebras, which allows to extend the partially defined projections 
$\tilde{\rho}_1^{\Lambda}:\Lambda^1(X_1\cup_f X_2)\supseteq(\pi^{\Lambda})^{-1}(\tilde{i}_1(X_1))\to\Lambda^1(X_1)$ and 
$\tilde{\rho}_2^{\Lambda}:\Lambda^1(X_1\cup_f X_2)\supseteq(\pi^{\Lambda})^{-1}(i_2(X_2))\to\Lambda^1(X_2)$ to the corresponding subsets of $\cl(\Lambda^1(X_1\cup_f X_2),g^{\Lambda})$: 
$$\tilde{\rho}_1^{\cl(\Lambda)}:\cl(\Lambda^1(X_1\cup_f X_2),g^{\Lambda})\supseteq(\pi^{\cl(\Lambda)})^{-1}(\tilde{i}_1(X_1))\to\cl(\Lambda^1(X_1),g_1^{\Lambda}),$$
$$\tilde{\rho}_2^{\cl(\Lambda)}:\cl(\Lambda^1(X_1\cup_f X_2),g^{\Lambda})\supseteq(\pi^{\cl(\Lambda)})^{-1}(i_2(X_2))\to\cl(\Lambda^1(X_2),g_2^{\Lambda}).$$

\begin{prop}
Let $x\in\tilde{i}_1(X_1)\subseteq X_1\cup_f X_2$. Then the map $\tilde{\rho}_1^{\Lambda}$ restricted to the fibre $\Lambda_x^1(X_1\cup_f X_2)$ determines a smooth algebra homomorphism 
$$\tilde{\rho}_1^{\cl(\Lambda),x}:\cl(\Lambda_x^1(X_1\cup_f X_2),g^{\Lambda}(x))\to\cl(\Lambda_{\tilde{i}_1^{-1}(x)}^1(X_1),g_1^{\Lambda}(\tilde{i}_1^{-1}(x))).$$ Likewise, if 
$x\in i_2(X_2)\subseteq X_1\cup_f X_2$ then the restriction of the map $\tilde{\rho}_2^{\Lambda}$ to the fibre $\Lambda_x^1(X_1\cup_f X_2)$ yields a smooth algebra homomorphism 
$$\tilde{\rho}_2^{\cl(\Lambda),x}:\cl(\Lambda_x^1(X_1\cup_f X_2),g^{\Lambda}(x))\to\cl(\Lambda_{i_2^{-1}(x)},g_2^{\Lambda}(i_2^{-1}(x))).$$
\end{prop}

\begin{proof}
This follows from the universal properties of the Clifford algebras; it suffices to observe that the maps $\tilde{\rho}_1^{\Lambda}$ and $\tilde{\rho}_2^{\Lambda}$ are isometries, which follows from the 
construction of $g^{\Lambda}$, and the compatibility of $g_1^{\Lambda}$ and $g_2^{\Lambda}$, the latter ensuring that for any $\alpha\in\Lambda^1(X_1\cup_f X_2)$ such that 
$\pi^{\Lambda}(\alpha)\in i_2(f(Y))$ we have 
$$g^{\Lambda}(\pi^{\Lambda}(\alpha))(\alpha,\alpha)=g_1^{\Lambda}(\tilde{i}_1^{-1}(\pi^{\Lambda}(\alpha)))(\tilde{\rho}_1^{\Lambda}(\alpha),\tilde{\rho}_1^{\Lambda}(\alpha))= 
g_2^{\Lambda}(i_2^{-1}(\pi^{\Lambda}(\alpha)))(\tilde{\rho}_2^{\Lambda}(\alpha),\tilde{\rho}_2^{\Lambda}(\alpha)).$$
\end{proof}

\subsubsection{Compatibility of Clifford actions by $\Lambda^1(X_1)$ and $\Lambda^1(X_2)$}

Let $c_1$ be a smooth action of $\cl(\Lambda^1(X_1),g_1^{\Lambda})$ on $E_1$, and let $c_2$ be a smooth action of $\cl(\Lambda^1(X_2),g_2^{\Lambda})$ on $E_2$. Consider the gluing of the two 
base spaces, $X_1$ and $X_2$, along a given smooth map $f:X_1\supseteq Y\to X_2$ (usually a diffeomorphism and satisfying the extendibility condition $\calD_1^{\Omega}=\calD_2^{\Omega}$, 
although these are not strictly necessary for the definition below). 

\begin{defn}\label{compatible:clifford:actions:of:lambdas:defn}
The actions $c_1$ and $c_2$ are \textbf{compatible} if for all $x\in i_2(f(Y))$, for all $\alpha\in(\pi^{\Lambda})^{-1}(x)$, and for all $e_1\in\chi_1^{-1}(\tilde{i}_1^{-1}(x))$ we have that
$$(c_2(\tilde{\rho}_2^{\Lambda}(\alpha))(i_2^{-1}(x))))(\tilde{f}'(e_1))=\tilde{f}'((c_1(\tilde{\rho}_1^{\Lambda}(\alpha))(\tilde{i}_1^{-1}(x)))(e_1)).$$
\end{defn}

The aim of this notion is to ensure that $E_1\cup_{\tilde{f}'}E_2$ carries a well-defined smooth Clifford action by $\Lambda^1(X_1\cup_f X_2)$, and in particular the middle line of the formula in Definition 
\ref {compatible:clifford:actions:of:dual:lambdas:defn} allows for the action to be smooth across both $\tilde{i}_1(X_1)$ and $i_2(X_2)$ (see the next section for the proof).

\subsubsection{The induced Clifford action of $\cl(\Lambda^1(X_1\cup_f X_2),g^{\Lambda})$ on $E_1\cup_{\tilde{f}'}E_2$}

As we have seen in Section 7, if there is an appropriate gluing of Clifford algebras then $E_1\cup_{\tilde{f}'}E_2$ is naturally a Clifford module over the result of that gluing. However, in the case of 
$\cl(\Lambda^1(X_1),g_1^{\Lambda})$ and $\cl(\Lambda^1(X_2),g_2^{\Lambda})$, the pseudo-bundle of algebras resulting from the gluing is not $\cl(\Lambda^1(X_1\cup_f X_2),g^{\Lambda})$, whereas 
we want $E_1\cup_{\tilde{f}'}E_2$ to be a Clifford module over the latter. In this section we construct the appropriate action.

\begin{defn}\label{induced:clifford:action:of:lambda:glued:defn}
Let $x\in X_1\cup_f X_2$, let $\alpha\in\Lambda_x^1(X_1\cup_f X_2)$, and let $e\in(\chi_1\cup_{(\tilde{f}',f)}\chi_2)^{-1}(x)$. Define the \textbf{induced action} $\tilde{c}$ of $\Lambda^1(X_1\cup_f X_2)$ on 
$E_1\cup_{\tilde{f}'}E_2$ by setting that
$$\tilde{c}(\alpha)(x)(e):=\left\{\begin{array}{cl} 
j_1^{E_1}((c_1(\tilde{\rho}_1^{\Lambda}(\alpha))(i_1^{-1}(x)))((j_1^{E_1})^{-1}(e))) & \mbox{if }x\in i_1(X_1\setminus Y), \\
j_2^{E_2}((c_2(\tilde{\rho}_2^{\Lambda}(\alpha))(i_2^{-1}(x)))((j_2^{E_2})^{-1}(e))) & \mbox{if }x\in i_2(f(Y)), \\
j_2^{E_2}((c_2(\tilde{\rho}_2^{\Lambda}(\alpha))(i_2^{-1}(x)))((j_2^{E_2})^{-1}(e))) & \mbox{if }x\in i_2(X_2\setminus f(Y)).
\end{array}\right.$$
\end{defn}

Observe that the compatibility condition ensures that over $\tilde{i}_1(X_1)$ the action $\tilde{c}$ is equivalent to $c_1$ (while by definition over $i_2(X_2)$ it is equivalent to $c_2$), which allows to show 
that $\tilde{c}$ is smooth; it being a linear action on each fibre is obvious from the construction.

\begin{thm}\label{clifford:action:of:lambda:glued:thm}
Let $X_1$ and $X_2$ be such that $\Lambda^1(X_1)$ and $\Lambda^1(X_2)$ are finite-dimensional, and let $f:X_1\supseteq Y\to X_2$ be a gluing diffeomorphism such that 
$\calD_1^{\Omega}=\calD_2^{\Omega}$. Let $g_1^{\Lambda}$ and $g_2^{\Lambda}$ be compatible pseudo-metrics on $\Lambda^1(X_1)$ and $\Lambda^1(X_2)$ respectively, and let $g^{\Lambda}$ be 
the induced pseudo-metric on $\Lambda^1(X_1\cup_f X_2)$. Let $\chi_1:E_1\to X_1$ be a pseudo-bundle of Clifford modules over $\cl(\Lambda^1(X_1),g_1^{\Lambda})$ with Clifford action $c_1$, let 
$\chi_2:E_2\to X_2$ be a pseudo-bundle of Clifford modules over $\cl(\Lambda^1(X_2),g_2^{\Lambda})$ with Clifford action $c_2$, let $\tilde{f}':\chi_1^{-1}(Y)\to E_2$ be a fibrewise linear diffeomorphism 
that covers $f$, and suppose that $c_1$ and $c_2$ are compatible with the gluing along $(\tilde{f}',f)$. Then the induced action $\tilde{c}$ yields a well-defined smooth Clifford action of 
$\cl(\Lambda^1(X_1\cup_f X_2),g^{\Lambda})$ on $E_1\cup_{\tilde{f}'}E_2$.
\end{thm}

\begin{proof}
The fact that $\tilde{c}$ is well-defined follows from the compatibility of $c_1$ and $c_2$. To show that it is smooth, it essentially suffices to observe that, again by definition of the compatibility of Clifford actions, 
over $\tilde{i}_1(X_1)$ it essentially (up to technicalities of the gluing construction) coincides with $c_1$, and over $i_2(X_2)$ it coincides, in the same sense, with $c_2$. To illustrate this, let 
$q:U'\to\Lambda^1(X_1\cup_f X_2)$ be a plot of $\Lambda^1(X_1\cup_f X_2)$ such that $\mbox{Range}(\pi^{\Lambda}\circ q)\subseteq\tilde{i}_1(X_1)$, and let $p:U\to E_1\cup_{\tilde{f}'}E_2$ be a plot of 
$E_1\cup_{\tilde{f}'}E_2$ such that $\mbox{Range}((\chi_1\cup_{(\tilde{f}',f)}\chi_2)\circ p)\subseteq i_2(X_2)$. We need to check that the evaluation function
$$(u',u)\mapsto\tilde{c}(q(u'))(p(u))$$ defined on the subset $Z_{q,p}:=\{(u',u)\,|\,\pi^{\Lambda}(q(u'))=(\chi_1\cup_{(\tilde{f}',f)}\chi_2)(p(u))\}\subseteq U'\times U$ considered with the subset diffeology is smooth 
as a map into $E_1\cup_{\tilde{f}'}E_2$. This evaluation function has form
$$\tilde{c}(q(u'))(p(u))=\left\{\begin{array}{ll}
j_1^{E_1}(c_1(\tilde{\rho}_1^{\Lambda}(q(u')))((j_1^{E_1})^{-1}(p(u)))) & \mbox{for }u\mbox{ such that }p(u)\in i_1(X_1\setminus Y), \\
(j_2^{E_2}\circ\tilde{f}')(c_1(\tilde{\rho}_1^{\Lambda}(q(u')))((j_2^{E_2}\circ\tilde{f}')^{-1}(p(u)))) & \mbox{for }u\mbox{ such that }p(u)\in\tilde{i}_1(Y)=i_2(f(Y)).
\end{array}\right.$$ 

Let $\tilde{j}_1:E_1\to E_1\cup_{\tilde{f}'}E_2$ be defined by 
$$\tilde{j}_1(e_1)=\left\{\begin{array}{ll} 
j_1^{E_1}(e_1) & \mbox{if }\chi_1(e_1)\in X_1\setminus Y, \\
(j_2^{E_2}\circ\tilde{f}')(e_1) & \mbox{if }\chi_1(e_1)\in Y.
\end{array}\right.$$
Observe that by the definition of gluing diffeology $\tilde{j}_1^{-1}\circ p$ is some plot $p_1$ of $E_1$, and that $\tilde{\rho}_1^{\Lambda}\circ q$ is a plot $q_1$ of $\Lambda^1(X_1)$. Since $c_1$ is a smooth 
action by assumption, the evaluation function $(u',u)\mapsto c_1(q_1(u'))(\tilde{j}_1^{-1}(p(u)))$ is smooth; in particular, its pre-composition with any plot of $Z_{p,q}$ is a plot of $E_1$. Since we have 
$$\tilde{c}(q(u'))(p(u))=\left\{\begin{array}{ll} 
j_1^{E_1}(c_1(q_1(u'))(\tilde{j}_1^{-1}(p(u)))) & \mbox{for }u\mbox{ such that }p(u)\in i_1(X_1\setminus Y), \\
(j_2^{E_2}\circ\tilde{f}')(c_1(q_1(u'))(\tilde{j}_1^{-1}(p(u)))) &\mbox{for }u\mbox{ such that }p(u)\in\tilde{i}_1(Y)=i_2(f(Y))
\end{array}\right.$$ (essentially one of the standard forms of plots of $E_1\cup_{\tilde{f'}}E_2$) we conclude that the evaluation function for $\tilde{c}$ relative to plots $q$ and $p$ is indeed smooth as a map 
into $E_1\cup_{\tilde{f'}}E_2$. The other case (when $q$ and $p$ take values in fibres over $i_2(X_2)$) is treated similarly (it is actually simpler), so we obtain the claim.
\end{proof}

Usually, Clifford actions involved in the constructions of Dirac operators are assumed to be unitary. The diffeological concept of a unitary Clifford action does not really differ from the usual notion and is 
as follows (we state it only for Clifford modules over some $\cl(\Lambda^1(X),g)$). 

\begin{defn}\label{unitary:clifford:action:defn}
Let $X$ be a diffeological space such that $\Lambda^1(X)$ carries a pseudo-metric $g^{\Lambda}$, let $\chi:E\to X$ be a pseudo-bundle of Clifford modules over $\cl(\Lambda^1(X),g^{\Lambda})$ with 
Clifford action $c$, and let $g$ be a pseudo-metric on $E$. The action $c$ is said to be \textbf{unitary} if
$$g(x)(c(\alpha)e_1,c(\alpha)(e_2))=g(x)(e_1,e_2)$$ for all $x\in X$, for all $\alpha\in\Lambda_x^1(X)$ such that $g^{\Lambda}(x)(\alpha,\alpha)=1$, and for all $e_1,e_2\in\chi^{-1}(x)=E_x$.
\end{defn}

It is essentially the consequence of the compatibility notion for pseudo-metrics on $\Lambda^1(X_1)$ and $\Lambda^1(X_2)$ (see Definition \ref{compatible:pseudo-metrics:on:lambda:defn}) that gluing 
together two unitary actions yields a unitary action.

\begin{prop}\label{induced:clifford:action:is:unitary:prop}
Let $\chi_1:E_1\to X_1$ be a pseudo-bundle of Clifford modules over $\cl(\Lambda^1(X_1),g_1^{\Lambda})$ with Clifford action $c_1$, let $\chi_2:E_2\to X_2$ be a pseudo-bundle of Clifford modules over 
$\cl(\Lambda^1(X_2),g_2^{\Lambda})$ with Clifford action $c_2$, let $f:X_1\supseteq Y\to X_2$ be a diffeomorphism such that $\calD_1^{\Omega}=\calD_2^{\Omega}$, and let $\tilde{f}':\chi_1^{-1}(Y)\to E_2$ 
be a fibrewise linear diffeomorphism that covers $f$. Assume that $g_1^{\Lambda}$ and $g_2^{\Lambda}$ are compatible with $f$ as pseudo-metrics on $\Lambda^1(X_1)$ and $\Lambda^1(X_2)$, and that 
$c_1$ and $c_2$ are compatible Clifford actions; let $g^{\Lambda}$ be the induced pseudo-metric on $\Lambda^1(X_1\cup_f X_2)$, and let $\tilde{c}$ be the induced Clifford action of 
$\cl(\Lambda^1(X_1\cup_f X_2),g^{\Lambda})$ on $E_1\cup_{\tilde{f}'}E_2$. Suppose that $E_1$ and $E_2$ are endowed with compatible pseudo-metrics $g_1$ and $g_2$, and let $\tilde{g}$ be the 
induced pseudo-metric on $E_1\cup_{\tilde{f}'}E_2$. If the actions $c_1$ and $c_2$ are unitary then $\tilde{c}$ is a unitary action as well.
\end{prop}

\begin{proof}
It suffices to show that $\alpha\in\Lambda^1(X_1\cup_f X_2)$ is unitary if and only if either both, or one of (since they may not be both defined), $\tilde{\rho}_1^{\Lambda}(\alpha)$, 
$\tilde{\rho}_2^{\Lambda}(\alpha)$ are unitary. Indeed, let $x=\pi^{\Lambda}(\alpha)$. Indeed, let $x=\pi^{\Lambda}(\alpha)$, and let $e_1,e_2\in(\chi_1\cup_{(\tilde{f}',f)}\chi_2)^{-1}(x)$. Then by definition
\begin{flushleft}
$\tilde{g}(x)(\tilde{c}(\alpha)e_1,\tilde{c}(\alpha)(e_2))=$
\end{flushleft}
$$\left\{\begin{array}{ll}
g_1(i_1^{-1}(x))((c_1(\tilde{\rho}_1^{\alpha}(\alpha))(i_1^{-1}(x)))((j_1^{E_1})^{-1}(e_1)),(c_1(\tilde{\rho}_1^{\alpha}(\alpha))(i_1^{-1}(x)))((j_1^{E_1})^{-1}(e_2))) & \mbox{if }x\in i_1(X_1\setminus Y), \\
g_2(i_2^{-1}(x))((c_2(\tilde{\rho}_2^{\alpha}(\alpha))(i_2^{-1}(x)))((j_2^{E_1})^{-1}(e_1)),(c_2(\tilde{\rho}_2^{\alpha}(\alpha))(i_2^{-1}(x)))((j_2^{E_1})^{-1}(e_2))) & \mbox{if }x\in i_2(X_2),
\end{array}\right.$$
$$\tilde{g}(x)(e_1,e_2)=\left\{\begin{array}{ll}
g_1(i_1^{-1}(x))((j_1^{E_1})^{-1}(e_1),(j_1^{E_1})^{-1}(e_2)) & \mbox{if }x\in i_1(X_1\setminus Y), \\
g_2(i_2^{-1}(x))((j_2^{E_1})^{-1}(e_1),(j_2^{E_1})^{-1}(e_2)) & \mbox{if }x\in i_2(X_2).
\end{array}\right.$$ It thus suffices to show that $\tilde{\rho}_1^{\Lambda}$ and $\tilde{\rho}_2^{\Lambda}$ preserve the unitarity of the actions. This follows from the definition of the pseudo-metric 
$g^{\Lambda}$. Indeed, 
$$g^{\Lambda}(x)(\alpha,\alpha)=\left\{\begin{array}{cl} 
g_1^{\Lambda}(i_1^{-1}(x))(\tilde{\rho}_1^{\Lambda}(\alpha),\tilde{\rho}_1^{\Lambda}(\alpha)) & \mbox{if }x\in i_1(X_1\setminus Y), \\
\frac12 g_1^{\Lambda}(\tilde{i}_1^{-1}(x))(\tilde{\rho}_1^{\Lambda}(\alpha),\tilde{\rho}_1^{\Lambda}(\alpha))+
\frac12 g_2^{\Lambda}(i_2^{-1}(x))(\tilde{\rho}_2^{\Lambda}(\alpha),\tilde{\rho}_2^{\Lambda}(\alpha)) &  \mbox{if }x\in i_2(f(Y)), \\
g_2^{\Lambda}(i_2^{-1}(x))(\tilde{\rho}_2^{\Lambda}(\alpha),\tilde{\rho}_2^{\Lambda}(\alpha)) &  \mbox{if }x\in i_2(X_2\setminus f(Y)),
\end{array}\right.$$ from which the claim is obvious.
\end{proof}

\subsubsection{A Clifford connection on $E_1\cup_{\tilde{f}'}E_2$ out of those on $E_1$ and $E_2$}

Let $\nabla^1$ be a Clifford connection on $E_1$, that carries a smooth Clifford action $c_1$ by $\cl(\Lambda^1(X_1),g_1^{\Lambda})$, and let $\nabla^2$ be a Clifford connection on $E_2$, that has a 
smooth Clifford action $c_2$ of $\cl(\Lambda^1(X_2),g_2^{\Lambda})$. Let $\nabla^{\Lambda,1}$ and $\nabla^{\Lambda,2}$ be the Levi-Civita connections on $(\Lambda^1(X_1),g_1^{\Lambda})$ and 
$(\Lambda^1(X_2),g_2^{\Lambda})$ respectively (whose existence is by assumption). Let $f:X_1\supseteq Y\to X_2$ be a diffeomorphism satisfying $\calD_1^{\Omega}=\calD_2^{\Omega}$. We assume that 
$c_1$ and $c_2$ are compatible in the sense of Definition \ref{compatible:clifford:actions:of:lambdas:defn}, $g_1^{\Lambda}$ and $g_2^{\Lambda}$ are compatible in the sense of Definition 
\ref{compatible:pseudo-metrics:on:lambda:defn}, $\nabla^1$ and $\nabla^2$ are compatible in the sense of Definition \ref{compatibility:of:connections:defn}, and $\nabla^{\Lambda,1}$ and 
$\nabla^{\Lambda,2}$ are compatible in the sense of Definition \ref{compatibility:of:connections:on:lambda:defn}.

\begin{thm}\label{induced:connection:is:clifford:thm}
Let $\nabla^{\cup}$ be the induced connection on $E_1\cup_{\tilde{f}'}E_2$, and let $\tilde{c}$ be the induced Clifford action of $\cl(\Lambda^1(X_1\cup_f X_2),g^{\Lambda})$ on it. Then $\nabla^{\cup}$ is 
a Clifford connection.
\end{thm}

\begin{proof}
The identity to verify is
$$\nabla^{\cup}_t(\tilde{c}(s)r)=\tilde{c}(\nabla^{\Lambda}_t s)(r)+\tilde{c}(s)(\nabla^{\cup}_t r).$$ Let $r\in C^{\infty}(X_1\cup_f X_2,E_1\cup_{\tilde{f}'}E_2)$, and let 
$s,t\in C^{\infty}(X_1\cup_f X_2,\Lambda^1(X_1\cup_f X_2))$. Define 
$$r_1:=\tilde{j}_1^{-1}\circ r\circ\tilde{i}_1\in C^{\infty}(X_1,E_1),\,\,\,r_2:=j_2^{E_2}\circ r\circ i_2,$$ and recall the sections $s_1,s_2,t_1,t_2$ associated to $s$ and $t$ via
$$s_1=\tilde{\rho}_1^{\Lambda}\circ s\circ\tilde{i}_1,\,\,s_2=\tilde{\rho}_2^{\Lambda}\circ s\circ i_2,\,\,t_1=\tilde{\rho}_1^{\Lambda}\circ t\circ\tilde{i}_1,\,\,t_2=\tilde{\rho}_2^{\Lambda}\circ t\circ i_2.$$ Let us check 
the desired equality at an arbitrary point $x\in X_1\cup_f X_2$.

Let first $x\in i_1(X_1\setminus Y)$. Then 
$$(\nabla^{\cup}_t(\tilde{c}(s)r))(x)=j_1^{E_1}((\nabla^1_{t_1}(\tilde{c}(s)r)_1)(i_1^{-1}(x))).$$ Observe that $(\tilde{c}(s)r)_1=c_1(s_1)r_1$ by construction, so we actually have 
$$(\nabla^{\cup}_t(\tilde{c}(s)r))(x)=j_1^{E_1}((\nabla^1_{t_1}c_1(s_1)r_1)(i_1^{-1}(x))).$$ On the right-hand side we have 
$$(\tilde{c}(\nabla^{\Lambda}_t s)(r)+\tilde{c}(s)(\nabla^{\cup}_t r))(x)=j_1^{E_1}\left(c_1(\tilde{\rho}_1^{\Lambda}(\nabla^{\Lambda}_t s))(r_1)(i_1^{-1}(x))+c_1(s_1)(\nabla^1_{t_1}r_1)(i_1^{-1}(x))\right)=$$
$$=j_1^{E_1}\left((c_1(\nabla^{\Lambda,1}_{t_1}s_1))(r_1)(i_1^{-1}(x))+c_1(s_1)(\nabla^1_{t_1}r_1)(i_1^{-1}(x))\right)=j_1^{E_1}\left(\nabla^1_{t_1}(c_1(s_1)r_1)(i_1^{-1}(x))\right),$$ since $\nabla^1$ is a 
Clifford connection by assumption. This yields the desired equality for $x\in i_1(X_1\setminus Y)$, and the case of $x\in i_2(X_2\setminus f(Y))$ is completely analogous. 

Thus, let $x\in i_2(f(Y))$. Let us write, first of all,
$$(\tilde{c}(s)r)(x)=\frac12\tilde{j}_1(c_1(s_1)r_1)(\tilde{i}_1^{-1}(x))+\frac12 j_2^{E_2}(c_2(s_2)r_2)(i_2^{-1}(x)),$$ which we can do by compatibility of the actions $c_1$ and $c_2$, the definition of $\tilde{c}$, 
and the constructions of $s_1,s_2,r_1,r_2$. Then we have
$$(\nabla^{\cup}_t(\tilde{c}(s)r))(x)=\tilde{j}_1\left((\nabla^1_{t_1}(c_1(s_1)r_1))(\tilde{i}_1^{-1}(x))\right)+j_2^{E_2}\left((\nabla^2_{t_2}(c_2(s_2)r_2))(i_2^{-1}(x))\right).$$  
On the right-hand side we 
have
\begin{flushleft}
$(\tilde{c}(\nabla^{\Lambda}_t s)(r)+\tilde{c}(s)(\nabla^{\cup}_t r))(x)=$
\end{flushleft}
\begin{flushright}
$\frac12(c_1(\nabla^{\Lambda,1}_{t_1}s_1))(r_1)(\tilde{i}_1^{-1}(x))+\frac12(c_2(\nabla^{\Lambda,2}_{t_2}s_2))(i_2^{-1}(x))+\frac12 c_1(s_1)(\nabla^1_{t_1}r_1)(\tilde{i}_1^{-1}(x))+
\frac12 c_2(s_2)(\nabla^2_{t_2}r_2)(i_2^{-1}(x)).$ 
\end{flushright}
The desired equality follows from the assumption of $\nabla^1$ and $\nabla^2$ being Clifford 
connections on $E_1$ and $E_2$ respectively.
\end{proof}

\subsection{The induced Clifford connection on $\bigwedge(\Lambda^1(X_1\cup_f X_2))$}

The pseudo-bundles $\bigwedge(\Lambda^1(X_1))$ and $\bigwedge(\Lambda^1(X_2))$ are specific instances of Clifford modules, over the Clifford algebras $\cl(\Lambda^1(X_1),g_1^{\Lambda})$ and 
$\cl(\Lambda^1(X_2),g_2^{\Lambda})$ respectively. As has been said already, $\bigwedge(\Lambda^1(X_1\cup_f X_2))$ is not the result of any gluing between them.\footnote{Although, as we have seen in
Section 7,the result of a gluing of $\bigwedge(\Lambda^1(X_1))$ to $\bigwedge(\Lambda^1(X_2))$ may coincide with $\bigwedge(\Lambda^1(X_1)\cup\Lambda^1(X_2))$, for some gluing between 
$\Lambda^1(X_1)$ and $\Lambda^1(X_2)$.} 

As a matter of standard reasoning, a connection on $\Lambda^1(X)$ provides us with a connection on $\bigwedge(\Lambda^1(X))$, and in particular, the Levi-Civita connection on 
$(\Lambda^1(X),g^{\Lambda})$ yields a Clifford connection on $\bigwedge(\Lambda^1(X))$, where the latter is considered as a Clifford module over $\cl(\Lambda^1(X),g^{\Lambda})$ with the standard 
Clifford action. Thus, if we assume that $(\Lambda^1(X_1),g_1^{\Lambda})$ and $(\Lambda^1(X_2),g_2^{\Lambda})$ admit compatible Levi-Civita connections then by Theorem 
\ref{induced:connection:on:lambda:is:levi-civita:thm}, they induce the Levi-Civita connection on $(\Lambda^1(X_1\cup_f X_2),g^{\Lambda})$, and this allows to endow $\bigwedge(\Lambda^1(X_1\cup_f X_2))$, 
seen as a Clifford module over $\cl(\Lambda^1(X_1\cup_f X_2),g^{\Lambda})$ with the standard Clifford action, with the corresponding Clifford connection.

\section{Diffeological Dirac operators}

In this concluding section we put together the standard definition of the Dirac operator and the above-described diffeological counterparts of its building blocks. The result fully mimics the standard notion and 
is well-behaved with respect to the gluing procedure, for which we comment on how it applies to Dirac operators on wedges of standard smooth manifolds.

\subsection{The definition and the gluing procedure}

We first spell out the abstract definition, although it is in almost complete analogy with the standard one (as we cited it in the Introduction), consider very simple examples, and point out that there is an almost 
trivial procedure of gluing of two Dirac operators (obtaining again a Dirac operator), as long as these are associated to all the compatible data.

\subsubsection{The definition}

This is the same as the standard definition (a version, more precisely); the only difference is that the diffeological versions of all the components are used.

\begin{defn}
Let $X$ be a diffeological space such that $\Lambda^1(X)$ is finite-dimensional and admits pseudo-metrics, let $g^{\Lambda}$ be a pseudo-metric on $\Lambda^1(X)$, and let $\chi:E\to X$ be a 
pseudo-bundle of Clifford modules over $\cl(\Lambda^1(X),g^{\Lambda})$ with Clifford action $c$. Suppose, furthermore, that $E$ admits a pseudo-metric $g$ and a Clifford connection $\nabla$ compatible 
with $g$. The operator 
$$D:C^{\infty}(X,E)\to C^{\infty}(X,E)\,\,\mbox{ given by }\,\, D=c\circ\nabla$$
is the \textbf{Dirac operator} on $E$ corresponding to the data $(X,g^{\Lambda},E,c,\nabla)$.
\end{defn} 

In the standard context it is also required that the Clifford action be unitary with respect to the given Riemannian metric on the given bundle of Clifford modules. For us, that would mean that $c$ should be an 
unitary action with respect to the pseudo-metrics $g^{\Lambda}$ on $\Lambda^1(X)$ and $g$ on $E$, see Definition \ref{unitary:clifford:action:defn}.

\subsubsection{Gluing of compatible Dirac operators}

This is akin to most of our gluing constructions (since it collects them all). The idea of the construction should be obvious by now. 

\paragraph{The assumptions} Let $X_1$ and $X_2$ be two diffeological spaces such that $\Lambda^1(X_1)$ and $\Lambda^1(X_2)$ are finite-dimensional. Let $f:X_1\supseteq Y\to X_2$ be a gluing map, 
that we need to assume to be a diffeomorphism and such that $\calD_1^{\Omega}=\calD_2^{\Omega}$ (the reason why we need, as opposed to want, these assumptions is that some of our constructions, such 
as $g^{\Lambda}$ and $\nabla^{\cup}$, were only defined in their presence). 

Let $g_1^{\Lambda}$ and $g_2^{\Lambda}$ be compatible (see Definition \ref{compatible:pseudo-metrics:on:lambda:defn}) pseudo-metrics on $\Lambda^1(X_1)$ and $\Lambda^1(X_2)$ respectively (thus 
assuming that each of them admits a pseudo-metric in the first place and, furthermore, that there exists at least one pair of compatible pseudo-metrics on them). Assume also that 
$(\Lambda^1(X_1),g_1^{\Lambda})$ and $(\Lambda^1(X_2),g_2^{\Lambda})$ both admit Levi-Civita connections, and that these Levi-Civita connections are compatible with each other in the sense of 
Definition \ref{compatibility:of:connections:on:lambda:defn}. By Theorem \ref{induced:connection:on:lambda:is:levi-civita:thm} they induce the Levi-Civita connection on 
$(\Lambda^1(X_1\cup_f X_2),g^{\Lambda})$.

Let now $\chi_1:E_1\to X_1$ be a pseudo-bundle of Clifford modules over $\cl(\Lambda^1(X_1),g_1^{\Lambda})$ with Clifford action $c_1$, and let $\chi_2:E_2\to X_2$ be, likewise, a pseudo-bundle of 
Clifford modules over $\cl(\Lambda^1(X_2),g_2^{\Lambda})$. Let $\tilde{f}':E_1\supseteq\chi_1^{-1}(Y)\to E_2$ be a fibrewise linear diffeomorphism that covers $f$. Assume that $c_1$ and $c_2$ are 
compatible with the gluing along $(\tilde{f}',f)$, in the sense of Definition \ref{compatible:clifford:actions:of:lambdas:defn}, and let $\tilde{c}$ be the induced Clifford action of $\Lambda^1(X_1\cup_f X_2)$ on 
$E_1\cup_{\tilde{f}'}E_2$, see Definition \ref{induced:clifford:action:of:lambda:glued:defn} and Theorem \ref{clifford:action:of:lambda:glued:thm}.
 
Let $g_1$ and $g_2$ be pseudo-metrics on $E_1$ and $E_2$ respectively, and assume that they are compatible with the gluing along $(\tilde{f}',f)$ in the sense of Definition 
\ref{compatible:pseudo-metrics:on:two:pseudo-bundles:defn}. Let $\tilde{g}$ be the induced pseudo-metric on $E_1\cup_{\tilde{f}'}E_2$, see Theorem \ref{glued:pseudometric:commutative:thm} and 
Theorem \ref{glued:pseudometric:noncommutative:thm}. Assume that $c_1$ and $c_2$ are unitary actions. Then by Proposition \ref{induced:clifford:action:is:unitary:prop} the induced action $\tilde{c}$ is a 
unitary action as well.

Let $\nabla^1$ be a Clifford connection on $E_1$, compatible with $g_1$, and let $\nabla^2$  be a Clifford connection on $E_2$, compatible with $g_2$. Assume furthermore that $\nabla^1$ and $\nabla^2$ 
are compatible with each other in the sense of Definition \ref{compatibility:of:connections:defn}, and let $\nabla^{\cup}$ be the induced connection on $E_1\cup_{\tilde{f}'}E_2$. By Theorem 
\ref{induced:connection:is:compatible:with:induced:pseudo-metric:thm}, $\nabla^{\cup}$ is compatible with the pseudo-metric $\tilde{g}$, and by Theorem \ref{induced:connection:is:clifford:thm} it is a 
Clifford connection on $E_1\cup_{\tilde{f}'}E_2$, considered as a pseudo-bundle of Clifford modules over $\cl(\Lambda^1(X_1\cup_f X_2),g^{\Lambda})$ with Clifford action $\tilde{c}$.

\paragraph{The Dirac operator obtained by gluing} The following three $5$-tuples
$$(X_1,g_1^{\Lambda},E_1,c_1,\nabla^1),\,\,\,(X_2,g_2^{\Lambda},E_2,c_2,\nabla^2),\,\,\,(X_1\cup_f X_2,g^{\Lambda},E_1\cup_{\tilde{f}'}E_2,\tilde{c},\nabla^{\cup})$$ provide each the data necessary to 
define a Dirac operator. Let 
$$D_1:C^{\infty}(X_1,E_1)\to C^{\infty}(X_1,E_1),\,D_1=c_1\circ\nabla^1\,\,\,\mbox{ and }\,\,\,D_2:C^{\infty}(X_2,E_2)\to C^{\infty}(X_2,E_2),\,D_2=c_2\circ\nabla^2$$ be the two given Dirac operators, 
\emph{i.e.}, corresponding to the first two tuples.

\begin{defn}
The Dirac operator 
$$\tilde{D}:=\tilde{c}\circ\nabla^{\cup}:C^{\infty}(X_1\cup_f X_2,E_1\cup_{\tilde{f}'}E_2)\to C^{\infty}(X_1\cup_f X_2,E_1\cup_{\tilde{f}'}E_2)$$ is said to be the result of \textbf{gluing of Dirac operators} $D_1$ 
and $D_2$. 
\end{defn}

The action of $\tilde{D}$ can easily be described in terms of the actions of $D_1$ and $D_2$. 

\begin{prop}\label{splitting:dirac:operator:prop}
Let $s\in C^{\infty}(X_1\cup_f X_2,E_1\cup_{\tilde{f}'}E_2)$, and let $s=s_1\cup_{(f,\tilde{f}')}s_2$ be its splitting as the result of gluing of compatible sections $s_1\in C^{\infty}(X_1,E_1)$ and 
$s_2\in C^{\infty}(X_2,E_2)$. Then
$$\tilde{D}(s)=D_1(s_1)\cup_{(f,\tilde{f}')}D_2(s_2).$$
\end{prop}

\begin{proof}
Let us compare $\tilde{D}(s)(x)$ and $(D_1(s_1)\cup_{(f,\tilde{f})}D_2(s_2))(x)$ for $x\in X_1\cup_f X_2$. For $x\in i_1(X_1\setminus Y)$ and $x\in i_2(X_2\setminus f(Y))$ the equality between the two is 
obvious, so let $x\in i_2(f(Y))$. By definition of the gluing construction (for diffeological spaces and maps between them), we need to compare $(\tilde{c}\circ\nabla^{\cup})(s)(x)$ with 
$D_2(s_2)(i_2^{-1}(x))=(c_2\circ\nabla^2)(s_2)(i_2^{-1}(x))$. 

Let us therefore calculate $(\tilde{c}\circ\nabla^{\cup})(s)(x)$. We have 
$$(\tilde{c}\circ\nabla^{\cup})(s)(x)=(c_2\circ(\tilde{\rho}_2^{\Lambda}\otimes(j_2^{E_2})^{-1})(\nabla^{\cup}s)(i_2^{-1}(x))=c_2(\nabla^2s_2)(i_2^{-1}(x)),$$ respectively by definition of the action $\tilde{c}$ and 
by the construction of the connection $\nabla^{\cup}$, whence the claim.
\end{proof}

\begin{rem}
The formula in Proposition \ref{splitting:dirac:operator:prop} could be taken as a definition of gluing of Dirac operators. Specifically, say that Dirac operators $D_1$ and $D_2$ on $E_1$ and $E_2$ are 
\textbf{compatible} if for every compatible pair of sections $s_1\in C^{\infty}(X_1,E_1)$, $s_2\in C^{\infty}(X_2,E_2)$ the sections $D_1(s_1)$ and $D_2(s_2)$ are again compatible. For any two compatible 
Dirac operators $D_1$ and $D_2$ define the result of their gluing to be the operator 
$$D_1\cup_{(f,\tilde{f}')}D_2:C^{\infty}(X_1\cup_f X_2,E_1\cup_{\tilde{f}'}E_2)\to C^{\infty}(X_1\cup_f X_2,E_1\cup_{\tilde{f}'}E_2)$$ acting by 
$$(D_1\cup_{(f,\tilde{f}')}D_2)(s)=D_1(s_1)\cup_{(f,\tilde{f}')}D_2(s_2)$$ for every section $s\in C^{\infty}(X_1\cup_f X_2,E_1\cup_{\tilde{f}'}E_2)$ written as $s=s_1\cup_{(f,\tilde{f}')}s_2$ (recall that, since $f$ 
and $\tilde{f}'$ are diffeomorphisms, this presentation of $s$ is unique). The operator is well-defined by compatibility of $D_1$ and $D_2$ and is again a Dirac operator, since by construction it coincides with 
$\tilde{D}$.
\end{rem}

\subsection{One-point wedges of smooth manifolds}

This is probably the simplest illustration of the gluing procedure for Dirac operators. Let $(M_1,g_1,E_1,c_1,\nabla^1)$ and $(M_2,g_2,E_2,c_2,\nabla^2)$ be two standard sets of data defining usual Dirac 
operators $D_1$ and $D_2$. Let $x_1\in M_1$ and $x_2\in M_2$ be two points, and let $f:\{x_1\}\to\{x_2\}$ be the obvious map. Then $M_1\cup_f M_2$ is an instance of gluing of diffeological spaces.

Assume now that the fibres over $x_1$ and $x_2$ are isomorphic, and choose an isomorphism $\tilde{f}'$ of these fibres, $\tilde{f}':(E_1)_{x_1}\to(E_2)_{x_2}$ (it is of course smooth, since it is just a 
diffeomorphism of standard vector spaces). To apply the gluing construction, we need the following compatibility conditions:
\begin{enumerate}
\item The map $\tilde{f}'$ should preserve the scalar products $g_1(x_1)$ and $g_2(x_2)$ (this corresponds to the compatibility condition for pseudo-metrics); and 
\item The actions $c_1$ and $c_2$ should be equivariant with respect to $\tilde{f}'$, that is, they should satisfy 
$$c_1|_{T_{x_1}^*M_1}(\alpha_1)((\tilde{f}')^*(\cdot))=c_2|_{T_{x_2}^*M_2}(\alpha_2)(\cdot)$$ for all $\alpha_1\in T_{x_1}^*M_1$ and $\alpha_2\in T_{x_2}^*M_2$ (this is the compatibility condition for the 
actions).
\end{enumerate}
Notice that the compatibility condition for the connections is empty. This is because it is based on the compatibility conditions for elements of $\Lambda^1(X_1)$ and $\Lambda^1(X_2)$, which is always 
empty in the case of a one-point gluing.

Obviously, outside of the wedge point $D_1\cup_{(f,\tilde{f'})}D_2$ acts either as $D_1$ or as $D_2$, whichever is appropriate. Let $s\in C^{\infty}(M_1\cup_f M_2,E_1\cup_{\tilde{f}'}E_2)$. The value 
$(D_1\cup_{(f,\tilde{f'})}D_2)(s)(x)$ of the image $(D_1\cup_{(f,\tilde{f'})}D_2)(s)$ at the wedge point $x$ has then form 
$$\sum(\omega_i^1(x_1)\oplus\omega_j^2(x_2))\otimes(\tilde{f}'(e_i^1)+e_j^2),$$ where $D_1(s_1)(x_1)=\sum\omega_i^1(x_1)\otimes e_i^1$ and $D_2(s_2)(x_2)=\sum\omega_j^2(x_2)\otimes e_j^2$.

\subsection{Concluding remarks}

A vast amount of topics has been omitted from this manuscript, including everything that has to do with Dirac operators and the Atiyah-Singer theory \emph{proper}, and most of the constructions developed 
herein come with strong limitations, such as restricting ourselves, particularly from Section 8 onwards, to gluings along diffeomorphisms and extendable differential forms. As far as the omissions are 
concerned, they have to do with finding a reasonable limit for the scope of this work. 

These omissions, in any case, are of two sorts. One concerns the notions that go most closely together with the constructions considered here, such as the curvature tensor of a diffeological connection and 
characteristic classes of diffeological vector pseudo-bundles endowed with connections. These were omitted mostly because they are not strictly necessary for the final purpose, and also for reasons of length. 
Another noticeable absence is that of diffeological counterparts of the standard instances of Dirac operators, such as the de Rham operator (although it is not impossible to have one). That would be nice to 
have, and it will probably get done in the future.

The limitations are a somewhat different matter, but in the end it was also a conscious choice to avoid reaching statements-in-maximal-generality all throughout (this is not to imply that I was able to obtain all 
the potential maximal generality statements; sometimes I wasn't). The reasoning was that, for the first approach to this subject matter (which is, after all, is not much more than just a systematic way of piecing 
together the usual Dirac operators and explaining in which sense, consistent with the existing theory, the result is again a Dirac operator) it appears, as a matter of opinion to limit the discussion to bluings 
along diffeological diffeomorphisms that are defined, although not just on usual open domains, on sets that are sufficiently well-behaved to ensure our extendability conditions.

\section*{Appendix: open questions}

There are some obvious open questions that got raised during this work. Here is a rather incomplete list of them; it is distinct from the list of omissions appearing above, and has little intersection with the list 
of limitations, also see above. Unlike the former two, where some objective considerations (the length and so on) may justify the absence, in this manuscript, of the answer, the list appearing just below is 
compiled on the basis of, it would be nice to include an answer here, but I don't know it.

\paragraph{Existence questions} We have mostly avoided dealing with existence issues throughout this manuscript; when existence was in doubt, we dealt with the matter by just asking for whatever required 
object as a matter of assumption. Two factors are behind this. One is the inherent breadth of the concept of a diffeological space; since this can be pretty much anything at all, \emph{some} assumptions should 
always be needed to ensure the existence of such-and-such object, and it might just happen that a list of such assumption is better replaced by a plain requirement for the object to exist. The other reason is 
more of a technical matter and applies to more concrete objects, such as pseudo-metrics or connections: the way in which the existence of their standard counterparts is checked is based on local coordinates 
of sort. These do not really exist in diffeology, so the standard procedure does not immediately apply. This is certainly not an insurmountable obstacle (it might be just a matter of using a different approach), so 
such existence questions are collected here, starting with the most technical one:
\begin{itemize}
\item is there a counterpart of the partition of unity theorem for diffeological spaces?
\item which diffeological vector pseudo-bundles admit pseudo-metrics? (My original hope in considering the diffeological gluing was to use it as a substitute of local trivializations, first of all, with the aim of 
obtaining pseudo-metrics by gluing them. This can certainly be done, but the range of pseudo-bundles that result from a finite sequence of bluings is rather limited, see below).
\item which diffeological vector pseudo-bundles admit diffeological connections?
\item which diffeological spaces (and for which pseudo-metrics) admit Levi-Civita connections? 
\item which diffeological vector pseudo-bundles admit Clifford connections, and for which pseudo-metrics? 
\end{itemize}

\paragraph{Extendibility questions} These regard specifically the spaces of sections of pseudo-bundles and the spaces $\Omega^1(X)$ of differential forms on diffeological spaces, and the behavior of these 
under gluing, such as, how frequently the condition $\calD_1^{\Omega}=\calD_2^{\Omega}$ is satisfied.

\paragraph{Interplay between various compatibility notions} A wealth of compatibility notions appears throughout the paper (too many, actually, for having a common name). These may, or may not, be 
unrelated to each other. Here are some specific questions in this respect:
\begin{itemize}
\item are the Levi-Civita connections corresponding to compatible pseudo-metrics themselves compatible (as connections, of course)?
\item are Clifford connections on Clifford modules endowed with compatible actions themselves compatible?
\end{itemize}

\paragraph{Strengthening the gluing diffeology} As we pointed out several times throughout, the gluing diffeology is a very weak diffeology, weaker in any case than one would expect in a given setting. The 
idea of it is thus of being a precursor to other, stronger diffeologies.

\paragraph{D-topology and gluing} This matter was not considered at all here, but it is a very natural one. There are some obvious questions such as whether the image $i_2(X_2)$ is a D-open subset of 
$X_1\cup_f X_2$ (in general it is not).

\paragraph{Locality and dimensions} It would have been suitable to at least mention local dimensions of the spaces $C^{\infty}(X,V)$ of smooth sections of pseudo-bundles in Section 7; we avoided doing this 
for reasons of space.

\vspace{5mm}

\noindent University of Pisa \\
Department of Mathematics \\
Via F. Buonarroti 1C\\
56127 PISA -- Italy\\
\ \\
ekaterina.pervova@unipi.it\\

\end{document}